\documentclass[10pt]{article}
\usepackage[english]{babel} % language
\usepackage[margin=1.5in]{geometry} % geometry of page
\usepackage{amsmath, amsfonts, amssymb, amsthm} % math
\usepackage[hidelinks,linktocpage=false]{hyperref} % cross referencing
\usepackage[scr=boondoxo,cal=euler]{mathalfa} % mathscr and mathcal fonts
\usepackage{tikz-cd} % commutative diagrams
\usepackage{pxfonts} % font

\usepackage{indentfirst} % indent first paragraph
\usepackage{fancyhdr} % header and footer
\usepackage{graphicx} % for includegraphics
\usepackage{enumitem} % better enumerate, itemize
\usepackage{microtype} % better kerning
\usepackage{pdfpages} % insert pdf documents
\usepackage{centernot} % slash through symbols
\usepackage{ragged2e} % for \RightRagged 
\usepackage{mathabx}

\allowdisplaybreaks % allow multiple page display math

\usetikzlibrary{decorations.pathmorphing}

\theoremstyle{plain}
\newtheorem{thm}{Theorem}[section]
\newtheorem*{thm*}{Theorem}
\newtheorem{lem}[thm]{Lemma}
\newtheorem{cor}[thm]{Corollary}
\newtheorem{prop}[thm]{Proposition}
\newtheorem*{prop*}{Proposition}

\theoremstyle{definition}

\newtheorem*{df*}{Definition}
\newtheorem*{dfs*}{Definitions}

\newtheorem*{exercise*}{Exercise}

\theoremstyle{remark}
\newtheorem{rem}[thm]{Remark}
\newtheorem*{rem*}{Remark}

\newtheorem*{example}{Example}
\newtheorem*{examples}{Examples}

\newcommand{\fk}[1]{\mathfrak{#1}}

\newcommand{\sr}[1]{\mathscr{#1}}
\newcommand{\bo}[1]{\boldsymbol{#1}} % bold math

\newcommand{\Z}{\mathbf{Z}} % integers
\newcommand{\Q}{\mathbf{Q}} % rationals
\newcommand{\R}{\mathbf{R}} % reals
\newcommand{\C}{\mathbf{C}} % complex numbers
\newcommand{\F}{\mathbf{F}}
\newcommand{\x}{\times}

\newcommand{\pt}{\mathrm{pt}}

\newcommand{\CFKhat}{\smash{\widehat{\mathrm{CFK}}}}
\newcommand{\HFKhat}{\smash{\widehat{\mathrm{HFK}}}}

\newcommand{\HFhat}{\smash{\widehat{\mathrm{HF}}}}
\newcommand{\HFLhat}{\smash{\widehat{\mathrm{HFL}}}}

\DeclareMathOperator{\Isharp}{I^\sharp}
\DeclareMathOperator{\Inat}{I^\natural}

\DeclareMathOperator{\SHI}{SHI}
\DeclareMathOperator{\KHI}{KHI}

\DeclareMathOperator{\rank}{rank}

\DeclareMathOperator{\Int}{Int}

\DeclareMathOperator{\Hom}{Hom}

\DeclareMathOperator{\Id}{Id}

\DeclareMathOperator{\U}{U}
\DeclareMathOperator{\SO}{SO}
\DeclareMathOperator{\Spin}{Spin}

\DeclareMathOperator{\HFK}{HFK}

\DeclareMathOperator{\HFL}{HFL}
\DeclareMathOperator{\CFL}{CFL}

\DeclareMathOperator{\Kh}{Kh}
\DeclareMathOperator{\Khr}{Khr}
\DeclareMathOperator{\KhC}{KhC}
\DeclareMathOperator{\KhrC}{KhrC}

\DeclareMathOperator{\SFH}{SFH}
\DeclareMathOperator{\SFC}{SFC}

\DeclareMathOperator{\Sym}{Sym}

\DeclareMathOperator{\Eig}{Eig}
\usepackage{caption,pinlabel}
\usepackage[numbers]{natbib}
\title{The cosmetic crossing conjecture for split links}
\author{Joshua Wang}
\date{}
\begin{document}
\maketitle
\begin{abstract}
	Given a band sum of a split two-component link along a nontrivial band, we obtain a family of knots indexed by the integers by adding any number of full twists to the band. We show that the knots in this family have the same Heegaard knot Floer homology and the same instanton knot Floer homology. In contrast, a generalization of the cosmetic crossing conjecture predicts that the knots in this family are all distinct. We verify this prediction by showing that any two knots in this family have distinct Khovanov homology. Along the way, we prove that each of the three knot homologies detects the trivial band.
\end{abstract}
\tableofcontents

%%%%%%%%%%%%%%%%%%%%%%%%%%%%%%%%%%%%%%%
%%%%%%%%%% Introduction %%%%%%%%%%%%%%%
%%%%%%%%%%%%%%%%%%%%%%%%%%%%%%%%%%%%%%%

\newpage
\section{Introduction}

The cosmetic crossing conjecture asserts that every crossing change of an oriented knot which does not change the oriented knot type is \textit{nugatory} (Figure~\ref{fig:nugatory}). 

\begin{figure}[!ht]
	\centering
	\includegraphics[width=.4\textwidth]{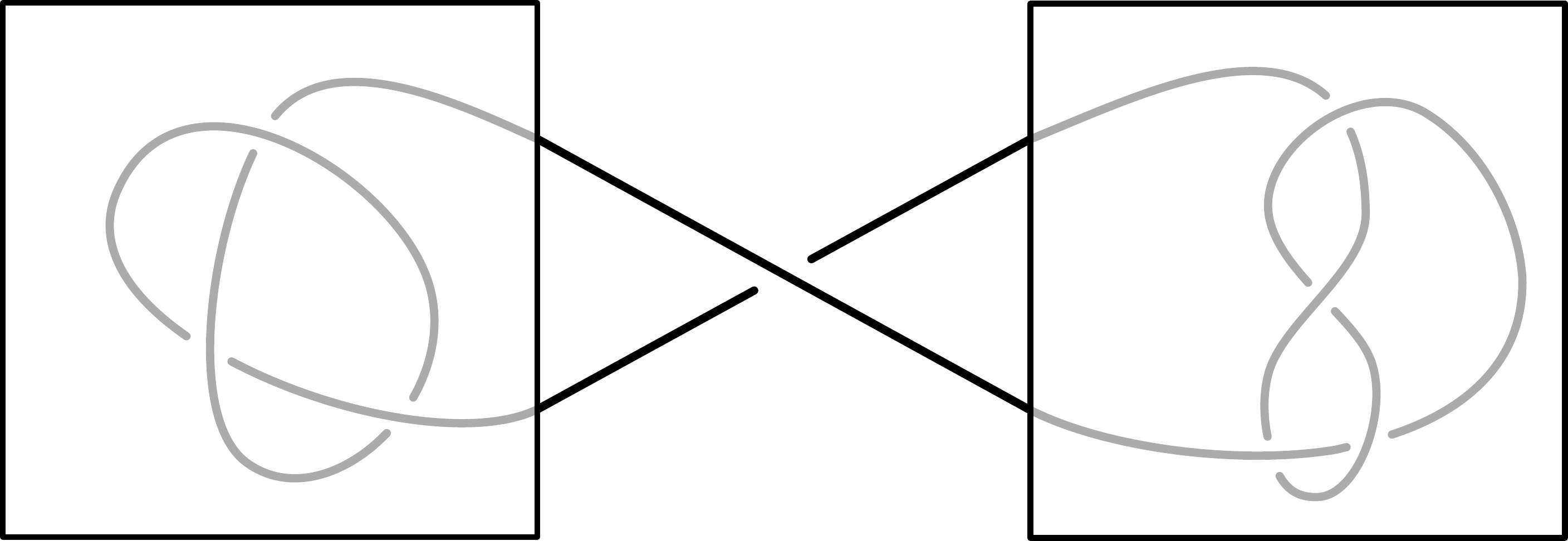}\\
	\vspace{5pt}
	\includegraphics[width=.4\textwidth]{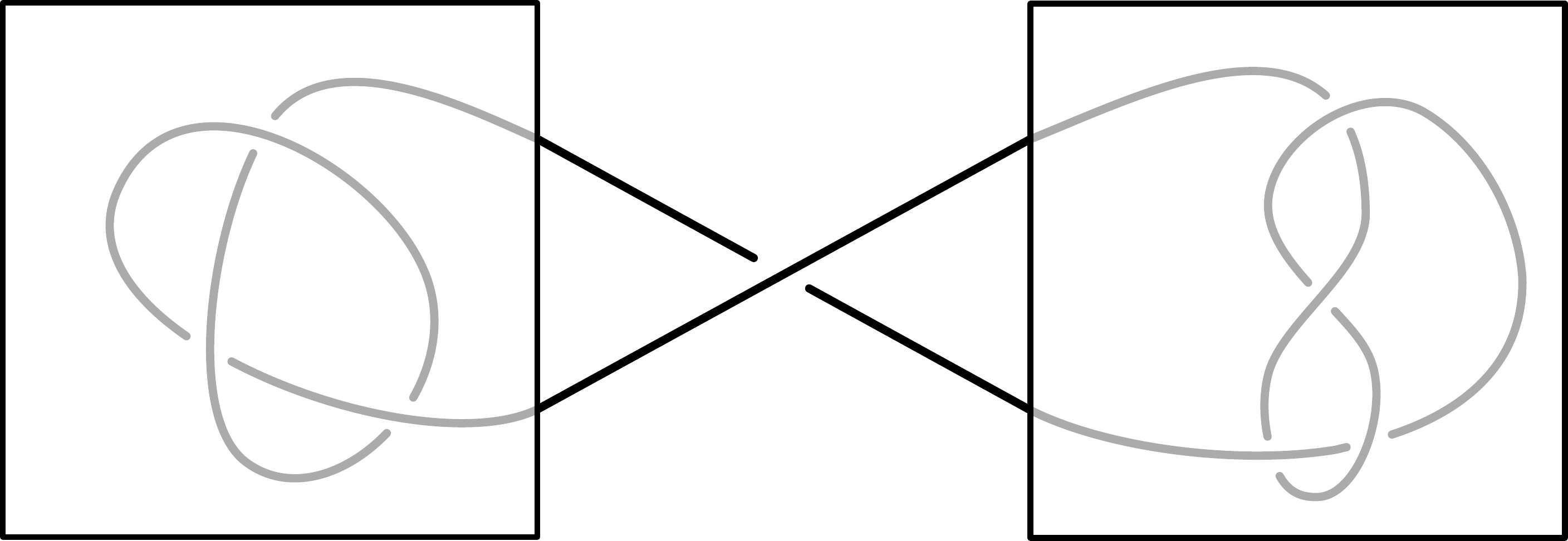}
	\caption{A nugatory crossing change.}
	\label{fig:nugatory}
\end{figure}

\noindent More precisely, suppose $K_-$ and $K_+$ are oriented knots in $S^3$ which agree outside of a small ball inside of which $K_-$ is a negative crossing and $K_+$ is a positive crossing. Let $L$ be the two-component link obtained by taking the oriented resolution of this crossing. Observe that $K_-$ can be obtained from $L$ by band surgery along a band $b$ in the small ball (Figure~\ref{fig:bandb}).

\begin{figure}[!ht]
	\centering
	\includegraphics[width=.13\textwidth]{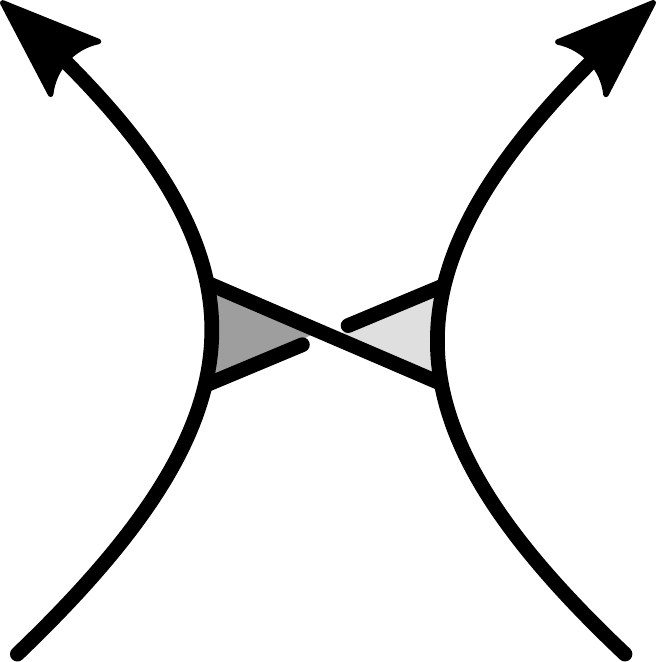}
	\captionsetup{width=.7\textwidth}
	\caption{$K_-$ is obtained from $L$ by band surgery along a band $b$.}
	\label{fig:bandb}
\end{figure}

\noindent The knot $K_+$ is then obtained from $K_-$ by adding a full twist to the band $b$. If $L$ happens to be a split link, then $b$ is said to be \textit{trivial} if there exists an embedded sphere which splits $L$ and intersects $b$ along a single arc. 

\theoremstyle{plain}
\newtheorem*{CCC}{Cosmetic Crossing Conjecture}

\begin{CCC}
	Suppose $K_-$ and $K_+$ are isotopic as oriented knots. Then \begin{enumerate}[itemsep=-.5ex, topsep=0.5ex]
		\item the link $L$ is split,
		\item the band $b$ is trivial. 
	\end{enumerate}
\end{CCC}

\noindent If $L$ is split and $b$ is trivial, then the crossing change from $K_-$ to $K_+$ is called \textit{nugatory}. 
The cosmetic crossing conjecture, which is also known as the nugatory crossing conjecture, appears as Problem 1.58 on Kirby's problem list \cite{MR1470751} and is attributed to X. S. Lin. 

The conjecture has the following generalization. If $K_b$ is an oriented knot obtained from a two-component link $L$ by band surgery along a band $b$, then we may consider the family of oriented knots $\{K_{b+n}\}_{n\in\Z}$ where $K_{b+n}$ is obtained from $K_b$ by adding $n$ full twists to the band. 

\theoremstyle{plain}
\newtheorem*{GCCC}{Generalized Cosmetic Crossing Conjecture}

\begin{GCCC}
	Suppose there are distinct integers $n, m$ for which $K_{b+n}$ and $K_{b+m}$ are isotopic as oriented knots. Then \begin{enumerate}[itemsep=-.5ex, topsep=0.5ex]
		\item the link $L$ is split,
		\item the band $b$ is trivial. 
	\end{enumerate}
\end{GCCC}

\subsection{Statement of results}

The conjecture is usually viewed as a claim about knots and their crossing changes. See for example \cite{MR1669827,MR2980586,MR2928712,MR3501264,MR3654492,MR3605982}. Alternatively, the conjecture can viewed as a claim about two-component links and their band surgeries. For example, the generalized cosmetic crossing conjecture is equivalent to the following two assertions: \begin{enumerate}[itemsep=-0.5ex, topsep=0.5ex]
	\item If $L$ is a \textit{nonsplit} two-component link, then for any band $b$, the knots $K_{b+n}$ for $n \in \Z$ are all distinct as oriented knots. 
	\item If $L$ is a \textit{split} two-component link, then for any \textit{nontrivial} band $b$, the knots $K_{b+n}$ for $n \in \Z$ are all distinct as oriented knots. 
\end{enumerate}
The main result of this paper is a proof of the second assertion, the generalized cosmetic crossing conjecture for split links. When $L$ is a split two-component link, a knot $K_b$ obtained from $L$ by band surgery is called a \textit{band sum} of $L$ (Figure~\ref{fig:bandSum}). The proof appears in section~\ref{subsec:KhovanovFullTwists}. 

\begin{thm}\label{thm:nonTrivBandFullTwistsDifferent}
	Let $K_b$ be the band sum of a split two-component link along a nontrivial band $b$, and let $K_{b+n}$ be obtained from $K_b$ by adding $n$ full twists to the band. The knots $K_{b+n}$ for $n \in \Z$ are all distinct as unoriented knots. 

	In fact, if $K_{b+n/2}$ is obtained from $K_b$ by adding $n$ half twists to the band, then $K_{b+n/2}$ for $n \in \Z$ are all distinct as unoriented knots.
\end{thm}

\begin{figure}[!ht]
	\centering
	\captionsetup{width=.7\textwidth}
	\includegraphics[width=.5\textwidth]{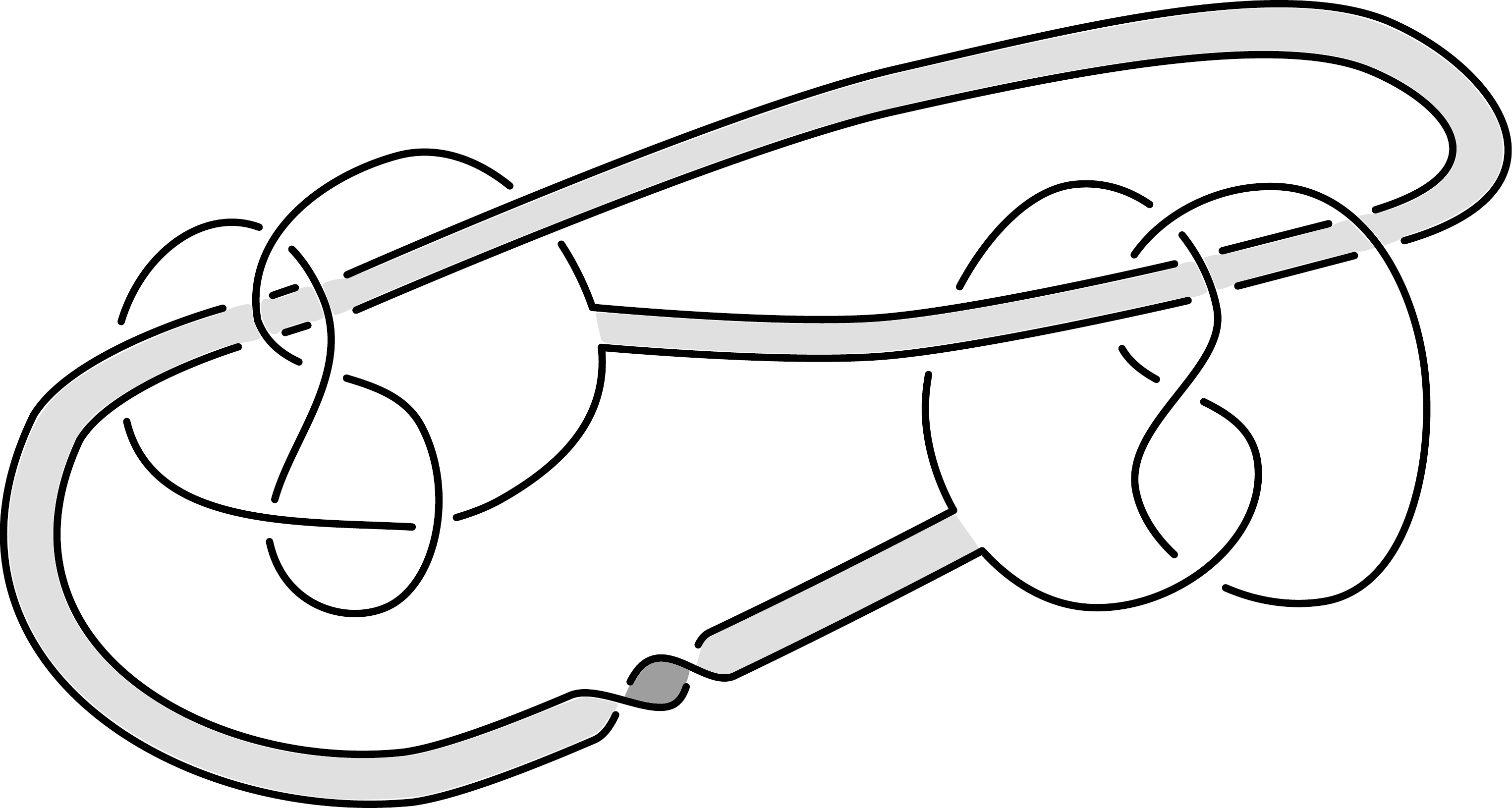}
	\caption{A band sum along a nontrivial band. The genus of this particular band sum is equal to the genus of the connected sum.}
	\label{fig:bandSum}
\end{figure}

In this paper, we study how a number of modern knot invariants of $K_b$ behave under adding full twists to the band. We first recall the behavior of two classical knot invariants, the Alexander and Jones polynomials, under this operation. The (symmetrized) Alexander polynomial $\Delta(J) \in \Z[t^{\pm1}]$ is an invariant of oriented links $J$ in $S^3$ which evaluates to $1$ on the unknot and satisfies the oriented skein relation \[
	\Delta(J_+) - \Delta(J_-) = (t^{-1/2} - t^{1/2})\Delta(J_0)
\]for three links $J_+,J_-,J_0$ which agree outside a small ball in which $J_+$ is a positive crossing, $J_-$ is a negative crossing, and $J_0$ is the oriented resolution of the crossing. When three links $J_+,J_-,J_0$ stand in this relation, they are said to form an \textit{oriented skein triple}. If $J$ is a knot, then $\Delta(J)$ is actually an invariant of the unoriented knot type of $J$. Let $K_b$ be a band sum of a split two-component link $L$, and let $K_\#$ be the connected sum of $L$. Then $K_\#,K_\#,L$ form an oriented skein triple (Figure~\ref{fig:nugatory}), so we see that $\Delta(L) = 0$ from the skein relation. It then follows from the skein relation for the triple $K_{b+1},K_b,L$ that $\Delta(K_b) = \Delta(K_{b+1})$. Although this elementary argument shows the Alexander polynomial fails to distinguish the knots $K_{b+n}$ for $n \in \Z$ when $L$ is split, the same argument proves the generalized cosmetic crossing conjecture for many nonsplit links. 

\begin{prop}
	Let $L$ be a two-component link for which $\Delta(L) \neq 0$. If $K_b$ is obtained from $L$ by band surgery, and $K_{b+n}$ is obtained from $K_b$ by adding $n$ full twists to the band, then the knots $K_{b+n}$ for $n \in \Z$ are all distinct as unoriented knots. In particular, the generalized cosmetic crossing conjecture is true for all two-component links $L$ with nonzero linking number. 
\end{prop}
\begin{proof}
	The skein relation implies that $\Delta(K_{b+n}) = \Delta(K_b) + n(t^{-1/2} - t^{1/2})\Delta(L)$, so the polynomials $\Delta(K_{b+n})$ for $n \in \Z$ are all distinct by the assumption that $\Delta(L) \neq 0$. 
	Thus, the unoriented knots $K_{b+n}$ for $n \in \Z$ are distinct. 

	The Conway polynomial $\nabla(J) \in \Z[z]$ is an invariant of oriented links $J$ for which the substitution $z = t^{-1/2} - t^{1/2}$ recovers $\Delta(J)$. Hence, $\nabla(J) = 0$ if and only if $\Delta(J) = 0$. If $L$ is a two-component link with nonzero linking number, then $\nabla(L) \neq 0$ because the coefficient of the linear term of $\nabla(L)$ is precisely the linking number of $L$, which follows from an elementary argument using the skein relation $\nabla(J_+) - \nabla(J_-) = z\nabla(J_0)$. 
\end{proof}

We return to the split link case. The (unnormalized) Jones polynomial $V(J) \in \Z[q^{\pm1}]$ is an invariant of oriented links $J$ in $S^3$ which evaluates to $q + q^{-1}$ on the unknot and satisfies a different oriented skein relation \[
	q^{-2}V(J_+) - q^2V(J_-) = (q^{-1} - q)V(J_0).
\]Let $L$ be a split two-component link, and as before, let $K_\#$ be the connected sum and $K_b$ be a band sum of $L$. From the skein relation for the triple $K_\#,K_\#,L$, we obtain $(q^{-1} + q)V(K_\#) = V(L)$, and in combination with the skein relation for the triple $K_{b+1},K_b,L$, we find that \[
	V(K_{b+1}) - V(K_\#) = q^4(V(K_b) - V(K_\#)).
\]Let $P_b$ be the polynomial $V(K_b) - V(K_\#)$ so that $V(K_b) = V(K_\#) + P_b$. The skein relation implies that $V(K_{b+1}) = V(K_\#) + q^4P_b$. The argument iterates to show that $V(K_{b+n}) = V(K_\#) + q^{4n}P_b$. Thus, if $P_b \neq 0$, then the polynomials $V(K_{b+n})$ are all distinct. The author does not know if $P_b \neq 0$ for all nontrivial bands $b$. 

\vspace{20pt}

Knot Floer homology \cite{MR2704683,MR2065507} in its most basic form associates to a knot $J$ in $S^3$ a finite-dimensional vector space $\HFKhat(J)$ over $\F_2 = \Z/2$. It is equipped with two $\Z$-gradings, called the Alexander and Maslov gradings. Up to bigraded isomorphism, $\HFKhat(J)$ is an invariant of the unoriented knot type of $J$. The Euler characteristic of $\HFKhat(J,s)$ in Alexander grading $s$ is the coefficient of $t^s$ in the Alexander polynomial $\Delta(J)$ of $J$. In section~\ref{subsec:HeegaardInvariance}, we prove the following result which categorifies the observation that $\Delta(K_b) = \Delta(K_{b+1})$. 

\begin{thm}\label{thm:HeegaardKnotFloerInvariantUnderFullTwists}
	Let $K_b$ be a band sum of a split two-component link, and let $K_{b+n}$ be obtained by adding $n$ full twists to the band. Then \[
		\HFKhat(K_b) \cong \HFKhat(K_{b+n})
	\]as bigraded vector spaces over $\F_2$. 
\end{thm}
\begin{rem}
	Another version of knot Floer homology takes the form of a bigraded module $\HFK^-(K_b)$ over the polynomial ring $\F_2[U]$. Theorem~\ref{thm:HeegaardKnotFloerInvariantUnderFullTwists} holds for this version as well. 
\end{rem}
\begin{rem}
	The special case of this result for both $\HFKhat$ and $\HFK^-$ when the split link is the unlink was proven by Hedden-Watson (Theorem 1 of \cite{MR3782416}). 
\end{rem}

Instanton knot Floer homology \cite{MR2652464,MR1171893} is a similar knot invariant and takes the form of a finite-dimensional vector space $\KHI(J)$ over $\C$ equipped with a $\Z$-grading called the Alexander grading, and a $\Z/2$-grading called the canonical mod $2$ grading. Again up to isomorphism respecting the two gradings, $\KHI(J)$ is an invariant of the unoriented knot type of $J$ in $S^3$. The Euler characteristic of $\KHI(J,s)$ in Alexander grading $s$ is also the coefficient of $t^s$ in the Alexander polynomial $\Delta(J)$ of $J$ \cite{MR2661575,MR2683750}. We will sometimes refer to knot Floer homology as Heegaard knot Floer homology to distinguish it from instanton knot Floer homology. In section~\ref{subsec:invarianceInstanton}, we prove the following result analogous to Theorem~\ref{thm:HeegaardKnotFloerInvariantUnderFullTwists}.

\begin{thm}\label{thm:SuturedInstantonInvariantUnderFullTwists}
	Let $K_b$ be a band sum of a split two-component link, and let $K_{b+n}$ be obtained by adding $n$ full twists to the band. Then \[
		\KHI(K_b) \cong \KHI(K_{b+n})
	\]as vector spaces over $\C$ equipped with Alexander gradings and canonical mod $2$ gradings. 
\end{thm}

Khovanov homology \cite{MR1740682} associates to an oriented link $J$ in $S^3$ a finitely-generated $R$-module $\Kh(J)$ equipped with two-bigradings labeled $q$ and $h$, where $R$ is any commutative ring. The $R$-module $\Kh(J)$ up to bigraded isomorphism is an invariant of the oriented link type of $J$, and in the special case when $J$ is a knot, depends only on the unoriented knot type of $J$. With coefficients in any field, the Euler characteristic of $\Kh^i(J)$ in $q$-grading $i$ is the coefficient of $q^i$ in the Jones polynomial $V(J)$. When $J$ is equipped with a basepoint $p$, then there is a reduced variant $\Khr(J,p)$ which is an invariant of the pointed link $(J,p)$. If $J$ is knot, then $\Khr(J) = \Khr(J,p)$ does not depend on the basepoint. 

Let $L$ be a split two-component link, let $K_b$ be a band sum of $L$, and let $K_\#$ be the connected sum. As we have observed, if we write $V(K_b) = V(K_\#) + P_b$ for some polynomial $P_b \in \Z[q^{\pm 1}]$, then $V(K_{b+n}) = V(K_\#) + q^{4n}P_b$. If $P_b \neq 0$, then the Jones polynomial distinguishes the knots in the family $K_{b+n}$. The main result of \cite{MR4041014} implies that $\Kh(K_\#)$ is isomorphic as a bigraded $R$-module to a direct summand of $\Kh(K_b)$. In particular, there is a ribbon concordance from $K_\#$ to $K_b$ \cite{MR1451821} and a ribbon concordance $C\colon K \to K'$ induces a bigrading-preserving $R$-module map $\Kh(C)\colon \Kh(K) \to \Kh(K')$ which is an isomorphism onto a direct summand \cite{MR4041014}. Thus, we may write $\Kh(K_b) \cong \Kh(K_\#) \oplus H_b$ for some finitely-generated bigraded $R$-module $H_b$. When the coefficient ring $R$ is a field, the $q$-graded Euler characteristic of $H_b$ is the polynomial $P_b$. For any integers $i,j$, let $h^iq^j\:H_b$ denote the bigraded $R$-module whose graded summand in $(h,q)$-grading $(x + i,y+j)$ is equal to the $(x,y)$-graded summand of $H_b$. We prove in sections~\ref{subsec:KhovanovDetectingTrivBand} and \ref{subsec:KhovanovFullTwists} the following two results, where the coefficient ring $R$ is the field $\F_2$. The second result categorifies the identity $V(K_{b+n}) = V(K_\#) + q^{4n}P_b$ and implies Theorem~\ref{thm:nonTrivBandFullTwistsDifferent}. 

\begin{thm}\label{thm:KhovanovHomologyDetectsTrivialBand}
	Let $K_b$ be the band sum of a split two-component link along a band $b$, and let $K_\#$ be the connected sum. Then \[
		\dim \Kh(K_\#) = \dim\Kh(K_b)
	\]over $\F_2$ if and only if $b$ is trivial. 
\end{thm}

\begin{thm}\label{thm:KhovanovHomologyChangesUnderFullTwists}
	Let $K_b$ be the band sum of a split two-component link along a nontrivial band $b$, let $K_\#$ be the connected sum, and let $K_{b+n}$ be obtained from $K_b$ by adding $n$ full twists to the band. Then there is a nonzero finite-dimensional bigraded vector space $H_b$ for which \[
		\Kh(K_{b+n}) \cong \Kh(K_\#) \oplus h^{2n}q^{4n}\: H_b
	\]as bigraded vector spaces over $\F_2$. In particular, the bigraded vector spaces $\Kh(K_{b+n})$ over $\F_2$ for $n \in \Z$ are all distinct. 

	In fact, if $K_{b+n/2}$ is obtained from $K_b$ by adding $n$ half twists to the band, then $\Kh(K_{b+n/2}) \cong \Kh(K_\#) \oplus h^{n}q^{2n}\: H_b$ so the bigraded vector spaces $\Kh(K_{b+n/2})$ for $n \in \Z$ are also all distinct.
\end{thm}

\begin{rem}
	Theorems \ref{thm:KhovanovHomologyDetectsTrivialBand} and \ref{thm:KhovanovHomologyChangesUnderFullTwists} also hold for $\Khr$ over $\F_2$ and $\Khr$ over $\Q$. 
\end{rem}
\begin{rem}
	Hedden-Watson prove that the bigraded vector spaces $\Khr(K_{b+n})$ over $\F_2$ are all distinct for any nontrivial band $b$ in the special case where the split link is the unlink (Theorem 3.2 of \cite{MR3782416}). 
\end{rem}
\begin{rem}
	Observe that \textit{a priori}, there could exist a nontrivial band $b$ for which $K_b$ and $K_\#$ are isotopic. Miyazaki \cite{MR4058258} proves that such a band does not exist, and our work gives an independent proof (Corollary~\ref{cor:miyazaki}).
\end{rem}
\begin{rem}
	For any oriented link $J$, the maximal and minimal $h$-gradings in which $\Kh(J)$ is nonzero provide lower bounds on the number of positive and negative crossings in any diagram for $J$. In particular, if $K_b$ is a nontrivial band sum of a split two-component link, then Theorem~\ref{thm:KhovanovHomologyChangesUnderFullTwists} implies that the minimal number of positive (resp. negative) crossings in any diagram for $K_{b+n}$ is unbounded as $n \to \infty$ (resp. $n \to -\infty$). 
\end{rem}

We prove Theorem~\ref{thm:KhovanovHomologyDetectsTrivialBand} using the spectral sequence from Khovanov homology to singular instanton homology, constructed by Kronheimer-Mrowka in their proof that Khovanov homology detects the unknot \cite{MR2805599}. Singular instanton homology \cite{MR2805599} associates to an unoriented link $J$ in $S^3$ a finitely-generated $R$-module $\Isharp(J)$ for any commutative ring $R$. The invariant is defined for links in more general $3$-manifolds, but for links $J$ in $S^3$, there is a spectral sequence whose $E_2$-page is the Khovanov homology of the mirror of $J$ which abuts to $\Isharp(J)$. Just as for Khovanov homology, if the link $J$ is equipped with a basepoint $p$, then there is a reduced variant $\Inat(J,p)$ call the \textit{reduced} singular instanton homology of $(J,p)$. When $J$ is a knot, we omit the basepoint $p$ from the notation $\Inat(J) = \Inat(J,p)$. 

Kronheimer-Mrowka show that there is an isomorphism $\Inat(J) \cong \KHI(J)$ of vector spaces over $\C$ when $J$ is a knot (Proposition 1.4 of \cite{MR2805599}). We prove that the dimension of $\KHI(K_b)$ over $\C$ detects the trivial band in section~\ref{subsec:DetectTrivBandInstanton}. We show that Khovanov homology detects the trivial band by reducing to this result.

\begin{thm}\label{thm:suturedInstantonHomologyDetectsTrivialBand}
	Let $K_b$ be the band sum of a split two-component link along a band $b$, and let $K_\#$ be the connected sum. Then \[
		\dim \KHI(K_\#) = \dim \KHI(K_b)
	\]over $\C$ if and only if $b$ is trivial. 
\end{thm}

\begin{cor}\label{cor:singularInstantonHomologyDetectsTrivialBand}
	Let $K_b$ be the band sum of a split two-component link along a band $b$, and let $K_\#$ be the connected sum. Then \[
		\dim \Inat(K_\#) = \dim \Inat(K_b)
	\]over $\Q$ if and only if $b$ is trivial. 
\end{cor}

\begin{rem}
	Corollary~\ref{cor:singularInstantonHomologyDetectsTrivialBand} also holds for $\Inat$ over $\F_2$ and $\Isharp$ over $\F_2$. See section~\ref{subsec:singularInstantonHomology}.
\end{rem}

We establish the same detection result for knot Floer homology. 

\begin{thm}\label{thm:HeegaardFloerHomologyDetectsTrivialBand}
	Let $K_b$ be the band sum of a split two-component link along a band $b$, and let $K_\#$ be the connected sum. Then \[
		\dim \HFKhat(K_\#) = \dim \HFKhat(K_b)
	\]over $\F_2$ if and only if $b$ is trivial. 
\end{thm}

\begin{rem}
	In fact, we have a strict inequality $\dim \HFKhat(K_b,g(K_b)) > \dim \HFKhat(K_\#,g(K_b))$ in Alexander grading $g(K_b)$, the Seifert genus of $K_b$ when $b$ is nontrivial (Theorem~\ref{thm:HeegaardTrivBandDetect}). The same is true for $\KHI$. Note that there are nontrivial bands $b$ for which $g(K_b) = g(K_\#)$. For example, see Figure~\ref{fig:bandSum}. This strict inequality in Alexander grading $g(K_b)$ recovers a result of Kobayashi (Theorem 2 of \cite{MR1177410}) that if $g(K_b) = g(K_\#)$ and $K_b$ is fibered, then $b$ is trivial. See Corollary~\ref{cor:kobayashi}.
\end{rem}

\subsection{Outline of the arguments}

We first outline the proofs that $\HFKhat(K_b)$ and $\KHI(K_b)$ are invariant under adding full twists to the band (Theorems~\ref{thm:HeegaardKnotFloerInvariantUnderFullTwists} and \ref{thm:SuturedInstantonInvariantUnderFullTwists}). See sections~\ref{subsec:HeegaardInvariance} and \ref{subsec:invarianceInstanton} respectively for full proofs. Let $K_b$ be the band sum of a split two-component link $L$ along a nontrivial band $b$. As we have observed, $K_{b+1},K_b,L$ form an oriented skein triple. Both $\HFKhat$ and $\KHI$ have extensions to links in $S^3$, and each satisfies an oriented skein exact triangle \cite{MR2065507,MR2683750}. In particular, there are exact triangles \[
	\begin{tikzcd}[column sep=0]
		\HFKhat(K_{b+1}) \ar[rr] & & \HFKhat(K_b) \ar[dl]\\
		& \HFKhat(L) \ar[ul] &
	\end{tikzcd} \qquad\qquad \begin{tikzcd}[column sep=0]
		\KHI(K_{b+1}) \ar[rr] & & \KHI(K_b) \ar[dl]\\
		& \KHI(L) \ar[ul] &
	\end{tikzcd} 
\]for which the maps $\HFKhat(K_{b+1}) \to \HFKhat(K_b)$ and $\KHI(K_{b+1}) \to \KHI(K_b)$ preserve the two gradings on each invariant. 

Zemke's work on ribbon concordances \cite{Zem19a} gives us a way to compute these maps. According to the main result of \cite{Zem19a}, a ribbon concordance $K_\# \to K_b$ induces a bigrading-preserving inclusion of $\HFKhat(K_\#)$ onto a direct summand of $\HFKhat(K_b)$. The argument extends to show that the skein exact triangle for $K_\#$ is a direct summand of the skein exact triangle for $K_b$. With splittings \[
	\HFKhat(K_{b+1}) \cong \HFKhat(K_\#) \oplus F_{b+1} \qquad \HFKhat(K_b) \cong \HFKhat(K_\#) \oplus F_b
\]as bigraded vector spaces, the skein exact triangle splits as the following direct sum. \[
	\begin{tikzcd}[column sep=-2ex]
		\HFKhat(K_{b+1}) \ar[rr] & & \HFKhat(K_b) \ar[dl]\\
		& \HFKhat(L) \ar[ul] &
	\end{tikzcd} \cong \begin{tikzcd}[column sep=-2ex]
		\HFKhat(K_\#) \ar[rr] & & \HFKhat(K_\#) \ar[dl] \\
		& \HFKhat(L) \ar[ul] &
	\end{tikzcd} \oplus \begin{tikzcd}[column sep=-2ex,row sep=5ex]
		H \ar[rr] & & H' \ar[dl]\\
		& \phantom{\HFK}0\phantom{\HFK} \ar[ul] &
	\end{tikzcd}
\]The isomorphism $H \to H'$ preserves the bigradings, so $\HFKhat(K_{b+1}) \cong \HFKhat(K_b)$ as bigraded vector spaces. The same argument works for $\KHI$. 

\vspace{20pt}

Next, we outline how Theorem~\ref{thm:KhovanovHomologyChangesUnderFullTwists} follows from the result that $\Kh(K_b)$ detects the trivial band (Theorem~\ref{thm:KhovanovHomologyDetectsTrivialBand}). The full proof appears in section~\ref{subsec:KhovanovFullTwists}. Let $K_{b+1/2}$ be the knot obtained by taking the unoriented resolution of the crossing. Khovanov homology satisfies an unoriented skein exact triangle (see Proposition 4.2 of \cite{MR2704683} for example) which in our case, takes the following form. \[
	\begin{tikzcd}[row sep=large]
		\Kh(K_b) \ar[rr,"{(1,2)}"] & & \Kh(K_{b+1/2}) \ar[rr,"{(1,2)}"] \ar[dl,"{(0,-1)}"] & & \Kh(K_{b+1}) \ar[dl,"{(0,-1)}"]\\
		& \Kh(L) \ar[ul,"{(0,-1)}"] & & \Kh(L) \ar[ul,"{(0,-1)}"] &
	\end{tikzcd}
\]The maps shift the bigradings $(h,q)$ by the displayed degrees. Ribbon concordances can be chosen to give us splittings \cite{MR4041014} \[
	\Kh(K_b) \cong \Kh(K_\#) \oplus H_b \quad \Kh(K_{b+1/2}) \cong \Kh(K_\#) \oplus H_{b+1/2} \quad \Kh(K_{b+1}) \cong \Kh(K_\#) \oplus H_{b+1}
\]as bigraded vector spaces over $\F_2$ so that the unoriented skein exact triangles for $K_b$ split as the following direct sum. \[
	\left( \begin{tikzcd}[column sep=-2ex]
		\Kh(K_\#) \ar[rr] & & \Kh(K_\#) \ar[rr] \ar[dl] & & \Kh(K_\#) \ar[dl]\\
		& \Kh(L) \ar[ul] & & \Kh(L) \ar[ul] &
	\end{tikzcd} \right) \oplus \left(\begin{tikzcd}[column sep=1ex]
		H_b \ar[rr] & & H_{b+1/2} \ar[rr] \ar[dl] & & H_{b+1} \ar[dl]\\
		& 0 \ar[ul] & & 0 \ar[ul] &
	\end{tikzcd} \right)
\]The isomorphisms $H_b \to H_{b+1/2} \to H_{b+1}$ each shift the bigrading by $(1,2)$. Thus $H_{b+1/2} \cong hq^2\:H_b$ and $H_{b+1} \cong h^2q^4\:H_b$. The argument iterates to show that $H_{b+n/2} \cong h^nq^{2n}H_b$. The fact that $H_b$ is of positive dimension whenever $b$ is nontrivial is exactly Theorem~\ref{thm:KhovanovHomologyDetectsTrivialBand}. 

\vspace{20pt}

To prove that Khovanov homology detects the trivial band (Theorem~\ref{thm:KhovanovHomologyDetectsTrivialBand}), we reduce to showing that singular instanton homology detects the trivial band in section~\ref{subsec:KhovanovDetectingTrivBand}. Using functoriality \cite{MR3903915} of Kronheimer-Mrowka's spectral sequence from Khovanov homology to singular instanton homology, we show that if Levine-Zemke's injective ribbon concordance map $\Kh(K_\#) \to \Kh(K_b)$ over $\F_2$ is an isomorphism, then $\dim \Isharp(K_\#) = \dim \Isharp(K_b)$ over $\F_2$ as well. Using the identity $\dim \Isharp(K) = 2\dim \Inat(K)$ special to the coefficient ring $\F_2$ (Lemma 7.7 of \cite{MR3880205}), we reduce the problem to showing that $\Inat(K_b)$ over $\F_2$ detects the trivial band. By a universal coefficient argument and the isomorphism $\Inat(K) \cong \KHI(K)$ over $\C$ for knots mentioned previously, we reduce to showing that $\dim \KHI(K_b)$ over $\C$ detects the trivial band. This reduction argument appears in section~\ref{subsec:KhovanovDetectingTrivBand}.

\vspace{20pt}

Lastly, we outline the argument that $\dim \KHI(K_b)$ and $\dim \HFKhat(K_b)$ detect the trivial band (Theorems~\ref{thm:suturedInstantonHomologyDetectsTrivialBand} and \ref{thm:HeegaardFloerHomologyDetectsTrivialBand}). Full proofs appear in sections~\ref{subsec:DetectTrivBandInstanton} and \ref{subsec:HeegaardDetecting}, respectively. The two proofs have the same general structure, and since the argument for knot Floer homology is slightly simpler than for $\KHI$, we outline a proof for $\HFKhat$. 

\begin{figure}[!ht]
	\centering
	\labellist
	\pinlabel $C$ at 480 205
	\endlabellist
	\includegraphics[width=.5\textwidth]{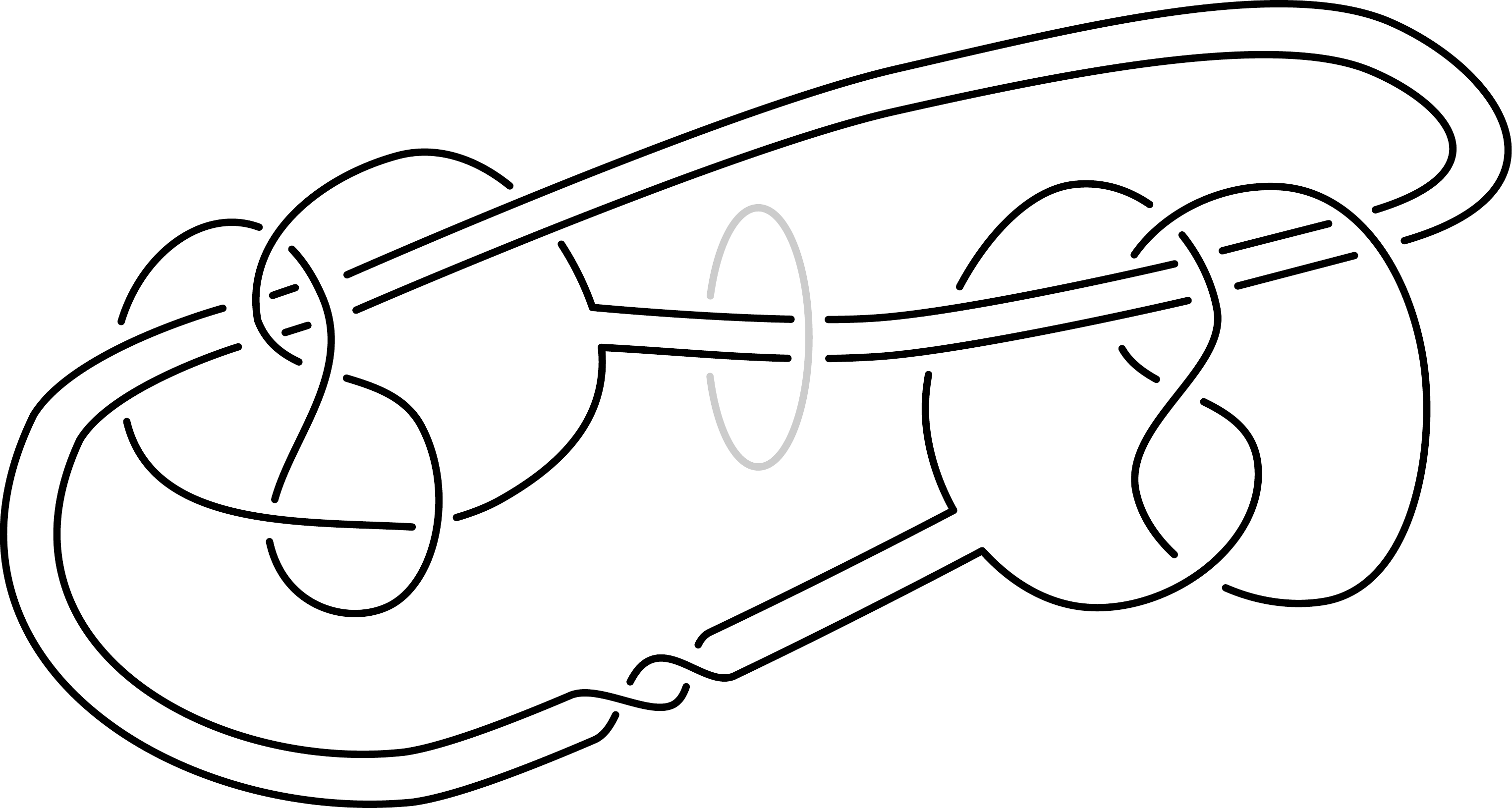}
	\caption{A band sum with a linking circle $C$.}
	\label{fig:linkingCircle}
\end{figure}

Let $C$ be an unknot in the complement of $K_b$ which bounds a disc that meets the band transversely in a single arc. Such a circle $C$ in the complement of $K_b$ is called a \textit{linking circle} for the band (Figure~\ref{fig:linkingCircle}). Note that the band is trivial if and only if $C$ bounds a disc disjoint from $K_b$. The link Floer homology $\HFLhat(K_b \cup C)$ of the link $K_b \cup C$ detects the trivial band since it detects the Thurston norm of the link exterior \cite{MR2393424}. Zemke's ribbon concordance argument \cite{Zem19a} gives an inclusion of $\HFLhat(K_\# \cup C)$ onto a direct summand of $\HFLhat(K_b \cup C)$ so we may write \[
	\HFLhat(K_b \cup C) \cong \HFLhat(K_\# \cup C) \oplus G_b
\]where $G_b \neq 0$. It does not immediately follow that $\HFKhat(K_b)$ also detects the trivial band; there is a spectral sequence from $\HFLhat(K_b \cup C)$ to $\HFKhat(K_b) \otimes \F_2^2$, and the higher differentials in the spectral sequence may kill $G_b$. We essentially prove the detection result by showing that a nonzero element of $G_b$ does, in fact, survive the spectral sequence. 

The spectral sequence from $\HFLhat(K_b \cup C)$ to $\HFKhat(K_b) \otimes \F_2^2$ is directly analogous to the spectral sequence from $\HFKhat(J)$ to $\HFhat(S^3)$ for a knot $J \subset S^3$. A Heegaard diagram for $J$ has two basepoints, and if we ignore one of them, we obtain a Heegaard diagram for $S^3$. The ignored basepoint gives a filtration on the chain complex for $S^3$, and the spectral sequence associated to this filtered complex has $E_2$-page $\HFKhat(J)$ and abuts to $\HFhat(S^3)$. Similarly, a Heegaard diagram for the link $K_b \cup C$ has four basepoints, two for each component. If we ignore a basepoint corresponding to $C$, we obtain a sutured Heegaard diagram for the sutured exterior of $K_b$ with a sutured puncture, which is the result of attaching a $2$-handle to the sutured exterior of $K_b \cup C$ along a suture lying on $\partial N(C)$. The sutured Floer homology of this sutured manifold is isomorphic to $\HFKhat(K_b) \otimes \F_2^2$. The spectral sequence associated to the ignored basepoint has $E_2$-page $\HFLhat(K_b \cup C)$ and abuts to $\HFKhat(K_b) \otimes \F_2^2$. 

There is a splitting of $\HFLhat(K_b \cup C)$ along relative $\Spin^c$ structures, which form an affine space over $H_1(S^3\setminus (K_b \cup C))$. Let $\sigma$ be one of the two parallel sutures on $\partial N(C)$, and let $\tau$ be one of the two parallel sutures on $\partial N(K_b)$. In the splitting $\HFLhat(K_b \cup C) \cong \HFLhat(K_\# \cup C) \oplus G_b$, all of $\HFLhat(K_\# \cup C)$ lies in $\Spin^c$ structures differing only by multiples of $[\tau] \in H_1(S^3\setminus(K_b \cup C))$. To show that $\HFKhat(K_b)$ detects the trivial band, it suffices to show that there exist two nonzero elements of $\HFLhat(K_b \cup C)$ lying in $\Spin^c$ structures differing by $[\sigma] \in H_1(S^3\setminus (K_b \cup C))$ that survive in the spectral sequence to the $E_\infty$-page. 

Let $(M_0,\gamma_0)$ be the sutured exterior of $K_b \cup C$, and let $(N_0,\beta_0)$ be obtained from $(M_0,\gamma_0)$ by attaching a $2$-handle the suture $\sigma$ on $\partial N(C)$. Let $S_1$ be a nice decomposing surface in $(M_0,\gamma_0)$ disjoint from $\partial N(C)$, and let $(M_1,\gamma_1)$ be the result of decomposing along $S_1$. The distinguished toral boundary component $\partial N(C)$ remains a toral boundary component of $(M_1,\gamma_1)$, and $\sigma$ is still a suture of $(M_1,\gamma_1)$. Let $(N_1,\beta_1)$ be obtained from $(M_1,\gamma_1)$ by attaching a $2$-handle along $\sigma$. Alternatively, $(N_1,\beta_1)$ is the result of decomposing $(N_0,\beta_0)$ along $S_1$. Just like the spectral sequence $\SFH(M_0,\gamma_0) \Rightarrow \SFH(N_0,\beta_0)$, we have a spectral sequence $\SFH(M_1,\gamma_1) \Rightarrow \SFH(N_1,\beta_1)$ and it is compatible with the direct summand inclusions given by Juh\'asz \cite{Juh08}. \[
	\begin{tikzcd}
		\SFH(M_1,\gamma_1) \ar[d,Rightarrow] \ar[r,hook] & {\SFH(M_0,\gamma_0) = \HFLhat(K_b \cup C)} \ar[d,Rightarrow]\\
		\SFH(N_1,\beta_1) \ar[r,hook] & {\SFH(N_0,\beta_0) = \HFKhat(K_b) \otimes \F_2^2}
	\end{tikzcd}
\]If two nonzero elements of $\SFH(M_1,\gamma_1)$ lying in $\Spin^c$ structures differing by $[\sigma] \in H_1(M_1)$ survive the spectral sequence to $\SFH(N_1,\beta_1)$, then the same is true for $\SFH(M_0,\gamma_0)$ and $\SFH(N_0,\beta_0)$. 

Using Gabai's proof of superadditivity of genus under band sum \cite{Gab87a,MR895573}, we show that if $b$ is nontrivial, then there is a sequence of nice surface decompositions \[
	S^3(K_b \cup C) = (M_0,\gamma_0) \overset{S_1}{\rightsquigarrow} (M_1,\gamma_1) \overset{S_2}{\rightsquigarrow} \cdots \overset{S_n}{\rightsquigarrow} (M_n,\gamma_n)
\]where each $S_i$ is disjoint from $\partial N(C)$, and $(M_n,\gamma_n)$ is the disjoint union of a product sutured manifold and the sutured exterior of a two-component Hopf link $H$ (Theorem~\ref{thm:SeqOfSurfDecomp}). This reduces the entire problem to a model computation for the Hopf link, which we do explicitly. 

There are a number of challenges in trying to adapt this argument to sutured instanton homology. The essential difficulty is the construction of spectral sequences $\SHI(M_i,\gamma_i) \Rightarrow \SHI(N_i,\beta_i)$; a spectral sequence from the sutured instanton homology of a knot $J \subset S^3$ to the sutured instanton homology of $S^3$ has not even been constructed. A way forward is to repackage the spectral sequence as a minus version of Floer homology. For example, in place of the spectral sequence $\HFKhat(J) \Rightarrow \HFhat(S^3)$ for a knot $J \subset S^3$, we have the $\F_2[U]$-module $\HFK^-(J)$ together with the exact triangles \[
	\begin{tikzcd}[column sep=0]
		\HFK^-(J) \ar[rr,"U"] & & \HFK^-(J) \ar[dl]\\
		& \HFKhat(J) \ar[ul] &
	\end{tikzcd}\qquad \begin{tikzcd}[column sep=0]
		\HFK^-(J) \ar[rr,"U-\Id"] & & \HFK^-(J) \ar[dl]\\
		& \HFhat(S^3) \ar[ul] &
	\end{tikzcd}
\]the second of which implies that the $\F_2[U]$-rank of $\HFK^-(J)$ equals the $\F_2$-dimension of $\HFhat(S^3)$. In section~\ref{subsec:HeegaardDetecting}, we prove Theorem~\ref{thm:HeegaardFloerHomologyDetectsTrivialBand} using a minus version of Heegaard Floer homology rather than spectral sequences as outlined above. By work of Baldwin-Sivek, Li, and Ghosh-Li \cite{MR3352794,MR3477339,1801.07634,1810.13071,1901.06679,1910.01758,1910.10842} inspired by ideas in \cite{MR3650078}, there is minus version of sutured instanton homology; in particular, there is a $\C[U]$-module $\KHI^-(J)$ whose $\C[U]$-rank equals the $\C$-dimension of the sutured instanton homology of $S^3$. We extend their work to our situation, and prove a number of technical results along the way. 

Li points out that one of our technical results has the following corollary. Although a conjectural spectral sequence $\KHI(Y,K) \Rightarrow \SHI(Y(1))$ for a null-homologous knot $K \subset Y$ has not yet been established, we have the following rank inequality.

\begin{prop}\label{prop:dimInequalitySuturedInstanton}
	Let $K$ be a null-homologous knot in a closed oriented $3$-manifold $Y$. Then \[
		\dim \KHI(Y,K) \ge \dim \SHI(Y(1))
	\]where $Y(1)$ is $Y$ with a sutured puncture.
\end{prop}

\vspace{20pt}

\theoremstyle{definition}
\newtheorem*{ack}{Acknowledgments}
\begin{ack}
	I would like to thank John Baldwin, Zhenkun Li, and Matthew Stoffregen for the helpful discussions. I would also like to thank my advisor Peter Kronheimer for his continued guidance, support, and encouragement. This material is based upon work supported by the NSF GRFP through grant DGE-1745303.
\end{ack}

%%%%%%%%%%%%%%%%%%%%%%%%%%%%%%%%%%%%%%%
%%%%%%%%% Sutured manifolds %%%%%%%%%%%
%%%%%%%%%%%%%%%%%%%%%%%%%%%%%%%%%%%%%%%

\section{Sutured manifolds}

After a review of the basics of sutured manifold theory, we construct a sequence of surface decompositions in section~\ref{subsec:seqOfSurfDecomp} which we use in an essential way in the proofs that knot Floer homology, instanton Floer homology, and Khovanov homology detect the trivial band. We also collect a number of consequences of the light bulb trick in section~\ref{subsubsec:LBT} which we will use later. 

\subsection{Preliminaries}

\begin{dfs*}
	A \textit{band} $b$ for a two-component link $J_1 \cup J_2$ in $S^3$ is an embedding of a rectangle $[0,1] \x [0,10]$ into $S^3$ so that one short edge lies on $J_1$, the other short edge lies on $J_2$, and the rest of the rectangle is disjoint from the link. Explicitly, we require that $[0,1] \x 0 \subset J_1$, $[0,1] \x 10 \subset J_2$, and $[0,1] \x (0,10) \cap (J_1 \cup J_2) = \emptyset$. The result of \textit{band surgery} on $J_1 \cup J_2$ along a band $b$ is the knot $K$ obtained by deleting the interior of the band along with the two short edges of its boundary. Explicitly, $K = (J_1 \cup b \cup J_2) \setminus (0,1) \x [0,10]$. An orientation on one of the three knots $J_1,J_2,K$ naturally induces orientations on the other two. 

	A two-component link $J_1 \cup J_2$ in $S^3$ is \textit{split} if there exists an embedded $2$-sphere $Q$ in $S^3$ disjoint from the link for which $J_1$ lies on one side of $Q$ and $J_2$ lies on the other. Any such embedded sphere is called a \textit{splitting sphere}. The result of band surgery on a split link $J_1 \cup J_2$ is called a \textit{band sum} of $J_1 \cup J_2$. A band $b$ for a link $J_1 \cup J_2$ is \textit{trivial} if $J_1 \cup J_2$ is split and there exists a splitting sphere for $J_1 \cup J_2$ that intersects the band transversely and in a single arc. The band sum along a trivial band is the connected sum $J_1 \# J_2$. 

	A \textit{linking circle} of a band $b$ for a two-component link $J_1 \cup J_2$ is a ``meridian'' of the band thought of as an unknot in the complement of $J_1 \cup b \cup J_2$. It bounds a disc in $S^3$ which is disjoint from $J_1 \cup J_2$ and intersects $b$ transversely along an arc. 
\end{dfs*}

\begin{rem}
	Our convention is that surgery on a linking circle with slope $+1$ adds a full twist to the band. If $K_b$ denotes the band sum, we write $K_{b+1}$ to denote the band sum with a full twist added to the band. 
\end{rem}

\begin{dfs*}
	A \textit{cobordism} $C\colon J_0 \to J_1$ between links $J_0,J_1 \subset S^3$ is a compact surface $C$ properly embedded in $[0,1] \x S^3$ for which $\partial C \cap i \x S^3 = J_i$ for $i = 0,1$. If $J_0,J_1$ are oriented, then $C\colon J_0 \to J_1$ is an \textit{oriented cobordism} if $C$ is oriented so that its boundary orientation satisfies $\partial C = J_1 - J_0$. Two cobordisms are equivalent if they are isotopic rel boundary. 

	A \textit{concordance} $C\colon J_0 \to J_1$ from a link $J_0 \subset S^3$ to another link $J_1 \subset S^3$ with the same number of components is a cobordism consisting of disjoint annuli $A_i$ where each $A_i$ has a boundary component in each of $0 \x S^3$ and $1 \x S^3$. A \textit{ribbon concordance} $R\colon J_0 \to J_1$ between links is a concordance $R$ for which, up to isotopy rel boundary, the function $R \hookrightarrow [0,1] \x S^3 \to [0,1]$ is Morse and has no critical points of index $2$. Such critical points are referred to as \textit{deaths}, while critical points of index $0$ and $1$ are called \textit{births} and \textit{saddles}, respectively.
\end{dfs*}

\subsubsection{The light bulb trick}\label{subsubsec:LBT}

The following argument, called the \textit{light bulb trick}, is well-known.

\begin{lem}[Light bulb trick]\label{lem:LBT}
	Let $K$ be a knot in $S^1 \x S^2$ which intersects $\pt \x S^2$ transversely and in a single point. Then $K$ is isotopic to $S^1 \x \pt$. 
\end{lem}
\begin{proof}
	Isotope $K$ so that the intersection of $K$ with $\pt \x S^2$ is a single transverse point $p$ and an arc. With the endpoints of the arc fixed, we may imagine pulling the arc within $\pt \x S^2$ across the point $p$ thereby doing a crossing change to $K$. Through such crossing changes we may certainly make $K$ become $S^1 \x \pt$. The light bulb trick is the observation that this crossing change can be achieved by an isotopy of the knot; simply swing the arc through the other side of $\pt \x S^2$. 
\end{proof}

We will use the following three corollaries of the light bulb trick. The first is elementary, the second is due to Thompson \cite{MR895572}, and the third is due to Miyazaki \cite{MR1451821}. 

\begin{cor}\label{cor:ZeroFillingMeridian}
	Let $M$ be the exterior of the two-component link $J \cup \mu$, where $J \subset S^3$ is a knot and $\mu$ is a meridian of $J$ thought of as an unknot in the complement of $J$. Filling the boundary component $\partial N(\mu)$ of $M$ along the zero slope results in the solid torus. 
\end{cor}
\begin{proof}
	Zero surgery on $\mu$ makes $J$ a knot in $S^1 \x S^2$. The disc that $\mu$ bounds in $S^3$ that intersects $J$ transversely and in a single point can be capped off with a disc in $S^1 \x S^2$ to obtain a copy of $\pt \x S^2$, so we may apply the light bulb trick. The exterior of $S^1 \x \pt$ in $S^1 \x S^2$ is a solid torus. 
\end{proof}

\begin{cor}[Claim (a) of \cite{MR895572}]\label{cor:ZeroSurgeryLinkingCircle}
	Let $K_b$ be a band sum of a two-component split link $J_1 \cup J_2$, and let $C$ be a linking circle. Let $M$ be the exterior of $K_b \cup C$. Filling the component $\partial N(C)$ along the zero slope results in a reducible manifold. In particular, the result is the connected sum of $S^1 \x S^2$ with the exterior of $J_1 \# J_2$ in $S^3$. 
\end{cor}
\begin{proof}
	The linking circle $C$ bounds a disc in $S^3$ which intersects the band of $K_b$ along an arc. After doing zero surgery on $C$, this disc can be capped off with a disc in $S^1 \x S^2$ to obtain a copy of $\pt \x S^2$. Using the light bulb trick, we may now do ``crossing changes'' where exactly one of the two ``strands'' is the band (Figure~\ref{fig:bandCrossing}). Through such moves, the band can be made trivial so that the resulting knot is a copy of $J_1 \# J_2$ in a $3$-ball contained in $S^1 \x S^2$. 
\end{proof}

\begin{figure}[!ht]
	\centering
	\includegraphics[width=.2\textwidth]{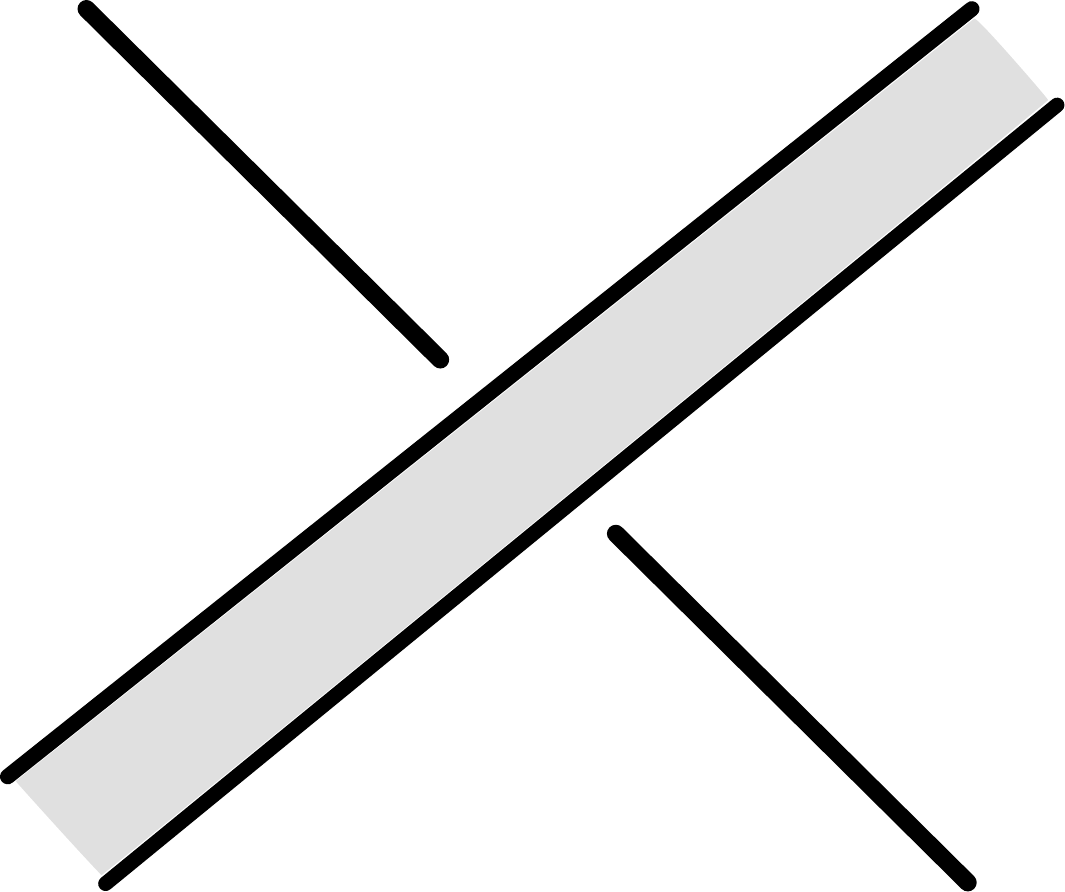}
	\hspace{30pt}
	\includegraphics[width=.2\textwidth]{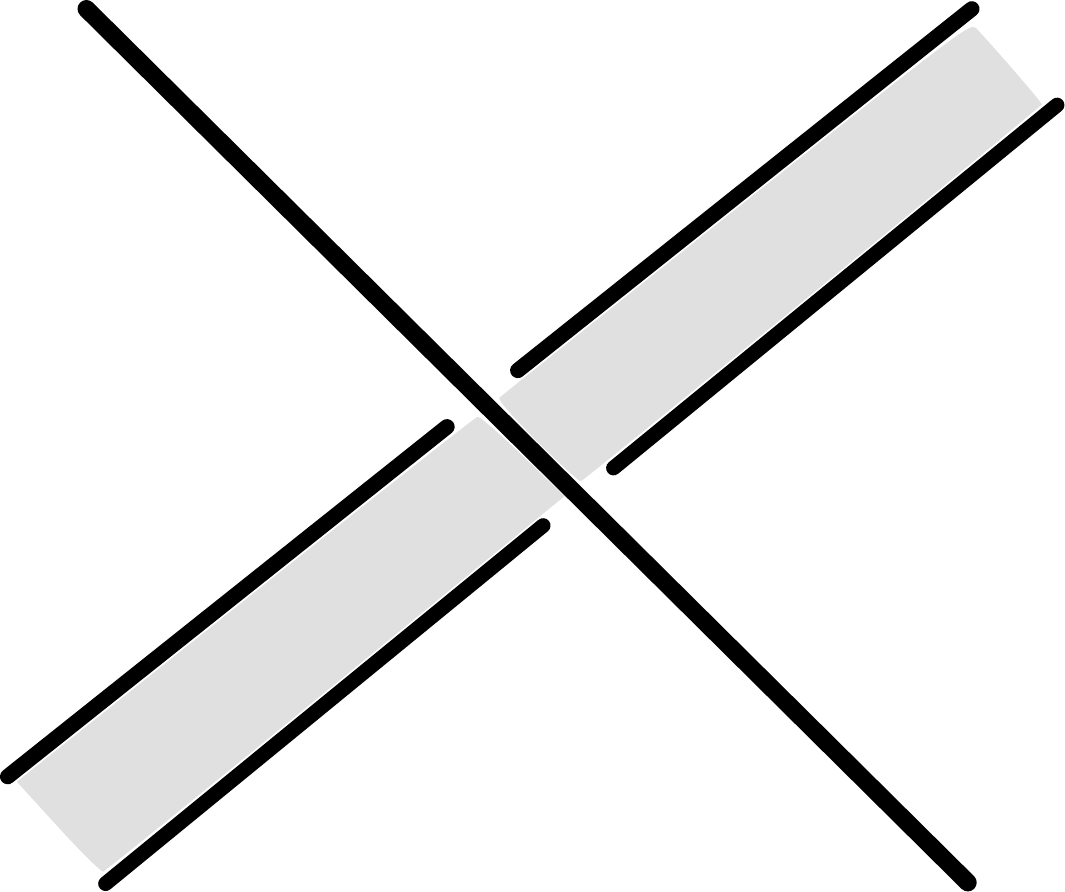}
	\caption{``Crossing change'' where one ``strand'' is the band.}
	\label{fig:bandCrossing}
\end{figure}

\begin{cor}[Theorem 1.1 of \cite{MR1451821}]\label{cor:RibbonConcordance}
	Let $K_b$ be a band sum of a two-component split link $J_1 \cup J_2$. Then there is a ribbon concordance from $J_1 \# J_2$ to $K_b$. 
\end{cor}
\begin{proof}
	We describe a movie presentation of a concordance from $K_b$ to $J_1 \# J_2$ consisting entirely of saddle moves and deaths. Reversing the movie provides the desired ribbon concordance. 

	Consider a ``crossing'' of $K_b$ where exactly one of the two ``strands'' is the band $b$ (Figure~\ref{fig:bandCrossing}). Via a saddle move, the ``crossing change'' can be achieved at the cost of a linking circle for the band (Figure~\ref{fig:saddle}). Through these moves, the band can be made trivial, after which all additional linking circles may be capped off with embedded discs, yielding deaths. The resulting knot is $J_1 \# J_2$. There are an equal number of saddles as deaths so the cobordism is indeed a concordance. 
\end{proof}

\begin{figure}[!ht]
	\centering
	\includegraphics[width=.25\textwidth]{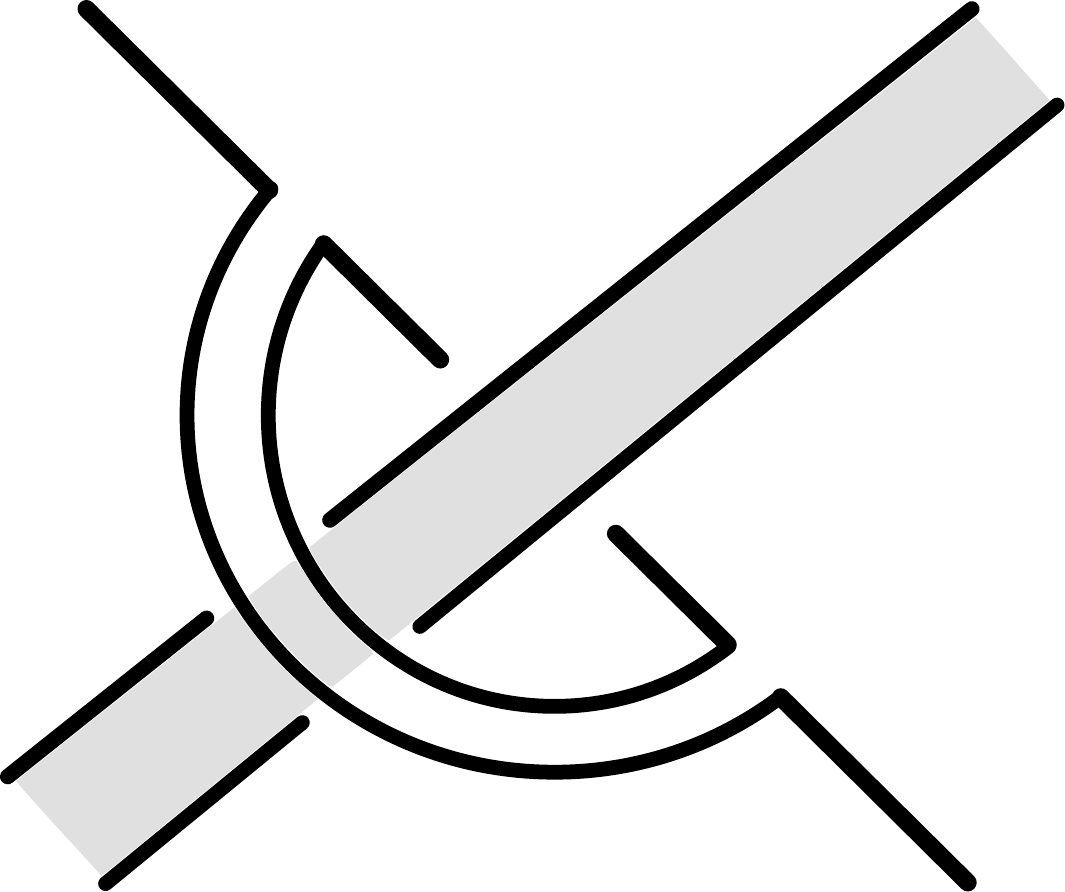}
	\captionsetup{width=.7\textwidth}
	\caption{The ``crossing change'' of Figure~\ref{fig:bandCrossing} achieved by a saddle at the cost of a linking circle for the band.}
	\label{fig:saddle}
\end{figure}

\subsubsection{Sutured manifolds}\label{subsubsec:suturedManifolds}

We review some basics of Gabai's theory of sutured manifolds \cite{Gab83,Gab87a}.

\begin{dfs*}
	A \textit{sutured manifold} $(M,\gamma)$ consists of a compact oriented $3$-manifold $M$ and subsurface $\gamma \subset \partial M$ consisting of pairwise disjoint annuli $A(\gamma)$ and tori $T(\gamma)$. Each annular component of $\gamma$ is equipped with an essential oriented simple closed curve, called a \textit{suture}, in its interior. The union of the sutures of $\gamma$ is denoted $s(\gamma)$. We let $R(\gamma)$ denote the complement of the interior of $\gamma$ in $\partial M$, and we require that each component $V$ of $R(\gamma)$ is given an orientation in such a way that each boundary component of $V$ with the induced orientation, thought of as lying in an annulus in $A(\gamma)$, is in the same homology class as the suture in that annulus. We define $R_+(\gamma)$ (resp. $R_-(\gamma)$) to be the union of the components of $R(\gamma)$ whose orientation agrees (resp. disagrees) with the induced boundary orientation on $R(\gamma)$ from $M$. 

	A sutured manifold $(M,\gamma)$ is \textit{balanced} if $M$ has no closed components, every closed component of $\partial M$ contains a suture, and $\chi(R_+) = \chi(R_-)$ where $\chi$ is the Euler characteristic. In particular, there are no toral components of $\gamma$. 
\end{dfs*}

\begin{examples}
	Let $\Sigma$ be a compact oriented surface with no closed components, and consider the three-manifold $[-1,1] \x \Sigma$. Let $\gamma = [-1,1] \x \partial \Sigma$ with $s(\gamma) = 0 \x \partial \Sigma$ oriented as the boundary of $\Sigma$. It follows that $R_+(\gamma) = 1 \x \partial \Sigma$ while $R_-(\gamma) = -1 \x \partial \Sigma$. This sutured manifold is balanced, and is called a \textit{product sutured manifold}. 

	Let $L$ be a link in a closed oriented connected $3$-manifold $Y$. We make the exterior of $L$ into a balanced sutured manifold by placing two oppositely oriented meridional sutures on each boundary component. We refer to this sutured manifold as \textit{the sutured exterior of $L$}. Another possible choice of sutures which yields a (non-balanced) sutured manifold is to place these meridians on only a subset of the toral boundary components, and label the rest as toral components in $T(\gamma)$. 

	Let $(M,\gamma)$ be a connected sutured manifold. Let $B^3$ be an open $3$-ball embedded in the interior of $M$. Let $(M,\gamma)(1)$ denote the sutured manifold $(M\setminus B^3,\gamma \cup \beta)$ where $\beta$ is a suture on the new $S^2$ boundary component of $M \setminus B^3$. We refer to $(M,\gamma)(1)$ as the sutured manifold $(M,\gamma)$ with a \textit{sutured puncture}. Similarly, $(M,\gamma)(n)$ is $(M,\gamma)$ with $n$ sutured punctures. 
\end{examples}

\begin{dfs*}
	Let $S$ be a compact oriented surface, properly embedded in a sutured manifold $(M,\gamma)$. Then $S$ is called a \textit{decomposing surface} if for every component $\lambda$ of $S \cap \gamma$, one of the following hold: \begin{enumerate}[nolistsep]
		\item[(1)] $\lambda$ is a properly embedded non-separating arc in $\gamma$. 
		\item[(2)] $\lambda$ is a simple closed curve in an annular component $A$ of $\gamma$ in the same homology class as its suture. 
		\item[(3)] $\lambda$ is a homotopically nontrivial curve in a toral component $T$ of $\gamma$, and every other component $\delta$ of $T \cap \partial S$ represents the same homology class in $H_1(T)$. 
	\end{enumerate}We also require that no component of $S$ is a disc $D$ with $\partial D \subset R(\gamma)$ and no component of $\partial S$ bounds a disc in $R(\gamma)$.

	A decomposing surface $S \subset (M,\gamma)$ defines a \textit{sutured manifold decomposition} (or \textit{surface decomposition}) \[
		(M,\gamma) \overset{S}{\rightsquigarrow} (M',\gamma')
	\]where $M' = M \setminus N(S)$ and \begin{align*}
		\gamma' &= (\gamma \cap M') \cup N(S_+' \cap R_-(\gamma)) \cup N(S_-'\cap R_+(\gamma))\\
		R_+(\gamma') &= ((R_+(\gamma) \cap M') \cup S_+') - \mathrm{Int}(\gamma')\\
		R_-(\gamma') &= ((R_-(\gamma) \cap M') \cup S_-') - \mathrm{Int}(\gamma')
	\end{align*}where $S_+'$ (resp. $S_-'$) consists of the components of $\partial N(S) \cap M'$ whose normal vector points out of (resp. into) $M'$, thereby agreeing (resp. disagreeing) with the induced boundary orientation. Note that $s(\gamma)$ is determined by the decomposition $R(\gamma') = R_+(\gamma') \cup R_-(\gamma')$. 
\end{dfs*}

\begin{examples}
	Let $\Sigma$ be a Seifert surface for a null-homologous link $L$ in a closed oriented connected $3$-manifold $Y$. The \textit{sutured exterior of $\Sigma$} is the exterior of $\Sigma$ with a suture running along each boundary component of $\Sigma$, with the boundary orientation. Explicitly, we thicken $\Sigma$ in $Y$ to obtain an embedded copy of $[-1,1] \x \Sigma$, and we delete its interior and let $s(\gamma) = 0 \x \partial \Sigma$. We may view $\Sigma$ as a decomposing surface in the sutured exterior of $L = \partial \Sigma$. The sutured manifold decomposition along $\Sigma$ gives the sutured exterior of $\Sigma$. 
\end{examples}

\begin{dfs*}
	The \textit{complexity} of a compact oriented surface $S$ is \[
		x(S) = \sum_{S_i} \max\{0,-\chi(S_i) \}
	\]where the sum is over the components $S_i$ of $S$. Let $M$ be a compact oriented $3$-manifold, and let $N$ be a subsurface of $\partial M$. We define the \textit{Thurston norm} $x(\alpha)$ of a relative homology class $\alpha \in H_2(M,N)$ to be the minimal complexity of a compact oriented surface $S$ that can be properly embedded in $(M,N)$ in the homology class $\alpha$. Typically $N$ is either empty, the entire boundary, or $\gamma \subset \partial M$. A properly embedded compact oriented surface $(S,\partial S) \subset (M,N)$ is called \textit{norm-minimizing} if it is incompressible and $x(S) = x([S])$ where $[S] \in H_2(M,N)$. 
\end{dfs*}

\begin{df*}
	A sutured manifold $(M,\gamma)$ is \textit{taut} if $M$ is irreducible and $R(\gamma)$ is norm-minimizing in $H_2(M,\gamma)$. 
\end{df*}

\begin{lem}[Lemma 0.4 of \cite{Gab87a}]\label{lem:TautImpliesTaut}
	Let $(M,\gamma) \overset{S}{\rightsquigarrow} (M',\gamma')$ be a sutured manifold decomposition. If $(M',\gamma')$ is taut, then $(M,\gamma)$ is taut.
\end{lem}

\begin{df*}
	A disc $D$ properly embedded in a balanced sutured manifold $(M,\gamma)$ for which $\partial D \cap s(\gamma)$ consists of two points and $\partial D \cap A(\gamma)$ consists of essential arcs in $A(\gamma)$ is called a \textit{product disc}. A product disc is \textit{trivial} if it is isotopic rel boundary to a disc in $\partial M$. An annulus $A$ properly embedded in $(M,\gamma)$ for which one component of $\partial A$ is contained in $R_+(\gamma)$ and the other in $R_-(\gamma)$ is called a \textit{product annulus} if neither component bounds a disc in $R(\gamma)$.
\end{df*}

\begin{lem}[Lemma 3.12 of \cite{Gab83}]\label{lem:decompAlongProdDiscAnnIsTaut}
	Let $(M,\gamma) \overset{S}{\rightsquigarrow} (M',\gamma')$ be decomposition along a product disc or a product annulus. Then $(M,\gamma)$ is taut if and only if $(M',\gamma')$ is taut. 
\end{lem}

\subsection{A sequence of surface decompositions to the Hopf link exterior}\label{subsec:seqOfSurfDecomp}

In section~\ref{subsubsec:suturedFloerHomology}, we will review Juh\'asz's theory of Floer homology for sutured manifolds \cite{Juh06,Juh08}. The sutured manifolds for which sutured Floer homology is defined are balanced sutured manifolds, and the relevant surface decompositions are those called \textit{nice}. Among other things, the result of decomposing a balanced sutured manifold by a nice surface is another balanced sutured manifold.

\begin{df*}[Definition 1.2 of \cite{Juh08} and Definition 3.20 of \cite{MR2653728}]
	A decomposing surface $S$ in a balanced sutured manifold $(M,\gamma)$ is \textit{nice} if $S$ has no closed components, and for each component $V$ of $R(\gamma)$, the set of closed components of $S \cap V$ consists of parallel, coherently oriented, and boundary-coherent simple closed curves. 

	An oriented simple closed curve $C$ in a compact oriented surface $R$ with no closed components is \textit{boundary-coherent} if it is homologically essential in $R$ or it is oriented as the boundary of its interior. The \textit{interior} of $C$ is the connected component of $R \setminus C$ which is disjoint from $\partial R$. 
\end{df*}
\begin{rem}
	Juh\'asz requires an additional condition that a normal vector to $S$ is never parallel to an auxiliary vector field on $\partial M$ used in defining $\Spin^c$ structures. This condition will be satisfied by a nice surface as defined here when placed in generic position with respect to $\partial M$.
\end{rem}

We obtain nice surface decompositions using the following two lemmas. In particular, any \textit{groomed} decomposing surface in a balanced sutured manifold with no closed components is nice. 

\begin{df*}[Definition 0.2 of \cite{Gab87a}]
	A decomposition $(M,\gamma) \overset{S}{\rightsquigarrow} (M',\gamma')$ is \textit{groomed} if both sutured manifolds are taut, no subset of toral components of $S \cup R(\gamma)$ is homologically trivial in $H_2(M)$ and for each component $V$ of $R(\gamma)$, either $S \cap V$ is a union of parallel, coherently oriented, non-separating closed curves or $S \cap V$ is a union of arcs such that for each component $\delta$ of $\partial V$, $|\delta \cap \partial S| = |\langle \delta,\partial S\rangle|$. A decomposing surface $S$ is \textit{groomed} if its corresponding sutured manifold decomposition is groomed. 
\end{df*}

\begin{lem}[Lemma 0.7 of \cite{Gab87a}]\label{lem:thereAreGroomedSurfaces}
	Let $(M,\gamma)$ be a taut sutured manifold. If $y \in H_2(M,\partial M)$ is nonzero, then there exists a properly embedded groomed surface $S$ in $(M,\gamma)$ representing $y$. 
\end{lem}

\begin{lem}\label{lem:throwingAwayClosedCompsOfGroomedIsNice}
	Suppose $(M,\gamma) \overset{S'}{\rightsquigarrow} (M'',\gamma'')$ is a groomed surface decomposition and that $(M,\gamma)$ is balanced. Let $S' = S \cup P$ where $P$ consists of the closed components of $S'$, so that $S$ is nice. If $(M',\gamma')$ is obtained from $(M,\gamma)$ by decomposing along $S$, then $(M',\gamma')$ is taut. 
\end{lem}
\begin{proof}
	Observe that $(M'',\gamma'')$ can be obtained from $(M',\gamma')$ by decomposing along $P$. Since $(M'',\gamma'')$ is taut, it follows that $(M',\gamma')$ is taut as well by Lemma~\ref{lem:TautImpliesTaut}. 
\end{proof}

Although sutured Floer homology is only defined for balanced sutured manifolds, the following lemma will allow us to utilize sutured manifolds with toral components of $\gamma$.

\begin{lem}\label{lem:ToralSuture}
	Let $(M,\gamma')$ be a taut sutured manifold with a toral component $T$ of $\gamma'$. Let $\mu$ be a simple closed curve on $T$ which does not bound a disc in $M$. If $\gamma$ is obtained from $\gamma'$ by replacing $T$ with annular neighborhoods of two oppositely oriented copies of $\mu$, then $(M,\gamma)$ is taut. 
\end{lem}
\begin{proof}
	Certainly $M$ is still irreducible, and $R(\gamma)$ is incompressible because $\mu$ does not bound a disc in $M$. Let $S$ be a compact oriented surface, properly embedded in $M$ with $\partial S \subset \gamma$ and $[S] = [R(\gamma)] \in H_2(M,\gamma)$. Starting from $S$, we construct a new surface that is disjoint from $T$ through moves that preserve the properties $\partial S \subset \gamma$ and $[S] = [R(\gamma)] \in H_2(M,\gamma)$ and that do not increase complexity.

	Inessential components of $\partial S \cap T$ may be capped off with discs, innermost components firsts. Hence we may assume that all components of $\partial S\cap T$ are essential, and we note that each component is a parallel copy of $\mu$. Since $[\partial S \cap T] = [R(\gamma) \cap T] = 0$ in $H_1(T)$, if $\partial S \cap T$ is nonempty, there are a pair of components $\lambda_1,\lambda_2$ of $\partial S \cap T$ which are parallel and oppositely oriented. They cobound an annulus which we may attach to $S$. By iterating this procedure, we obtain the desired surface. Since $[R(\gamma)] = [R(\gamma')] \in H_2(M,\gamma')$ and since $R(\gamma')$ is norm-minimizing, we find that $x(S) \ge x(R(\gamma')) = x(R(\gamma))$ so $R(\gamma)$ is norm-minimizing. 
\end{proof}

\vspace{20pt}

If a sequence of decomposing surfaces is chosen carefully, then any taut sutured manifold can be reduced to a product sutured manifold. 

\begin{thm}[Theorem 4.2 of \cite{Gab83} and Theorem 8.2 of \cite{Juh08}]\label{thm:niceHierarchy}
 	If $(M,\gamma)$ is a taut sutured manifold with nonempty boundary, then there is a sequence of surface decompositions \[
 		(M,\gamma) \overset{S_1}{\rightsquigarrow} (M_1,\gamma_1) \overset{S_2}{\rightsquigarrow} \cdots \overset{S_n}{\rightsquigarrow} (M_n,\gamma_n)
 	\]where $(M_n,\gamma_n)$ is a product sutured manifold. If $(M,\gamma)$ is balanced, then the surfaces may be chosen to be nice. 
\end{thm}

\noindent Gabai defines a notion of \textit{complexity} $C(M,\gamma)$ for taut sutured manifolds $(M,\gamma)$ (Definition 4.3 of \cite{Gab83}) which takes values in a totally ordered set. If $C_1\ge C_2 \ge\cdots$ is a sequence of taut sutured manifold complexities, then Proposition 4.4 of \cite{Gab83} states that the sequence eventually stabilizes. Gabai then shows that if $(M,\gamma)$ is a taut sutured manifold which is not a product sutured manifold, then there is a sequence of surface decompositions resulting in a taut sutured manifold of strictly lower complexity. A sequence of surface decompositions resulting in a product sutured manifold is called a \textit{sutured manifold hierarchy}. We say that a sutured manifold hierarchy for a balanced sutured manifold is \textit{nice} if all of the surface decompositions are nice. 

The following two lemmas provide a way to ensure that a sequence of surface decompositions strictly lowers complexity. The first lemma is an elaboration of Step 1 of the proof of Theorem 4.2 of \cite{Gab83}.

\begin{lem}\label{lem:existenceOfMaximalSetProdDiscs}
	Let $(M,\gamma)$ be a taut sutured manifold. Then any set of disjoint pairwise-nonparallel nontrivial product discs in $(M,\gamma)$ is contained in a finite maximal one. 
\end{lem}
\begin{proof}
	First note that if $(M',\gamma')$ is obtained by decomposing $(M,\gamma)$ along a nontrivial product disc $D$, then any product disc $D'$ in $(M',\gamma')$ can be isotoped through products discs to not intersect $\partial N(D)$. In particular, we may view $D'$ as a disc in $(M,\gamma)$ disjoint from $D$. Note that $D'$ is trivial in $(M',\gamma')$ if and only if $D'$ is either trivial or parallel to $D$ in $(M,\gamma)$. It therefore suffices to show that any sequence of nontrivial product disc decompositions is finite. 

	Since $(B^3,\beta) = ([-1,1] \x D^2,[-1,1] \x \partial D^2)$ has no nontrivial product discs, we may assume that no component of $\partial M$ is a sphere. 
	Let $(M',\gamma')$ be obtained by decomposing $(M,\gamma)$ along a nontrivial product disc $D$. Then $\chi(\partial M') = \chi(M) + 2$. Suppose that a boundary component of $M'$ is a $2$-sphere, from which it follows that $(B^3,\beta)$ is a component of $M'$. Let $Q$ be the boundary component of $M$ which contains $\partial D$. If $\partial D$ were separating in $Q$, then $D$ would be trivial, so $\partial D$ must be non-separating in $Q$. It follows that $Q$ is a torus and that the component of $(M,\gamma)$ containing $Q$ is the sutured exterior of the unknot $U$ in $S^3$. 

	The result follows by induction on $|\chi(\partial M)| + |\partial M|$ for $(M,\gamma)$ with no closed components and no spherical boundary components but with a nontrivial product disc. The base case $|\chi(\partial M)| + |\partial M| = 1$ is true because $(M,\gamma)$ must be $S^3(U)$. Assume the result is true when $|\chi(\partial M)| + |\partial M| \leq k$, and suppose $(M,\gamma)$ has a nontrivial product disc $D$ and satisfies the stated conditions with $|\chi(\partial M)| + |\partial M| = k+1$. Let $(M',\gamma')$ be obtained by decomposing $(M,\gamma)$ along $D$. If no boundary component of $M'$ is a sphere, then since $|\chi(\partial M')| + |\partial M'| \leq k$, the result follows from the induction hypothesis. If a boundary component of $M'$ is a sphere, then $(M,\gamma) = (M'',\gamma'') \cup S^3(U)$ and $(M',\gamma') = (M'',\gamma'') \cup (B^3,\beta)$. Since $|\chi(\partial M'')| + |\partial M''| = k$, the result also follows. 
\end{proof}

\begin{lem}[Lemma 4.12 of \cite{Gab83}]\label{lem:dropComplexity}
	Let $(N,\delta) \overset{S}{\rightsquigarrow} (N',\delta')$ be a decomposition for which $(N,\delta)$ and $(N,\delta')$ are taut. Suppose that every product disc in $(N,\delta)$ is trivial, and that some component of $S$ is not boundary parallel. Let $D$ be a maximal set of disjoint pairwise-nonparallel nontrivial product discs in $(N',\delta')$, and let $(N'',\delta'')$ be obtained by decomposing $(N',\delta')$ along $D$. Then $C(N,\delta) > C(N'',\delta'')$.
\end{lem}

We now prove the main result of this section. Essentially, it is Gabai's proof of superadditivity of genus under band sum \cite{MR895573} but adapted for sutured Floer homology. We make essential use of this result in the argument that Heegaard knot Floer homology and instanton knot Floer homology detect the trivial band (see sections~\ref{subsec:HeegaardDetecting} and \ref{subsec:DetectTrivBandInstanton}). 

\begin{thm}\label{thm:SeqOfSurfDecomp}
	Let $(M_0,\gamma_0)$ be the sutured exterior of $K_b \cup C$ where $K_b$ is obtained from a two-component link by band surgery along a nontrivial band $b$ and $C$ is a linking circle. Then there is a sequence of surface decompositions \[
		(M_0,\gamma_0) \overset{S_1}{\rightsquigarrow} (M_1,\gamma_1) \overset{S_2}{\rightsquigarrow} \cdots \overset{S_n}{\rightsquigarrow} (M_n,\gamma_n)
	\]with the following properties: \begin{enumerate}[itemsep=-0.5ex]
		\item[(1)] each $S_i$ is nice and disjoint from $\partial N(C)$,
		\item[(2)] each $(M_i,\gamma_i)$ is taut, 
		\item[(3)] $(M_n,\gamma_n)$ is a union of a product sutured manifold and a taut sutured manifold $(H,\delta)$ containing $\partial N(C)$ as a boundary component of the following special form. The manifold $H$ is the exterior of a link $J \cup \mu$ in $S^3$ where $J$ is a knot and $\mu$ is a meridian of $J$ with $\partial N(C) = \partial N(\mu)$, and $\delta$ consists of two meridional sutures on $\partial N(\mu)$ and an even number of parallel sutures on $\partial N(J)$.  
	\end{enumerate}
	If the original two-component link is split, then there exists a sequence of surface decompositions with the stated properties where $(H,\delta)$ is the sutured exterior of a two-component Hopf link in $S^3$. 
\end{thm}
\begin{proof}
	Let $(M_0,\gamma_0')$ be the exterior of $K_b \cup C$ with two meridional sutures on $\partial N(K_b)$ and with $\partial N(C)$ a toral component of $\gamma_0'$. Following Gabai's proof of Theorem 1.7 in \cite{Gab87a}, we will construct a sequence of surface decompositions \begin{equation}\label{eq:SeqOfSurfDecomps}
		\tag{$\ast$}
		(M_0,\gamma_0') \overset{S_1}{\rightsquigarrow} (M_1,\gamma_1') \overset{S_2}{\rightsquigarrow} \cdots \overset{S_n}{\rightsquigarrow} (M_n,\gamma_n')
	\end{equation}satisfying the following properties: \begin{enumerate}[nolistsep]
		\item[(1)] Each $S_i$ is nice and disjoint from $\partial N(C)$.
		\item[(2')] Each $(M_i,\gamma_i')$ is taut. 
		\item[(3')] $(M_n,\gamma_n')$ is a union of a product sutured manifold and a taut sutured manifold $(H,\delta')$ containing $\partial N(C)$ as a component of $\delta'$ of the following special form. The manifold $H$ is the exterior of a link $J\cup \mu$ in $S^3$ where $J$ is a knot and $\mu$ is a meridian of $J$ with $\partial N(C) = \partial N(\mu)$, and $\delta'$ consists of $\partial N(C)$ as a toral component and an even number of parallel sutures on $\partial N(J)$. 
	\end{enumerate}
	Given the sequence~(\ref{eq:SeqOfSurfDecomps}), the desired sequence in the proposition may be obtained in the following way. Let $\gamma_i$ be obtained from $\gamma_i'$ by replacing the toral component $\partial N(C)$ of $\gamma_i'$ with two oppositely oriented meridional sutures on $\partial N(C)$. The meridian of $\partial N(C)$ does not bound a disc in $M_i$ because it does not bound a disc in $M_0$. Thus each $(M_i,\gamma_i)$ is taut by Lemma~\ref{lem:ToralSuture}. It follows that the sequence \[
		(M_0,\gamma_0) \overset{S_1}{\rightsquigarrow} (M_1,\gamma_1) \overset{S_2}{\rightsquigarrow} \cdots \overset{S_n}{\rightsquigarrow} (M_n,\gamma_n)
	\]satisfies properties (1), (2), and (3) as desired. 

	To construct the sequence~(\ref{eq:SeqOfSurfDecomps}), we hand-pick the first surface $S_1$ and use an inductive procedure to produce the rest. Let $S_1$ be a Seifert surface for $K_b$ disjoint from $C$ which has minimal genus among all such Seifert surfaces. We may assume that $S_1$ contains the band $b$. Indeed, let $D$ be a disc in $S^3$ whose boundary is $C$ and which intersects the band $b$ along an arc. The intersection $S_1 \cap D$ consists of simple closed curves and an arc whose endpoints coincide with the arc $b \cap D$. By an isotopy of $S_1$, we may assume that the arc of $S_1 \cap D$ coincides with $b \cap D$. Now that $b \cap D$ lies in $S_1$, by a further isotopy of $S_1$ we may assume that all of $b$ lies in $S_1$. Decomposing along $S_1$ gives a taut sutured manifold $(M_1,\gamma_1')$. 

	Now suppose we have already constructed the sequence up to $(M_k,\gamma_k')$ so that it satisfies (1) and (2'). Let $S_{k+1}$ be a maximal set of disjoint pairwise-nonparallel nontrivial product discs in $(M_k,\gamma_k')$, which exists by Lemma~\ref{lem:existenceOfMaximalSetProdDiscs}. This extension of the sequence \[
		(M_k,\gamma_k') \overset{S_{k+1}}{\rightsquigarrow} (M_{k+1},\gamma_{k+1}')
	\]still satisfies (1) and (2'), and now every product disc in $(M_{k+1},\gamma_{k+1}')$ is trivial. 

	Now let $(H,\delta')$ be the component of $(M_{k+1},\gamma_{k+1}')$ containing $\partial N(C)$. Note that $\partial H \neq \partial N(C)$ because $H$ is embedded in the exterior of $K_b \cup C$; filling in $H$ with $N(C)$ would give a closed $3$-dimensional submanifold in the exterior of $K_b$ which is impossible. Let $T = \partial H \setminus \partial N(C)$ so that $\partial H$ is the disjoint union of the closed oriented surface $T$ with the torus $\partial N(C)$. There are two cases: \begin{enumerate}
		\item[(a)] The map $H_1(T) \to H_1(H)$ is not injective. 

		Then there exists a nonzero $y \in H_2(H,\partial H)$ such that $[y \cap \partial N(C)] = 0$ and $[y \cap T] \neq 0$. By Lemma~\ref{lem:thereAreGroomedSurfaces}, there exists a groomed surface $S_{k+2}'$ for which $[S_{k+2}'] = y \in H_2(H,\partial H)$. Let $S_{k+2}' = S_{k+2} \cup P$ where $P$ consists of the closed components of $S_{k+2}'$. The decomposing surface $S_{k+2}$ is nice, and the result of decomposing along $S_{k+2}$ is taut by Lemma~\ref{lem:throwingAwayClosedCompsOfGroomedIsNice}. Note that by the definition of a decomposing surface, $S_{k+2} \cap \partial N(C)$ consists of parallel essential simple closed curves in the same homology class of $\partial N(C)$ because $\partial N(C)$ is a toral component of $\gamma_{k+1}'$. Since $[S_{k+2} \cap \partial N(C)] = 0$, it follows that $S_{k+2} \cap \partial N(C) = \emptyset$. Thus we can extend our sequence by \[
			(M_{k+1},\gamma_{k+1}') \overset{S_{k+2}}{\rightsquigarrow} (M_{k+2},\gamma_{k+2}')
		\]and it still satisfies (1) and (2'). 

		We now return to the inductive step of the argument and repeat. If we always return to case (a), then the complexities $C(M_i,\gamma_i')$ constructed in the infinite sequence are decreasing and never stabilize by Lemma~\ref{lem:dropComplexity} because every other decomposing surface is a maximal set of disjoint pairwise-nonparallel nontrivial product discs. Since every infinite decreasing sequence of sutured manifold complexities must stabilize by Proposition 4.4 of \cite{Gab83}, we eventually must reach case (b). 
		\item[(b)] The map $H_1(T) \to H_1(H)$ is injective. 

		\item[$-$] We will show that $T$ is a torus. Recall that ``half lives, half dies'': the kernel of the map $\iota\colon H_1(\partial H;\Q) \to H_1(H;\Q)$ has half the rank of $H_1(\partial H;\Q)$. Also note that the rank of $\iota$ is at least the dimension of $H_1(T)$ because $H_1(T) \to H_1(H)$ is injective by assumption. With $\Q$-coefficients, we have \begin{align*}
			\dim H_1(T) + \dim H_1(\partial N(C)) &= \dim H_1(\partial H)\\
			&= \rank(\iota) + \dim(\ker \iota)\\
			&\ge \dim H_1(T) + \textstyle\frac12 \dim H_1(\partial H)
		\end{align*}from which it follows that $\dim H_1(T) \leq 2$. Since $(H,\delta')$ is taut, no component of $\partial H$ is a $2$-sphere. Thus $T$ is a torus because $T$ is a nonempty closed oriented surface with no sphere components. 

		Since $H \subset M_0$, we may view $T$ as a torus embedded in the exterior of $K_b \cup C$. In fact, the torus $T$ will be disjoint from the first decomposing surface $S_1$ so $T$ is disjoint from the band $b$. The torus $T \subset S^3$ separates $S^3$ into two $3$-manifolds with boundary. One side contains $C$; deleting a regular neighborhood of $C$ from this side gives $H$. The other side contains $S_1$ and in particular $K_b$ and the band $b$. 

		\item[$-$] We will show that the side of $T$ containing $K_b$ is a solid torus. Let $D$ be a disc in $S^3$ with boundary the linking circle $C$ which intersects $b$ in a single arc. As $T$ is disjoint from the band, $T \cap D$ consists of simple closed curves in $D \setminus (D \cap b)$. If a component of $T \cap D$ is inessential in $D \setminus (D \cap b)$, let $\lambda$ be an innermost one, and compress $T$ along the disc that $\lambda$ bounds in $D\setminus (D \cap b)$. If the result is connected, it is a splitting sphere for $K_b \cup C$ so $b$ is trivial. As $b$ is nontrivial, the result of the compression is disconnected. One component must be a sphere in the complement of $K_b \cup C$ so it bounds a ball disjoint from $K_b \cup C$. Thus we may remove the circle of intersection $\lambda$ via an isotopy of $T$. By iterating this procedure, we may assume there are no inessential circles in $T \cap D \subset D\setminus(D \cap b)$. 

		The remaining components are all essential and parallel to the boundary. There are also an odd number of such components because $T$ separates $K_b$ from $C$. Let $\lambda$ be the innermost such component and consider the disc $D_\lambda$ that $\lambda$ bounds in $D$. Note that $D_\lambda$ intersects the band along an arc. If $\lambda$ is inessential in $T$, then the union of a disc in $T$ with $D_\lambda$ is an embedded sphere in $S^3$ which meets the band $b$ in an arc, contradicting the fact that $b$ is nontrivial. Since the band is nontrivial, $\lambda$ must be essential in $T$. Thus $D_\lambda$ is an essential compressing disc on the side of $T$ containing $K_b$. It follows that the side of $T$ containing $K_b$ is a solid torus, and the side of $T$ containing $C$ is a knot exterior. The other components of $T \cap D$ thought of as circles in $T$ must be parallel copies of $\lambda$. Since $\lambda$ is the meridian of the knot exterior, we see that $H$ is indeed homeomorphic to the exterior of a link of the form $J \cup \mu$ where $\mu$ is a meridian of a knot $J$. Furthermore, $\partial N(\mu) = \partial N(C)$ is a toral component of $\delta'$ and the sutures on the other boundary component $T = \partial N(J)$ consists of an even number of parallel sutures. Indeed, because each decomposing surface has no closed components, we know that $T$ is not a toral component of $\delta'$, and since $(H,\delta')$ is taut, all sutures on $T$ are essential. 
		\item[$-$] Finally, choose a nice sutured manifold hierarchy which exists by Theorem~\ref{thm:niceHierarchy} for the taut sutured manifold $(M_{k+1},\gamma_{k+1}')\setminus (H,\delta')$. Extend the sequence by applying this hierarchy to $(M_{k+1},\gamma_{k+1}')$. The end result is the union of a product sutured manifold and $(H,\delta')$. 
	\end{enumerate}
	In the case that the original two-component link is split, we first show that the parallel sutures on $\partial N(J)$ are meridional following Gabai's proof of superadditivity of genus under band sum \cite{MR895573}. Given the sequence \[
		(M_0,\gamma_0) \overset{S_1}{\rightsquigarrow} (M_1,\gamma_1) \overset{S_2}{\rightsquigarrow} \cdots \overset{S_n}{\rightsquigarrow} (M_n,\gamma_n)
	\]with the properties stated in the proposition, we fill the torus boundary component $\partial N(C)$ of each $(M_i,\gamma_i)$ with the zero slope filling to obtain the sequence of surface decompositions \[
		(N_0,\gamma_0'') \overset{S_1}{\rightsquigarrow} (N_1,\gamma_1'') \overset{S_2}{\rightsquigarrow} \cdots \overset{S_n}{\rightsquigarrow} (N_n,\gamma_n'').
	\]By the light bulb trick (Corollary~\ref{cor:ZeroFillingMeridian}), the manifold $N_n$ is the disjoint union of a product sutured manifold and a solid torus, and the sutures on this solid torus bound discs if and only if the parallel sutures on $\partial N(J)$ were originally meridional. If these sutures were not meridional, then $(N_n,\gamma_n'')$ is taut, so by Lemma~\ref{lem:TautImpliesTaut} all of the manifolds $(N_i,\gamma_i'')$ are taut. In particular, $N_0$ is irreducible. However, $N_0$ is the zero-filling of $\partial N(C)$ in the exterior of $K_b \cup C$ which again by the light bulb trick (Corollary~\ref{cor:ZeroSurgeryLinkingCircle}) is reducible. Thus the parallel sutures on $\partial N(J)$ are meridional. 

\begin{figure}[!ht]
	\centering
	\labellist
	\pinlabel $\partial N(C)$ at 480 470
	\pinlabel $\partial N(J)$ at 200 700
	\pinlabel $A$ at 30 160
	\endlabellist
	\includegraphics[width=.3\textwidth]{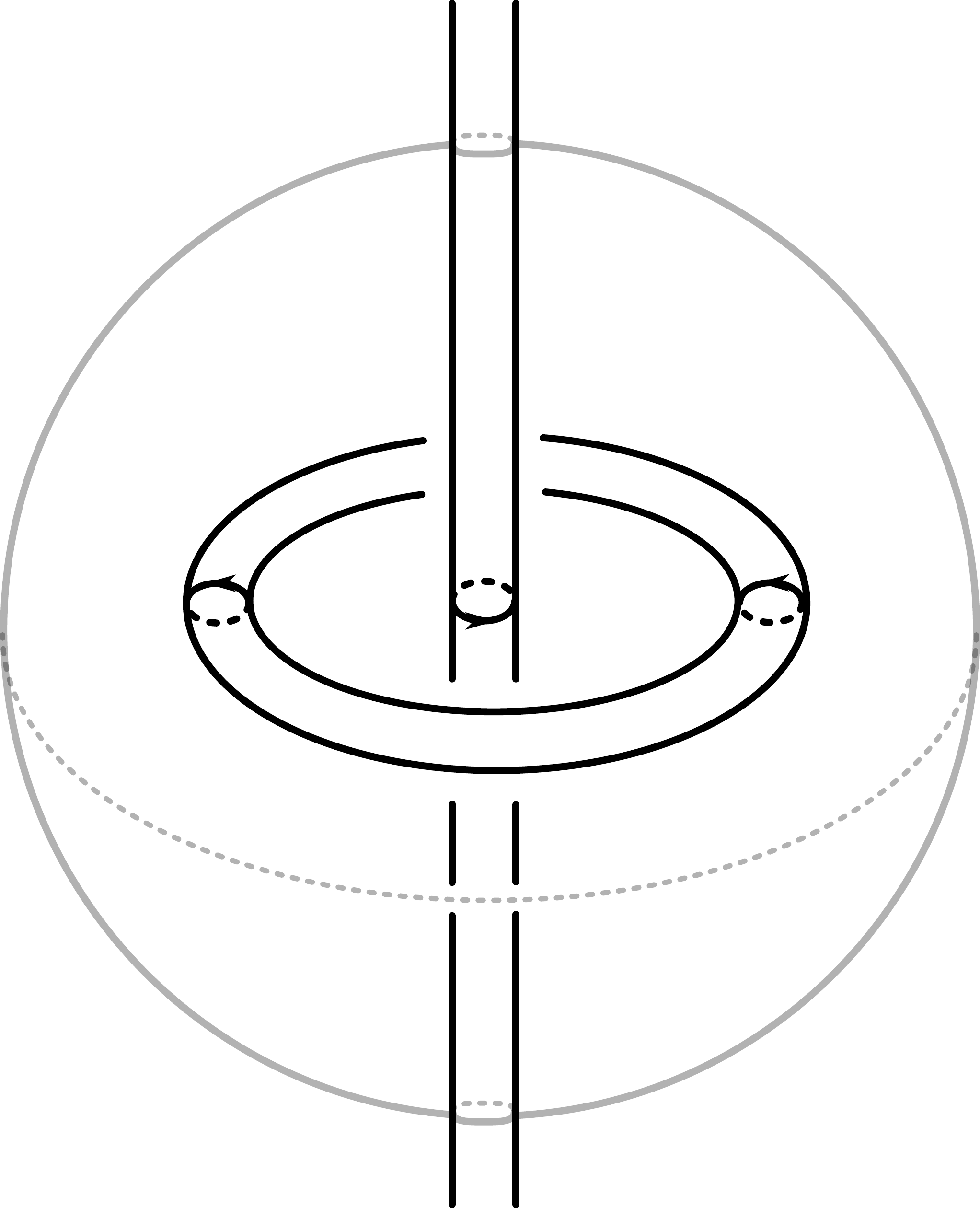}
	\captionsetup{width=.7\textwidth}
	\caption{A product annulus $A$.}
	\label{fig:prodAnnulus}
\end{figure}

	Since the sutures on $\partial N(J)$ are meridional, there is a product annulus $A$ whose boundary lies on $\partial N(J)$ for which the decomposition of $(H,\delta)$ along $A$ yields the union of the sutured exterior of a Hopf link and the exterior of $J$ with an even number of meridional sutures (Figure~\ref{fig:prodAnnulus}). We continue the sequence of decompositions with a nice hierarchy which turns the exterior of $J$ with its sutures into a product sutured manifold using Theorem~\ref{thm:niceHierarchy}. 
\end{proof}

%%%%%%%%%%%%%%%%%%%%%%%%%%%%%%%%%%%%%%%
%%%% Heegaard knot Floer homology %%%%%
%%%%%%%%%%%%%%%%%%%%%%%%%%%%%%%%%%%%%%%

\section{Heegaard Floer homology}

In section~\ref{subsec:HeegaardInvariance}, we show that the knot Floer homology of a band sum is invariant under adding full twists to the band. The proof combines Zemke's ribbon concordance argument \cite{Zem19a} with a surgery exact triangle. Then in section~\ref{subsec:HeegaardDetecting}, we show that the knot Floer homology of a band sum detects the trivial band. We again use Zemke's ribbon concordance argument, but the main argument utilizes a suitable minus version of sutured Floer homology and the sequence of surface decompositions constructed in Theorem~\ref{thm:SeqOfSurfDecomp}. We first review sutured Floer homology \cite{Juh06,Juh08} and Zemke's functoriality of link Floer homology \cite{Zem19c,Zem19b}. 

\subsection{Preliminaries}\label{subsec:HeegaardPreliminaries}

The simplest version of the knot Floer homology \cite{MR2065507,MR2704683}, referred to as the hat version, takes the form of a finite-dimensional bigraded vector space $\HFKhat(K)$ over $\F_2$. We will also make use of the minus version $\HFK^-(K)$ taking the form of a finitely-generated bigraded $\F_2[U]$-module. Both arise from suitable versions of Lagrangian intersection theory in a symmetric product of a Heegaard diagram for the knot. 

The two gradings are called the Maslov and Alexander gradings. The Maslov grading $d$ may be thought of as the homological grading since the differential of the chain complex drops Maslov grading by one, while the Alexander grading $s$ may be thought of the splitting of knot Floer homology along relative $\Spin^c$ structures (see section~\ref{subsubsec:suturedFloerHomology}). The splitting along bigradings is written \[
	\HFKhat(K) = \bigoplus_{(s,\,d) \in \Z^2} \HFKhat_{\,d}(K,s).
\]We also write $\HFKhat(K,s) = \bigoplus_d \HFKhat_d(K,s)$. For knots in $S^3$ for example, knot Floer homology detects the genus of a knot (Theorem 1.2 of \cite{MR2023281}), and it also detects whether the knot is fibered (Theorem 1.1 of \cite{MR2357503}). Specifically, the largest Alexander grading $s$ for which $\HFKhat(K,s)$ is nonzero is the Seifert genus of $K$, and $\HFKhat(K,g(K))$ is $1$-dimensional if and only if the knot is fibered. See section~\ref{subsubsec:suturedFloerHomology} for the definition of knot Floer homology. 

In section~\ref{subsubsec:suturedFloerHomology}, we review sutured Floer homology \cite{Juh06,Juh08}, a version of Heegaard Floer homology for balanced sutured manifolds, encompassing the hat version for closed oriented $3$-manifolds and the hat version of knot Floer homology and its generalization to links. Then in section~\ref{subsubsec:functorialityLinkFloerHomology}, we review Zemke's functoriality for a minus version of link Floer homology \cite{Zem19c,Zem19b}. 

\subsubsection{Sutured Floer homology}\label{subsubsec:suturedFloerHomology}

We recall the basics of Juh\'asz's sutured Floer homology \cite{Juh06,Juh08}. For the basics of sutured manifolds, see section~\ref{subsubsec:suturedManifolds}. 

\begin{df*}[sutured Heegaard diagram]
	Let $\Sigma$ be a closed oriented surface with collections $\bo\alpha = \{\alpha_1,\ldots,\alpha_k\}$ and $\bo\beta = \{\beta_1,\ldots,\beta_k\}$ of pairwise disjoint simple closed curves along with a finite set of basepoints $\bo p = \{p_1,\ldots,p_n\}$ disjoint from $\bigcup{\bo\alpha}\cup\bigcup\bo\beta$. Assume that each curve in $\bo\alpha$ intersects each curve in $\bo\beta$ transversely. If there is a basepoint in every component of $\Sigma\setminus\bigcup\bo\alpha$ and in every component of $\Sigma\setminus\bigcup\bo\beta$, then we say that $(\Sigma,\bo\alpha,\bo\beta,\bo p)$ is a \textit{balanced sutured Heegaard diagram}. 

	Consider the following construction of a balanced sutured manifold from a balanced sutured Heegaard diagram. Delete a regular neighborhood of the set of basepoints $\bo p$ to obtain a compact oriented surface with boundary $\Sigma'$. Attach $3$-dimensional $2$-handles to $[-1,1] \x \Sigma'$ along $\{-1\} \x \alpha_i$ and $\{1\} \x \beta_i$ and call the resulting $3$-manifold $M$. View $\{0\} \x \partial \Sigma'$ as an oriented $1$-manifold $\gamma$ in $\partial M$, oriented as the boundary of $\Sigma'$. Then $(M,\gamma)$ is a balanced sutured manifold by Proposition 2.9 of \cite{Juh06}, and we say that $(\Sigma,\bo\alpha,\bo\beta,\bo p)$ is a balanced diagram for $(M,\gamma)$. By Proposition 2.14 of \cite{Juh06}, every balanced sutured manifold has a balanced diagram. 
\end{df*}

\begin{rem*}
	Juh\'asz's definition of a balanced diagram is the diagram with boundary $(\Sigma',\bo\alpha,\bo\beta)$ obtained from $(\Sigma,\bo\alpha,\bo\beta,\bo p)$ by deleting a regular neighborhood of $\bo p$. Clearly $(\Sigma,\bo\alpha,\bo\beta,\bo p)$ can be recovered from $(\Sigma',\bo\alpha,\bo\beta)$ by simply capping off each boundary component of $\Sigma'$ with a disc with a basepoint. 

	Note that the basepoints $\bo p$ are in one-to-one correspondence with the sutures on $(M,\gamma)$ when $(\Sigma,\bo\alpha,\bo\beta,\bo p)$ is a balanced sutured Heegaard diagram for $(M,\gamma)$. 
\end{rem*}

\begin{dfs*}[Whitney discs, domains, admissibility]
	Fix a balanced diagram $(\Sigma,\bo\alpha,\bo\beta,\bo p)$. Let $\mathbf{T}_\alpha,\mathbf{T}_\beta$ be tori $\alpha_1 \x \cdots \x \alpha_k,\beta_1 \x \cdots \x \beta_k$ in the $k$-fold symmetric product $\Sym\!{}^k(\Sigma)$ of $\Sigma$. The intersection points $\mathbf{x} \in \mathbf{T}_\alpha \cap \mathbf{T}_\beta$ are called \textit{generators}. For generators $\mathbf{x},\mathbf{y}$, a \textit{Whitney disc from $\mathbf{x}$ to $\mathbf{y}$} is a continuous map $\phi\colon D^2 \to \Sym\!{}^k(\Sigma)$ for which \begin{align*}
		\phi(+i) = \mathbf{y} & \qquad\quad \phi(\{z \in \partial D^2 \:|\: \mathrm{Re}(z) \ge 0\}) \subset \mathbf{T}_\alpha\\
		\phi(-i) = \mathbf{x} &\qquad\quad \phi(\{z \in \partial D^2 \:|\: \mathrm{Re}(z) \leq 0\}) \subset \mathbf{T}_\beta
	\end{align*}and $\pi_2(\mathbf{x},\mathbf{y})$ denotes the set of such Whitney discs up to homotopy through Whitney discs. A \textit{domain} $D = \sum_{i=1}^n a_iD_i$ with $a_i \in \Z$ is a formal sum of the closures $D_1,\ldots,D_n$ of the connected components of $\Sigma \setminus (\bigcup\bo\alpha\cup\bigcup\bo\beta)$. A domain is \textit{nonnegative} if all coefficients are nonnegative. The boundary $\partial D$ of a domain $D$, thought of as a $2$-chain, is a sum of arcs where each arc lies in either $\bigcup \bo\alpha$ or $\bigcup\bo\beta$ so according to this division we write $\partial D = \partial_\alpha D + \partial_\beta D$. We say that $D$ is a \textit{domain from $\mathbf{x}$ to $\mathbf{y}$} if $\partial(\partial_\alpha D)) = \mathbf{y} - \mathbf{x}$ and $\partial(\partial_\beta D)) = \mathbf{x} - \mathbf{y}$. For a Whitney disc $\phi$ and a point $p \in \Sigma \setminus (\bigcup\bo\alpha \cup \bigcup\bo\beta)$, we define $n_p(\phi)$ to be the algebraic intersection number of the image of $\phi$ with $p \x \Sym\!{}^{k-1}(\Sigma)$. Similarly, define $n_p(D)$ to be the coefficient of $D_i$ in $D$ where $p \in D_i$. The association $\phi\mapsto \sum_i^n n_{p_i}(\phi) D_i$ where $p_i \in D_i$ gives a map $\pi_2(\mathbf{x},\mathbf{y}) \to D(\mathbf{x},\mathbf{y})$. 

	A domain $P$ is called \textit{periodic} if $\partial_\alpha P$ consists of a sum of $\alpha$ curves and $\partial_\beta P$ consists of a sum of $\beta$ curves. A balanced diagram $(\Sigma,\bo\alpha,\bo\beta,\bo p)$ is called \textit{admissible} if every nonnegative periodic domain $P$ with $n_p(P) = 0$ for all $p \in \bo p$ is zero. By Corollaries 3.12 and 3.15 of \cite{Juh06}, any balanced diagram defining a balanced sutured manifold with trivial second homology is admissible, and every balanced sutured manifold in general has an admissible diagram. 
\end{dfs*}

\begin{df*}[relative $\Spin^c$ structures]
	Let $(M,\gamma)$ be a connected balanced sutured manifold. Fix a nowhere-vanishing vector field $v_0$ on $\partial M$ which points into (resp. out of) $M$ along $R_-(\gamma)$ (resp. $R_+(\gamma)$) and which is the gradient of $[-1,1] \x s(\gamma) \to [-1,1]$ on $\gamma = [-1,1] \x s(\gamma)$, where $\{\pm1\} \x s(\gamma) \subset \partial R_{\pm}(\gamma)$. 

	Let $v$ and $w$ be nowhere-vanishing vector fields on $M$ that agree with $v_0$ on $\partial M$. They are \textit{homologous} if they are homotopic rel $\partial M$ through nowhere-vanishing vector fields on the complement of an open ball in the interior of $M$. A homology class of such a vector field is called a relative $\Spin^c$ structure on $(M,\gamma)$. It turns out that the set of relative $\Spin^c$ structures $\Spin^c(M,\gamma)$ is an affine space over $H^2(M,\partial M;\Z) = H_1(M)$. 

	Associated to each generator $\mathbf{x}$ in a balanced diagram $(\Sigma,\bo\alpha,\bo\beta,\bo p)$ for $(M,\gamma)$ is a relative $\Spin^c$ structure $\fk{s}(\mathbf{x})$ (see Definition 4.5 of \cite{Juh06}). Since $\Spin^c(M,\gamma)$ is affine over $H_1(M)$, there is a well-defined difference $\fk{s}(\mathbf{x}) - \fk{s}(\mathbf{y})$ in $H_1(M)$. This homology class is denoted $\epsilon(\mathbf{x},\mathbf{y}) \in H_1(M)$ and may be constructed as follows. Choose paths $a\colon I \to \mathbf{T}_\alpha,b\colon I \to \mathbf{T}_\beta$ with $\partial a = \partial b = \mathbf{x} - \mathbf{y}$. Then $a - b$ viewed as a $1$-cycle in $\Sigma' = \Sigma \setminus N(\bo p)$ represents $\epsilon(\mathbf{x},\mathbf{y})$, viewing $\Sigma' \subset M$. 
\end{df*}

\begin{rem}\label{rem:SpinCSutures}
	Let $(\Sigma,\bo\alpha,\bo\beta,\bo p)$ be a balanced diagram for $(M,\gamma)$. Recall that the basepoints $\bo p$  are in one-to-one correspondence with the sutures $s(\gamma)$. Let $[\sigma_p] \in H_1(M)$ denote the homology class of the suture $\sigma_p$ corresponding to $p \in \bo p$. Then if $D \in D(\mathbf{x},\mathbf{y})$ is a domain from $\mathbf{x}$ to $\mathbf{y}$, then \[
		\epsilon(\mathbf{x},\mathbf{y}) = \sum_{p \in \bo p} n_p(D)\:[\sigma_p].
	\]In particular, if $n_p(D) = 0$ for all $p \in \bo p$, then $\epsilon(\mathbf{x},\mathbf{y}) = 0$. 
\end{rem}

For a suitable family of almost complex structures $J_s$ on $\Sym\!{}^k(\Sigma)$, the moduli space $\sr M(\phi)$ of $J_s$-holomorphic representatives of $\phi$ is a smooth manifold of dimension $\mu(\phi)$, the Maslov index of $\phi$. When $\mu(\phi) = 1$, the quotient $\widehat{\sr M}(\phi)$ of $\sr M(\phi)$ by the free $\R$-action given by reparametrization is a compact $0$-dimensional manifold. See \cite{Juh06,MR2113019} for details. 

\begin{df*}[sutured Floer homology]
	Let $(\Sigma,\bo\alpha,\bo\beta,\bo p)$ be an admissible balanced diagram for $(M,\gamma)$. For each relative $\Spin^c$ structure $\fk{s}$ on $(M,\gamma)$, let $\SFC(\Sigma,\bo\alpha,\bo\beta,\bo p,\fk{s})$ be the vector space over $\F_2$ with basis the set of generators $\mathbf{x} \in \mathbf{T}_\alpha \cap \mathbf{T}_\beta$ with $\fk{s}(\mathbf{x}) = \fk{s}$ equipped with the differential \[
		\partial\mathbf{x} = \sum_{\mathbf{y} \in \mathbf{T}_\alpha\cap\mathbf{T}_\beta} \sum_{\substack{\phi \in \pi_2(\mathbf{x},\mathbf{y}),\:\mu(\phi) = 1\\n_p(\phi) = 0 \:|\: p \in \bo p}} \#\widehat{\sr M}(\phi) \cdot \mathbf{y}.
	\]Any $\mathbf{y}$ appearing in $\partial \mathbf{x}$ with nonzero coefficient satisfies $\fk{s}(\mathbf{y}) = \fk{s}$ by Remark~\ref{rem:SpinCSutures}. The formula defines a differential whose homology is independent of the choice of $J_s$ and the choice of admissible balanced diagram by Theorems 7.2 and 7.5 of \cite{Juh06}. Denote the homology by $\SFH(M,\gamma,\fk{s})$, and set \[
		\SFC(\Sigma,\bo\alpha,\bo\beta,\bo p) = \bigoplus_{\fk{s}\in \Spin^c(M,\gamma)} \SFC(\Sigma,\bo\alpha,\bo\beta,\bo p,\fk{s}) \qquad \SFH(M,\gamma) = \bigoplus_{\fk{s} \in \Spin^c(M,\gamma)} \SFH(M,\gamma,\fk{s}).
	\]
\end{df*}

\begin{example}
	For a knot $K$ in $S^3$, recall that $S^3(K)$ is the exterior of $K$ equipped with two meridional sutures. Knot Floer homology $\HFKhat(K)$ is defined to be the sutured Floer homology of $S^3(K)$. The relative $\Spin^c$ structures form an affine space over $H_1(S^3 \setminus K) = \Z$. It turns out that there is a symmetry among the groups $\SFH(S^3(K),\fk{s})$: there is a unique $\Spin^c$ structure $\fk{s}_0$ for which $\dim \SFH(S^3(K),\fk{s}_0 - n) = \dim\SFH(S^3(K),\fk{s}_0 + n)$ for all $n$. The relative $\Z$-grading determined by the splitting along $\Spin^c$ structures and choice of generator of $H_1(S^3\setminus K)$ is upgraded to an absolute $\Z$-grading called the Alexander grading. 

	In this situation, there is also an absolute $\Z$-grading, called the Maslov grading, on each of the complexes $\SFC(\Sigma,\bo\alpha,\bo\beta,\bo p)$ with respect to which the differential is homogeneous of degree $-1$. This may be thought of as the homological grading of the chain complex. The knot Floer homology group in Alexander grading $s$ and Maslov grading $d$ is denoted $\HFKhat_d(K,s)$. We will only use the Maslov grading in the context of proving invariance of knot Floer homology under full twists of the band. For more general sutured manifolds and for the gauge-theoretic invariants of sutured manifolds, an absolute $\Z$-grading is absent; there is often only a relative $\Z/2$ or $\Z/4$ grading. 
\end{example}

The following definition and theorem exhibits the behavior of sutured Floer homology under nice surface decompositions (see section~\ref{subsubsec:suturedManifolds}) \cite{Juh08}.

\begin{df*}[outer $\Spin^c$ structure]
	Let $S$ be a decomposing surface in a balanced sutured manifold $(M,\gamma)$. A $\Spin^c$ structure $\fk{s}$ is \textit{outer with respect to $S$} if there is a unit vector field $v$ on $M$ whose homology class is $\fk{s}$ and $v_p \neq (\nu_S)_p$ for every $p \in S$ where $\nu_S$ is a unit normal vector field of $S$ with respect to a Riemannian metric on $M$. The set of outer $\Spin^c$ structures is denoted $O_S$. 
\end{df*}

\begin{thm*}[Theorem 1.3 of \cite{Juh08}]
	Let $(M,\gamma) \overset{S}{\rightsquigarrow} (M',\gamma')$ be a nice surface decomposition of balanced sutured manifolds. Then \[
		\SFH(M',\gamma') \cong \bigoplus_{\fk{s} \in O_S} \SFH(M,\gamma,\fk{s}).
	\]
\end{thm*}

\subsubsection{Functoriality and ribbon concordances}\label{subsubsec:functorialityLinkFloerHomology}

Knot Floer homology, and its generalization to links, is a functorial theory. We will use Zemke's formulation \cite{Zem19c,Zem19b} since we need cobordism maps for the minus version of link Floer homology. These maps are an extension of the maps that Juh\'asz defines for the hat version \cite{MR3519484,MR4080483}. 

\begin{dfs*}
	If $L \subset Y$ and $L' \subset Y'$ are oriented links in closed oriented $3$-manifolds $Y,Y'$, then a link cobordism from $L$ to $L'$ is a compact oriented $4$-manifold $W$ whose boundary is equipped with an identification $\partial W = -Y \amalg Y'$ along with a properly embedded compact oriented surface $F \subset W$ for which $\partial F = -L \amalg L'$. For functoriality of knot Floer homology, we need \textit{decorated} link cobordisms between \textit{multi-based} links.

	A \textit{multi-based link} $\sr L = (L,\mathbf{w},\mathbf{z})$ is an oriented link $L \subset Y$ in a closed oriented $3$-manifold $Y$ with two disjoint sets of basepoints $\mathbf{w},\mathbf{z}$ on $L$ such that each component of $L$ has at least two basepoints, and adjacent basepoints are not in the same set. A \textit{decorated link cobordism} from a multi-based link $\sr L \subset Y$ to another $\sr L'\subset Y'$ is a link cobordism $(W,F)$ with a properly embedded $1$-manifold $\sr A \subset F$ dividing $F$ into two subsurfaces $F_{\mathbf{w}}$ and $F_{\mathbf{z}}$ meeting along $\sr A$ so that $\mathbf{w},\mathbf{w}' \subset F_{\bo w}$ and $\mathbf{z},\mathbf{z}' \subset F_{\bo z}$. Two cobordisms between the same multi-based links are equivalent if they are diffeomorphic rel boundary preserving the decorations. 
\end{dfs*}

\begin{examples}
	Let $L,L'$ be links, and suppose $F$ is a link cobordism with $\partial F = -L \amalg L'$ where $F$ is a disjoint union of annuli. We put basepoints on $L$ and $L'$ in such a way that each component contains exactly two basepoints. Then choose an arc on each annular component of $F$ joining the $\bo z$ basepoints and let $F_{\bo z}$ be a regular neighborhood of the arcs. The $1$-manifold $\sr A \subset F$ will consist of two parallel copies of the chosen arc in each annular component of $F$. The choice of decorations is not canonical, as it depends on the initial arc joining the $\bo z$ basepoints. This decoration on a concordance between knots is used in \cite{MR3590358}.

	Let $\sr L \subset Y$ be a multi-based link, and consider the cobordism $W$ obtained by attaching a $4$-dimensional $2$-handle to $[0,1] \x Y$ along $\{1\} \x K \subset \{1\} \x Y$ where $K \subset Y$ is a framed knot disjoint from $L$. Then there is a natural link cobordism $F = [0,1] \x L \subset [0,1] \x Y \subset W$ consisting of disjoint annuli. We may choose our decorations in the way described in the previous example by taking $F_{\bo z}$ to be a regular neighborhood of the collection of arcs $[0,1] \x \bo{z}$ within $F$. 
\end{examples}

Associated to each multi-based link $\sr L$ in $Y$ is a Floer homology group $\HFLhat(\sr L,Y)$ and associated to each decorated link cobordism $(W,F)$ together with $\fk{s} \in \Spin^c(W)$ is a functorial cobordism map \[
	\HFLhat(\sr L,Y,\fk{s}|_Y) \to \HFLhat(\sr L',Y',\fk{s}|_{Y'}).
\]When the $\Spin^c$ structure on $Y$ is unique, as it is for $Y = S^3$, we omit it from our notation. A cobordism map without a specified $\Spin^c$ structure is just the sum over all $\Spin^c$ structures. The Floer homology group is just the sutured Floer homology of the exterior of the link with a meridional suture for each basepoint. There is a general grading formula (Theorem 1.4 \cite{Zem19b}) when $\fk{s}|_Y, \fk{s}|_{Y'}$ are torsion and $\sr L,\sr L'$ are null homologous, but we only record here two cases where both the Maslov and Alexander gradings are preserved. If $K_+$ is a knot in $S^3$ with a local positive crossing, then there is a $2$-handle attachment cobordism from $K_+$ to the corresponding knot $K_-$ with a local negative crossing. Viewing this $2$-handle attachment as a decorated link cobordism as above, the induced map $\HFKhat(K_+) \to \HFKhat(K_-)$ preserves both the Maslov and Alexander grading (Example 12.7 of \cite{Zem19b}). If $C\colon K \to K'$ is a concordance between knots in $S^3$, and $C$ is decorated as above, then the induced map $\HFKhat(K) \to \HFKhat(K')$ preserves both the Maslov and Alexander gradings. We will use the following result in sections~\ref{subsec:HeegaardInvariance} and \ref{subsec:HeegaardInvariance}.

\begin{thm}[Theorem 1.1 of \cite{Zem19a}]\label{thm:zemke}
	Let $C\colon K_0 \to K_1$ be a ribbon concordance between knots in $S^3$. Let $C'\colon K_1 \to K_0$ denote the concordance in reverse. Then the composite of the maps induced by $C$ and $C'$ \[
		\HFKhat(K_0) \to \HFKhat(K_1) \to \HFKhat(K_0)
	\]is the identity. In particular, the map induced by $C$ is a bigrading-preserving inclusion onto a direct summand. 
\end{thm}

Zemke defines cobordism maps for a minus version of link Floer homology. We review basic features of his maps only the in case of links in $S^3$. A Heegaard diagram for such a multi-based link $\sr L = (L,\bf{w},\bf{z})$ is simply a Heegaard diagram for the exterior of the link with a meridional suture for each basepoint. The sutures and thereby the basepoints on the diagram are labelled $w_1,\ldots,w_n,z_1,\ldots,z_n$ according to their correspondence with $\bf{w},\bf{z}$ on $L$. In section 3.3 of \cite{Zem19c}, Zemke associates to a diagram an $\F_2[U_1,\ldots,U_n,V_1,\ldots,V_n]$-module $CFL^-(\sr L)$ freely generated by $\bf{x} \in \bf{T}_\alpha \cap \bf{T}_\beta$, equipped with a module endomorphism $\partial$ defined by \[
	\partial \mathbf{x} = 
	\sum_{\bf{y}\in\bf{T}_\alpha\cap\bf{T}_\beta} 
	\sum_{\substack{\phi\in\pi_2(\bf{x},\bf{y})\\\mu(\phi) = 1}} 
	\#\widehat{\sr M}(\phi) U_1^{n_{z_1}(\phi)} \cdots U_n^{n_{z_n}(\phi)}V_1^{n_{w_1}(\phi)}\cdots V_n^{n_{w_n}(\phi)} \cdot\mathbf{y}.
\]In general, the endomorphism is not a differential, but its square $\partial^2$ is multiplication by an element in the ring $\F_2[U_1,\ldots,U_n,V_1,\ldots,V_n]$ (Lemma 3.4 of \cite{Zem19c}). In section~\ref{subsec:HeegaardDetecting}, we will consider a minus version of link Floer homology defined when each component of $L$ has exactly two basepoints, and all of the $U_i,V_i$ variables are set to zero, except for one $U_i$ variable. The induced endomorphism on this quotient is a differential. Associated to a decorated link cobordism $(W,F)\colon (S^3,\sr L_1) \to (S^3,\sr L_2)$ and a $\Spin^c$ structure $\fk{s} \in \Spin^c(W)$ is an equivariant map $CFL^-(L_1) \to CFL^-(L_2)$ which commutes with the endomorphisms. The map up to equivariant chain homotopy is a diffeomorphism invariant of $(W,F)$ and is functorial. There are also Maslov and Alexander gradings on $CFL^-$, and the cobordism map induced by a concordance between links preserves these gradings (Theorem 2.14 of \cite{Zem19b}).

When the multi-based link is a knot $K$ in $S^3$ with exactly two basepoints, we let $\HFK^-(K)$ denote the homology of the complex $\CFL^-(K)$ after setting $V = 0$. In this paper, the \textit{minus} version of knot Floer homology refers to this finitely-generated bigraded $\F_2[U]$-module $\HFK^-(K)$. 

\subsection{Invariance under full twists of the band}\label{subsec:HeegaardInvariance}

Let $K_b$ be a band sum of a split two-component link $L$ along a band $b$, and let $K_\#$ denote the connected sum. Let $K_{b+1}$ be obtained by adding a full twist to the band, and recall that $K_{b+1},K_b,L$ form an oriented skein triple. If $C$ is a linking circle for the band, then $+1$ surgery on $C$ in the complement of $K_b$ yields $K_{b+1}$. The $\infty,0$, and $1$ surgeries on $C$ fit into an exact triangle (Theorem 8.2 of \cite{MR2065507}). \[
	\begin{tikzcd}[column sep=0]
		\HFKhat(K_{b+1}) \ar[rr] & & \HFKhat(K_b) \ar[dl]\\
		& \HFKhat(S^1 \x S^2,K_\#) \ar[ul] &
	\end{tikzcd}
\]Observe that $0$-surgery on $C$ yields a copy of $K_\#$ contained in a ball in $S^1 \x S^2$ by the light bulb trick (Corollary~\ref{cor:ZeroSurgeryLinkingCircle}). The maps in the triangle are the functorial $2$-handle attachment maps with the standard decorations described in section~\ref{subsubsec:functorialityLinkFloerHomology}. The map $\HFKhat(K_{b+1}) \to \HFKhat(K_b)$ preserves the Alexander grading by Theorem 8.2 of \cite{MR2065507} and it preserves the Maslov grading by Example 12.7 of \cite{Zem19b}. 

\begin{lem}\label{lem:HeegaardTrivTriangle}
	In the surgery triangle for the trivial band \[
		\begin{tikzcd}[column sep=0]
			\HFKhat(K_\#) \ar[rr] & & \HFKhat(K_\#) \ar[dl]\\
			& \HFKhat(S^1 \x S^2,K_\#) \ar[ul] &
		\end{tikzcd}
	\]the map $\HFKhat(K_\#) \to \HFKhat(K_\#)$ is zero. 
\end{lem}
\begin{proof}
	By the connected sum formula for sutured manifolds (Proposition 9.15 of \cite{Juh06}), we know that \[
		\dim\HFKhat(S^1 \x S^2,K_\#) = \dim \HFKhat(K_\#) \cdot \dim\HFhat(S^1 \x S^2) = 2\cdot\dim \HFKhat(K_\#)
	\] which implies that the map $\HFKhat(K_\#) \to \HFKhat(K_\#)$ is zero. 
\end{proof}

\begin{lem}\label{lem:HeegaardInjectBetweenTriangles}
	There are injective ribbon concordance maps making the following diagram commute. \[
		\begin{tikzcd}[column sep=tiny,row sep=1.5em]
			\HFKhat(K_{b+1}) \ar[rr] & & \HFKhat(K_b) \ar[dl]\\
			& \HFKhat(S^1 \x S^2,K_\#) \ar[lu] & \\
			\HFKhat(K_\#) \ar[uu,hook] \ar[rr] & & \HFKhat(K_\#) \ar[dl] \ar[uu,hook]\\
			& \HFKhat(S^1 \x S^2,K_\#) \ar[uu,hook,crossing over] \ar[ul] &
		\end{tikzcd}
	\]
\end{lem}
\begin{proof}
	Let $R\colon K_\# \to K_b$ be a ribbon concordance in $[0,1] \x S^3$ arising from the construction in \cite{MR1451821} (see Corollary~\ref{cor:RibbonConcordance}), and view it as a movie consisting of births followed by saddle moves. Observe that we may assume that all of the births and saddle moves are completely disjoint from a fixed $3$-ball containing a short segment of the band. Playing the movie except with a full twist added to the band within the fixed ball is a ribbon concordance $K_\# \to K_{b+1}$. Similarly we may think of the linking circle $C$ as lying in the fixed ball; after doing $0$-surgery on $C$, playing the movie is a ribbon concordance within $[0,1] \x S^1 \x S^2$ from $K_\#$ to $K_\#$. 

	Now consider the cobordisms maps induced from attaching a $2$-handle along $C$. We view these cobordisms as knot cobordisms in the standard way. The associated cobordism maps are the maps in the exact triangle. Pre- and post-composing these $2$-handle attachments with the ribbon concordances yield diffeomorphic knot cobordisms, so commutativity now follows from functoriality. 

	Zemke's ribbon concordance argument \cite{Zem19a} applies to the ribbon concordance inducing the map $\HFKhat(S^1 \x S^2,K_\#) \to \HFKhat(S^1 \x S^2,K_\#)$ so it is injective. 
\end{proof}

\begin{rem}\label{rem:HeegaardSurjTri}
	If $C\colon K_\# \to K_b$ is a ribbon concordance, and $C'\colon K_b \to K_\#$ is the concordance in reverse, then the composite $C'\circ C\colon K_\# \to K_\#$ induces the identity on knot Floer homology (Theorem~\ref{thm:zemke}). The argument of Lemma~\ref{lem:HeegaardInjectBetweenTriangles} yields the following commutative diagram. \[
		\begin{tikzcd}[column sep=tiny,row sep=1.5em]
			\HFKhat(K_\#) \ar[rr] & & \HFKhat(K_\#) \ar[dl]\\
			& \HFKhat(K_\#,S^1 \x S^2) \ar[ul]\\
			\HFKhat(K_{b+1}) \ar[rr] \ar[uu,two heads] & & \HFKhat(K_b) \ar[uu,two heads] \ar[dl]\\
			& \HFKhat(K_\#,S^1 \x S^2) \ar[lu] \ar[uu,two heads,crossing over] & \\
			\HFKhat(K_\#) \ar[uu,hook] \ar[rr] & & \HFKhat(K_\#) \ar[dl] \ar[uu,hook]\\
			& \HFKhat(K_\#,S^1 \x S^2) \ar[uu,hook,crossing over] \ar[ul] &
		\end{tikzcd}
	\]
\end{rem}

\theoremstyle{plain}
\newtheorem*{thm:HFKinvariant}{Theorem~\ref{thm:HeegaardKnotFloerInvariantUnderFullTwists}}
\begin{thm:HFKinvariant}
	Let $K_b$ be a band sum of a split two component link, and let $K_{b+n}$ be obtained by adding $n$ full twists to the band. Then \[
		\HFKhat(K_b) \cong \HFKhat(K_{b+n})
	\]as bigraded vector spaces over $\F_2$. 
\end{thm:HFKinvariant}

\begin{proof}
	It suffices to prove the result for $n = 1$. Using the ribbon concordances in the proof of Lemma~\ref{lem:HeegaardInjectBetweenTriangles}, we obtain a copy of $\HFKhat(K_\#)$ as a direct summand of $\HFKhat(K_b)$ and similarly for $\HFKhat(K_{b+n})$ by Theorem~\ref{thm:zemke}. Hence we may write \[
		\HFKhat(K_b) \cong \HFKhat(K_\#) \oplus F_b \qquad \HFKhat(K_{b+1}) \cong \HFKhat(K_\#) \oplus F_{b+1}
	\]as bigraded vector spaces for some bigraded vector spaces $F_b$ and $F_{b+1}$. It suffices to show that $F_b$ and $F_{b+1}$ are isomorphic as bigraded vector spaces. It follows from Lemmas \ref{lem:HeegaardTrivTriangle}, \ref{lem:HeegaardInjectBetweenTriangles}, and Remark~\ref{rem:HeegaardSurjTri} that the exact triangle \[
		\begin{tikzcd}[column sep=tiny]
			\HFKhat(K_{b+1}) \ar[rr] & & \HFKhat(K_b) \ar[dl]\\
			& \HFKhat(K_\#,S^1 \x S^2)\ar[ul] &
		\end{tikzcd}
	\]splits into the following direct sum of two exact triangles. \[
		\left(\begin{tikzcd}[column sep=0]
			\HFKhat(K_\#) \ar[rr,"0"] & & \HFKhat(K_\#) \ar[dl,hook]\\
			& \HFKhat(K_\#,S^1 \x S^2)\ar[ul,two heads] &
		\end{tikzcd}\right) \bigoplus \left(\begin{tikzcd}[column sep=large]
			F_{b+1} \ar[rr,hook,two heads] & & F_{b} \ar[dl]\\
			& 0 \ar[ul] &
		\end{tikzcd} \right)
	\]The isomorphism $F_{b+1} \to F_{b}$ respects both gradings, so the result follows. The fact that the isomorphism preserves Alexander grading appears in Theorem 8.2 of \cite{MR2065507}, and the fact that it preserves Maslov grading appears in Example 12.7 of \cite{Zem19b}. 
\end{proof}

The author thanks Ian Zemke for explaining that there is also a surgery exact triangle for the minus version of knot Floer homology \[
	\begin{tikzcd}[column sep=0]
		\HFK^-(K_{b+1}) \ar[rr] & & \HFK^-(K_b) \ar[dl]\\
		& \HFK^-(S^1 \x S^2,K_\#) \ar[ul] &
	\end{tikzcd}
\]where the maps are the functorial $2$-handle attachment maps with the standard decorations. Although the maps are \textit{a priori} a sum of infinitely many homogeneous maps indexed by $\Spin^c$ structures on the $2$-handle attachment cobordisms, but it follows from an adjunction inequality (Example 12.7 of \cite{Zem19b}) that only finitely many of the maps are possibly nonzero. The only possibly nonzero contributions to the map $\HFK^-(K_{b+1}) \to \HFK^-(K_b)$ preserve both the Maslov and Alexander gradings. The argument of Proposition 11.5 of \cite{1011.1317} shows that the triangle is exact after tensoring with the power series ring $\F_2[[U]]$. Since $\F_2[[U]]$ is faithfully flat as a module over $\F_2[U]$, the surgery triangle over $\F_2[U]$ is also exact. Note that there is a similar exact triangle (Theorem 1.1 of \cite{0707.1165}) where the maps are defined using special Heegaard diagrams rather than functorial $2$-handle attachment maps. Using the surgery exact triangle described here, the argument in Theorem~\ref{thm:HeegaardKnotFloerInvariantUnderFullTwists} works equally well for $\HFK^-$. 

\subsection{Detecting the trivial band}\label{subsec:HeegaardDetecting}

Let $C$ be a linking circle for the band $b$, and consider the link $K_b \cup C$. Recall that $C$ bounds a disc disjoint from $K_b$ if and only if the band is trivial. We will show that the knot Floer homology of $K_b$ detects the trivial band by relating it to the link Floer homology of $K_b \cup C$. In section~\ref{subsubsec:ComparisonWithMinusLinkFloer}, we show that the $\F_2[U]$-rank of a suitable minus version of link Floer homology $\HFL^-(K_b \cup C,\sigma)$ is exactly twice the $\F_2$-dimension of $\HFKhat(K_b)$. To show that $\dim\HFKhat(K_b) > \dim\HFKhat(K_\#)$ when $b$ is nontrivial, it therefore suffices to show that $\rank \HFL^-(K_b \cup C,\sigma) > \rank \HFL^-(K_\# \cup C,\sigma)$. To obtain this strict rank inequality, we use the sequence of surface decompositions constructed in Theorem~\ref{thm:SeqOfSurfDecomp} along with a suitable minus version of sutured Floer homology developed in section~\ref{subsubsec:minusHeegaardSuturedFloer} and a model computation made in section~\ref{subsubsec:HopfHeegaard}. 

\subsubsection{Comparison with a minus version of link Floer homology}\label{subsubsec:ComparisonWithMinusLinkFloer}

\begin{df*}[a minus version of link Floer homology]
	Fix an admissible balanced diagram $(\Sigma,\bo\alpha,\bo\beta,\bo p)$ with four basepoints for the link $K_b \cup C$. Let $S^3(K_b \cup C)$ be the sutured exterior of the link $K_b \cup C$, and of the two sutures on $\partial N(C) \subset \partial S^3(K_b \cup C)$, let $\sigma$ be the one oriented as the meridian of $C$. Let $z \in \bo p$ be the basepoint corresponding to $\sigma$. Also assume that $(\Sigma,\bo\alpha,\bo\beta,\bo p\setminus\{z\})$ is an admissible diagram. 

	As usual, let ${\bf T}_\alpha$, $\mathbf{T}_\beta$ be the tori $\alpha_1 \x \cdots \alpha_k$, $\beta_1 \x \cdots \beta_k$ in the $k$-fold symmetric product of $\Sigma$, where $k$ is the number of curves in $\bo\alpha$. The version of link Floer homology that we will use is a chain complex over $\F_2[U]$ with underlying module \[
		\CFL^-(\Sigma,\bo\alpha,\bo\beta,\bo p,z) = \bigoplus_{\mathbf{x} \in \mathbf{T}_\alpha \cap \mathbf{T}_\beta} \F_2[U]\cdot\mathbf{x}
	\]with $\F_2[U]$-linear differential \[
		\partial\mathbf{x} = \sum_{\mathbf{y} \in \mathbf{T}_\alpha\cap\mathbf{T}_\beta} \sum_{\substack{\phi \in \pi_2(\mathbf{x},\mathbf{y}),\:\mu(\phi) = 1\\n_p(\phi) = 0 \:|\: p \in \bo p\setminus\{z\}}} \#\widehat{\sr M}(\phi) \cdot U^{n_z(\phi)}\cdot\mathbf{y}. 
	\]Here $\mu(\phi)$ is the Maslov index of a Whitney disc $\phi \in \pi_2(\bo x,\bo y)$, the number $\#\widehat{\sr M}(\phi) \in \F_2$ is the mod $2$ count of the $J$-holomorphic representatives of $\phi$ up to reparametrization of the disc for a suitable family $J = (J_t)_{t \in [0,1]}$ of almost complex structures on $\Sym\!{}^k(\Sigma)$, and $n_p(\phi)$ is the algebraic intersection number of $\phi$ and $p \x \Sym\!{}^{k-1}(\Sigma)$.

	Note that the differential is blocked by all basepoints except $z$, and algebraic intersection with $z \x \Sym\!{}^{k-1}(\Sigma)$ is recorded in the $U$-variable. Let $\HFL^-(K_b \cup C,\sigma)$ denote the homology of this complex. 
\end{df*}

\begin{rem}
	To see that we can arrange that $(\Sigma,\bo\alpha,\bo\beta,\bo p)$ and $(\Sigma,\bo\alpha,\bo\beta,\bo p\setminus\{z\})$ are both admissible, see Lemma~\ref{lem:existsAdmissibleDiagram}. Invariance of $\HFL^-(K_b \cup C,\sigma)$ follows from Theorem 4.7 of \cite{MR2443092}. See section~\ref{subsubsec:minusHeegaardSuturedFloer} for a generalization of this construction. 
\end{rem}

Associated to each generator $\mathbf{x}\in \mathbf{T}_\alpha\cap \mathbf{T}_\beta$ is a $\Spin^c$ structure $\fk{s}(\bf{x})$. Recall that the set of $\Spin^c$ structures on a balanced sutured manifold $M$ is an affine space over $H_1(M)$. The homology class of $\sigma$, the suture on $\partial N(C) \subset \partial S^3(K_b \cup C)$ oriented as the meridian of $C$, is infinite-cyclic in $H_1(S^3(K_b \cup C))$. We define a decomposition of $\CFL^-(\Sigma,\bo\alpha,\bo\beta,\bo p,z)$ along $\Spin^c$ structures by declaring that $U^n \mathbf{x}$ lies in the $\Spin^c$ structure $\fk{s}(\mathbf{x}) - n[\sigma]$. The differential now preserves $\Spin^c$ structure so $\HFL^-(K_b \cup C,\sigma) = \bigoplus_{\fk s} \HFL^-(K_b \cup C,\sigma,\fk{s})$ splits along $\Spin^c$ structures as well. The map $U$ sends $\HFL^-(K_b \cup C,\sigma,\fk{s})$ to $\HFL^-(K_b \cup C,\sigma,\fk{s} - [\sigma])$. The behavior of $U$ with respect to the splitting gives restrictions on the module structure of $\HFL^-(K_b \cup C,\sigma)$ by the following well-known lemma. 

\begin{lem}\label{lem:HomologyStruct}
	Let $\F$ be a field, and let $H$ be a finitely-generated $\F[U]$-module. Also assume that there is a set $S$ with a free $\Z$-action $S \x \Z \to S$ denoted $(\fk{s},n) \mapsto \fk{s} + n$ for which there is a splitting $H = \bigoplus_{\fk{s}} H(\fk{s})$ and $U$ maps $H(\fk{s})$ to $H(\fk{s} - 1)$. Then $H$ is isomorphic to \[
		\F[U]^k \oplus\bigoplus_{i=1}^n \frac{\F[U]}{U^{k_i}}
	\]for a nonnegative integer $k$ and positive integers $k_1,\ldots,k_n$. The isomorphism can be chosen so that generators of the summands are homogeneous elements of $H$ with respect to the splitting along $S$.
\end{lem}
\begin{proof}
	Proposition A.4.3 of \cite{MR3381987} gives the result when free $\Z$-action on $S$ is transitive. We split $H$ as a direct sum over the orbits of $S \x \Z \to S$ to reduce to this case. 
\end{proof}

Let $(\Sigma,\bo\alpha,\bo\beta,\bo p)$ be a Heegaard diagram for $K_b \cup C$ with four basepoints, and let $z \in \bo p$ be the basepoint corresponding to the suture on $\partial N(C)$ oriented as the meridian. The short exact sequence \[
	\begin{tikzcd}
		0 \ar[r] & \CFL^-(\Sigma,\bo\alpha,\bo\beta,\bo p,z) \ar[r,"U - \Id"] &[0.5em] \CFL^-(\Sigma,\bo\alpha,\bo\beta,\bo p,z) \ar[r] & \displaystyle\frac{\CFL^-(\Sigma,\bo\alpha,\bo\beta,\bo p,z)}{U - \Id} \ar[r] & 0
	\end{tikzcd}
\]induces the following exact triangle. \[
	\begin{tikzcd}[column sep=0]
		\HFL^-(K_b \cup C,\sigma) \ar[rr,"U - \Id"] & & \HFL^-(K_b \cup C,\sigma) \ar[dl]\\
		& {H\left(\displaystyle\frac{\CFL^-(\Sigma,\bo\alpha,\bo\beta,\bo p,z)}{U - \Id} \right)} \ar[ul] &
	\end{tikzcd}
\]We will identify the quotient complex in Lemma~\ref{lem:quotientComplex}. 

\begin{lem}\label{lem:basepointErasing}
	The sutured Heegaard diagram $(\Sigma,\bo\alpha,\bo\beta,\bo p \setminus \{z\})$ obtained by erasing $z$ represents the sutured exterior of $K_b$ with a sutured puncture. 
\end{lem}
\begin{proof}
	Recall that the sutured exterior of $K_b \cup C$ may obtained from the diagram $(\Sigma,\bo\alpha,\bo\beta,\bo p)$ by first deleting a regular neighborhood of $\bo p$ to obtain $\Sigma'$, and then attaching $2$-handles to $[-1,1] \x \Sigma'$ along $\{-1\} \x \alpha_i$ and $\{1\} \x \beta_i$, and viewing $\{0\} \x \partial\Sigma'$ as the collection of sutures. 

	Erasing $z$ from the diagram therefore corresponds to filling the corresponding boundary component of $\Sigma'$ with a disc. Thus the corresponding sutured manifold is obtained from the sutured exterior of $K_b \cup C$ by attaching a $3$-dimensional $2$-handle along the suture on $\partial N(C)$. This sutured manifold is precisely the sutured exterior of $K_b$ with a sutured puncture. 
\end{proof}

\begin{lem}\label{lem:quotientComplex}
	The quotient complex \[
		\frac{\CFL^-(\Sigma,\bo\alpha,\bo\beta,\bo p,z)}{U - \Id}
	\]is isomorphic to the sutured Floer complex associated to $(\Sigma,\bo\alpha,\bo\beta,\bo p\setminus\{z\})$. 
\end{lem}
\begin{proof}
	We may identify the underlying $\F_2[U]$-module of the quotient complex as \[
		\frac{\CFL^-(\Sigma,\bo\alpha,\bo\beta,\bo p,z)}{U - \Id} = \bigoplus_{\mathbf{x} \in \mathbf{T}_\alpha\cap\mathbf{T}_\beta} \F_2\cdot\mathbf{x}
	\]where $U$ acts as the identity. Under this identification, the differential on the quotient is \[
		\partial \mathbf{x} = \sum_{\mathbf{y} \in \mathbf{T}_\alpha\cap\mathbf{T}_\beta} \sum_{\substack{\phi \in \pi_2(\mathbf{x},\mathbf{y}),\:\mu(\phi) = 1\\n_p(\phi) = 0 \:|\: p \in \bo{p}\setminus \{z\}}} \#\widehat{\sr M}(\phi)\cdot\mathbf{y}.
	\]This is precisely the sutured Floer chain complex associated to $(\Sigma,\bo\alpha,\bo\beta,\bo p\setminus\{z\})$. 
\end{proof}

\begin{prop}\label{prop:rankEqualsTwiceDim}
	The rank of $\HFL^-(K_b \cup C,\sigma)$ as an $\F_2[U]$-module is two times the dimension of $\HFKhat(K_b)$ as an $\F_2$-vector space. 
\end{prop}
\begin{proof}
	By Lemma~\ref{lem:quotientComplex}, we have a short exact sequence \[
		\begin{tikzcd}[column sep=small]
			0 \ar[r] & \CFL^-(\Sigma,\bo\alpha,\bo\beta,\bo p,z) \ar[r,"U - \Id"] &[2em] \CFL^-(\Sigma,\bo\alpha,\bo\beta,\bo p,z) \ar[r] & \SFC(\Sigma,\bo\alpha,\bo\beta,\bo p\setminus\{z\}) \ar[r] & 0.
		\end{tikzcd}
	\]Let $S^3(K_b)(1)$ denote the sutured exterior of $K_b$ with a sutured puncture. By Proposition 9.14 of \cite{Juh06}, we know that $\SFH(S^3(K_b)(1)) \cong \SFH(S^3(K_b)) \otimes \F_2^2 = \HFKhat(K_b) \otimes \F_2^2$ so by Lemma~\ref{lem:basepointErasing} we have the following exact triangle. \[
		\begin{tikzcd}[column sep=0]
			\HFL^-(K_b \cup C,\sigma) \ar[rr,"U - \Id"] & & \HFL^-(K_b \cup C,\sigma) \ar[dl]\\
			& \HFKhat(K_b) \otimes \F_2^2 \ar[ul] &
		\end{tikzcd}
	\]By Lemma~\ref{lem:HomologyStruct}, we know that $\HFL^-(K_b \cup C,\sigma)$ is isomorphic to \[
		\F_2[U]^k \oplus \bigoplus_{i=1}^n \frac{\F_2[U]}{U^{k_i}}
	\]as an $\F_2[U]$-module where $k$ is the rank of $\HFL^-(K_b \cup C,\sigma)$ and $k_1,\ldots,k_n$ are positive integers. It follows that $U - \Id$ is injective on $\HFL^-(K_b \cup C,\sigma)$, so it suffices to show that the cokernel of $U - \Id$ is $k$-dimensional over $\F_2$. But this is true because $U - \Id$ is invertible on $\F_2[U]/U^n$ for positive $n$ and the cokernel of $U - \Id$ on $\F_2[U]$ is $1$-dimensional. 
\end{proof}

Thus, in order to prove that $\HFKhat(K_b) \neq \HFKhat(K_\#)$ for $b$ nontrivial, it suffices to prove that the $\F_2[U]$-rank of $\HFL^-(K_b \cup C,\sigma)$ is strictly larger than the $\F_2[U]$-rank of $\HFL^-(K_\# \cup C,\sigma)$. We give a sufficient condition (Proposition~\ref{prop:SpinCstructDiff}) on the structure of $\HFL^-(K_b \cup C,\sigma)$ for this to occur, and we later verify it. First observe that we have a rank inequality from an application of Zemke's ribbon concordance argument. 

\begin{lem}\label{lem:HeegaardMinusInject}
	Let $R\colon K_\# \cup C \to K_b\cup C$ be a ribbon concordance, with $R'\colon K_b \cup C \to K_\# \cup C$ the concordance in reverse. There are induced $\F_2[U]$-equivariant maps \[
		\HFL^-(K_\# \cup C,\sigma) \to \HFL^-(K_b \cup C,\sigma) \to \HFL^-(K_\# \cup C,\sigma)
	\]respecting $\Spin^c$ structure whose composite is the identity on $\HFL^-(K_\# \cup C,\sigma)$. The maps respect the action of $[\sigma]$ on $\Spin^c$ structures. 
\end{lem}
\begin{proof}
	This follows from Zemke's ribbon concordance argument. For example, see Theorem 1.7 of \cite{Zem19a}. 
\end{proof}

\begin{cor}
	We have an $\F_2[U]$-rank inequality \[
		\rank \HFL^-(K_\# \cup C,\sigma) \leq \rank \HFL^-(K_b \cup C,\sigma).
	\]
\end{cor}

By Lemma~\ref{lem:HomologyStruct}, we know that $\HFL^-(K_b\cup C,\sigma)$ is isomorphic to $\F_2[U]^k \oplus \bigoplus_i \F_2[U]/U^{k_i}$ where the generators of the summands can be chosen to be homogeneous with respect to the splitting along $\Spin^c$ structures. Recall that $x \in \HFL^-(K_b \cup C,\sigma)$ is $U$-torsion if $U^n x = 0$ for some $n$. 

\begin{df*}[generating a free summand]
	We say that a $\Spin^c$ structure $\fk{s}$ of the sutured exterior of $K_b \cup C$ \textit{generates a free summand} of $\HFL^-(K_b \cup C,\sigma)$ if there exists an element $x \in \HFL^-(K_b \cup C,\sigma,\fk{s})$ which is not $U$-torsion and there is no $y \in \HFL^-(K_b \cup C,\sigma)$ for which $Uy = x + z$ where $z \in \HFL^-(K_b \cup C,\sigma,\fk{s})$ is $U$-torsion. 
\end{df*}
\begin{rem}
	Given an isomorphism of $\HFL^-(K_b \cup C,\sigma)$ with $\F_2[U]^k \oplus \bigoplus_i \F_2[U]/U^{k_i}$ where the generators of the free-summands are homogeneous with respect to the splitting along $\Spin^c$ structures, these $\Spin^c$ structures are the ones generating free summands in the sense of the definition. The isomorphism with $\F_2[U]^k \oplus \bigoplus_i \F_2[U]/U^{k_i}$ is non-canonical, but the $\Spin^c$ structures generating free summands are well-defined. 
\end{rem}

\begin{prop}\label{prop:SpinCstructDiff}
	If there are $\Spin^c$ structures $\fk{t},\fk{s}$ generating free summands of $\HFL^-(K_b \cup C,\sigma)$ satisfying $\fk{t} = \fk{s} + [\sigma]$, then the rank of $\HFL^-(K_b \cup C,\sigma)$ is strictly larger than the rank of $\HFL^-(K_\# \cup C,\sigma)$. 
\end{prop}
\begin{proof}
	We first describe the structure of $\HFL^-(K_\# \cup C,\sigma)$. Fix a Heegaard diagram $(\Sigma,\bo\alpha,\bo\beta,\bo p)$ with two basepoints for $K_\#$, and in a region of $\Sigma$ containing one of the basepoints $p$, add the two curves $\alpha_C,\beta_C$ and the two basepoints $z,w$ appearing in Figure~\ref{fig:unknotDiagram}. The new diagram $(\Sigma,\bo\alpha',\bo\beta',\bo p',z) = (\Sigma,\bo\alpha \cup\{\alpha_C\},\bo\beta \cup \{\beta_C\},\bo p \cup \{z,w\},z)$ with distinguished basepoint $z$ is an admissible diagram for the sutured exterior of the split link $K_\# \cup C$ where $z$ corresponds to the suture on $\partial N(C)$ oriented as the meridian. 

	\begin{figure}[!ht]
		\centering
		\labellist
		\pinlabel $\alpha_C$ at 565 210
		\pinlabel $\beta_C$ at 50 210
		\pinlabel $p$ at 30 405
		\pinlabel $z$ at 325 265
		\pinlabel $w$ at 335 155
		\endlabellist
		\includegraphics[width=.35\textwidth]{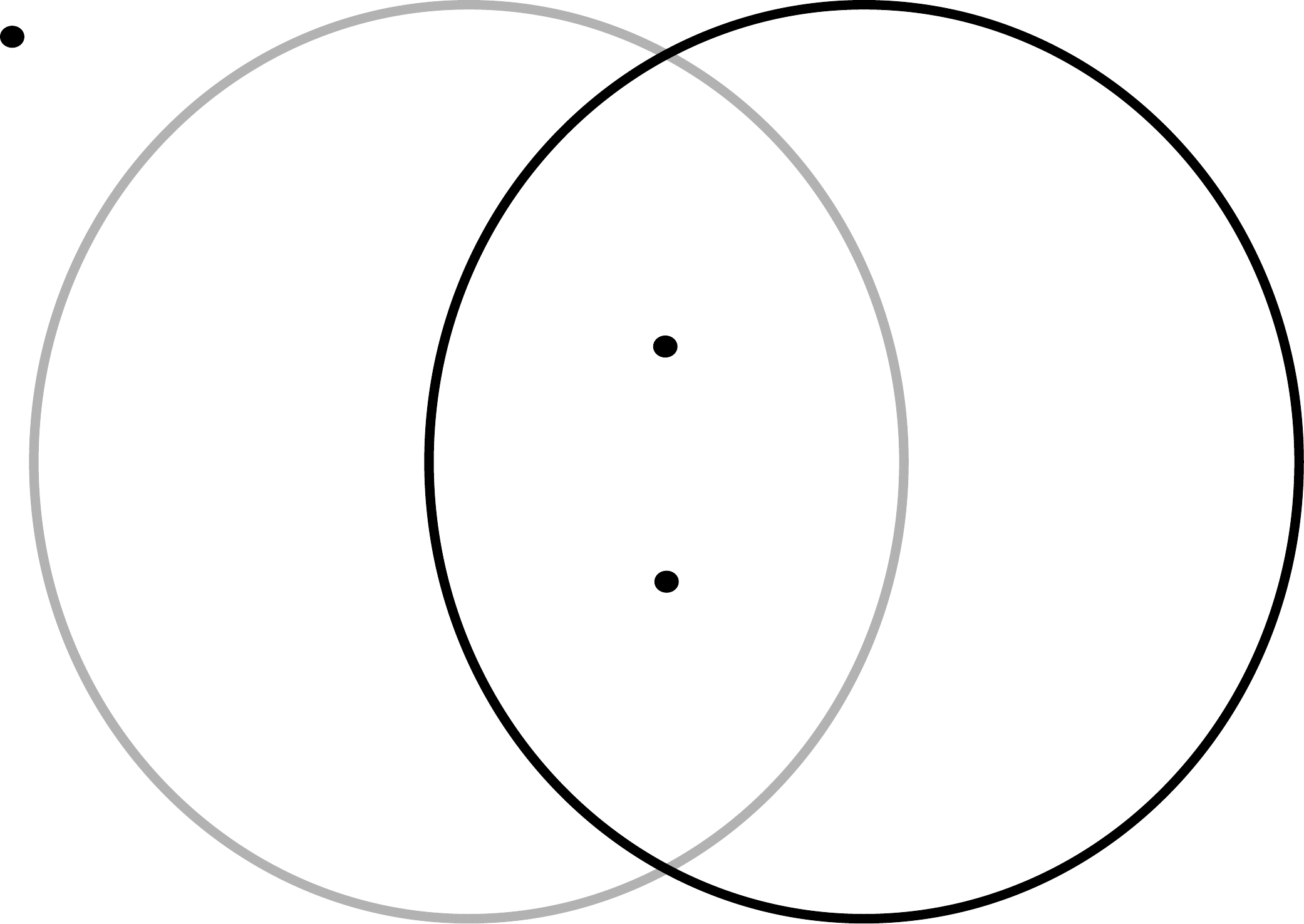}
		\caption{Forming a Heegaard diagram for $K_\# \cup C$ from a diagram for $K_\#$.}
		\label{fig:unknotDiagram}
	\end{figure}
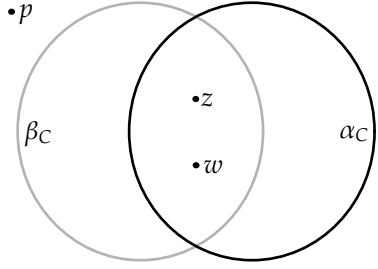

	To each generator $\mathbf{x}$ of the complex $\CFKhat(\Sigma,\bo\alpha,\bo\beta,\bo p)$, there are two generators $\mathbf{x}_1,\mathbf{x}_2$ of the complex $\CFL^-(\Sigma,\bo\alpha',\bo\beta',\bo p',z)$. For suitable families of almost complex structures on the symmetric products, we can arrange that if the differential acts as $\partial \mathbf{x} = \sum_{\mathbf{y}} c_{\mathbf{y}} \cdot \mathbf{y}$ on $\CFKhat(\Sigma,\bo\alpha,\bo\beta,\bo p)$, then $\partial \mathbf{x}_i = \sum_{\mathbf{y}} c_{\mathbf{y}} \cdot \mathbf{y}_i$ for $i = 1,2$ in $\CFL^-(\Sigma,\bo\alpha',\bo\beta',\bo p',z)$. Indeed, any Whitney disc passing over the basepoint $z$ will be blocked by the basepoint $w$, discs connecting generators of differing subscript cancel in pairs, and no previously existing discs will be blocked because the additional decorations were placed in a region containing a basepoint. 

	It follows that $\HFL^-(K_\# \cup C,\sigma)$ is a free $\F_2[U]$-module, and all $\Spin^c$ structures generating free-summands differ by multiples of the homology class of the meridian of $K_\#$. In particular, no two $\Spin^c$ structures generating free-summands satisfy $\fk{t} = \fk{s} + [\sigma]$. From Lemma~\ref{lem:HeegaardMinusInject}, it follows that free summands of $\HFL^-(K_\# \cup C,\sigma)$ are mapped isomorphically onto free summands of $\HFL^-(K_b \cup C,\sigma)$. Hence if $\HFL^-(K_b\cup C,\sigma)$ has $\Spin^c$ structures generating free-summands satisfying $\fk{t} = \fk{s} + [\sigma]$, then there must be a free-summand not in the image of $\HFL^-(K_\#\cup C,\sigma)$. 
\end{proof}

\subsubsection{A minus version of sutured Floer homology}\label{subsubsec:minusHeegaardSuturedFloer}

The sutured exterior of $K_b\cup C$ with a distinguished suture $\sigma$ on $\partial N(C)$ is an example of a balanced sutured manifold with a suitable distinguished suture for which we will define a suitable minus version of sutured Floer homology. See Alishahi and Eftekhary \cite{MR3412088} for a more general minus version of sutured Floer homology. 

\begin{df*}[suitable distinguished suture]
	Let $(M,\gamma)$ be a balanced sutured manifold. We say that one of the sutures $\sigma$ is a \textit{suitable distinguished suture} if \begin{enumerate}[itemsep=-0.5ex]
		\item $\sigma$ lies on a toral boundary component $T$ of $M$,
		\item $s(\gamma) \cap T$ consists of two oppositely oriented parallel copies of $\sigma$,
		\item there exists a properly embedded compact oriented surface $\Sigma \subset M$ for which $\partial \Sigma \cap T$ is a simple closed curve that intersects $\sigma$ in a single point transversely.
	\end{enumerate}
	
	We say that $(\Sigma,\bo\alpha,\bo\beta,\bo p,z)$ is a balanced sutured Heegaard diagram \textit{with a suitable distinguished basepoint} representing $(M,\gamma,\sigma)$ if $(\Sigma,\bo\alpha,\bo\beta,\bo p)$ is a balanced sutured Heegaard diagram for $(M,\gamma)$ and $z \in \bo p$ is the basepoint corresponding to the suture $\sigma$. It is \textit{admissible} if both $(\Sigma,\bo\alpha,\bo\beta,\bo p)$ and $(\Sigma,\bo\alpha,\bo\beta,\bo p\setminus\{z\})$ are admissible.
\end{df*}

\begin{examples}
	Let $K \cup C$ be a two component link in $S^3$, and let $S^3(K \cup C)$ be its sutured exterior. Let $\sigma$ be a suture on $\partial N(C)$. Then $\sigma$ is a suitable distinguished suture for $S^3(K \cup C)$. 

	Suppose $K$ has a Seifert surface $\Sigma$ disjoint from $C$. Then $\sigma$ is a suitable distinguished suture on the sutured manifold obtained by decomposing $S^3(K \cup C)$ along $\Sigma$.
\end{examples}

\begin{rem}
	Let $(\Sigma,\bo\alpha,\bo\beta,\bo p,z)$ be a balanced diagram with a suitable distinguished suture representing $(M,\gamma,\sigma)$. The diagram $(\Sigma,\bo\alpha,\bo\beta,\bo p \setminus \{z\})$ obtained by forgetting the basepoint $z$ represents the sutured manifold obtained from $(M,\gamma,\sigma)$ by attaching a two-handle to $\sigma$. See Lemma~\ref{lem:basepointErasing}. 
\end{rem}

\begin{lem}\label{lem:existsAdmissibleDiagram}
	Every balanced sutured manifold with a suitable distinguished suture has an admissible balanced diagram with a suitable distinguished basepoint. 
\end{lem}
\begin{proof}
	Let $(M,\gamma,\sigma)$ be a balanced sutured manifold with a suitable distinguished suture, and let $(N,\beta)$ be the sutured manifold obtained by attaching a $2$-handle to $\sigma$.

	Start with an arbitrary diagram for $(M,\gamma)$, and view it as a diagram for $(N,\beta)$ by ignoring the basepoint $z$ corresponding to $\sigma$. Apply Juh\'asz's procedure of Proposition 3.15 of \cite{Juh06} to obtain an admissible diagram for $(N,\beta)$, ensuring that each move is disjoint from $z$. The resulting diagram for $(N,\beta)$ is admissible by construction, and by including the basepoint $z$, the resulting diagram for $(M,\gamma)$ is also admissible. 
\end{proof}

\begin{df*}[a minus version of sutured Floer homology for balanced sutured manifolds with a suitable distinguished suture]
	Let $(\Sigma,\bo\alpha,\bo\beta,\bo p,z)$ be an admissible balanced diagram with suitable distinguished basepoint. As usual, let ${\bf T}_\alpha$, $\mathbf{T}_\beta$ be the tori $\alpha_1 \x \cdots \alpha_k$, $\beta_1 \x \cdots \beta_k$ in the $k$-fold symmetric product of $\Sigma$, where $k$ is the number of curves in $\bo\alpha$. Let $\SFC^-(\Sigma,\bo\alpha,\bo\beta,\bo p,z)$ be the free $\F_2[U]$-module \[
		\SFC^-(\Sigma,\bo\alpha,\bo\beta,\bo p,z) = \bigoplus_{{\bf x} \in \mathbf{T}_\alpha \cap \mathbf{T}_\beta} \F_2[U]\cdot {\bf x}
	\]with differential \[
		\partial \mathbf{x} = \sum_{\mathbf{y} \in \mathbf{T}_\alpha\cap\mathbf{T}_\beta} \sum_{\substack{\phi \in \pi_2(\mathbf{x},\mathbf{y}),\:\mu(\phi) = 1\\n_p(z) = 0 \:|\: p \in \bo p\setminus\{z\}}} \#\widehat{\sr M}(\phi) \cdot U^{n_z(\phi)}\cdot \mathbf{y}.
	\]where $\mu(\phi)$ is the Maslov index of $\phi \in \pi_2(\mathbf{x},\mathbf{y})$, $\#\widehat{\sr M}(\phi)$ is the mod $2$ count of the $J$-holomorphic representatives of $\phi$ up to reparametrization for a suitable family $J = (J_t)_{t \in [0,1]}$ of almost complex structures on $\Sym\!{}^k(\Sigma)$, and $n_p(\phi)$ is the algebraic intersection number of $\phi$ and $p \x \Sym\!{}^{k-1}(\Sigma)$.

	Let $\SFH^-(M,\gamma,\sigma)$ denote the homology of the complex. Letting $[\sigma] \in H_1(M)$ denote the homology class of $\sigma \subset \partial M$, we declare that $U^n\mathbf{x}$ lies in the $\Spin^c$ structure $\fk{s}(\mathbf{x}) - n[\sigma]$. We obtain a splitting $\SFH^-(M,\gamma,\sigma) = \bigoplus_{\fk s} \SFH^-(M,\gamma,\sigma,\fk{s})$ along $\fk{s} \in \Spin^c(M,\gamma)$ and $U$ sends $\SFH^-(M,\gamma,\sigma,\fk{s})$ to $\SFH^-(M,\gamma,\sigma,\fk{s} - [\sigma])$. 
\end{df*}

\begin{rem}
	Invariance from the Heegaard diagram follows in the same way as the proof of invariance for knot Floer homology. Let $(N,\beta)$ be the sutured manifold obtained by attaching a $2$-handle to $(M,\gamma)$ along $\sigma$. The minus complex for $(M,\gamma,\sigma)$ is essentially the hat complex for $(N,\beta)$ with the filtration induced by the basepoint $z$. Alternatively, our construction is a special case of the more general minus version of sutured Floer homology appearing in \cite{MR3412088}. 
\end{rem}

The purpose of this minus version of sutured Floer homology is its behavior under nice surface decompositions when the decomposing surface is disjoint from the distinguished torus. 

\begin{prop}\label{prop:HeegaardMinusSurfDecomp}
	Let $(M,\gamma,\sigma)$ be a balanced sutured manifold with a suitable distinguished suture, and consider a nice surface decomposition \[
		(M,\gamma) \overset{S}{\rightsquigarrow} (M',\gamma')
	\]where $S \subset M$ is disjoint from the toral boundary component of $M$ that contains $\sigma$. Then $(M',\gamma',\sigma)$ is a balanced sutured manifold with a suitable distinguished suture, and $\SFH^-(M',\gamma',\sigma)$ is a direct summand of $\SFH^-(M,\gamma,\sigma)$ as an $\F_2[U]$-module, compatible with the action of the homology class of $\sigma$ on $\Spin^c$ structures. 

	More precisely, if $x,y \in \SFH^-(M',\gamma',\sigma)$ are homogeneous with respect to the splitting along $\Spin^c$ structures, then their images in $\SFH^-(M,\gamma,\sigma)$ are as well. If the $\Spin^c$ structures supporting $x$ and $y$ differ by $[\sigma] \in H_1(M')$, then the the $\Spin^c$ structures supporting their images differ by $[\sigma] \in H_1(M)$. 
\end{prop}
\begin{proof}
	Note that $M'$ is obtained from $M$ by deleting a regular neighborhood of $S$, so the distinguished toral boundary component $T$ of $M$ remains a toral boundary component of $M'$ with the same two sutures. Let $\Sigma$ be a properly embedded surface in $M$ for which $\partial \Sigma \cap T$ is a simple closed curve intersecting $\sigma$ in single point. We may assume that $\Sigma$ and $S$ intersect transversely. The surface $\Sigma \cap M'$ verifies that $\sigma$ is a suitable distinguished suture for $(M',\gamma')$. 

	As before, let $(N,\beta)$ be obtained from $(M,\gamma)$ by attaching a $2$-handle to $\sigma$. Let $(N',\beta')$ be obtained by decomposing $(N,\beta)$ along $S$. Certainly $(N',\beta')$ can also be obtained by attaching a $2$-handle to $(M',\gamma')$ along $\sigma$. We now reduce to Juh\'asz's proof of Theorem 1.3 of \cite{Juh08}. First by Lemma 4.5 of \cite{Juh08}, we may assume that $S \subset (N,\beta)$ is \textit{good}. Next, by Proposition 4.4 of \cite{Juh08}, we may choose a good surface diagram of $(M,\gamma)$ adapted to $S$. By Proposition 4.8 of \cite{Juh08} and the argument used in the proof of Lemma~\ref{lem:existsAdmissibleDiagram}, we may assume that the diagram is admissible, and that the surface diagram of $(N,\beta)$ adapted to $S$ obtained by deleting the distinguished basepoint is admissible. By Theorem 6.4 of \cite{Juh08}, we may suppose that the diagram for $(N,\beta)$ is nice, from which it follows that the associated diagram for $(M,\gamma)$ is nice as well. The result now follows from Proposition 7.6 and Lemma 5.4 of \cite{Juh08}. 
\end{proof}

\begin{cor}\label{cor:HeegaardSpinCDiffSurf}
	Let $(M,\gamma,\sigma)$ and $(M',\gamma',\sigma)$ be balanced sutured manifolds with suitable distinguished sutures where $(M',\gamma')$ is obtained from $(M,\gamma)$ by a nice surface decomposition along a surface $S$ disjoint from the toral boundary component of $M$ containing $\sigma$. Suppose there are $\Spin^c$ structures $\fk{s}',\fk{t}'$ generating free-summands of $\SFH^-(M',\gamma',\sigma)$ for which $\fk{t}' = \fk{s}' + [\sigma]$. Then there are $\Spin^c$ structures $\fk{s},\fk{t}$ generating free-summands of $\SFH^-(M,\gamma,\sigma)$ for which $\fk{t} = \fk{s} + [\sigma]$. 
\end{cor}

\subsubsection{The Hopf link model computation}\label{subsubsec:HopfHeegaard}

We will show that $\HFL^-(K_b \cup C,\sigma)$ satisfies the condition given in Proposition~\ref{prop:SpinCstructDiff} by using Corollary~\ref{cor:HeegaardSpinCDiffSurf} and the sequence of surface decompositions constructed in Theorem~\ref{thm:SeqOfSurfDecomp}. 

\begin{lem}\label{lem:HeegaardHopf}
	Let $H$ be a two-component Hopf link, and let $\sigma$ be any one of the four sutures in the sutured exterior of $H$. Then there are $\Spin^c$ structures $\fk{s},\fk{t}$ generating free summands of $\HFL^-(H,\sigma)$ satisfying $\fk{t} = \fk{s} + [\sigma]$. 
\end{lem}
\begin{figure}[!ht]
		\centering
		\labellist
		\pinlabel $z$ at 240 70
		\pinlabel $a$ at 150 300
		\pinlabel $b$ at 150 155
		\pinlabel $c$ at 300 150
		\pinlabel $d$ at 295 295
		\pinlabel $\alpha$ at 160 430
		\pinlabel $\beta$ at 20 290
		\endlabellist
		\includegraphics[width=.35\textwidth]{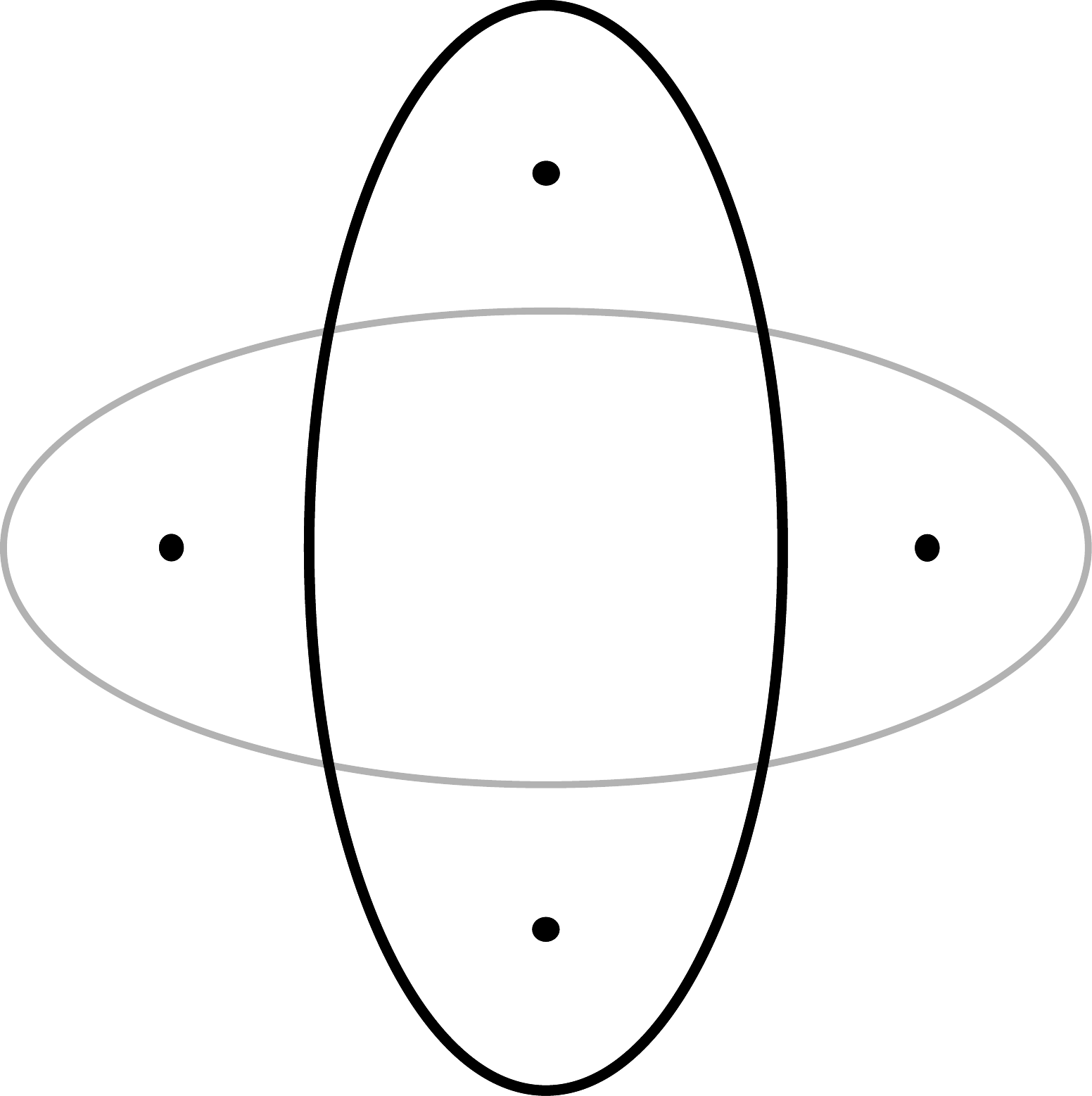}
		\caption{A nice Heegaard diagram for the Hopf link.}
		\label{fig:HopfLinkDiagram}
	\end{figure}
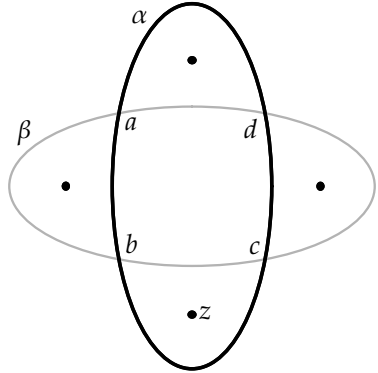
\begin{proof}
	A nice diagram in $S^2$ of the Hopf link is given in Figure~\ref{fig:HopfLinkDiagram}. There are four generators $a,b,c,d$, all lying in distinct $\Spin^c$ structures, and one nonzero differential $\partial b = Uc$. The generators $a,d$ lie in $\Spin^c$ structures differing by $[\sigma]$, and each generate a free summand in $\HFL^-(H,\sigma)$. 
\end{proof}
\begin{rem}
	In Lemma~\ref{lem:instantonHopfLink}, we prove the analogous result for sutured instanton homology. The argument there works equally well here. 
\end{rem}

\begin{thm}[Theorem~\ref{thm:HeegaardFloerHomologyDetectsTrivialBand}]\label{thm:HeegaardTrivBandDetect}
	Let $K_b$ be the band sum of a split two-component link along a band $b$, and let $K_\#$ be the connected sum. Then $b$ is trivial if and only if $\HFKhat(K_b)$ and $\HFKhat(K_\#)$ have the same total dimension with coefficients in $\F_2$. In particular, if $b$ is nontrivial then there is a strict inequality\[
		\dim \HFKhat(K_b,g(K_b)) > \dim \HFKhat(K_\#,g(K_b))
	\]in Alexander grading $g(K_b)$, the Seifert genus of $K_b$. 
\end{thm}
\begin{proof}
	To show that $\dim \HFKhat(K_b) > \dim \HFKhat(K_\#)$, it suffices to show that \[
		\rank \HFL^-(K_b \cup C,\sigma) > \rank \HFL^-(K_\# \cup C,\sigma)
	\]where $\sigma$ is a suture on $\partial N(C)$ in both cases by Proposition~\ref{prop:rankEqualsTwiceDim}. By Proposition~\ref{prop:SpinCstructDiff}, it suffices to show that there are free summands of $\HFL^-(K_b \cup C,\sigma)$ whose generators lie in $\Spin^c$ structures differing by $[\sigma]$. Assume that the band $b$ is nontrivial, Theorem~\ref{thm:SeqOfSurfDecomp} gives a sequence of nice surface decompositions \[
		S^3(K_b \cup C) \overset{S_1}{\rightsquigarrow} (M_1,\gamma_1) \overset{S_2}{\rightsquigarrow} \cdots \overset{S_n}{\rightsquigarrow} (M_n,\gamma_n)
	\]where each $S_i$ is disjoint from $\partial N(C)$, and $(M_n,\gamma_n)$ is the disjoint union of a product sutured manifold and $S^3(H)$, the sutured exterior of a Hopf link. By Lemma~\ref{lem:HeegaardHopf}, we know that $\SFH^-(M_n,\gamma_n,\sigma)$ has free summands with generators lying in $\Spin^c$ structures differing by $[\sigma]$, so by Corollary~\ref{cor:HeegaardSpinCDiffSurf}, we know that $\SFH^-(S^3(K_b \cup C),\sigma) = \HFL^-(K_b \cup C,\sigma)$ does as well. 

	The fact that this inequality may be taken to be in Alexander grading $g(K_b)$ follows from Theorem 1.5 of \cite{Juh08} and the fact that the first surface $S_1$ in the sequence of surface decompositions is along a minimal genus Seifert surface for $K_b$. 
\end{proof}

We obtain Floer-theoretic proofs of the following two known topological results. 

\begin{cor}[Theorem 1 of \cite{MR4058258}]\label{cor:miyazaki}
	The band sum of a two-component split link is isotopic to the connected sum if and only if the band is trivial. 
\end{cor}

\begin{cor}[Theorem 2 of \cite{MR1177410}]\label{cor:kobayashi}
	If $g(K_b) = g(K_\#)$ and $K_b$ is fibered, then $b$ is trivial. 
\end{cor}
\begin{proof}
	Let $g = g(K_b) = g(K_\#)$, and suppose that $b$ is nontrivial. Then by Theorem~\ref{thm:HeegaardTrivBandDetect} and the genus detection result for $\HFKhat$ (Theorem 1.2 of \cite{MR2023281}) \[
		\dim \HFKhat(K_b,g) > \dim \HFKhat(K_\#,g) \ge 1.
	\]Thus $\dim \HFKhat(K_b,g) \ge 2$ so $K_b$ is not fibered (Theorem 1.1 of \cite{MR2357503}).
\end{proof}

%%%%%%%%%%%%%%%%%%%%%%%%%%%%%%%%%%%%%%%
%%%%%%%%%% Instanton &&&&&&&&&%%%%%%%%%
%%%%%%%%%%%%%%%%%%%%%%%%%%%%%%%%%%%%%%%

\section{Instanton Floer homology}

We will use two different versions of instanton Floer homology, which was originally defined by Floer in \cite{MR956166}. The first is \textit{sutured instanton homology} defined in \cite{MR2652464} for balanced sutured manifolds. The invariant $\KHI$ is a special case of this construction. The instanton and Heegaard versions of sutured Floer homology share many formal properties, including the behavior under nice surface decomposition. In section~\ref{subsec:DetectTrivBandInstanton}, we will use the sequence of surface decompositions constructed in Theorem~\ref{thm:SeqOfSurfDecomp} to show that the sutured instanton homology of the sutured exterior of a band sum detects the trivial band. The structure of the argument is the same as for sutured Heegaard Floer homology, but the construction of a suitable minus version of sutured instanton homology is not as straightforward. Section~\ref{subsec:MinusVersionSuturedInstanton} is dedicated towards this task. In section~\ref{subsec:invarianceInstanton}, we show that the sutured Floer homology of a band sum is invariant under adding full twists to the band. 

The other version of instanton homology we consider is \textit{singular instanton homology} \cite{MR2805599} defined for links in closed oriented $3$-manifolds. For links in $S^3$, there is a spectral sequence from Khovanov homology to singular instanton homology, originally used to prove that Khovanov homology detects the unknot. An important feature of the instanton homology package is that the reduced variant of singular instanton homology for a knot in $S^3$ is isomorphic to the sutured instanton homology of its sutured complement. One side of the isomorphism is a theory connected to Khovanov homology through a spectral sequence, and the other side is a theory in which we can use sutured manifold techniques. In section~\ref{subsec:singularInstantonHomology}, we show that reduced and unreduced singular instanton homology detect the trivial band.  

\subsection{Preliminaries}

We review Kronheimer-Mrowka's sutured instanton homology \cite{MR2652464}. For a review of sutured manifolds, see section~\ref{subsubsec:suturedManifolds}. Singular instanton homology is reviewed in section~\ref{subsec:singularInstantonHomology}.

\begin{df*}[closure of a balanced sutured manifold]
	Let $(M,\gamma)$ be a balanced sutured manifold. An \textit{auxiliary surface} for $(M,\gamma)$ is connected compact oriented surface $B$ equipped with an orientation-reversing diffeomorphism $\partial B = s(\gamma)$. Associated to an auxiliary surface $B$ is the \textit{preclosure}\[
		M \cup [-1,1] \x B
	\]of $(M,\gamma)$ where $A(\gamma)$ is identified with $[-1,1] \x \partial B$ using the identification $\partial B = s(\gamma)$. The gluing is done so that $1 \x B$ is glued to $R_+(\gamma)$. The preclosure has two boundary components, labeled $\bar{R}_+$ and $\bar{R}_-$ where $\bar{R}_\pm$ is the boundary component containing $R_\pm(\gamma)$. Because $(M,\gamma)$ is balanced, it follows that $\bar{R}_+$ and $\bar{R}_-$ are diffeomorphic. Fix a diffeomorphism $h\colon \bar{R}_+ \to \bar{R}_-$ which preserves orientations inherited from $R_\pm(\gamma)$ and which sends a point $1 \x q \in 1 \x B \subset \bar{R}_+$ to $-1 \x q \subset -1 \x B \subset \bar{R}_-$. Identify the boundary components of the preclosure using $h$ to form a connected closed oriented manifold $Y$. The boundary components $\bar{R}_\pm$ of the preclosure become a single connected closed oriented surface $\bar{R}$ embedded in $Y$, and the arc $[-1,1] \x q \subset [-1,1] \x B$ becomes a closed oriented loop $\alpha$ in $Y$ intersecting $\bar{R}$ once positively. The triple $(Y,\bar{R},\alpha)$ is called a closure of $(M,\gamma)$ if the genus of $\bar{R}$ is at least $1$. 

	It is useful to record the information of a closure in a slightly different way following \cite{MR3352794}. First of all, instead of forming the closure $Y$ by gluing the two boundary components of the preclosure together, we instead glue $[1,3] \x \bar{R}_+$ to the preclosure so that $1 \x \bar{R}_+$ is identified with $\bar{R}_+$ and $3 \x \bar{R}_+$ is identified with $\bar{R}_-$ using $h$. This way, the sutured manifold $(M,\gamma)$ is embedded in $Y$. We think of $\bar{R} = 2 \x \bar{R}_+$ as the distinguished surface, and we let $\alpha$ be the union of the arcs $[-1,1] \x q \cup [1,3] \x q$. We assume that $g(\bar{R}) \ge 1$, and we choose a non-separating curve $\eta \subset \bar{R} \setminus q$. The tuple $(Y,\bar{R},\eta,\alpha)$ together with the embeddings $M \hookrightarrow Y$ and $[1,3] \x \bar{R} \hookrightarrow Y$ is called a \textit{closure} of $(M,\gamma)$. 
\end{df*}

\begin{df*}[admissible pair]
	Let $Y$ be a connected closed oriented $3$-manifold, and let $\beta \subset Y$ be a closed oriented $1$-manifold. Then $(Y,\beta)$ is an \textit{admissible pair} if $\beta$ intersects some embedded surface transversely in an odd number of points. 
\end{df*}

There is a well-defined instanton Floer homology $I_*(Y)_\beta$ for an admissible pair $(Y,\beta)$ taking the form of a $\Z/8\Z$-graded complex vector space. It is defined by a Morse theory for a Chern-Simons functional defined on a space of $\SO(3)$ connections on a $\U(2)$ bundle over $Y$ whose first Chern class is Poincar\'e dual to $\alpha$. Associated to each class $[S] \in H_2(Y)$ is a linear operator $\mu([S])\colon I_*(Y)_\beta \to I_*(Y)_\beta$. If $[T] \in H_2(Y)$ is another class, then $\mu([S])$ and $\mu([T])$ commute and $\mu([S] + [T]) = \mu([S]) + \mu([T])$. There is an additional operator $\mu(y)$ thought of as associated to a point $y \in Y$, but is independent of the point. This operator also commutes with $\mu([S])$ for each $[S] \in H_2(Y)$. 

\begin{prop}[Corollary 7.2 of \cite{MR2652464} referencing \cite{MR1670396}]
	Let $(Y,\beta)$ be an admissible pair, and let $R$ be a connected closed oriented surface embedded in $Y$ of positive genus with $R \cdot \beta$ odd. The simultaneous eigenvalues of the pair of operators $\mu(R)$ and $\mu(y)$ on $I_*(Y)_\beta$ are a subset of the pairs \[
		(i^r(2k),(-1)^r2)
	\]for integers $k$ with $0 \leq k \leq g(R)-1$ and $r = 0,1,2,3$.
\end{prop}

Let $I_*(Y|R)_\beta$ be the $(2g(R)-2,2)$-generalized eigenspace of $(\mu(R),\mu(y))$. When $g(R) \ge 2$, this is simply the $(2g(R)-2)$-generalized eigenspace of $\mu(R)$.  

\begin{prop}[Proposition 7.5 of \cite{MR2652464}]\label{prop:eigenvalsOfMuS}
	Let $(Y,\beta)$ be an admissible pair, with $R \subset Y$ a connected closed oriented surface with $R \cdot \beta$ odd. 
	Let $S$ be a connected closed oriented surface of positive genus properly embedded in $Y$. The eigenvalues of the operator \[
		\mu(S)\colon I^*(Y|R)_\beta \to I^*(Y|R)_\beta
	\]are a subset of the even integers $2i$ for which $|2i| \leq 2g(S) - 2$. 
\end{prop}

A cobordism $(W,\nu)\colon (Y_1,\beta_1) \to (Y_2,\beta_2)$ between admissible pairs induces a map \[
	I_*(W)_\nu\colon I_*(Y_1)_{\beta_1} \to I_*(Y_2)_{\beta_2}
\] which depends up to sign only on $[\nu] \in H_2(W,\partial W)$ and the isomorphism class of $(W,\nu)$. If surfaces $\Sigma_1 \subset Y_1$ and $\Sigma_2 \subset Y_2$ are homologous in $W$, then $I_*(W)_\nu$ intertwines $\mu(\Sigma_1)$ and $\mu(\Sigma_2)$. See section 2 of \cite{1801.07634} for more. 

Given a closure $(Y,\bar{R},\eta,\alpha)$ of $(M,\gamma)$, both $(Y,\alpha)$ and $(Y,\eta \cup \alpha)$ are admissible pairs. Let the resulting instanton Floer homology groups be $I_*(Y)_\alpha$ and $I_*(Y)_{\eta + \alpha}$. Kronheimer-Mrowka define the sutured instanton homology of $(M,\gamma)$ to be $I_*(Y|\bar{R})_\alpha$. The Floer group $I_*(Y)_{\eta + \alpha}$ is used as an auxiliary tool in a number of their arguments. Following the terminology of \cite{MR3352794}, call $I_*(Y|\bar{R})_\alpha$ (resp. $I_*(Y|\bar{R})_{\eta+\alpha}$) the \textit{untwisted} (resp. \textit{twisted}) sutured instanton homology of the closure. Kronheimer-Mrowka \cite{MR2652464} prove that up to isomorphism, the twisted and untwisted Floer groups are same and depend only on the balanced sutured manifold. If $J$ is a knot in $S^3$, then the \textit{(sutured) instanton knot homology} of $J$, denoted $\KHI(J)$, is the sutured instanton homology of its sutured exterior. Similarly, if $J$ is a link in $S^3$, then $\KHI(J)$ is the sutured instanton homology of its exterior. 

Baldwin-Sivek \cite{MR3352794} establish naturality of Kronheimer-Mrowka's construction of sutured instanton homology. Two linear maps between complex vector spaces are \textit{projectively equivalent} if they differ by a scalar in $\C^\x$. Projective equivalence classes of maps can be composed, and the notions of injectivity, surjectivity, isomorphism, etc. are well-defined. 

\begin{thm}[Section 9 of \cite{MR3352794}]\label{thm:BaldwinSivekNaturality}
	For each pair of closures $\sr D = (Y,\bar{R},\eta,\alpha)$, $\sr D' = (Y',\bar{R}',\eta',\alpha')$ of a balanced sutured manifold $(M,\gamma)$, there is a projective equivalence class $\Psi_{\sr D,\sr D'}$ of isomorphisms \[
		\Psi_{\sr D,\sr D'}\colon I_*(Y|\bar{R})_{\eta + \alpha} \to I_*(Y'|\bar{R}')_{\eta' +\alpha'}
	\]for which $\Id \in \Psi_{\sr D,\sr D}$ and $\Psi_{\sr D',\sr D''} \circ \Psi_{\sr D,\sr D'} = \Psi_{\sr D,\sr D''}$. 

	Fix an isotopy class of diffeomorphisms $(M,\gamma) \to (N,\beta)$ of balanced sutured manifolds. Then for each pair of closures $(Y,\bar{R},\eta,\alpha), (Y'',\bar{R}'',\eta'',\alpha'')$ of $(M,\gamma),(N,\beta)$, there is an equivalence class of isomorphisms \[
		I_*(Y|\bar{R})_{\eta + \alpha} \to I_*(Y''|\bar{R}'')_{\eta''+\alpha''}.
	\]These maps commute with the above $\Psi_{\sr D,\sr D'}$ classes. These maps are also functorial with respect to composing isotopy classes of diffeomorphisms. 
\end{thm}

\begin{rem}\label{rem:projectiveEquiv}
 	We will let $\SHI(M,\gamma)$ denote the (twisted) sutured instanton homology of $(M,\gamma)$. One can view $\SHI(M,\gamma)$ either as a particular group defined using a closure, or as the groupoid defined by Theorem~\ref{thm:BaldwinSivekNaturality}. 

 	A map between sutured instanton homology groups will mean a projective equivalence class of maps between groups defined by particular closures. A \textit{natural} map will mean a map defined for every pair of closures which commute with the \textit{canonical isomorphisms} $\Psi_{\sr D,\sr D'}$. A natural map $\SHI(M,\gamma) \to \SHI(N,\beta)$ is typically constructed in the following way. Given a closure $\sr D_M = (Y,R,\eta,\alpha)$ of $(M,\gamma)$ with special properties specific to the situation, we then construct a closure $\sr D_N = (Y_*,R_*,\eta_*,\alpha_*)$ of $(N,\beta)$ along with a map \[
 		I_*(Y|R)_{\eta + \alpha} \to I_*(Y_*|R_*)_{\eta_* + \alpha_*}.
 	\]This particular map defined using $\sr D$ then defines a natural map $\SHI(M,\gamma) \to \SHI(N,\beta)$ just by pre- and post-composing with the canonical isomorphisms. However, if we start with a different closure $\sr D'_M$ of $(M,\gamma)$ which still satisfies the same special properties, the resulting natural map may differ. The work in showing that a map is natural is that these resulting natural maps are the same. To do this, it suffices to show that the maps defined using $\sr D_M$ and $\sr D'_M$ intertwine the canonical isomorphisms $\Psi_{\sr D_M,\sr D_M'}$ and $\Psi_{\sr D_N,\sr D_N'}$. 
\end{rem}

\subsection{A minus version of sutured instanton Floer homology}\label{subsec:MinusVersionSuturedInstanton}

In this section, we define a minus version of sutured instanton homology for balanced sutured manifolds with a suitable distinguished suture. Recall the following definition from section~\ref{subsubsec:minusHeegaardSuturedFloer}. 

\begin{df*}[suitable distinguished suture]
	Let $(M,\gamma)$ be a balanced sutured manifold. We say that one of the sutures $\sigma$ is a \textit{suitable distinguished suture} if \begin{enumerate}[itemsep=-0.5ex]
		\item $\sigma$ lies on a toral boundary component $T$ of $M$,
		\item $s(\gamma) \cap T$ consists of two oppositely oriented parallel copies of $\sigma$,
		\item There exists a properly embedded compact oriented surface $\Sigma \subset M$ for which $\partial \Sigma \cap T$ is a simple closed curve that intersects $\sigma$ in a single point transversely.
	\end{enumerate}
\end{df*}

The definition of a minus version of sutured instanton homology for $(M,\gamma,\sigma)$ is based on a construction that first appeared in \cite{MR3650078} in the context of Heegaard Floer homology, which was subsequently adapted for instanton (and monopole) Floer homology in \cite{MR3352794,MR3477339,1801.07634,1810.13071,1901.06679,1910.01758}. In section~\ref{subsubsec:contactStructures}, we briefly explain the alternative construction of the minus version of knot Floer homology in \cite{MR3650078} and the analogous results in instanton Floer homology needed to adapt the construction to define a minus version of instanton knot Floer homology. We construct our minus version of sutured instanton homology in section~\ref{subsubsec:constructionOfTheInvariant} and prove a number of features of the invariant in section~\ref{subsubsec:featuresOfTheInvariant}. Finally in section~\ref{subsubsec:gradingSuturedInstanton}, we prove a few technical results about gradings used in previous sections. 

\subsubsection{Contact structures and a minus version of knot Floer homology}\label{subsubsec:contactStructures}

Let $\xi$ be a (cooriented and positive) contact structure on a compact oriented $3$-manifold $M$. A surface $F$ properly embedded in $M$ is \textit{convex} if there is a contact vector field $v$ on $M$ transverse to $F$. The surface $F$ and the vector field $v$ determine a \textit{dividing set} $\Gamma$ on $F$ which is a $1$-manifold in $F$ which divides $F$ into two subsurfaces $F_+ \cup F_-$ that meet along $\Gamma$. Furthermore, if $F$ is given an orientation, there is a natural orientation on $\Gamma$ which agrees with the boundary orientation induced from $F_+$. Any other contact vector field $v$ transverse to $F$ will yield an isotopic dividing set. 

\begin{df*}[sutured contact manifold]
	A contact structure $\xi$ on a balanced sutured manifold $(M,\gamma)$ is a contact structure on $M$ for which $\partial M$ is convex and has dividing set isotopic to $\gamma$.
	Two contact structures $\xi_0$ and $\xi_1$ on $(M,\gamma)$ are equivalent if they lie in a $1$-parameter family $\xi_t$ for $t \in [0,1]$ of contact structures on $(M,\gamma)$. 
\end{df*}

\begin{thm}[Theorem 1.1 of \cite{0807.2431}]\label{thm:HKMgluing}
	Let $(M_1,\gamma_1)$ and $(M_2,\gamma_2)$ be balanced sutured manifolds. Suppose that $M_1$ is embedded in the interior of $M_2$, and every component of $M_2\setminus \Int(M_1)$ contains a boundary component of $M_2$. Associated to an equivalence class of contact structures $\xi$ on the balanced sutured manifold $(M_2\setminus \Int(M_1),\gamma_1 \cup \gamma_2)$ is a natural ``gluing'' map \[
		\phi_{\xi}\colon \SFH(-M_1,-\gamma_1) \to \SFH(-M_2,-\gamma_2)
	\]defined over $\F_2$. The map is also defined over $\Z$, but with a sign ambiguity. 
\end{thm}
Naturality refers to a suitable independence of the choices of auxiliary data needed to define the sutured Floer homology groups. If $M_1$ is the complement of an open tubular neighborhood $(0,1] \x \partial M_2$ of the boundary of $M_2$ with the same set of sutures, then there is a $[0,1]$-translation invariant contact structure $\xi$ on $[0,1] \x \partial M_2$ with dividing set $\{0,1\} \x \gamma_2$. In this case, the gluing map $\phi_\xi$ is the identity. If we have two inclusions $(M_1,\gamma_1) \subset (M_2,\gamma_2)$ and $(M_2,\gamma_2) \subset (M_3,\gamma_3)$ satisfying the hypotheses of Theorem~\ref{thm:HKMgluing} along with contact structures $\xi_{12}$ and $\xi_{23}$ on $M_2\setminus \Int(M_1)$ and $M_3\setminus \Int(M_2)$, respectively, compatible with the sutures, then there is a composite contact structure $\xi_{13} = \xi_{12} \cup \xi_{23}$ on $M_3\setminus\Int(M_1)$ compatible with the sutures, and $\phi_{\xi_{13}} = \phi_{\xi_{23}}\circ\phi_{\xi_{12}}$. 

\vspace{20pt}

Let $K$ be a null-homologous oriented knot in a closed oriented $3$-manifold $Y$.
We outline how the authors of \cite{MR3650078} recover $\HFK^-(-Y,K)$ as a $\F_2[U]$-module with an Alexander grading from a limit of certain sutured Floer homology groups. Let $N(K)$ be an tubular neighborhood of $K$, and parametrize the knot complement as $Y\setminus N(K) \cup [0,\infty) \x \partial N(K)$. For each $n \ge 1$, let $X_n = Y\setminus N(K) \cup [0,n] \x \partial N(K)$, and let $X_n$ have a pair of parallel oppositely oriented sutures $\gamma_n$ of slope $\lambda - n\mu$ where $\mu$ is the meridian and $\lambda$ is the longitude. For each $n < m$, the sutured manifold $(X_n,\gamma_n)$ is embedded in the interior of $(X_m,\gamma_m)$, and $X_m\setminus \Int(X_n)$ is the thicken torus $[n,m] \x \partial N(K)$. For each $n$, a particular contact structure $\xi_n^-$ on $[n,n+1] \x \partial N(K)$ compatible with $\gamma_n \cup \gamma_{n+1}$ is chosen. These contact structures induce gluing maps $\psi_n^- = \phi_{\xi_n^-}$ \[
	\begin{tikzcd}[column sep=large]
		\cdots \ar[r,"\psi_{n-1}^-"] & \SFH(-X_n,-\gamma_n) \ar[r,"\psi_n^-"] & \SFH(-X_{n+1},-\gamma_{n+1}) \ar[r,"\psi_{n+1}^-"] & \cdots
	\end{tikzcd}
\]and we let $H^-$ be the colimit of this diagram in the category of vector spaces over $\F_2$. 

There is another special contact structure $\xi_n^+$ on $[n,n+1] \x \partial N(K)$ compatible with $\gamma_n \cup \gamma_{n+1}$ which has the property that the two contact structures $\xi_n^+ \cup \xi_{n+1}^-$ and $\xi_n^- \cup \xi_{n+1}^+$ on $([n,n+2] \x \partial N(K), \gamma_n \cup \gamma_{n+2})$ are equivalent. For $\psi_n^+ = \phi_{\xi_n^+}$, it follows that the diagram \[
	\begin{tikzcd}[column sep=large,row sep =large]
		\cdots \ar[r,"\psi_{n-1}^-"] & \SFH(-X_n,-\gamma_n) \ar[r,"\psi_{n}^-"] \ar[d,"\psi_n^+"] & \SFH(-X_{n+1},-\gamma_{n+1}) \ar[r,"\psi_{n+1}^-"] \ar[d,"\psi_{n+1}^+"] & \cdots\\
		\cdots \ar[r,"\psi_{n}^-"] & \SFH(-X_{n+1},-\gamma_{n+1}) \ar[r,"\psi_{n+1}^-"] & \SFH(-X_{n+2},-\gamma_{n+2}) \ar[r,"\psi_{n+2}^-"] & \cdots
	\end{tikzcd}
\]commutes. The $\psi_n^+$ maps therefore induce a map $U\colon H^-\to H^-$ on the colimit, and it is with respect to this map that $H^-$ is made an $\F_2[U]$-module. Finally, the relative homology class of a Seifert surface yields gradings on each sutured Floer homology group in the sequence with respect to which the $\psi_-$ and $\psi_+$ maps are homogeneous, so that there is an induced grading on $H^-$ with respect to which $U$ is homogeneous of degree $-1$. The graded $\F_2[U]$-module $H^-$ defined in this way is isomorphic to the $\F_2[U]$-module $\HFK^-(-Y,K)$ with the Alexander grading. 

\vspace{20pt}

The contact structures $\xi_n^-,\xi_n^+$ on $([n,n+1] \x \partial N(K), \gamma_n \cup \gamma_{n+1})$ are tight, and there is a characterization of them in terms of a classification of tight contact structures on thickened tori. See section 2.1.3 of \cite{MR3650078} which references \cite{MR1779622,MR1786111}. We will instead specify them through a sequence of \textit{contact handle attachments}. An explicit construction of the Honda-Kazez-Mati\'c gluing maps in terms of contact handles was given in \cite{MR4080483}. 

We first explain the topology of contact handle attachments. Let $h^i$ for $i = 0,1,2,3$ be a $3$-ball with a single suture. We next specify how attaching regions for each handle must intersect the suture. \begin{itemize}[itemsep=-0.5ex]
	\item The attaching region for $h^0$ is empty. 
	\item The attaching region for $h^1$ consists of two disjoint discs on the boundary centered at points on the suture.
	\item The attaching region for $h^2$ is an annular neighborhood of a simple closed curve on the boundary that intersects the suture twice transversely.
	\item The attaching region for $h^3$ is the entire boundary.
\end{itemize}Let $(M,\gamma)$ be a balanced sutured manifold. A contact handle attachment of index $i \in \{0,1,2,3\}$ is a smooth handle attachment subject to the constraint that the attaching regions of $M$ and $h^i$ are identified in a way that identifies their intersections with the sutures. The sutures within the attaching regions must also be identified in an orientation-reversing way. The result of a contact handle attachment is another sutured manifold. It will be balanced except for the possibility that a $3$-handle attachment creates a closed component of the sutured manifold. 

For each contact handle attachment, a \textit{contact handle attachment map} \[
	\SFH(-M,-\gamma) \to \SFH(-M',-\gamma')
\]is constructed in \cite{MR4080483}  where $(M',\gamma')$ is obtained by attaching a contact handle to $(M,\gamma)$. If $(M'',\gamma'')$ is obtained from $(M,\gamma)$ by two different sequences of contact handle attachments, we have two maps $\SFH(-M,-\gamma) \to \SFH(-M'',-\gamma'')$ obtained by the composites of contact handle attachment maps. These two maps can be different. However, each contact handle carries a tight contact structure compatible with the suture, unique up to equivalence, and if the two sequences of contact handle attachments determine the equivalent contact structures then the two maps will be the same. More precisely, we shrink $(M,\gamma)$ slightly within $(M'',\gamma'')$ along a collar neighborhood of its boundary so that $M$ is embedded in the interior of $M''$. We may then glue an $I$-invariant contact structure on this collar neighborhood to the contact structures on the handles to obtain a contact structure on $M''\setminus \Int(M)$ compatible with the sutures. If the two sequences determine the same contact structure on $M''\setminus \Int(M)$, then they will induce the same map. In particular, it is shown in \cite{MR4080483} that they will induce the Honda-Kazez-Mati\'c gluing map. 

The contact structures $\xi_n^-,\xi_n^+$ on $([n,n+1] \x \partial N(K),\gamma_n \cup \gamma_{n+1})$ arise as \textit{bypass attachments} \cite{MR1786111}, so we first briefly explain bypass attachments more generally and Honda's bypass exact triangle following \cite{1801.07634,MR2774055}. Let $\alpha$ be an arc embedded in the boundary of a balanced sutured manifold $(M,\gamma)$, transverse to $s(\gamma)$. We require that the endpoints of $\alpha$ lie on $s(\gamma)$ and the interior of $\alpha$ intersects $s(\gamma)$ in a single point. An arc of this form is called a (\textit{bypass}) \textit{attaching arc} for $(M,\gamma)$. Two attaching arcs are equivalent if they are isotopic through attaching arcs. Given an attaching arc $\alpha$, first attach a contact $1$-handle to $(M,\gamma)$ along the endpoints of $\alpha$. Then attach a canceling contact $2$-handle along a circle consisting of the arc $\alpha$ and an arc which runs along the $1$-handle. This contact $2$-handle attachment is uniquely specified once we require that the sutures change in the way shown in Figure~\ref{fig:bypass}. The composite of this pair of contact handles is referred to as a bypass attachment along $\alpha$. 

\begin{figure}[!ht]
	\centering
	\vspace{7.5pt}
	\includegraphics[width=.15\textwidth]{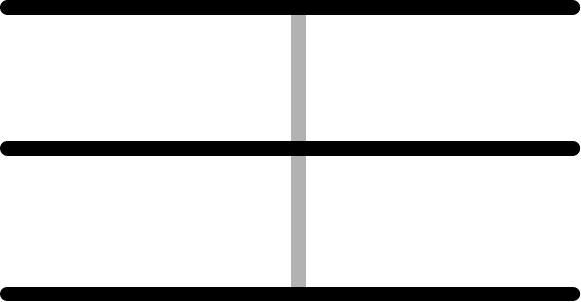}
	\hspace{20pt}
	\includegraphics[width=.15\textwidth]{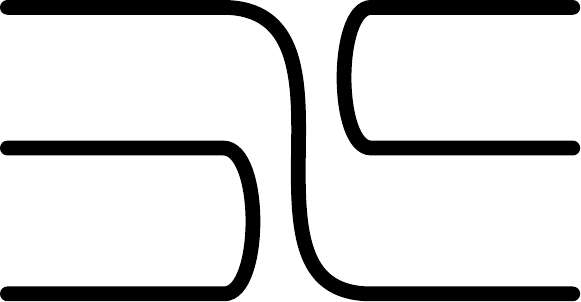}
	\caption{A bypass attachment. The outward normal points out of the page.}
	\label{fig:bypass}
\end{figure}

The intersection of the cocore of the $1$-handle with the boundary of the resulting sutured manifold is new bypass arc. By applying a bypass attachment to this arc, and repeating this procedure, we obtain a $3$-periodic sequence of sutured manifolds shown in Figure~\ref{fig:bypassTriangle}. According to unpublished work of Honda, the resulting triangle of bypass attachment maps form an exact triangle. 

\begin{figure}[!ht]
	\centering
	\vspace{7.5pt}
	\includegraphics[width=.35\textwidth]{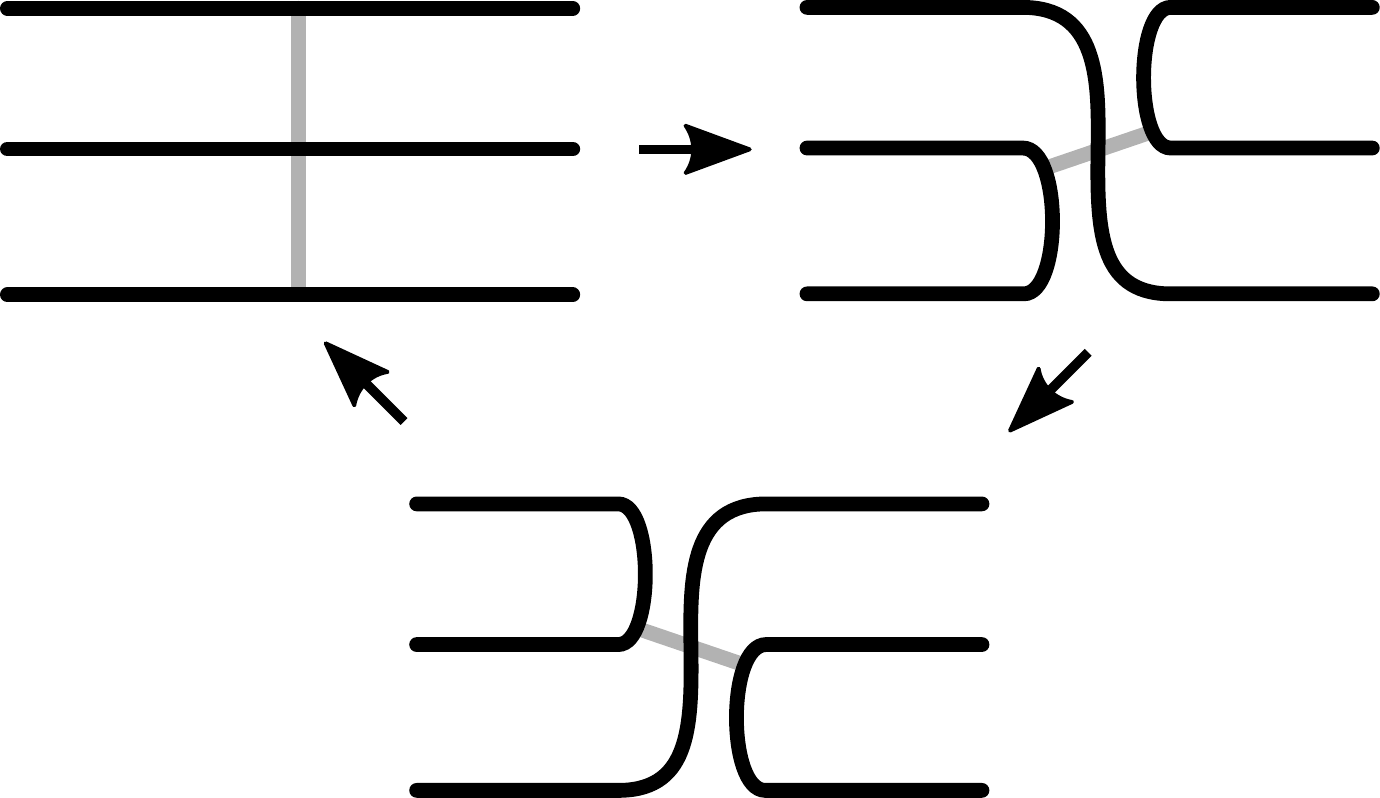}
	\caption{The bypass triangle.}
	\label{fig:bypassTriangle}
\end{figure}

We return to specifying the contact structures $\xi_n^-,\xi_n^+$ on $[n,n+1] \x \partial N(K)$ which define the maps $\psi_-,\psi_+\colon \SFH(-X_n,-\gamma_n) \to \SFH(-X_{n+1},-\gamma_{n+1})$. The sutures $\gamma_n$ on $X_n$ are parallel oppositely oriented curves of slope $\lambda - n \mu$ on $\partial X_n = \{n\} \x \partial N(K)$. Let $\lambda_n$ be the suture oriented as $\lambda - n\mu$ and let $-\lambda_n$ be the other suture. For each $n \ge 1$, let $\alpha_{n+1}^-$ be a bypass attaching arc on $\partial X_{n+1}$ which intersects $\lambda_{n+1}$ precisely at its endpoints and which is disjoint from both $\mu$ and $\lambda$. This arc is unique up to isotopy through attaching arcs. The sutured manifolds in the associated bypass exact triangle are \[
	\begin{tikzcd}[column sep=small]
		(X_{n+1},\gamma_{n+1}) \ar[rr,"\alpha_{n+1}^-"] & & Y(K) \ar[dl]\\
		& (X_{n},\gamma_{n}) \ar[ul] &
	\end{tikzcd}
\]where $Y(K)$ denotes the exterior of $K$ with two meridional sutures. The bypass attachment $(X_{n},\gamma_{n}) \to (X_{n+1},\gamma_{n+1})$ is the contact structure $\xi_{n}^-$. Similarly, let $\alpha_{n+1}^+$ be the unique bypass attaching arc on $\partial X_{n+1}$ which intersects $-\lambda_{n+1}$ precisely at its endpoints and which is disjoint from both $\mu$ and $\lambda$. The exact triangle involves the same three sutured manifolds and the bypass attachment $(X_{n},\gamma_{n}) \to (X_{n+1},\gamma_{n+1})$ in this triangle is $\xi_{n}^+$. 

\vspace{20pt}

As noted in Remark 1.11 of \cite{MR3650078}, this limiting invariant can be defined for any homology theory for sutured manifolds with a contact gluing result. We explain the relevant results in sutured instanton homology that have been developed towards this end. 

First of all, in order to make sense of maps between sutured instanton homology groups, a version of naturality for $\SHI(M,\gamma)$ is needed; the relevant result is proven in \cite{MR3352794} and stated here in Theorem~\ref{thm:BaldwinSivekNaturality}. As noted before, since the canonical isomorphisms between the sutured Floer homology groups constructed from different closures are well-defined up to a unit, all equalities of maps are taken to mean up equality to rescaling by an element of $\C^\x$. With this naturality, Baldwin-Sivek construct contact handle attachment maps. 

\begin{prop}[Section 3 of \cite{MR3477339}]\label{prop:instantonContactHandleAttachmentMaps}
	If $(M',\gamma')$ is obtained by attaching a contact handle to the balanced sutured manifold $(M,\gamma)$, then there is a natural contact handle attachment map \[
		\SHI(-M,-\gamma) \to \SHI(-M',\gamma').
	\]The contact $0$- and $1$-handle attachment maps are isomorphisms.
\end{prop}
\begin{df*}[Section 4 of \cite{1801.07634}]
	Let $(M',\gamma')$ be obtained from $(M,\gamma)$ by a bypass attachment along an attaching arc. The bypass attachment can realized by a contact $1$-handle attachment followed by a contact $2$-handle attachment. Let the \textit{bypass attachment map} \[
		\SHI(-M,-\gamma) \to \SHI(-M',-\gamma')
	\]be the composite of the two associated contact handle attaching maps in Proposition~\ref{prop:instantonContactHandleAttachmentMaps}. 
\end{df*}

Given a null-homologous oriented knot $K$ in a closed oriented $3$-manifold $Y$, let $X$ be its exterior, and let $\gamma_n$ consist of two parallel oppositely oriented sutures on $\partial X$ of slope $\lambda - n\mu$ where $\lambda$ is the longitude and $\mu$ is the meridian. Let $\psi_n^-$ and $\psi_n^+$ be the bypass attachment maps associated to the bypass attachments $\xi_n^-$ and $\xi_n^+$, respectively, from $(X,\gamma_n)$ to $(X,\gamma_{n+1})$. Then consider the colimit $H^-$ of \[
	\begin{tikzcd}[column sep=large]
		\cdots \ar[r,"\psi_{n-1}^-"] & \SHI(-X,-\gamma_n) \ar[r,"\psi_n^-"] & \SHI(-X,-\gamma_{n+1}) \ar[r,"\psi_{n+1}^-"] & \cdots
	\end{tikzcd}
\]in the category of vector spaces over $\C$. This colimit $H^-$ is the underlying vector space of the candidate minus version of instanton knot Floer homology. To define the $U$ action, we need the identity $\psi_{n+1}^-\circ\psi_n^+ = \psi_{n+1}^+ \circ \psi_n^-$. This result was proven for sutured Heegaard Floer homology using the Honda-Kazez-Mati\'c gluing maps. Although this particular identity follows from their construction of the bypass attachment maps, Baldwin-Sivek do not show in general that the composite of a sequence of contact handle attachment maps depends only on the resulting contact structure. However, they do prove the exactness of the bypass triangle.

\begin{thm}[Theorem 1.20 of \cite{1801.07634}]\label{thm:instantonByPassTri}
	Given a bypass triangle of sutured manifolds \[
		\begin{tikzcd}[column sep=small,row sep=small]
			(M,\gamma_1) \ar[rr] & & (M,\gamma_2) \ar[dl]\\
			& (M,\gamma_3) \ar[ul] &
		\end{tikzcd}
	\]the triangle of bypass attachment maps \[
		\begin{tikzcd}[column sep=-1ex,row sep=small]
			\SHI(-M,-\gamma_1) \ar[rr] & & \SHI(-M,-\gamma_2) \ar[dl]\\
			& \SHI(-M,-\gamma_3) \ar[ul] & 
		\end{tikzcd}
	\]is exact.
\end{thm}

In \cite{1810.13071}, Li proves that the composite of a sequence of contact handle attachment maps depends only on the resulting contact structure, thereby defining gluing maps for sutured instanton homology. 

\begin{thm}[Theorem 1.1 of \cite{1810.13071}]\label{thm:instantonGluingMaps}
	Let $(M_1,\gamma_1)$ and $(M_2,\gamma_2)$ be balanced sutured manifolds. Suppose that $M_1$ is embedded in the interior of $M_2$, and every component of $M_2 \setminus \Int(M_1)$ contains a boundary component of $M_2$. Associated to an equivalence class of contact structures $\xi$ on the balanced sutured manifold $(M_2\setminus \Int(M_1),\gamma_1 \cup \gamma_2)$ is a natural gluing map \[
		\phi_{\xi}\colon \SHI(-M_1,-\gamma_1) \to \SHI(-M_2,-\gamma_2)
	\]satisfying the following properties: \begin{enumerate}
		\item If $M_1$ is the complement of an open tubular neighborhood $(0,1] \x \partial M_2$ of the boundary of $M_2$ with the same set of sutures, then gluing map for the $I$-invariant contact structure is the identity. 
		\item If $(M_1,\gamma_1) \subset (M_2,\gamma_2)$ and $(M_2,\gamma_2) \subset (M_3,\gamma_3)$ satisfy the hypotheses of the theorem, then the gluing map $\phi_{\xi_{13}}$ associated to the composite contact structure $\xi_{13}$ on $M_3\setminus \Int(M_1)$ of contact structures $\xi_{12}$ on $M_2\setminus \Int(M_1)$ and $\xi_{23}$ on $M_3\setminus \Int(M_2)$ respecting the sutures satisfies $\phi_{\xi_{13}} = \phi_{\xi_{23}}\circ\phi_{\xi_{12}}$.
		\item If $(M_2,\gamma_2)$ is obtained by attaching a contact handle to $(M_1,\gamma_1)$, then after shrinking $M_1$ slightly within a collar neighborhood of its boundary within $M_2$, the gluing map $\phi_{\xi}$ associated to the contact structure $\xi$ obtained from an $I$-invariant contact structure on the collar neighborhood and the contact handle coincides with the contact handle attachment map. 
	\end{enumerate}
\end{thm}

\begin{cor}\label{cor:instantonDisjointAttachmentArcs}
	Suppose $\alpha$ and $\beta$ are disjoint attaching arcs on $(M,\gamma)$. Then the diagram \[
		\begin{tikzcd}
			\SHI(-M,-\gamma) \ar[r] \ar[d] & \SHI(-M',-\gamma') \ar[d]\\
			\SHI(-M'',-\gamma'') \ar[r] & \SHI(-M''',-\gamma''')
		\end{tikzcd}
	\]commutes where the horizontal arrows are bypass attachment maps along $\alpha$ and the vertical maps are by bypass attachment maps along $\beta$. 
\end{cor}
\begin{proof}
	Theorem~\ref{thm:instantonGluingMaps} applies because the contact structures induced by the composite of the two bypass attachments does not depend on the order of the two attachments. This can also be seen directly from the construction of the bypass attachment maps. 
\end{proof}

We also record here a lemma about trivial bypass attachment arcs. 

\begin{lem}\label{lem:trivialBypassArc}
	Let $\alpha_0$ be an attaching arc on $(M,\gamma)$ for which all three intersection points of $\alpha_0 \cap s(\gamma)$ lie on a single component of $s(\gamma)$. Assume that there is an isotopy $\alpha_t$ in $\partial M$ for $t \in [0,1]$ with $\alpha_t$ a bypass attachment arc for $t \in [0,1)$, and $\alpha_1$ a subarc of $s(\gamma)$. If the bypass attachment along $\alpha$ does not change the balanced sutured manifold, then the bypass attachment map is the identity.
\end{lem}
\begin{proof}
	The attaching arc is \textit{trivial} in the sense used in \cite{MR1940409}, so its contact structure is the standard $I$-invariant contact structure (Lemma 1.9 of \cite{MR1940409}). Theorem~\ref{thm:instantonGluingMaps} then implies the result. 
\end{proof}

By Corollary~\ref{cor:instantonDisjointAttachmentArcs}, the diagram of bypass attachment maps \[
	\begin{tikzcd}[column sep=large, row sep=large]
		\cdots \ar[r,"\psi_{n-1}^-"] & \SHI(-X,-\gamma_n) \ar[r,"\psi_n^-"] \ar[d,"\psi_n^+"] & \SHI(-X,-\gamma_{n+1}) \ar[r,"\psi_{n+1}^-"] \ar[d,"\psi_{n+1}^+"] & \cdots\\
		\cdots \ar[r,"\psi_n^-"] & \SHI(-X,-\gamma_{n+1}) \ar[r,"\psi_{n+1}^-"] & \SHI(-X,-\gamma_{n+2}) \ar[r,"\psi_{n+2}^-"] & \cdots
	\end{tikzcd}
\]commutes where as before $X$ is the exterior of the null-homologous knot $K$ in $Y$, and $\gamma_n$ is the pair of oppositely oriented sutures of slope $\lambda - n\mu$. The induced map $H^- \to H^-$ on colimits is declared to be the $U$ action, which makes $H^-$ into a $\C[U]$-module. 

\vspace{20pt}

The last ingredient is the $\Z$-grading, with respect to which $U$ is homogeneous of degree $-1$. Since we will need a more than a $\Z$-grading in section~\ref{subsec:MinusVersionSuturedInstanton}, we discuss gradings here in the generality we require there. Because the chain complex underlying sutured instanton homology is not known to have a splitting analogous to a splitting along $\Spin^c$ structures, a different approach is used to define a grading. In \cite{MR2652464}, a $\Z$-grading on $\SHI(Y(K))$ is defined using a Seifert surface $\Sigma$ for $K$ in the following way. They first choose a particular closure $(Z,\bar{R},\alpha)$ for $Y(K)$ and choose a particular extension of $\Sigma$ to a closed surface $\bar{\Sigma}$ in $Z$. The grading on $I_*(Z|\bar{R})_{\alpha}$ is defined by the splitting along generalized eigenspaces of $\mu(\bar{\Sigma})$. In \cite{1801.07634}, Baldwin-Sivek extend this construction by defining extensions of $\Sigma$ in closures of any genus and showing that their canonical isomorphisms between sutured Floer homology groups defined using different closures respect these gradings. In the situations considered in \cite{MR2652464,1801.07634}, the Seifert surface defining the grading always intersects the sutures of the ambient manifold in exactly two points. In order to define gradings on $\SHI(-X,-\gamma_n)$ with respect to which $\psi_n^\pm$ are homogeneous, Li defines in \cite{1901.06679} an a further generalization to handle the case where the boundary intersects the sutures in more points. Finally, in \cite{1910.10842}, gradings are defined in the generality that we will need in section~\ref{subsec:MinusVersionSuturedInstanton}. The following definition appears as Definition 2.25 in \cite{1910.10842}. 

\begin{df*}[admissible surface]
	Let $S$ be a compact oriented surface properly embedded in a balanced sutured manifold $(M,\gamma)$. Then $S$ is \textit{admissible} if every component of $\partial S$ intersects $s(\gamma)$, and $\frac12|S \cap s(\gamma)| - \chi(S))$ is an even integer. The second condition is just the requirement that $\pi_0(\partial S)$ and $\frac12|S \cap s(\gamma)|$ have the same parity. 
\end{df*}

Let $S$ be an admissible surface in a balanced sutured manifold $(M,\gamma)$, and let $B$ be an auxiliary surface for $(M,\gamma)$ with an orientation-reversing identification $\partial B = s(\gamma)$. Let $X = M \cup [-1,1] \x B$ be the associated preclosure with $\bar{R}_{\pm} = R_\pm(\gamma) \cup \{\pm1\} \x B$. Let $\delta$ be a collection of properly embedded disjoint arcs in $B$ for which $\partial \delta = \partial S \cap \partial B = \partial S \cap s(\gamma)$. Then on $\bar{R}_+$, we have a collection of closed oriented curves $\partial S \cap R_+(\gamma) \cup \{1 \} \x \delta$ and similarly on $\bar{R}_-$ we have $\partial S \cap R_-(\gamma) \cup \{-1\} \cup \delta$. Observe that $\delta$ determines a partition of $\partial S \cap s(\gamma)$ into sets of two; this partition is called a \textit{balanced pairing} for $S$ if the resulting collections of closed curves on $\bar{R}_+$ and $\bar{R}_-$ have the same number of components. The arcs in $\delta$ inherit orientations from these curves on $\bar{R}_+$, and if the arcs in $\delta$ oriented in this way are linearly independent in $H_1(B,\partial B)$ and $\delta$ induces a balanced pairing for $S$, then there exists a diffeomorphism $h\colon \bar{R}_+ \to \bar{R}_-$ which sends $\partial S \cap R_+(\gamma) \cup \{1\} \x \delta$ to $\partial S \cap R_-(\gamma) \cup \{-1\} \x \delta$ in an orientation-reversing way. 

If the genus of $B$ is large enough, the arcs $\delta$ can be chosen to be independent in $H_1(B,\partial B)$ and to induce a balanced pairing. Assume the genus of $B$ is large enough, and choose $\delta$ to satisfy these conditions. Let $h\colon \bar{R}_+ \to \bar{R}_-$ be a diffeomorphism sending the curves on $\bar{R}_+$ to those on $\bar{R}_-$, and let $Y$ be the closure obtained by gluing $[1,3] \x \bar{R}_+$ to $X$ using $\Id\colon 1 \x \bar{R}_+ \to \bar{R}_+$ and $h\colon 3 \x \bar{R}_+ \to \bar{R}_-$. Let $\bar{S}$ be the union of $S' = S \cup [-1,1] \x \delta \subset X$ and $[1,3] \x (\partial S' \cap \bar{R}_+)\subset [1,3] \x\bar{R}_+$ in $Y$. With suitable $\eta,\alpha \subset Y$, the sutured instanton homology of $(M,\gamma)$ is the $(2g(\bar{R}) - 2,2)$-generalized eigenspace $I_*(Y|\bar{R})_{\eta +\alpha}$ of the operators $\mu(\bar{R}),\mu(y)$ on the instanton homology of the admissible pair $(Y,\eta \cup\alpha)$. The operator $\mu(\bar{S})$ commutes with $\mu(\bar{R})$ and $\mu(y)$, and on $I_*(Y|\bar{R})_{\eta+\alpha}$, every eigenvalue $\lambda$ of $\mu(\bar{S})$ is an even integer satisfying $-(2g(\bar{S}) - 2) \leq\lambda \leq 2g(\bar{S}) - 2$ by Proposition~\ref{prop:eigenvalsOfMuS}. The $i$-th graded portion of $\SHI(M,\gamma)$ in the grading defined by $S$ is declared to be the $2i$-generalized eigenspace of $\mu(\bar{S})$ acting on $I_*(Y|\bar{R})_{\eta+\alpha}$. 

\begin{prop}[Theorem 2.27 of \cite{1910.10842} referencing \cite{1801.07634,1901.06679,1910.11954}]\label{prop:gradingWellDefined}
	The $\Z$-grading on $\SHI(M,\gamma)$ defined by $S$ is well-defined. For any two sets of choices leading to gradings on two closures of $(M,\gamma)$, the canonical isomorphism between the two corresponding instanton homology groups is grading-preserving. 
\end{prop}

\noindent We note that the grading depends on how $S$ intersects $s(\gamma)$, but if $T$ is in the same relative homology class as $S$ and $\partial T = \partial S$, then $T$ and $S$ induce the same grading on $\SHI(M,\gamma)$. The grading is written \[
	\SHI(M,\gamma) = \bigoplus_{i\in\Z} \SHI(M,\gamma,S,i).
\]Similarly, if $S$ and $T$ are isotopic through properly embedded surfaces whose boundaries always intersect $\gamma$ transversely, then $S$ and $T$ intersect $\gamma$ in the same way and define the same grading. 

Recall that a relative $\Z$-grading of a vector space $V$ is a splitting $V = \bigoplus_{i} V_i$ where the index set is an affine space over $\Z$. In particular, an absolute $\Z$-grading of $V = \bigoplus_{i \in \Z} V_i$ determines a relative $\Z$-grading. A different $\Z$-grading $V = \bigoplus_{i \in \Z} V_i'$ determines the same relative $\Z$-grading if there is an integer $k \in \Z$ for which $V_{i+k} = V_i'$ for each $i \in \Z$. We provide an alternate proof of the following result due to Ghosh-Li in section~\ref{subsubsec:gradingSuturedInstanton}. 

\begin{prop}[Theorem 1.12 of \cite{1910.10842}]\label{prop:sameRelGrading}
	Suppose $S$ and $T$ are admissible surfaces in a balanced sutured manifold $(M,\gamma)$ in the same relative homology class. Then $S$ and $T$ determine the same relative $\Z$-grading on $\SHI(M,\gamma)$. 
\end{prop}

In section~\ref{subsubsec:gradingSuturedInstanton}, we also provide a proof of the following result generalizing an argument appearing the proof of Theorem 1.12 of \cite{1910.10842}. 

\begin{prop}\label{prop:simultaneousClosures}
	Let $S$ be an admissible surface in a balanced sutured manifold $(M,\gamma)$, and let $\beta_1,\ldots,\beta_n$ be relative homology class in $H_2(M,\partial M)$. Then there are admissible surfaces $S',T_1,\ldots,T_n$ in $(M,\gamma)$ for which \begin{enumerate}[itemsep=-.5ex]
		\item decomposing $(M,\gamma)$ along either $S$ or $S'$ yields the same sutured manifold,
		\item $T_i$ represents $\beta_i$, 
		\item there exists a closure of $(M,\gamma)$ in which all of the surfaces $S',T_1,\ldots,T_n$ simultaneously extend to closed surfaces in the manner used to define gradings. 
	\end{enumerate} 
\end{prop}

As explained in \cite{1910.10842}, these gradings give rise to a splitting of $\SHI(M,\gamma)$ somewhat analogous to the splitting of $\SFH(M,\gamma)$ along relative $\Spin^c$ structures. An element $x \in \SHI(M,\gamma)$ is declared to be \textit{homogeneous} if it is homogeneous with respect to the relative $\Z$-grading determined by each class in $H_2(M,\partial M)$. Any two homogeneous elements $x,y$ determine a function $H_2(M,\partial M) \to \Z$ given by sending $\alpha \in H_2(M,\partial M)$ to the difference in grading between $x$ and $y$ in the grading determined by $\alpha$. 

\begin{prop}\label{prop:instantonSplittingAlongAffineSpace}
	Let $(M,\gamma)$ be a balanced sutured manifold. Then there is a splitting \[
		\SHI(M,\gamma) = \bigoplus_{h} \SHI(M,\gamma,h)
	\]along an affine space over $\Hom(H_2(M,\partial M),\Z)$. 
\end{prop}
\begin{proof}
	This is essentially Proposition 1.14 of \cite{1910.10842}. Choose a basis $\beta_1,\ldots,\beta_n$ for $H_2(M,\partial M)$, and by Proposition~\ref{prop:simultaneousClosures} we may choose admissible representatives $T_1,\ldots,T_n$ which can be simultaneously closed in a closure $Y$. Let $\bar{T}_1,\ldots,\bar{T}_n$ be their closures. Given any two of these surfaces $\bar{T}_i$ and $\bar{T}_j$, we may make them transverse and smooth out the circles of double points of $\bar{T}_i \cup \bar{T}_j$ to obtain a smooth representative of the homology class $[\bar{T}_i] + [\bar{T}_j]$. This smooth representative is a closed extension of an admissible surface in $(M,\gamma)$ representing $\beta_i + \beta_j$. Using additivity of $\mu$, it follows that if $x \in \SHI(M,\gamma)$ is homogeneous with respect to $\beta_1,\ldots,\beta_n$, then it is homogeneous with respect to each class in $H_2(M,\partial M)$. It also follows that the function $H_2(M,\partial M) \to \Z$ determined by two homogeneous elements of $\SHI(M,\gamma)$ is a group homomorphism. The splitting of $\SHI(M,\gamma)$ along generalized eigenspaces of $\mu(\bar{T}_i),\ldots,\mu(\bar{T}_j)$ thought of as an absolute $\Z^n$ grading depends on the choices of $T_1,\ldots,T_n$, but the same splitting with index set thought of as an affine space over $\Hom(H_2(M,\partial M),\Z)$ is an invariant of the balanced sutured manifold. 
\end{proof}

\vspace{20pt}

We return to Li's definition of a $\Z$-grading on the colimit $H^-$ with respect to which $U$ is homogeneous of degree $1$. The following result relates the gradings we have discussed to bypass attachments. 

\begin{prop}[Proof of Proposition 5.5 of \cite{1901.06679}]\label{prop:bypassGradings}
	Let $S$ be an admissible surface in a balanced sutured manifold $(M,\gamma)$, and let $\alpha$ be an attaching arc on $\partial M$ which is disjoint from $\partial S$. Let $(M,\gamma')$ be the balanced sutured manifold obtained by a bypass attachment along $\alpha$, and observe that $S$ is also an admissible surface in $(M,\gamma')$. The bypass attachment map \[
		\SHI(-M,-\gamma) \to \SHI(-M,-\gamma')
	\]preserves the $\Z$-gradings defined by $S$. 
\end{prop}

\noindent Fix a Seifert surface $\Sigma$ for the null-homologous knot $K \subset Y$, and view it as a surface in the exterior $X$ of $K$. For $\lambda = \partial \Sigma \cap \partial X$, let $\gamma_n$ as before consist of two parallel oppositely oriented sutures on $\partial X$ of slope $\lambda - n\mu$. We may arrange $\gamma_n$ so that it intersects $\partial \Sigma$ in exact $2n$ points on $T$. Note that $\Sigma$ is admissible in $(X,\gamma_n)$ if and only if $n$ is odd. By Proposition~\ref{prop:gradingWellDefined}, $\Sigma$ defines a $\Z$-grading on $\SHI(-X,-\gamma_n)$ when $n$ is odd. 
\begin{figure}[!ht]
	\centering
	\labellist
	\pinlabel $\alpha_n^-$ at 680 610
	\pinlabel $\lambda$ at 465 720
	\pinlabel $-\mu$ at 760 410
	\pinlabel $\alpha_n^+$ at 1555 610
	\pinlabel $\lambda$ at 1335 720
	\pinlabel $-\mu$ at 1630 410
	\endlabellist
	\includegraphics[width=.45\textwidth]{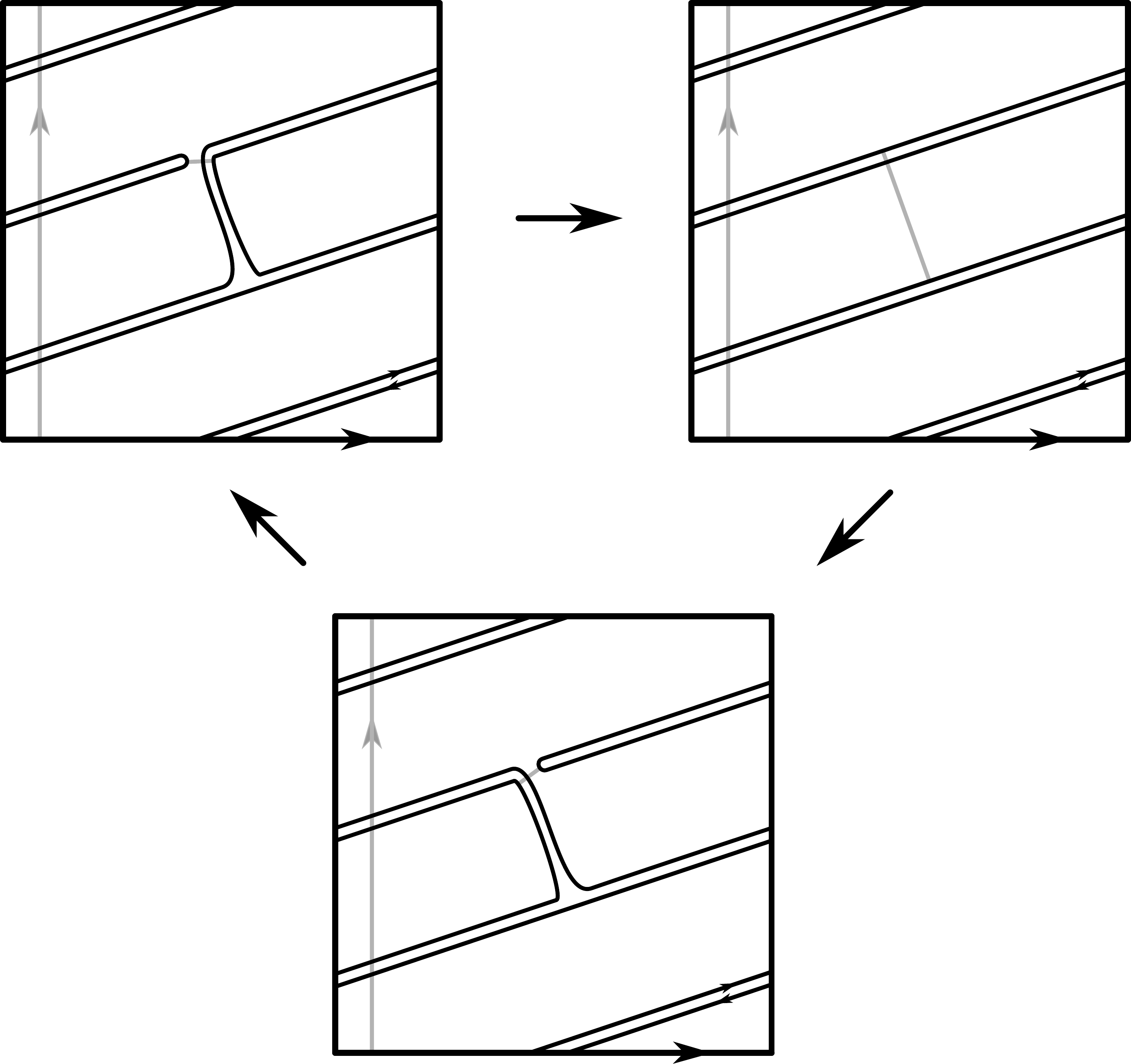}
	\hspace{10pt}
	\includegraphics[width=.45\textwidth]{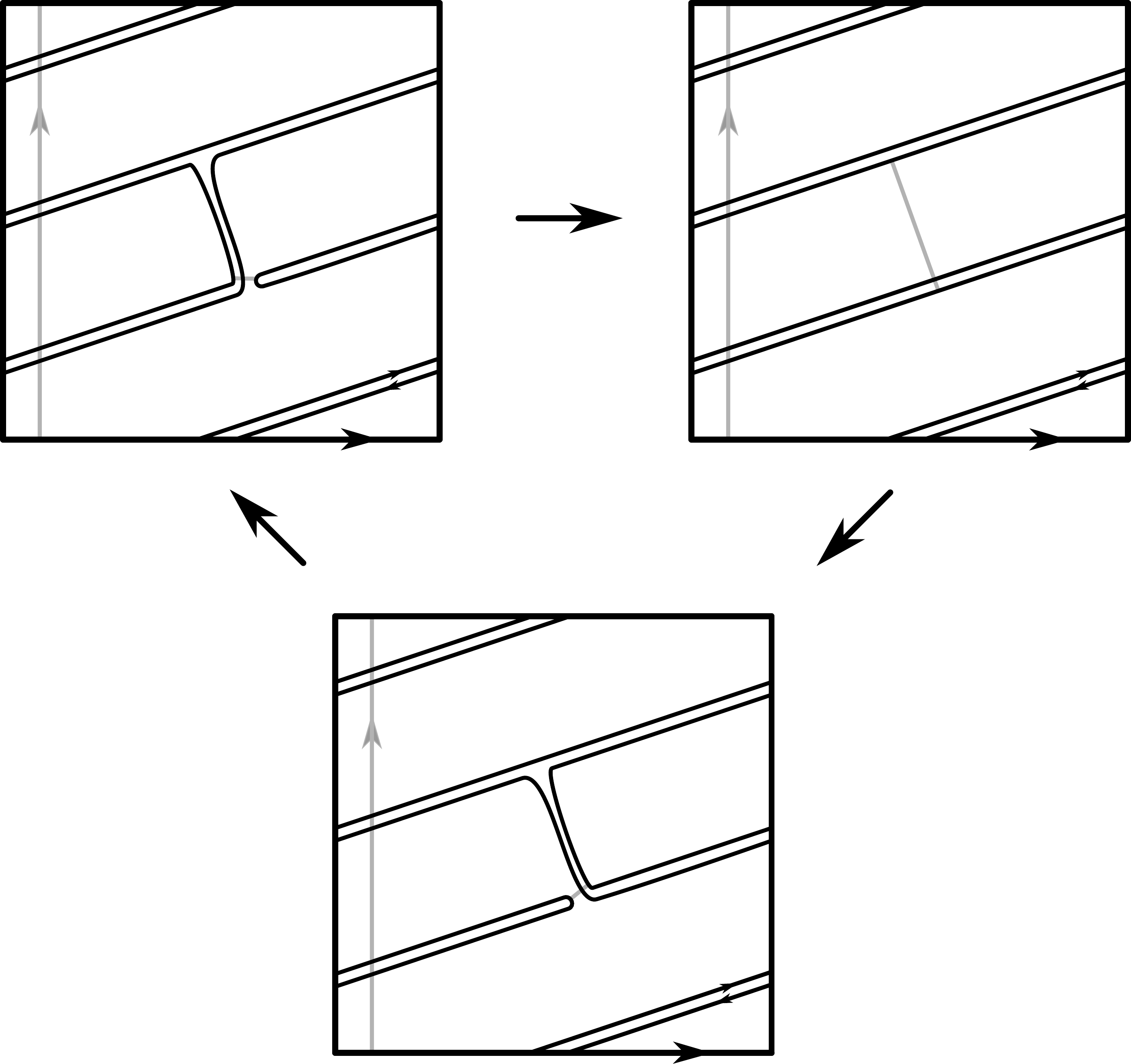}
	\caption{(Figure 3 of \cite{1901.06679}). Bypass attachment triangles associated to $\alpha_n^-,\alpha_n^+$.}
	\label{fig:triangleGradings}
\end{figure}
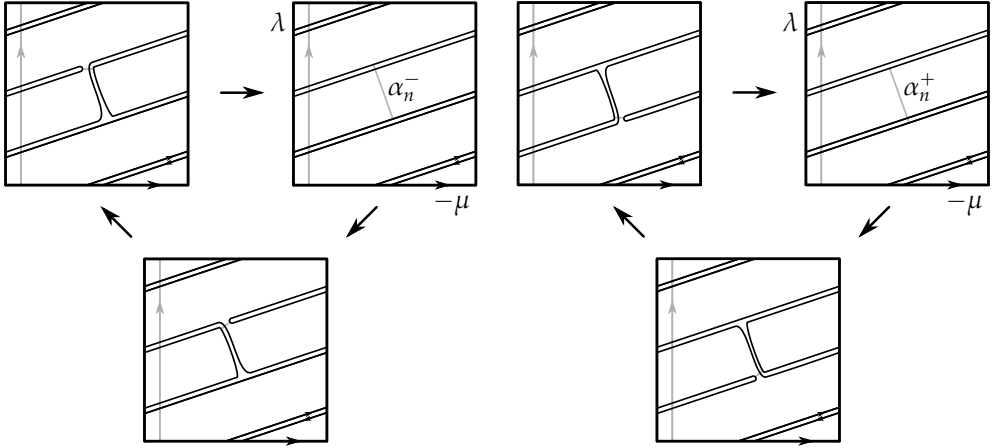
The attaching arcs $\alpha_n^-,\alpha_n^+$ are disjoint from $\partial \Sigma = \lambda$ by definition so by Proposition~\ref{prop:bypassGradings} when $n$ is odd, the bypass attachment maps in their associated exact triangles \[
	\begin{tikzcd}[column sep=-1.5em]
		\SHI(-X,-\gamma_{n-1}) \ar[rr,"\psi_{n-1}^-"] & & \SHI(-X,-\gamma_{n}) \ar[dl]\\
		& \SHI(-Y(K)) \ar[ul] &
	\end{tikzcd}\qquad \begin{tikzcd}[column sep=-1.5em]
		\SHI(-X,-\gamma_{n-1}) \ar[rr,"\psi_{n-1}^+"] & & \SHI(-X,-\gamma_{n}) \ar[dl]\\
		& \SHI(-Y(K)) \ar[ul] &
	\end{tikzcd}
\]respect the $\Z$-gradings defined by $\Sigma$ where $\Sigma$ intersects each collection of sutures on $\partial X$ in the ways shown in Figure~\ref{fig:triangleGradings}. Observe that in the two triangles, $\Sigma$ intersects $\gamma_{n-1}$ admissibly in two different ways. The two $\Z$-gradings on $\SHI(-X,-\gamma_{n-1})$ induced by $\Sigma$ intersecting $s(-\gamma_{n-1})$ in these two ways turn out to be different. The following definition follows the terminology of Definition 3.1 of \cite{1901.06679}. 

\begin{df*}[stabilization]
	Let $S$ be a compact oriented surface properly embedded in a balanced sutured manifold $(M,\gamma)$. Suppose $D$ is a disc in $\partial M$ for which \begin{enumerate}[itemsep=-0.5ex]
		\item $D \cap s(\gamma)$ is a single arc,
		\item $D \cap \partial S$ is a single arc,
		\item the two arcs $D \cap s(\gamma)$ and $D \cap \partial S$ are disjoint. 
	\end{enumerate}Consider a finger move on $\partial S$ supported within $D$ which introduces two new intersection points between $\partial S$ and $s(\gamma)$ (Figure~\ref{fig:stabilization}). A \textit{stabilization} of $S$ is an isotopy of $S$ through properly embedded surfaces which realizes this finger move on $\partial S$. We refer to $D$ as a disc which \textit{supports} the stabilization.
	After the stabilization, there are arcs along $\partial S$ and $s(\gamma)$ whose endpoints are the two new intersection points; these two arcs form the boundary of a disc in $D$ called the \textit{Whitney disc} of the stabilization. If the orientations of the two arcs inherited from $\partial S$ and $s(\gamma)$ match a boundary orientation of the Whitney disc, then the stabilization is called \textit{negative}. Otherwise, the stabilization is \textit{positive}.
\end{df*}

\begin{figure}[!ht]
	\centering
	\labellist
	\pinlabel $s(\gamma)$ at 330 0 
	\pinlabel $\partial S$ at 135 2
	\pinlabel $D$ at 50 445
	\pinlabel $s(\gamma)$ at 925 0 
	\pinlabel $\partial S$ at 730 2
	\pinlabel $D$ at 635 445
	\endlabellist
	\includegraphics[width=.2\textwidth]{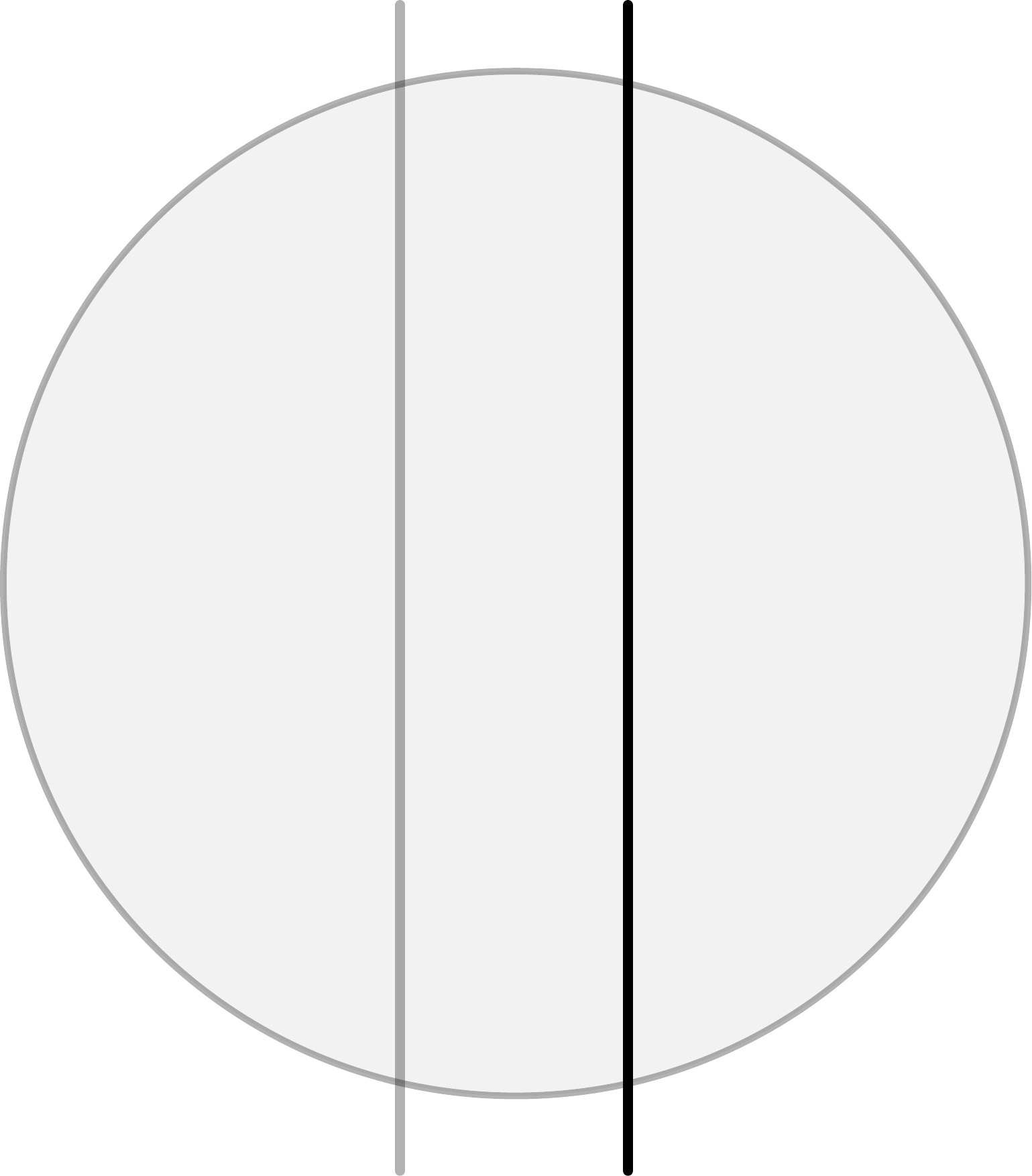}
	\hspace{20pt}
	\includegraphics[width=.2\textwidth]{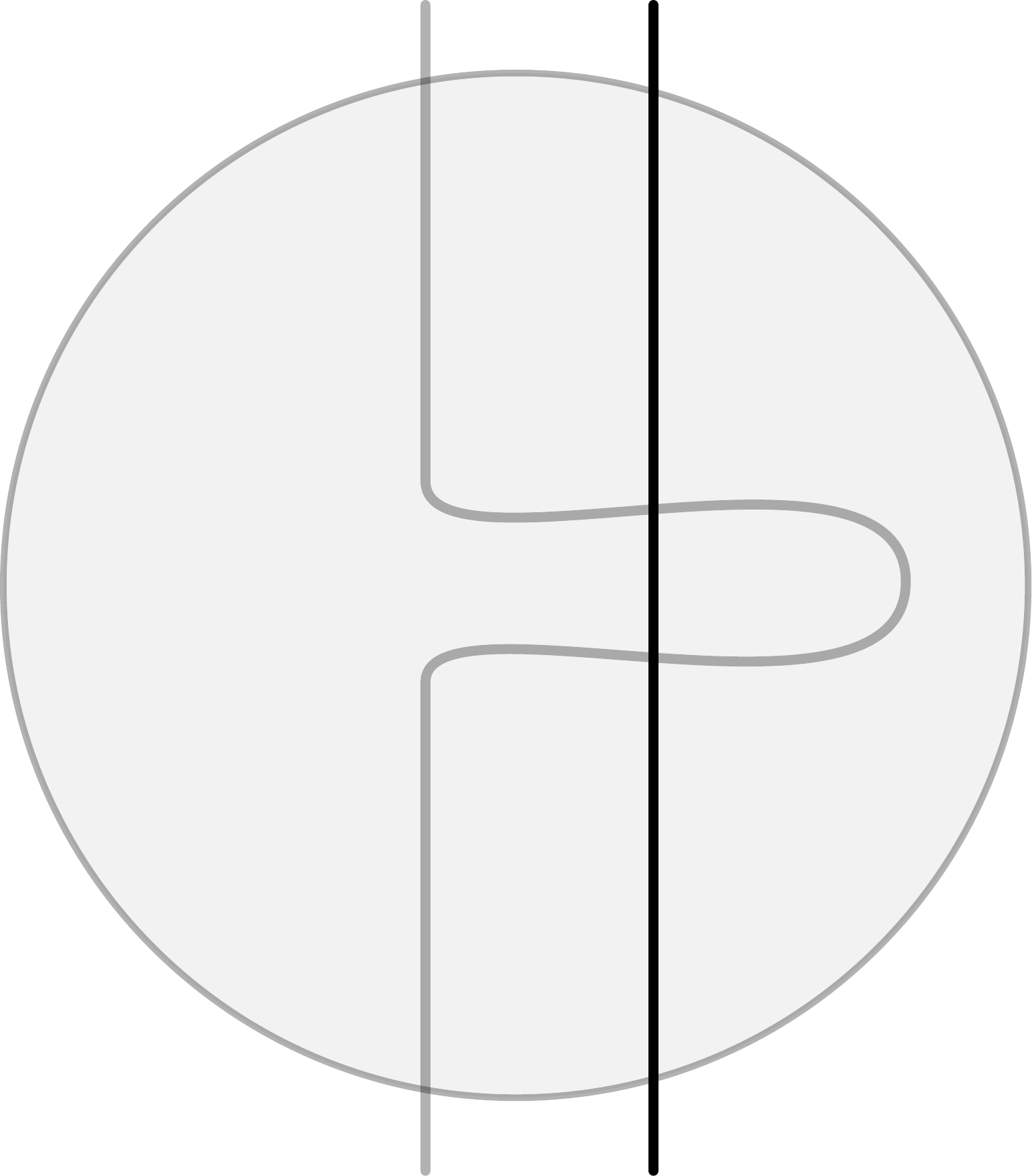}
	\caption{A stabilization.}
	\label{fig:stabilization}
\end{figure}
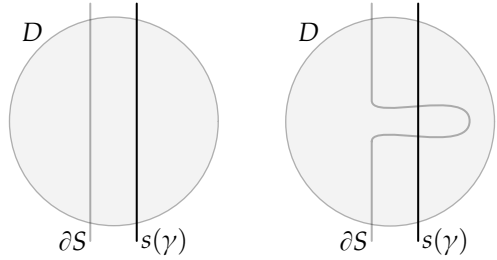

\begin{rem}
	Stabilizations are used in Lemma 4.5 of \cite{Juh08}. The main observation is that if $T$ is obtained by a negative stabilization of $S$, then decomposing along either $S$ or $T$ results in the same sutured manifold. Juh\'asz essentially shows that any nice surface $S$ can be isotoped through negative stabilizations to a surface $T$ for which each component of $\partial T$ intersects $s(\gamma)$. 
\end{rem}

In section~\ref{subsubsec:gradingSuturedInstanton}, we prove the following proposition which answers Conjecture 4.2 of \cite{1910.10842} affirmatively for instanton Floer homology. A special case that assumes that $S$ and $-S$ define taut surface decompositions is Proposition 4.1 of \cite{1910.10842}. 

\begin{prop}\label{prop:gradingShift}
	Let $S$ be an admissible surface in a balanced sutured manifold $(M,\gamma)$. Suppose $T$ is obtained from $S$ by $p$ positive stabilizations and $q$ negative stabilizations with $p - q = 2k$. Then for each $i \in \Z$ \[
		\SHI(M,\gamma,S,i) = \SHI(M,\gamma,T,i + k).
	\]
\end{prop}
\begin{cor}\label{cor:plusminusStab}
	Let $T$ and $T'$ be admissible surface each obtained by a single stabilization of a surface $S$ in a balanced sutured manifold $(M,\gamma)$. If the two stabilizations are the same sign, then $T$ and $T'$ induce the same $\Z$-grading. If $T$ is obtained by a negative stabilization, and $T'$ by a positive stabilization, then \[
		\SHI(M,\gamma,T,i) = \SHI(M,\gamma,T',i+1).
	\]
\end{cor}
\begin{proof}
	Let $D$ and $D'$ be discs in $\partial M$ supporting the two stabilizations. We may assume that $D$ and $D'$ are disjoint. Let $D''$ be another disc in $\partial M$ disjoint from $D$ and $D'$ which supports some other stabilization of $S$. Let $V$ be the surface obtained from $S$ by three stabilizations, one within each of $D,D',D''$. Proposition~\ref{prop:gradingShift} determines the grading shift between $T$ and $V$ and between $V$ and $T'$. 
\end{proof}

\noindent Proposition~\ref{prop:sameRelGrading} itself implies that each $\psi_n^\pm$ is homogeneous with respect to the relative $\Z$-gradings defined by the relative homology class of $\Sigma$, so the colimit $H^-$ can be given a $\Z$-grading. A version of Proposition~\ref{prop:gradingShift} is needed to show that $U$ is homogeneous of degree $-1$. The main point is that the two gradings on $\SHI(-X,-\gamma_{n-1})$ induced by $\Sigma$ intersecting $s(-\gamma_{n-1})$ in the two different ways are related by stabilizations, so Proposition~\ref{prop:gradingShift} allows Li to compute that for any homogeneous element $x \in \SHI(-X,-\gamma_{n-1})$, the homogeneous element $\psi_{n-1}^-(x) \in \SHI(-X,-\gamma_n)$ lies in $1$ relative grading greater than $\psi_{n-1}^+(x)$. We will repeat this argument in the proof of Proposition~\ref{prop:bypasstrianglesGradings}. 

\subsubsection{Construction and basic properties}\label{subsubsec:constructionOfTheInvariant}

The construction in this section is an adaptation of Li's construction in the special case of the sutured exterior of a knot. Let $(M,\gamma,\sigma)$ be a balanced sutured manifold with a suitable distinguished suture, and let $T \subset \partial M$ be the toral boundary component containing $\sigma$. Let $\lambda$ be an oriented simple closed curve on $T$ that intersects $\sigma$ once so that $\{\lambda,\sigma\}$ is a positive basis for $T$. For each $n \ge 1$, let $\lambda_n$ be a simple closed curve on $T$ in the homology class $\lambda - n\sigma$, and let $\gamma_n$ be the set of sutures obtained from $\gamma$ by replacing the sutures on $T$ with oppositely oriented copies of $\lambda_n$. We refer to the two sutures as $\lambda_n$ and $-\lambda_n$. 

Arrange $s(\gamma_n)$ so that it intersects $\lambda$ in exactly $2n$ points. Just as in section~\ref{subsubsec:contactStructures}, let $\alpha_{n+1}^-$ (resp. $\alpha_{n+1}^+$) be a bypass attaching arc on $T$ which intersects $\lambda_{n+1}$ (resp $-\lambda_{n+1}$) in precisely its endpoints and which is disjoint from both $\mu$ and $\lambda$. These arcs are unique up to isotopy through bypass arcs. The bypass exact triangle (Theorem~\ref{thm:instantonByPassTri}) for each of $\alpha_{n+1}^-,\alpha_{n+1}^+$ is of the form \[
	\begin{tikzcd}[column sep=0]
		\SHI(-M,-\gamma_{n}) \ar[rr] & & \SHI(-M,-\gamma_{n+1})\ar[dl]\\
		& \SHI(-M,-\gamma) \ar[ul] &
	\end{tikzcd}
\]where the map $\SHI(-M,-\gamma_{n+1}) \to \SHI(-M,-\gamma)$ is the one induced by the arc itself. Let $\psi_n^-$ (resp. $\psi_n^+$) be the bypass attaching map $\SHI(-M,-\gamma_n) \to \SHI(-M,-\gamma_{n+1})$ in the triangle associated to $\alpha_{n+1}^-$ (resp. $\alpha_{n+1}^+$). The bypass arcs which induce $\psi_n^-$ and $\psi_n^+$ can be made disjoint, so by Corollary~\ref{cor:instantonDisjointAttachmentArcs} we have the following commutative diagram. \[
	\begin{tikzcd}[column sep=large,row sep=large]
		\cdots \ar[r,"\psi_{n-1}^-"] & \SHI(-M,-\gamma_n) \ar[r,"\psi_{n}^-"] \ar[d,"\psi_n^+"] & \SHI(-M,-\gamma_{n+1}) \ar[r,"\psi_{n+1}^-"] \ar[d,"\psi_{n+1}^+"] & \cdots\\
		\cdots \ar[r,"\psi_{n}^-"] & \SHI(-M,-\gamma_{n+1}) \ar[r,"\psi_{n+1}^-"] & \SHI(-M,-\gamma_{n+2}) \ar[r,"\psi_{n+2}^-"] & \cdots
	\end{tikzcd}
\]

\begin{df*}[a minus version of sutured instanton homology for balanced sutured manifolds with a suitable distinguished suture]
	Define $\SHI^-(-M,-\gamma,-\sigma)$ be the colimit of \[
		\begin{tikzcd}
			\SHI(-M,-\gamma_1) \ar[r,"\psi_1^-"] & \cdots \ar[r,"\psi_{n-1}^-"] & \SHI(-M,-\gamma_n) \ar[r,"\psi_n^-"] & \cdots
		\end{tikzcd}
	\]in the category of complex vector spaces. Let $U\colon \SHI^-(-M,-\gamma,-\sigma) \to \SHI^-(-M,-\gamma,-\sigma)$ be the map induced by the sequence of maps $\psi_n^+\colon \SHI(-M,-\gamma_n) \to \SHI(-M,-\gamma_{n+1})$, and view $\SHI^-(-M,-\gamma,-\sigma)$ as a module over $\C[U]$. 
\end{df*}

\begin{rem}
	That the module $\SHI^-(-M,-\gamma,-\sigma)$ is independent of the choice of longitude $\lambda$ follows from the characterization of the contact structures $\xi_n^\pm$ just as in \cite{MR3650078}. This can also be seen directly by checking that the arcs $\alpha_n^\pm(\lambda)$ on $(M,\gamma_n(\lambda))$ defined in terms of $\lambda$ agree with the arcs $\alpha_{n+1}^\pm(\lambda')$ on $(M,\gamma_{n+1}(\lambda'))$ defined by a longitude $\lambda'$ of slope $\lambda + \sigma$. 
\end{rem}

\vspace{20pt}

We now turn to gradings on $\SHI^-(-M,-\gamma,-\sigma)$. First, we show that there is a splitting of $\SHI^-(-M,-\gamma,-\sigma)$ along an affine space over $\Hom(H_2(-M,\partial (-M)),\Z)$ analogous to the splitting in Proposition~\ref{prop:instantonSplittingAlongAffineSpace}. We then show that $U$ is homogeneous with respect to this splitting. In particular, if $x \in \SHI^-(-M,-\gamma,-\sigma)$ is homogeneous, then the difference in gradings between $x$ and $Ux$ is the image of $[\sigma] \in H_1(-M)$ under the composite \[
	H_1(-M) \to H^2(-M,\partial (-M)) \to \Hom(H_2(-M,\partial (-M)),\Z)
\]where the first map is Poincar\'e duality and the second is the natural map in the universal coefficients theorem. 

\begin{prop}\label{prop:splittingofSuturedInstantonMinus}
	There is a well-defined splitting \[
		\SHI^-(-M,-\gamma,-\sigma) = \bigoplus_h \SHI^-(-M,-\gamma,-\sigma,h)
	\]for $h$ in an affine space over $\Hom(H_2(-M,\partial (-M)),\Z)$. 
\end{prop}
\begin{proof}
	The sutured instanton homology of $(-M,-\gamma_n)$ has a splitting over an affine space over $H_2(-M,\partial (-M))$ by Proposition~\ref{prop:instantonSplittingAlongAffineSpace}. It suffices to show that \[
		\psi_n^-\colon \SHI(-M,-\gamma_n) \to \SHI(-M,-\gamma_{n+1})
	\]is homogeneous with respect to this splitting. Given any $a \in H_2(-M,\partial (-M))$, we can find an admissible surface $A \subset (-M,-\gamma_n)$ representing $a$ which is disjoint from the bypass attaching arc defining $\psi_n^-$. By Proposition~\ref{prop:bypassGradings}, it follows that $\psi_n^-$ respects the relative $\Z$-grading defined by $a$. Since this is true for each $a \in H_2(-M,\partial (-M))$, the result follows. 
\end{proof}

\begin{rem}
	Given an admissible surface $S \subset (-M,-\gamma)$, there is a relative $\Z$-grading on $\SHI^-(-M,-\gamma,-\sigma)$ denoted \[
		\SHI^-(-M,-\gamma,-\sigma) = \bigoplus_{i} \SHI^-(-M,-\gamma,-\sigma,S,i)
	\]which is compatible with the splitting in Proposition~\ref{prop:splittingofSuturedInstantonMinus}. Each $\SHI^-(-M,-\gamma,-\sigma,h)$ is a direct summand of one of the $\SHI^-(-M,-\gamma,-\sigma,S,i)$, and another term $\SHI^-(-M,-\gamma,-\sigma,h+\phi)$ also lies in the same $\SHI^-(-M,-\gamma,-\sigma,S,i)$ if and only if $\phi([S]) = 0$. 
\end{rem}

\begin{df*}[generalized Seifert surface]
	Let $(M,\gamma,\sigma)$ be a balanced sutured manifold with a suitable distinguished suture. A \textit{generalized Seifert surface} for $(M,\gamma,\sigma)$ is an admissible surface $\Sigma$ in $(M,\gamma)$ for which $\partial \Sigma \cap T$ is a simple closed curve $\lambda$ that intersects $\sigma$ in a single point transversely with $\lambda,\sigma$ a positive basis for $T$. The existence of a generalized Seifert surface for $(M,\gamma,\sigma)$ is guaranteed by the definition of a suitable distinguished suture. Note that a generalized Seifert surface is not required to be connected, and may have multiple boundary components. 
\end{df*}

Let $\Sigma$ be a generalized Seifert surface for $(M,\gamma,\sigma)$, and let $\lambda = \partial \Sigma \cap T$. We may use $\lambda$ to define the sutures $\gamma_n$. 
Since the way $\partial \Sigma$ intersects the sutures is relevant for the gradings, we let $\Sigma_n$ be the surface $\Sigma$ in $(M,\gamma_n)$ arranged to intersect $s(\gamma_n) \cap T$ in exactly $2n$ points. If $n$ is odd, then $\Sigma_n$ is admissible. Otherwise, any stabilization of $\Sigma_n$ is admissible. By Corollary~\ref{cor:plusminusStab}, any two negative stabilizations define the same grading, as do any two positive gradings. We let $\Sigma_n^-$ (resp. $\Sigma_n^+$) denote a surface obtained from $\Sigma$ by a negative (resp. positive) stabilization with this understanding. Similarly, for any $k \in \Z$, we let $\Sigma_n^{k}$ denote a surface obtained from $\Sigma$ by $m$ positive stabilizations and $m-k$ negative stabilizations. The grading that $\Sigma_n^k$ defines is independent of $m$ and the way the stabilizations are performed by Proposition~\ref{prop:gradingShift}. 

In the following proposition, we use $\Sigma$ to define a $\Z$-grading on $\SHI(-M,-\gamma)$, and we use $\Sigma_n$ and $\Sigma_n^-$ to define a grading on $\SHI(-M,-\gamma_n)$ when $n$ is odd and even, respectively. 

\begin{prop}[Proposition 5.5 of \cite{1901.06679}]\label{prop:bypasstrianglesGradings}
	The bypass exact triangles associated to $\alpha_{2k}^-,\alpha_{2k+1}^-$ are homogeneous of following degrees with respect to the gradings defined by $\Sigma,\Sigma_n$, and $\Sigma_n^-$. \[
		\begin{tikzcd}[column sep=-3em]
			\SHI(-M,-\gamma_{2k-1},i-1) \ar[rr,"\psi_{2k-1}^-"] & & \SHI(-M,-\gamma_{2k},i) \ar[rr,"\psi_{2k}^-"] \ar[dl] & & \SHI(-M,-\gamma_{2k+1},i) \ar[dl]\\
			& \SHI(-M,-\gamma,i + k - 1) \ar[ul] & & \SHI(-M,-\gamma,i+k) \ar[ul] &
		\end{tikzcd}
	\]Similarly, the bypass exact triangles associated to $\alpha_{2k}^+,\alpha_{2k+1}^+$ are homogeneous of the following degrees. \[
		\begin{tikzcd}[column sep=-3em]
			\SHI(-M,-\gamma_{2k-1},i+1) \ar[rr,"\psi_{2k-1}^+"] & & \SHI(-M,-\gamma_{2k},i+1) \ar[rr,"\psi_{2k}^+"] \ar[dl] & & \SHI(-M,-\gamma_{2k+1},i) \ar[dl]\\
			& \SHI(-M,-\gamma,i-k+1) \ar[ul] & & \SHI(-M,-\gamma,i-k) \ar[ul]
		\end{tikzcd}
	\]
\end{prop}
\begin{proof}
	Using Figure~\ref{fig:triangleGradings} and Proposition~\ref{prop:bypassGradings}, we may compute the grading shifts of the maps in the triangle with respect to stabilizations of $\Sigma_n$ for $\alpha_{2k+1}^\pm$. Together with similar diagrams for $\alpha_{2k}^\pm$, we have \[
		\begin{tikzcd}[column sep=-3.5em]
			\SHI(-M,-\gamma_{2k-1},\Sigma_{2k-1}^{-2},i) \ar[rr,"\psi_{2k-1}^-"] & & \SHI(-M,-\gamma_{2k},\Sigma_{2k}^-,i) \ar[rr,"\psi_{2k}^-"] \ar[dl] & & \SHI(-M,-\gamma_{2k+1},\Sigma_{2k+1},i) \ar[dl]\\
			& \SHI(-M,-\gamma,\Sigma^{2k-2},i) \ar[ul] & & \SHI(-M,-\gamma,\Sigma^{2k},i) \ar[ul] &
		\end{tikzcd}
	\]and \[
		\begin{tikzcd}[column sep=-3.8em]
			\SHI(-M,-\gamma_{2k-1},\Sigma_{2k-1}^{+2},i) \ar[rr,"\psi_{2k-1}^+"] & & \SHI(-M,-\gamma_{2k},\Sigma_{2k}^+,i) \ar[rr,"\psi_{2k}^+"] \ar[dl] & & \SHI(-M,-\gamma_{2k+1},\Sigma_{2k+1},i) \ar[dl]\\
			& \SHI(-M,-\gamma,\Sigma^{-2k+2},i) \ar[ul] & & \SHI(-M,-\gamma,\Sigma^{-2k},i). \ar[ul]
		\end{tikzcd}
	\]The claim now follows from Proposition~\ref{prop:gradingShift} keeping in mind that if $T$ is obtained from $S$ by a positive stabilization in $(M,\gamma)$, then $T$ is obtained from $S$ by a negative stabilization in $(-M,-\gamma)$. 
\end{proof}

\begin{cor}\label{cor:UgradingDropForG}
	Let $\Sigma$ be a generalized Seifert surface for $(M,\gamma)$. Then $U\colon \SHI^-(-M,-\gamma,-\sigma) \to \SHI^-(-M,-\gamma,-\sigma)$ is homogeneous of degree $-1$ with respect to the relative $\Z$-grading defined by $[\Sigma] \in H_2(-M,\partial (-M))$. 
\end{cor}
\begin{proof}
	From Proposition~\ref{prop:bypasstrianglesGradings}, if $x \in \SHI(-M,-\gamma_n)$ is homogeneous with respect to $[\Sigma]$, then $\psi_n^-(x)$ lies in one relative grading higher than $\psi_n^+(x)$. The relative $\Z$-grading on $\SHI^-(-M,-\gamma,-\sigma)$ induced by $[\Sigma] \in H_2(-M,\partial (-M))$ is defined using the $\psi_n^-$ maps so $U$ drops the relative grading by $1$. 
\end{proof}

\begin{prop}\label{prop:instantonUisHomogeneous}
	With respect to the splitting in Proposition~\ref{prop:splittingofSuturedInstantonMinus}, $U$ sends $\SHI^-(-M,-\gamma,-\sigma,h)$ to $\SHI^-(-M,-\gamma,-\sigma,h - [-\sigma])$ where $[-\sigma] \in \Hom(H_2(-M,\partial (-M)),\Z)$ is the image of the distinguished suture $-\sigma$ under $H_1(-M) \to H^2(-M,\partial (-M))\to \Hom(H_2(-M,\partial (-M)),\Z)$. 
\end{prop}
\begin{proof}
	The map $H_2(-M,\partial (-M)) \to \Z$ given by $[-\sigma]$ is the algebraic intersection number $[A] \mapsto - A \cdot \sigma$ where $A$ is a properly embedded compact oriented surface representing $[A] \in H_2(-M,\partial (-M))$. Fix a homogeneous element $x \in \SHI^-(-M,-\gamma,-\sigma,h)$. We will show that for each relative homology class $[A] \in H_2(-M,\partial (-M))$, that $Ux$ is homogeneous with respect to the relative grading defined by $[A]$ and the grading difference from $x$ to $Ux$ is given by $A \cdot \sigma$. 

	Let $\Sigma$ be a generalized Seifert surface for $(M,\gamma,\sigma)$. Then $\Sigma \cdot \sigma = 1$ in $(M,\gamma)$ so $\Sigma \cdot \sigma = -1$ in $(-M,-\gamma)$. By Corollary~\ref{cor:UgradingDropForG}, we know that $Ux$ is homogeneous with respect to the grading defined by $[\Sigma] \in H_2(-M,\partial (-M))$, and the grading difference from $x$ to $Ux$ is $-1 = \Sigma \cdot \sigma$ as required. 

	Suppose $[A] \in H_2(-M,\partial(-M))$ is a class for which $[\partial A \cap T] \in H_1(T)$ is a multiple of $\sigma$. Then we may find an admissible representative $A$ of $[A]$ for which $\partial A \cap T$ is just a number of parallel copies of $\sigma$. Since $\partial A$ can be made disjoint from the attaching arc which defines $\psi_n^+$, it follows that $U$ preserves the $\Z$-grading defined by $[A]$ so again $Ux$ is homogeneous with respect to the grading defined by $[A]$ and there is no grading difference between $Ux$ and $x$ as required because $A \cdot \sigma = 0$. 

	Now let $[B] \in H_2(-M,\partial (-M))$ be an arbitrary class. We may write $[B]$ as $k[\Sigma] + [A]$ for some integer $k$ and for some class $[A] \in H_2(-M,\partial (-M))$ for which $[\partial A \cap T] \in H_1(T)$ is a multiple of $\sigma$. By the proof of Proposition~\ref{prop:instantonSplittingAlongAffineSpace}, if $Ux$ is homogeneous with respect to the relative gradings defined by a basis for $H_2(-M,\partial (-M))$, then it is homogeneous with respect to the relative grading defined by any class in $H_2(-M,\partial (-M))$. Thus $Ux$ is homogeneous with respect to the relative grading defined by $[B]$. Since the function $H_2(-M,\partial(-M)) \to \Z$ defined by the grading difference between $x$ and $Ux$ is a group homomorphism, and since it agrees with $[\sigma]$ on a basis, it is equal to $[\sigma]$ as required. 
\end{proof}

\vspace{20pt}

Next, we prove a structural result for $\SHI^-(-M,-\gamma,-\sigma)$ as a $\C[U]$-module, and we relate $\SHI^-(-M,-\gamma,-\sigma)$ to $\SHI(-M,-\gamma)$.

\begin{prop}\label{prop:instantonMinusModuleFinitelyGenerated}
	The $\C[U]$-module $\SHI^-(-M,-\gamma,-\sigma)$ is finitely-generated. 
\end{prop}
\begin{proof}
	The following argument is based on the proof of Proposition 5.10 of \cite{1901.06679}. Let $\Sigma$ be a generalized Seifert surface for $(M,\gamma)$. Just like in Proposition~\ref{prop:bypasstrianglesGradings}, we use $\Sigma,\Sigma_n$, and $\Sigma_n^-$ to define absolute $\Z$-gradings on $\SHI(-M,-\gamma)$ and $\SHI(-M,-\gamma_n)$. Because $\SHI(-M,-\gamma_1)$ and $\SHI(-M,-\gamma)$ are finite-dimensional, there is a constant $\ell \in \Z$ for which $\SHI(-M,-\gamma_1,i) = 0 = \SHI(-M,-\gamma,i)$ for all $i \ge \ell$. By Proposition~\ref{prop:bypasstrianglesGradings}, we have the following diagram of exact triangles. \[
		\begin{tikzcd}[column sep=-3em]
			\SHI(-M,-\gamma_1,i) \ar[rr,"\psi_1^-"] & & \SHI(-M,-\gamma_2,i+1) \ar[rr,"\psi_2^-"] \ar[dl] & & \SHI(-M,-\gamma_3,i+1) \ar[dl] \ar[rr,"\psi_3^-"] & & \cdots \\
			& \SHI(-M,-\gamma,i+1) \ar[ul] & & \SHI(-M,-\gamma,i+2) \ar[ul] & & \phantom{\SHI(a)}\cdots\phantom{\SHI(M)} \ar[ul] &
		\end{tikzcd}
	\]Since the bottom row of vector spaces $\SHI(-M,-\gamma, i + n)$ are all zero for $i \ge \ell$, exactness of the triangles implies all terms in the top row are also zero for $i \ge \ell$. It follows that $\SHI^-(-M,-\gamma,-\sigma,\Sigma,i) = 0$ for $i \ge \ell$ as well. Although in general $[\Sigma]$ only defines a relative $\Z$-grading, here we have an absolute lift by declaring the $i$-graded portion to be the colimit $\psi_n^-$ maps starting at $\SHI(-M,-\gamma_1,i)$.

	Now observe that the same triangles show that for any fixed grading $i \in \Z$, if $k$ is large enough then \[
		\begin{tikzcd}[column sep=large]
			\SHI(-M,-\gamma_{2k-1},i+k-1) \ar[r,"\psi_{2k-1}^-"] & \SHI(-M,-\gamma_{2k},i+k) \ar[r,"\psi_{2k}^-"] & \cdots
		\end{tikzcd}
	\]are isomorphisms. In particular, if we fix $i$ and make $k$ large enough then we have \begin{align*}
		\SHI(-M,-\gamma_{2k},i+k) &\cong \SHI^-(-M,-\gamma,-\sigma,\Sigma,i)\\
		\SHI(-M,-\gamma_{2k+1},i+k-1) &\cong \SHI^-(-M,-\gamma,-\sigma,\Sigma,i-1)
	\end{align*}The exact triangle \[
		\begin{tikzcd}[column sep=-2em]
			\SHI(-M,-\gamma_{2k},i+k) \ar[rr,"\psi_{2k}^+"] & & \SHI(-M,-\gamma_{2k+1},i+k-1) \ar[dl]\\
			& \SHI(-M,-\gamma,i-1) \ar[ul] &
		\end{tikzcd}
	\]from Proposition~\ref{prop:bypasstrianglesGradings} shows that if $i$ is be sufficiently negative, say $i \leq m$, then \[
		U\colon \SHI^-(-M,-\gamma,-\sigma,\Sigma,i) \to \SHI(-M,-\gamma,-\sigma,\Sigma,i-1)
	\]is an isomorphism. Thus a finite complex basis for $\bigoplus_{i=m}^\ell \SHI^-(-M,-\gamma,-\sigma,\Sigma,i)$ generates $\SHI^-(-M,-\gamma,-\sigma)$ as a $\C[U]$-module. 
\end{proof}

\begin{cor}\label{cor:structMinusInstantonSutured}
	There is an isomorphism \[
		\SHI^-(-M,-\gamma,-\sigma) \cong \C[U]^k \oplus \bigoplus_{i=1}^n \frac{\C[U]}{U^{k_i}}
	\]where $k$ is the $\C[U]$-module rank of $\SHI^-(-M,-\gamma,-\sigma)$ and the integers $k_1,\ldots,k_n$ are positive. The generators of the summands may be chosen to be homogeneous with respect to the splitting over an affine space for $\Hom(H_2(-M,\partial (-M)),\Z)$ of Proposition~\ref{prop:splittingofSuturedInstantonMinus}. 
\end{cor}
\begin{proof}
	The result follows from Lemma~\ref{lem:HomologyStruct} and Propositions \ref{prop:instantonUisHomogeneous} and \ref{prop:instantonMinusModuleFinitelyGenerated}. 
\end{proof}

\begin{prop}[Proposition 5.14 of \cite{1901.06679}]\label{prop:instantonMinusUTriangle}
	The sequence of bypass exact triangles \[
		\begin{tikzcd}[column sep=-1em]
			\SHI(-M,-\gamma_{n}) \ar[rr,"\psi_n^+"] & & \SHI(-M,-\gamma_{n+1}) \ar[dl]\\
			& \SHI(-M,-\gamma) \ar[ul]
		\end{tikzcd}
	\]induces an exact triangle \[
		\begin{tikzcd}[column sep=-1em]
			\SHI^-(-M,-\gamma,-\sigma) \ar[rr,"U"] & & \SHI^-(-M,-\gamma,-\sigma) \ar[dl]\\
			& \SHI(-M,-\gamma). \ar[ul] &
		\end{tikzcd}
	\]The affine spaces over $\Hom(H_2(-M,\partial (-M)),\Z)$ for $\SHI^-(-M,-\gamma,-\sigma)$ and $\SHI(-M,-\gamma)$ can be identified, and with respect to this identification, map $\SHI^-(-M,-\gamma,-\sigma) \to \SHI(-M,-\gamma)$ is grading-preserving while the map $\SHI(-M,-\gamma) \to \SHI^-(-M,-\gamma,-\sigma)$ is homogeneous of degree $-[\sigma] \in \Hom(H_2(-M,\partial(-M)),\Z)$. 
\end{prop}

\begin{proof}
	The attaching arcs $\alpha_n^-$ and $\alpha_n^+$ can be made disjoint so any of the three attaching arcs in the bypass triangle for $\alpha_n^-$ can be made disjoint from any of the three attaching arcs in the triangle for $\alpha_n^+$. From Corollary~\ref{cor:instantonDisjointAttachmentArcs} and Figure~\ref{fig:Utriangle}, we have the following commutative diagram of bypass attachment maps. \[
		\begin{tikzcd}[column sep=small]
			& \SHI(-M,-\gamma_{n+1}) \ar[dd] \ar[rrr,"\psi_{n+1}^-"] & & & \SHI(-M,-\gamma_{n+2}) \ar[dd]\\
			\SHI(-M,-\gamma_n) \ar[ur,"\psi_n^+"] \ar[rrr,crossing over,"\quad\psi_n^-"] & & & \SHI(-M,-\gamma_{n+1}) \ar[ur,"\psi_{n+1}^+"] &\\
			& \SHI(-M,-\gamma) \ar[ul] \ar[rrr] & & & \SHI(-M,-\gamma) \ar[ul]
		\end{tikzcd}
	\]Now observe that the attaching arc inducing the map $\SHI(-M,-\gamma) \to \SHI(-M,-\gamma)$ is a trivial attaching arc so it induces the identity by Lemma~\ref{lem:trivialBypassArc}. The desired exact triangle on colimits now follows. The maps are homogeneous of the stated degrees by Proposition~\ref{prop:bypasstrianglesGradings} and the proof of Proposition~\ref{prop:splittingofSuturedInstantonMinus}. 
\end{proof}

\begin{figure}[!ht]
	\centering
	\labellist
	\pinlabel $\alpha_{n+1}^+$ at 740 630
	\endlabellist
	\includegraphics[width=.6\textwidth]{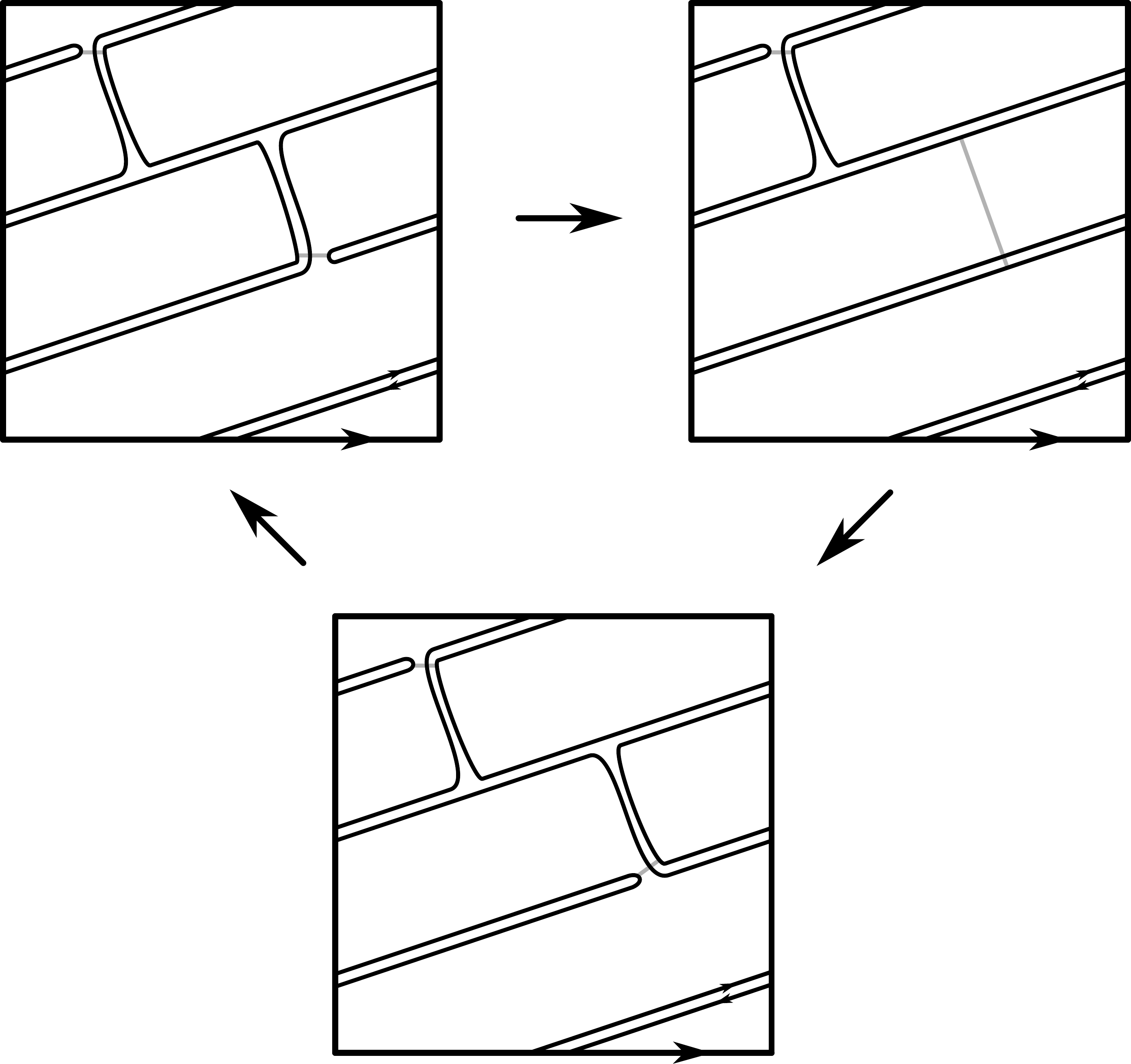}
	\captionsetup{width=.8\textwidth}
	\caption{The bypass triangle associated to $\alpha_{n+1}^+$. The short bypass arc in the upper left hand corner of the two top diagrams induce $\psi_n^-$ and $\psi_{n+1}^-$. In the bottom diagram, it is a trivial attaching arc.}
	\label{fig:Utriangle}
\end{figure}

\begin{rem}
	Proposition~\ref{prop:instantonMinusUTriangle} implies that $\SHI^-(-M,-\gamma,-\sigma)$ determines the dimension of each summand of $\SHI(-M,-\gamma)$ in the splitting of Proposition~\ref{prop:instantonSplittingAlongAffineSpace}. In particular, the total dimension of $\SHI(-M,-\gamma)$ is $k + 2n$ where $k$ is the $\C[U]$-module rank of $\SHI^-(-M,-\gamma,-\sigma)$ and $n$ is the number of $U$-torsion towers (Corollary~\ref{cor:structMinusInstantonSutured}).
\end{rem}

\vspace{20pt}

Consider attaching a smooth $2$-handle to $(M,\gamma)$ along the distinguished suture $\sigma$ to obtain a new balanced sutured manifold $(N,\beta)$. The toral boundary component of $M$ containing $\sigma$ becomes a $2$-sphere boundary component of $N$ with a single suture. We now show that the rank of $\SHI^-(M,\gamma,\sigma)$ is equal to the dimension of $\SHI(N,\beta)$, following section 3 of \cite{1910.01758}. 

Let $\sigma^*$ be a slight pushoff of $\sigma$ into the interior of $M$, and give it the framing induced by a tubular neighborhood $\sigma$ within $\partial M$. Surgery on $\sigma^*$, thought of as a knot in $(M,\gamma_n)$ with this framing, yields a balanced sutured manifold $(M^0,\gamma_n^0)$. Surgery on $\sigma^* \subset (M,\gamma_n)$ with framing $-1$ with respect to the given framing and meridian yields $(M,\gamma_{n-1})$. Given a closure of $(Y,R,\alpha)$ of $(M,\gamma)$, the knot $\sigma^* \subset \Int(M) \subset Y$ is disjoint from $R$. The two surgeries on $\sigma^*$ yield closures of $(M^0,\gamma_n^0)$ and $(M,\gamma_{n-1})$ and since $R$ is disjoint from $\sigma^*$, the associated surgery exact triangle preserves the gradings induced by $R$. In particular, we obtain an exact triangle \[
	\begin{tikzcd}[column sep=-2em,row sep=small]
		\SHI(-M,-\gamma_{n-1}) \ar[rr] & & \SHI(-M,-\gamma_n) \ar[dl]\\
		& \SHI(-M^0,-\gamma_n^0). \ar[ul] &
	\end{tikzcd}
\]Baldwin-Sivek prove in section 3.3 of \cite{MR3477339} that the maps in this triangle are natural. 

Observe that $\sigma \subset \partial M^0$ is the boundary of a product disc in $(M^0,\gamma_n^0)$. In fact, decomposing $(M^0,\gamma_n^0)$ along this product disc yields $(N,\beta)$. Decomposing along a product disc is the operation reverse to attaching a contact $1$-handle, so we may obtain $(M^0,\gamma_n^0)$ from $(N,\beta)$ by a contact $1$-handle attachment. Contact $1$-handle attachment maps are isomorphisms by construction, so we obtain an isomorphism \[
	\SHI(-N,-\beta) \to \SHI(-M^0,-\gamma_n^0)
\]which we use to identify these two vector spaces. Note that we can also obtain $(N,\beta)$ by attaching a contact $2$-handle to $(M,\gamma_n)$. The associated contact $2$-handle attachment map is by construction the composite of the map in the surgery triangle with the inverse to above $1$-handle attachment map. 

A bypass attachment arc on $(M,\gamma_n)$ naturally yields bypass attachment arcs on the two other sutured manifolds in the surgery triangle. The associated bypass attachment maps commute with the surgery exact triangle maps because the composite maps are ultimately defined by the same cobordisms. 
If $\alpha$ is an attachment arc on $(M^0,\gamma_n^0)$ which is disjoint from $\sigma \subset \partial M^0$, then $\alpha$ determines an attachment arc on $(N,\beta)$. The bypass attachment maps associated to these arcs commute with the $1$-handle attachment map by Theorem~\ref{thm:instantonGluingMaps}. Since the attaching arc on $(N,\beta)$ lies on a $2$-sphere boundary component of $N$ which contains a single suture, its associated bypass attachment map is either the identity or zero by Lemma~\ref{lem:trivialBypassArc}. Since the attaching arcs defining the bypass attachment maps $\psi_n^-,\psi_n^+$ can be made disjoint from $\sigma$, we obtain the following result. 

\begin{prop}\label{prop:NbetaCommutativeDiagram}
	We have the following commutative diagram \[
		\begin{tikzcd}[column sep=small]
			& \SHI(-M,-\gamma_n) \ar[dd] \ar[rrr,"\psi_n^-"] & & & \SHI(-M,-\gamma_{n+1}) \ar[dd] \\
			\SHI(-M,-\gamma_{n-1}) \ar[ur] \ar[rrr,crossing over,"\quad\psi_{n-1}^-"] & & & \SHI(-M,-\gamma_n) \ar[ur] &\\
			& \SHI(-N,-\beta) \ar[ul] \ar[rrr,"\Id"] & & & \SHI(-N,-\beta) \ar[ul]
		\end{tikzcd}
	\]where the two triangles are surgery exact triangles. The same diagram with $\psi_n^-$ and $\psi_{n-1}^-$ replaced by $\psi_n^+$ and $\psi_{n-1}^+$, respectively, is commutative. 
\end{prop}
\begin{proof}
	Let $\beta^-$ be the attaching arc on $(N,\beta)$ coming from the attaching arc on $(M,\gamma_n)$ defining $\psi_n^-$. The sutured manifold resulting from a bypass attachment along $\beta^-$ is still $(N,\beta)$ so by Lemma~\ref{lem:trivialBypassArc} the induced map is the identity. The same is true for the attaching arc defining $\psi_n^+$. The preceding discussion completes proof, which is just a rephrasing of the proof of Lemmas 3.4 and 3.6 in \cite{1910.01758}. 
\end{proof}

\begin{cor}
	There is an exact triangle \[
		\begin{tikzcd}[column sep=-1em]
			\SHI^-(-M,-\gamma,-\sigma) \ar[rr] & & \SHI^-(-M,-\gamma,-\sigma) \ar[dl]\\
			& \SHI(-N,-\beta) \ar[ul] &
		\end{tikzcd}
	\]of $\C[U]$-module maps, where $U$ acts as the identity on $\SHI(-N,-\beta)$. 
\end{cor}
\begin{proof}
	The exact triangle of $\C$-linear maps is obtained by taking the colimit along the $\psi_n^-$ maps of the surgery exact triangle using Proposition~\ref{prop:NbetaCommutativeDiagram}. That the maps commute with $U$ also follows from Proposition~\ref{prop:NbetaCommutativeDiagram}. The analogous result in \cite{1910.01758} is Corollary 3.6. 
\end{proof}

\begin{prop}\label{prop:rankOfMinusBiggerThanDim}
	There is an inequality \[
		\rank \SHI^-(-M,-\gamma,-\sigma) \ge \dim \SHI(-N,-\beta).
	\]
\end{prop}
\begin{proof}
	Because $\SHI^-(-M,-\gamma,-\sigma)$ can be given a $\Z$-grading with respect to which $U$ is homogeneous of degree $-1$, no nonzero element $x \in \SHI^-(-M,-\gamma,-\sigma)$ satisfies $Ux = x$. Since $\SHI(-N,-\beta) \to \SHI^-(-M,-\gamma,-\sigma)$ is $U$-equivariant and $U$ acts as the identity on $\SHI(-N,-\beta)$, the map must be zero. Thus the map $\SHI^-(-M,-\gamma,-\sigma) \to \SHI(-N,-\beta)$ is surjective and factors through the quotient $\SHI^-(-M,-\gamma,-\sigma)/(U - \Id)$. The dimension of $\SHI^-(-M,-\gamma,-\sigma)/(U - \Id)$ is precisely the $\C[U]$-module rank of $\SHI^-(-M,-\gamma,-\sigma)$ so the result follows. 
\end{proof}
\begin{rem}
	When $(M,\gamma,\sigma)$ is the sutured exterior of a knot in $S^3$, the manifold $(N,\beta)$ is just the $3$-ball with a single suture. Li proves the analogous result (Lemmas 3.2 and 3.4 of \cite{1910.01758}) using the fact that the contact invariant for $(N,\beta)$ generates $\SHI(N,\beta)$, which does not generalize to this situation. 
\end{rem}

Li points out that Proposition~\ref{prop:rankOfMinusBiggerThanDim} has the following corollary. The analogous result in Heegaard Floer homology follows from a spectral sequence from $\HFKhat(Y,K)$ to $\HFhat(Y)$ which is expected to exist in the instanton Floer context.

\theoremstyle{plain}
\newtheorem*{prop:dimIneq}{Proposition~\ref{prop:dimInequalitySuturedInstanton}}

\begin{prop:dimIneq}
	Let $K$ be a null-homologous knot in a closed oriented $3$-manifold $Y$. Then \[
		\dim \SHI(Y(K)) \ge \dim \SHI(Y(1))
	\]where $Y(K)$ is the sutured exterior of $K$ and $Y(1)$ is $Y$ with a sutured puncture. 
\end{prop:dimIneq}
\begin{proof}
	From Corollary~\ref{cor:structMinusInstantonSutured} and Proposition~\ref{prop:instantonMinusUTriangle}, we have the inequality \[
		\dim \SHI(Y(K)) \ge \rank \SHI^-(Y(K),\sigma)
	\]where $\sigma$ is either of the two sutures on $Y(K)$. Proposition~\ref{prop:rankOfMinusBiggerThanDim} gives \[
		\rank \SHI^-(Y(K),\sigma) \ge \dim \SHI(Y(1))
	\]from which the result follows.
\end{proof}

\begin{prop}\label{prop:rankOfMinusSmallerThanDim}
	There is an inequality \[
		\rank \SHI^-(-M,-\gamma,-\sigma) \leq \dim \SHI(-N,-\beta).
	\]
\end{prop}
\begin{proof}
	The proof is just a rephrasing of Li's proof of Proposition 3.3 of \cite{1910.01758}. Let $\Sigma$ be a generalized Seifert surface for $(M,\gamma)$. By Proposition~\ref{prop:bypasstrianglesGradings}, we have the following two graded exact triangles \[
		\begin{tikzcd}[column sep=-1em]
			\SHI(-M,-\gamma_{2k},i) \ar[rr,"\psi_{2k}^-"] & & \SHI(-M,-\gamma_{2k+1},i) \ar[dl]\\
			& \SHI(-M,-\gamma, i + k) \ar[ul] & 
		\end{tikzcd}
	\] \[
		\begin{tikzcd}[column sep=-1em]
			\SHI(-M,-\gamma_{2k},i+1) \ar[rr,"\psi_{2k}^+"] & & \SHI(-M,-\gamma_{2k+1},i) \ar[dl]\\
			& \SHI(-M,-\gamma,i-k) \ar[ul] &
		\end{tikzcd}	
	\]with gradings determined by $\Sigma,\Sigma_n$, and $\Sigma_n^-$. Suppose $N$ is large enough that $\SHI(-M,-\gamma,i)$ is zero if $|i| > N$. Then for $k = N$ in the above triangles, we have isomorphisms \begin{align*}
		&\psi_{2N}^-\colon \SHI(-M,-\gamma_{2N},i) \to \SHI(-M,-\gamma_{2N+1},i) && \text{when }i > 0\\
		&\psi_{2N}^+\colon \SHI(-M,-\gamma_{2N},i+1) \to \SHI(-M,-\gamma_{2N+1},i) && \text{when }i < 0.
	\end{align*}It follows that $\dim \SHI(-M,-\gamma_{2N+1}) = \dim \SHI(-M,-\gamma_{2N+1},0) + \dim \SHI(-M,-\gamma_{2N})$ as $\C$-vector spaces. The surgery exact triangle \[
		\begin{tikzcd}[column sep=0]
			\SHI(-M,-\gamma_{2N}) \ar[rr] & & \SHI(-M,-\gamma_{2N+1}) \ar[dl]\\
			& \SHI(-N,-\beta) \ar[ul] &
		\end{tikzcd}
	\]therefore gives the inequality \[
		\dim \SHI(-M,-\gamma_{2N+1},0) \leq \dim \SHI(-N,-\beta)
	\]so it suffices to show that $\dim \SHI(-M,-\gamma_{2N+1},0) = \rank \SHI^-(-M,-\gamma,-\sigma)$. But observe that the assumption that $\SHI(-M,-\gamma,i) = 0$ when $|i| > N$ implies that \[
		\begin{tikzcd}[column sep=large]
			\SHI(-M,-\gamma_{2N+1},0) \ar[r,"\psi_{2N+1}^-"] & \SHI(-M,-\gamma_{2N+2},1) \ar[r,"\psi_{2N+2}^-"] & \cdots
		\end{tikzcd}
	\]are all isomorphisms so $\SHI(-M,-\gamma_{2N+1},0) \cong \SHI^-(-M,-\gamma,-\sigma,-N)$. By Corollary~\ref{cor:structMinusInstantonSutured}, we may choose an isomorphism \[
		\SHI^-(-M,-\gamma,-\sigma) = \C[U]^k \oplus \bigoplus_{i=1}^n \frac{\C[U]}{U^{k_i}}
	\]for $k = \rank \SHI^-(-M,-\gamma,-\sigma)$ and positive $k_1,\ldots,k_n$ where the generators of the summands are homogeneous. By Proposition~\ref{prop:instantonMinusUTriangle}, the generator of each free summand must lie in grading at least $-N$, and all torsion is supported in gradings strictly larger than $-N$. Thus $\dim \SHI^-(-M,-\gamma,-\sigma,-N) = \rank \SHI^-(-M,-\gamma,-\sigma)$ as required. 
\end{proof}

Propositions \ref{prop:rankOfMinusBiggerThanDim} and \ref{prop:rankOfMinusSmallerThanDim} together give the following equality. 
\begin{cor}\label{cor:rankOfMinusInstantonEqualsDimOfHatNBeta}
	Let $(M,\gamma,\sigma)$ be a balanced sutured manifold with a suitable distinguished suture. Let $(N,\beta)$ be the balanced sutured manifold obtained by attaching a $2$-handle to $\sigma$. Then \[
		\rank \SHI^-(M,\gamma,\sigma) = \dim \SHI(N,\beta).
	\]
\end{cor}

\subsubsection{Surface decompositions and ribbon concordances}\label{subsubsec:featuresOfTheInvariant}

Let $(M,\gamma,\sigma)$ be a balanced sutured manifold with a suitable distinguished suture. If $S$ defines a nice surface decomposition \[
	(M,\gamma) \overset{S}{\rightsquigarrow} (M',\gamma')
\]where $S$ is disjoint from the distinguished toral boundary component $T$ of $M$, then $\sigma$ is a suitable distinguished suture for $(M',\gamma')$. We show that $\SHI^-(-M',-\gamma',-\sigma)$ is a direct $\C[U]$-module summand of $\SHI^-(-M,-\gamma,-\sigma)$ in a way compatible with the action of $[-\sigma]$ on relative gradings. There are more refined grading statements, but we only explicitly make note of those required for our application. 

\begin{lem}\label{lem:surfDecomp1handle}
	Let $(M,\gamma,\sigma)$ be a balanced sutured manifold with a suitable distinguished suture. If $(M',\gamma',\sigma)$ is obtained by attaching a contact $1$-handle away from the distinguished toral boundary component of $(M,\gamma,\sigma)$, then there is a $\C[U]$-module isomorphism \[
		\SHI^-(-M,-\gamma,-\sigma) \to \SHI^-(-M',-\gamma',-\sigma).
	\]If two homogeneous elements of $\SHI^-(-M,-\gamma,-\sigma)$ differ in grading by $[-\sigma]$, then their images in $\SHI^-(-M',-\gamma',-\sigma)$ are also homogeneous and differ in grading by $[-\sigma]$.
\end{lem}
\begin{proof}
	Fix a longitude $\lambda$ for $(M,\gamma,\sigma)$, and define $\gamma_n$ as before using $\lambda$. Let $\gamma_n'$ be obtained from $\gamma'$ by similarly swapping the two sutures on the distinguished toral boundary component for two oppositely oriented parallel curves of slope $\lambda - n\sigma$. Clearly $(M',\gamma_n')$ may also be obtained by attaching a contact $1$-handle to $(M,\gamma_n)$. The contact $1$-handle attachment map $\SHI^-(-M,-\gamma_n) \to \SHI^-(-M',-\gamma_n')$ is an isomorphism and commutes with the bypass attachment maps $\psi_n^\pm$. The limit of these maps is the desired $\C[U]$-module isomorphism. 

	Let $V$ be an admissible surface in $(M',\gamma')$. We may assume that $V$ is disjoint from the $1$-handle, so we may also think of $V$ as an admissible surface in $(M,\gamma)$. The contact $1$-handle attachment map preserves the grading that $V$ defines. Since $\sigma$ intersects $V$ in the same way in either $(M',\gamma')$ or $(M,\gamma)$, the grading claim follows. 
\end{proof}

\begin{lem}\label{lem:surfDecompProdAnnulus}
	Let $(M,\gamma,\sigma)$ be a balanced sutured manifold with a suitable distinguished suture, and let $A$ be a product annulus in $M$ that is disjoint from the toral boundary component $T$ of $M$ containing $\sigma$. Also suppose that $\partial A \cap R_+(\gamma)$ and $\partial A \cap R_-(\gamma)$ are homologically nontrivial simple closed curves in $R_+(\gamma)$ and $R_-(\gamma)$, respectively. If $(M',\gamma',\sigma)$ is obtained by decomposing along $A$, then there is an isomorphism \[
		\SHI^-(-M',-\gamma',-\sigma) \to \SHI^-(-M,-\gamma,-\sigma) 
	\]of $\C[U]$-modules. 
	If two homogeneous elements of $\SHI^-(-M',-\gamma',-\sigma)$ differ in grading by $[-\sigma]$, then their images are also homogeneous and differ by $[-\sigma]$. 
\end{lem}
\begin{proof}
	Let $B$ be an auxiliary surface for $(M',\gamma')$ with an identification $\partial B = s(\gamma')$. Let $s_1,s_2$ be the two new sutures in $s(\gamma')$ introduced in the decomposition along the product annulus $A$. These two sutures correspond to boundary components of $B$; let $c_1,c_2$ be copies of these two boundary components pushed into the interior of $B$. We construct a closure $Y$ using an identification $h\colon \bar{R}_+ \to \bar{R}_-$ which sends $c_1$ to itself and $c_2$ to itself. There are two tori $c_1 \x S^1$ and $c_2 \x S^1$ in $Y$ along which we may do excision. The result is a closure of $(M,\gamma)$. Floer excision gives an isomorphism $\phi\colon \SHI(M',\gamma') \to \SHI(M,\gamma)$ where these sutured Floer homology groups are defined by the described closures. 

	This isomorphism is natural; we point out the main observations needed to see this. Suppose $h'\colon \bar{R}_+\to\bar{R}_-$ is another diffeomorphism sending $c_i$ to itself. Following \cite{MR3352794}, the canonical isomorphism between the Floer homologies defined by these two closures can be viewed as a composite of cobordism maps defined by surgery along curves in $\bar{R}$ disjoint from $c_i$. The Floer excision map $\phi$ commutes with surgery along such a curve, so the isomorphism is natural under closures of the same genus. Similarly, the canonical isomorphisms defined to identify the Floer homologies of closures of different genera are defined by excision along tori disjoint from the tori defining $\phi$. 

	Fix homogeneous elements $x,y \in \SHI(M',\gamma')$ which differ in grading by $[\sigma]$. Let $V$ be an admissible surface in $(M,\gamma)$. We may assume that $V \cap A$ consists of homologically essential arcs in $A$. The intersection $W = V \cap M'$ is an admissible surface in $(M',\gamma')$. Choose a closure of $(M',\gamma')$ in which $W$ extends to a closed surface $\bar{W}$. Floer excision takes $\bar{W}$ to a closed extension of $V$ in the closure of $(M,\gamma)$, so the isomorphism $\phi$ respects the grading defined by $W$ and $V$. Because $V$ and $W$ have the same algebraic intersection with $\sigma$, the images of $x$ and $y$ in $\SHI(M,\gamma)$ differ by $[\sigma]$ in grading. 

	We have established the isomorphism with the specified grading behavior for sutured Floer homology. To obtain the same isomorphism on the minus version, we simply take the limit of the isomorphisms $\SHI(-M',-\gamma_n') \to \SHI(-M,-\gamma_n)$ defined in this way. These isomorphisms commute with the $\psi_n^\pm$ maps so the desired $\C[U]$-module isomorphism follows, along with the claim on grading.
\end{proof}

\begin{thm}\label{thm:instantonMinusSurfDecomp}
	Let $(M,\gamma,\sigma)$ be a balanced sutured manifold with a suitable distinguished suture, and let $S$ define a nice surface decomposition $(M,\gamma) \rightsquigarrow (M',\gamma')$ where $S$ is disjoint from the toral boundary component $T$ of $M$ containing $\sigma$. Then $\SHI^-(-M',-\gamma',-\sigma)$ is a direct $\C[U]$-module summand of $\SHI^-(-M,-\gamma,-\sigma)$. If two homogeneous elements in $\SHI^-(-M',-\gamma',-\sigma)$ differ in grading by $[-\sigma]$, then their images in $\SHI^-(-M,-\gamma,-\sigma)$ are homogeneous and also differ by $[-\sigma]$.
\end{thm}
\begin{proof}
	By Lemma 4.5 of \cite{Juh08}, there is an admissible surface $S_*$ isotopic to $S$ for which the result of decomposing $(M,\gamma)$ along $S_*$ is still $(M',\gamma')$. Following the proof of Proposition 6.9 of \cite{MR2652464}, choose a balanced pairing for $S_*$ and attach a contact $1$-handle to $(M,\gamma)$ along each pair to obtain a new balanced sutured manifold $(\tilde{M},\tilde{\gamma})$. Extend $S_*$ to a surface $\tilde{S}_*$ in $\tilde{M}$ by attaching a band within each contact $1$-handle. Let $(\tilde{M}',\tilde{\gamma}')$ be the result of decomposing $(\tilde{M},\tilde{\gamma})$ along $\tilde{S}_*$. Observe that $(\tilde{M}',\tilde{\gamma}')$ may obtained from $(M',\gamma')$ by attaching contact $1$-handles. In particular, we can use two contact $1$-handles for each $1$-handle used to construct $(\tilde{M},\tilde{\gamma})$ from $(M,\gamma)$. 

	Since $\partial \tilde{S}_* \cap s(\tilde{\gamma}) = \emptyset$, each suture of $(\tilde{M},\tilde{\gamma})$ is a suture of $(\tilde{M}',\tilde{\gamma}')$, and there is an extra suture of $(\tilde{M}',\tilde{\gamma}')$ for each component of $\partial\tilde{S}_* \cap (R_+(\tilde{\gamma}) \cup R_-(\tilde{\gamma}))$. Choose a bijection between the components of $\partial\tilde{S}_* \cap R_+(\tilde{\gamma})$ and the components of $\partial\tilde{S}_* \cap R_-(\tilde{\gamma})$. Still following the proof of Proposition 6.9 of \cite{MR2652464}, we glue a copy of $[-1,1] \x A$ where $A$ is an annulus to $(\tilde{M}',\tilde{\gamma}')$ along each extra pair of sutures, where the pairing is determined by the chosen bijection. Let $(\tilde{M}'',\tilde{\gamma}'')$ be the resulting sutured manifold. Observe that we may obtain $(\tilde{M}',\tilde{\gamma}')$ from $(\tilde{M}'',\tilde{\gamma}'')$ by decomposing along product annuli.

	Now choose an auxiliary surface $B$ for $(\tilde{M},\tilde{\gamma})$ and a homeomorphism $h\colon \bar{R}_+ \to \bar{R}_-$ which sends $\partial \tilde{S}_* \cap \bar{R}_+$ to $\partial \tilde{S}_* \cap \bar{R}_-$ matching the chosen bijection. Within the closure $Y$ determined by $B$ and $h$, we have two closed oriented surfaces. We have $\bar{R}$ as usual and $\bar{S}_*$, a closed extension of $\tilde{S}_*$. From this data, Kronheimer-Mrowka construct a closure of $(\tilde{M}'',\tilde{\gamma}'')$ whose ambient $3$-manifold is $Y$ but whose distinguished surface is the double-curve sum of $\bar{R}$ and $\bar{S}_*$. Name this surface $F$, and since $\mu(F) = \mu(\bar{R}) + \mu(\bar{S}_*)$ and $\chi(F) = \chi(\bar{R}) + \chi(\bar{S}_*)$, it follows that the $(2g(F) - 2)$-generalized eigenspace of $\mu(F)$ is just the intersection of the $(2g(\bar{R})-2)$-generalized eigenspace of $\mu(\bar{R})$ with the $(2g(\bar{S}_*) - 2)$-generalized eigenspace of $\mu(\bar{S}_*)$. Thus the sutured instanton homology of $(\tilde{M}'',\tilde{\gamma}'')$ includes as a direct summand of the sutured instanton homology of $(\tilde{M},\tilde{\gamma})$ defined by these closures. 

	There are a number of choices involved in establishing this direct summand relationship. We show that these choices can be made so that the inclusion respects the action of $[\sigma]$ on gradings. By Proposition~\ref{prop:simultaneousClosures}, we may choose a basis for $H_2(M,\partial M)$ along with representatives $T_1,\ldots,T_n$ such that after negative stabilizations of $S_*$, the surfaces $S_*,T_1,\ldots,T_n$ all admit closed extensions in a closure $Y$ of $(M,\gamma)$. From the proof of the Proposition, we may also assume that each $T_i$ intersects $S_*$ only in arcs. Let $h\colon \bar{R}_+ \to \bar{R}_-$ the diffeomorphism used in the construction of the closure $Y$ in which $S_*,T_1\ldots,T_n$ all admit closed extensions. The data determines a balanced pairing for $S_*$ and a bijection between the components of $\partial \tilde{S}_* \cap R_+(\tilde{\gamma})$ and $\partial \tilde{S}_* \cap R_-(\tilde{\gamma})$. The manifold $(\tilde{M},\tilde{\gamma})$ defined using this bijection is naturally embedded in the fixed preclosure of $(M,\gamma)$. The surfaces $T_1,\ldots,T_n$ and cocores of the contact $1$-handles are representatives for a basis for $H_2(\tilde{M},\partial \tilde{M})$. If the genus of the auxiliary surface is large enough, there is a homeomorphism $h'$ determining the same bijection between components of $\partial \tilde{S}_* \cap R_+(\tilde{\gamma})$ and $\partial \tilde{S}_* \cap R_-(\tilde{\gamma})$ for which all of these representatives have closed extensions in the associated closure. Since the representatives $T_i$ each intersect $\tilde{S}_*$ only along arcs, the surfaces $T_i \cap \tilde{M}'$ are admissible in $\tilde{M}'$. These surfaces can then be extended to admissible surfaces in $\tilde{M}''$ which close up to the given closures of $T_i$ in $Y$. It follows that if two elements in $\SHI(\tilde{M}'',\tilde{\gamma}'')$ are homogeneous and differ by $[\sigma]$ in grading, then their images in $\SHI(\tilde{M},\tilde{\gamma})$ are homogeneous and differ by $[\sigma]$ in grading as well. Combined with Lemmas \ref{lem:surfDecomp1handle} and \ref{lem:surfDecompProdAnnulus}, we have shown that there is an inclusion \[
		\SHI(M',\gamma') \hookrightarrow \SHI(M,\gamma)
	\]as a direct summand which respects the action of $[\sigma]$ on gradings. 

	To extend this result to the minus version, first observe that we have inclusions \[
		\SHI(-M',-\gamma_n') \hookrightarrow \SHI(-M,-\gamma_n)
	\]for each $n$. In particular, the data of $S_*$, a balanced pairing, and a bijection between the components of $\partial \tilde{S}_* \cap R_+(\tilde{\gamma})$ and $\partial \tilde{S}_* \cap R_-(\tilde{\gamma})$ can be chosen to be the same for each $n$. With these compatible choices, the inclusions commute with $\psi_n^\pm$ so they induce a inclusion \[
		\SHI^-(-M',-\gamma',-\sigma) \hookrightarrow \SHI^-(-M,-\gamma,-\sigma)
	\]of $\C[U]$-modules. To see that it is a direct summand, recall the splitting $\SHI^-(-M,-\gamma,-\sigma) = \bigoplus_h \SHI^-(-M,-\gamma,-\sigma,h)$ along an affine space. This splitting is a splitting as a $\C$-vector space. If we instead consider the splitting \[
		\SHI^-(-M,-\gamma,-\sigma) = \bigoplus_{h+\Z[-\sigma]} \SHI^-(-M,-\gamma,-\sigma,h+\Z[-\sigma])
	\]where $\SHI^-(-M,-\gamma,-\sigma,h+\Z[-\sigma]) = \bigoplus_{k \in \Z}\SHI^-(-M,-\gamma,-\sigma,h+k[-\sigma])$, we find that this is a splitting as a $\C[U]$-module by Proposition~\ref{prop:instantonUisHomogeneous}. The inclusion is onto a direct sum of a number of $\SHI^-(-M,-\gamma,-\sigma,h+\Z[-\sigma])$ so $\SHI^-(-M',-\gamma',-\sigma)$ is indeed a direct summand of $\SHI^-(-M,-\gamma,-\sigma)$ as a $\C[U]$-module.
\end{proof}

\vspace{20pt}

Consider a ribbon concordance $R\colon K_\# \cup C \to K_b \cup C$. By definition, $R$ consists of two disjoint concordances, one from $K_\#$ to $K_b$, the other between linking circles. Furthermore, they may be simultaneously arranged so that with respect to the projection $[0,1] \x S^3 \to [0,1]$ the annuli have critical points only of index $0$ and $1$. It follows that the exterior of $R$, viewed as a cobordism from $S^3(K_\# \cup C)$ to $S^3(K_b \cup C)$ can be constructed from $1$- and $2$-handles. 

\begin{thm}\label{thm:instantonMinusRibbon}
	Let $R\colon K_\#\cup C \to K_b \cup C$ be an oriented ribbon concordance. On each of the sutured exteriors $S^3(K_\# \cup C)$ and $S^3(K_b \cup C)$, let $\sigma$ be the suture on $\partial N(C)$ oriented as the meridian of $C$. Then there is an inclusion \[
		\SHI^-(S^3(K_\# \cup C),\sigma) \hookrightarrow \SHI^-(S^3(K_b \cup C),\sigma)
	\]as a direct $\C[U]$-module summand. Furthermore, the inclusion respects the action of $[\sigma]$ on gradings. 
\end{thm}
\begin{proof}
	Let $X^\#$ be the exterior of $K_\# \cup C$ in $S^3$, and similarly let $X^b$ be the exterior of $K_b \cup C$. Let $(X^\#,\gamma^\#)$ and $(X^b,\gamma^b)$ be the sutured exteriors with two meridional sutures on each boundary component. Using a longitude for $C$, we may define the sutures $\gamma^\#_n$ and $\gamma^b_n$ as before by replacing the meridional sutures on $\partial N(C)$ with oppositely oriented sutures of slope $\lambda - n\sigma$. 

	Given a closure $Y^\#$ of $(X^\#,\gamma^\#_n)$, the ribbon concordance induces a cobordism $W$ from $Y^\#$ to a closure $Y^b$ of $(X^b,\gamma^b_n)$. The ribbon concordance in reverse induces another cobordism $W'$ from $Y^b$ to $Y^\#$. By Theorem 6.7 of \cite{1904.09721}, the composite map $\SHI(X^\#,\gamma^\#_n) \to \SHI(X^b,\gamma^b_n) \to \SHI(X^\#,\gamma^\#_n)$ is the identity up to a unit. These cobordism maps commute with $\psi_n^\pm$ so we have induced $\C[U]$-module maps \[
		\SHI^-(S^3(K_\# \cup C),\sigma) \to \SHI^-(S^3(K_b \cup C),\sigma) \to \SHI^-(S^3(K_\# \cup C),\sigma)
	\]whose composite is the identity up to a unit. This proves that the first map is an inclusion onto a direct $\C[U]$-summand. The argument that we may choose our closure $Y^\#$ so that the induced $\C[U]$-module map is grading-preserving is the same as in the proof of Proposition 3.8 of \cite{1910.01758}. We chose the closure $Y^\#$ of $(X^\#,\gamma_n^\#)$ in such a way that Seifert surfaces for $K_\#$ and $C$ extend to closed surfaces. The closures of these Seifert surfaces are homologous within the cobordism $W$ to closures of Seifert surfaces of $K_b$ and $C$ so the inclusion respects the splitting along an affine space of Proposition~\ref{prop:splittingofSuturedInstantonMinus}. 
\end{proof}

\subsubsection{Gradings in sutured instanton Floer homology}\label{subsubsec:gradingSuturedInstanton}

We now prove Propositions \ref{prop:gradingShift}, \ref{prop:sameRelGrading}, and \ref{prop:simultaneousClosures}, restated here for the reader's convenience. 

\theoremstyle{plain}
\newtheorem*{prop:gradingShift}{Proposition~\ref{prop:gradingShift}}
\begin{prop:gradingShift}
	Let $S$ be an admissible surface in a balanced sutured manifold $(M,\gamma)$. Suppose $T$ is obtained from $S$ by $p$ positive stabilizations and $q$ negative stabilizations with $p - q = 2k$, then for each $i \in \Z$ \[
		\SHI(M,\gamma,S,i) = \SHI(M,\gamma,T,i + k).
	\]
\end{prop:gradingShift}

\newtheorem*{prop:sameRelGrading}{Proposition~\ref{prop:sameRelGrading}}
\begin{prop:sameRelGrading}[Theorem 1.12 of \cite{1910.10842}]
	Suppose $S$ and $T$ are admissible grading surfaces in a balanced sutured manifold $(M,\gamma)$ in the same relative homology class. Then $S$ and $T$ determine the same relative $\Z$-grading on $\SHI(M,\gamma)$. 
\end{prop:sameRelGrading}

\newtheorem*{prop:simultaneousClosures}{Proposition~\ref{prop:simultaneousClosures}}
\begin{prop:simultaneousClosures}
	Let $S$ be an admissible surface in a balanced sutured manifold $(M,\gamma)$, and let $\beta_1,\ldots,\beta_n$ be relative homology class in $H_2(M,\partial M)$. Then there are admissible surfaces $S',T_1,\ldots,T_n$ in $(M,\gamma)$ for which \begin{enumerate}[itemsep=-.5ex]
		\item decomposing $(M,\gamma)$ along either $S$ or $S'$ yields the same sutured manifold,
		\item $T_i$ represents $\beta_i$, 
		\item there exists a closure of $(M,\gamma)$ in which all of the surfaces $S',T_1,\ldots,T_n$ simultaneously extend to closed surfaces in the manner used to define gradings. 
	\end{enumerate} 
\end{prop:simultaneousClosures}

Recall that the complexity of a closed oriented surface $\bar{R}$ is $x(\bar{R}) = \sum \max(0,-\chi(\bar{R}_i))$ where the sum is over connected components of $\bar{R}$, and the Thurston norm $x(\alpha)$ of a second homology class $\alpha$ in a closed oriented $3$-manifold is the least complexity of an embedded representative. 

\begin{lem}\label{lem:muMapThurstonNorm}
	Let $(Y,\bar{R},\eta,\alpha)$ be a closure of a balanced sutured manifold $(M,\gamma)$, and let $\bar{S}$ and $\bar{T}$ be closed oriented surfaces embedded in $Y$ with no sphere components. If the Thurston norms of the homology classes $\bar{R} - \bar{T} + \bar{S}$ and $\bar{R} + \bar{T} - \bar{S}$ are less than or equal to the complexity of $\bar{R}$, then the gradings that they define on $I_*(Y|\bar{R})_{\eta +\alpha}$ are the same. 
	
	More generally, if the Thurston norm of $\bar{R} - \bar{T} + \bar{S}$ is at most $x(\bar{R}) - 2k$ while the Thurston norm of $\bar{R} + \bar{T} - \bar{S}$ is at most $x(\bar{R}) + 2k$, then the $i$-graded part of $I_*(Y|\bar{R})_{\eta+\alpha}$ with respect to $\bar{S}$ is equal to the $(i+k)$-graded part with respect to $\bar{T}$. 
\end{lem}

\begin{proof}
	We restrict attention to the $2$-generalized eigenspace $\Eig(\mu(y),2)$ of $\mu(y)$. The $i$-graded part of $I_*(Y|\bar{R})_{\eta+\alpha} = \Eig(\mu(\bar{R}),x(\bar{R}))$ with respect to $\bar{S}$ is by definition \[
		\Eig(\mu(\bar{R}),x(\bar{R})) \cap \Eig(\mu(\bar{S}),2i).
	\]Observe that $\Eig(\mu(\bar{R}),x(\bar{R})) \cap \Eig(\mu(\bar{S}),2i) \cap \Eig(\mu(\bar{T}),2j)$ is contained in \[
		\Eig(\mu(\bar{R} + \bar{S} - \bar{T}),x(\bar{R}) + 2i - 2j) \cap \Eig(\mu(\bar{R} - \bar{S} + \bar{T}),x(\bar{R}) - 2i + 2j)
	\]by additivity of $\mu$. Since the Thurston norm of the class $\bar{R} + \bar{S} - \bar{T}$ is at most $x(\bar{R}) - 2k$, the first generalized eigenspace in the intersection is zero when $j < i + k$ by Proposition~\ref{prop:eigenvalsOfMuS}. Similarly, the second is zero when $j > i + k$ because the Thurston norm of $\bar{R} - \bar{S} + \bar{T}$ is at most $x(\bar{R}) + 2k$. Thus \begin{align*}
		\Eig(\mu(\bar{R}),x(\bar{R})) \cap \Eig(\mu(\bar{S}),2i) &= \bigoplus_{j\in\Z} \Eig(\mu(\bar{R}),x(\bar{R})) \cap \Eig(\mu(\bar{S}),2i) \cap \Eig(\mu(\bar{T}),2j)\\
		&= \Eig(\mu(\bar{R}),x(\bar{R})) \cap \Eig(\mu(\bar{S}),2i) \cap \Eig(\mu(\bar{T}),2(i+k)).
	\end{align*}By the same argument with the roles of $\bar{S}$ and $\bar{T}$ reversed, we find that \[
		\Eig(\mu(\bar{R}),x(\bar{R})) \cap \Eig(\mu(\bar{S}),2i) = \Eig(\mu(\bar{R}),x(\bar{R})) \cap \Eig(\mu(\bar{T}),2(i+k)).\qedhere
	\]
\end{proof}
\begin{rem}
	If $\SHI(M,\gamma) \neq 0$, then the Thurston norm of $[\bar{R}]$ is exactly $x(\bar{R}) = 2g(\bar{R}) - 2$. It follows that the upper bounds on the Thurston norms of $\bar{R} - \bar{T} + \bar{S}$ and $\bar{R} + \bar{T} - \bar{S}$ are actually equalities. 
\end{rem}

In applications of the following lemma, $Q$ will be a surface representing the difference of homology classes of surfaces $\bar{S},\bar{T}$ used to define a grading. 

\begin{lem}\label{lem:ThurstonNormCalcForQ}
	Let $(Y,\bar{R},\eta,\alpha)$ be a closure of a balanced sutured manifold $(M,\gamma)$, and let $Q$ be a closed oriented surface embedded in $Y$ which meets $\bar{R}$ transversely in $n$ circles. Let $X$ be the exterior of $\bar{R}$ in $Y$ so that $\partial X = \bar{R}_+ \cup - \bar{R}_-$ where each of $\bar{R}_+,\bar{R}_-$ is oriented in the same way as $\bar{R}$. Let $Q_X$ be the surface $Q$ viewed in $X$, so that it is properly embedded in $X$ with $n$ boundary components on each of $\bar{R}_+$ and $\bar{R}_-$. 

	Suppose $Q_X$ may be isotoped rel boundary so that it is the union along boundary components of three surfaces with boundary $Q_X^+ \cup Q_X^A \cup Q_X^-$ with the following descriptions. The surface $Q_X^A$ is the disjoint union of $m$ properly embedded annuli in $X$, where each annulus connects $\bar{R}_+$ with $\bar{R}_-$. The surface $Q_X^+$ is a compact subsurface of $\bar{R}_+$ with $n + m$ boundary components with the same orientation as $\bar{R}_+$, and $Q_X^-$ is a compact subsurface of $\bar{R}_-$ with $n + m$ boundary components with the opposite orientation as $\bar{R}_-$. We also require that the complement of $Q_X^+$ in $\bar{R}_+$ and the complement of $Q_X^-$ in $\bar{R}_-$ are connected. 

	Then the homology class $\bar{R} + Q$ has Thurston norm at most $x(\bar{R}) - \chi(Q_X^+) + \chi(Q_X^-)$ while $\bar{R} - Q$ has Thurston norm at most $x(\bar{R}) + \chi(Q_X^+) - \chi(Q_X^-)$, as long as these two numbers are nonnegative. 
\end{lem}
\begin{proof}
	Assume that $Q_X$ has been isotoped to the form described. Up to isotopy, the closed surface $Q$ is recovered from $Q_X \subset Y$ by attaching $n$ disjoint annuli within the regular neighborhood of $\bar{R}$ which we parametrize as $[-1,1] \x \bar{R}$ as oriented manifolds. It follows that $\bar{R}_+ = \{-1\} \x \bar{R}$ while $\bar{R}_- = \{1\} \x \bar{R}$. If $c_1,\ldots,c_n$ are the original circles of intersection between $Q$ and $\bar{R}$, then the $n$ annuli that we attach to $Q_X$ are $[-1,1] \x c_i$. 

	Now consider the $2$-cycle $(\{0\} \x \bar{R}) \cup Q$ which represents the homology class $\bar{R} + Q$. Within $[0,1] \x \bar{R}$, this $2$-cycle consists of $\{0\} \x \bar{R}$, the annuli $[0,1] \x c_i$ and the surface $\{1\} \x Q_X^-$ with the orientation opposite to that of $\{1\}\x \bar{R} = \bar{R}_-$. Let $d_1,\ldots,d_m$ be the $m$ boundary components of $Q_X^-$ which are the $m$ boundary components of $Q_X^A$ lying on $\bar{R}_-$. The $3$-chain $[0,1] \x Q_X^-$ shows us that we may swap $\{0,1\}\x Q_X^- \cup \bigcup_i [0,1] \x c_i$ for $\bigcup_i [0,1] \x d_i$ and obtain a homologous surface. The effect is that the Euler characteristic has dropped by $2\cdot \chi(Q_X^-)$, and the resulting surface can be smoothed out. This resulting embedded surface has Euler characteristic $\chi(\bar{R}) + \chi(Q_X^+) - \chi(Q_X^-)$. From the assumption that the complement of $Q_X^-$ is connected, the desired claim on Thurston norm follows. 

	By applying the same argument to the $2$-cycle $(\{0\} \x \bar{R}) \cup -Q$ with a modification within $[-1,0] \x R$, the other claim follows as well.
\end{proof}

\begin{lem}\label{lem:gradingShiftF}
	Let $S$ and $T$ be admissible surfaces in a balanced sutured manifold $(M,\gamma)$ in the same relative homology class. 
	Suppose that $\partial T$ and $\partial S$ are disjoint, and that $\partial T - \partial S$ is the boundary of a subsurface $F$ of $\partial M$, where the orientation of $F$ matches that of $\partial M$. Then $S$ and $T$ determine the same relative $\Z$-grading on $\SHI(M,\gamma)$. In particular \[
		\SHI(M,\gamma,S,i) = \SHI(M,\gamma,T,i+k)
	\]where $k = \chi(F \cap R_-(\gamma)) - \chi(F \cap R_+(\gamma))$. 
\end{lem}
\begin{proof}
	Let $B$ be a connected auxiliary surface for $(M,\gamma)$ with a fixed identification $\partial B = s(\gamma)$, and let $X$ be the associated preclosure $M \cup [-1,1] \x B$ with oriented boundary $\bar{R}_+ - \bar{R}_-$ where $\bar{R}_\pm = R_\pm(\gamma) \cup \{\pm1\} \x B$. Choose two collections $\delta_T$ and $\delta_S$ of properly embedded arcs in $B$ so that all of the arcs in $\delta_T \cup \delta_S$ are disjoint, and $\partial \delta_T = \partial T \cap \partial B$ and $\partial \delta_S = \partial S \cap \partial B$. The simple closed curves $((\partial T \cup \partial S) \cap R_+(\gamma)) \cup (\{1\} \x (\delta_T \cup \delta_S))$ on $\bar{R}_+$ inherit orientations from $\partial T \cup \partial S$, and we orient $\delta_T \cup \delta_S$ to match these orientations. Furthermore, we assume that the genus of $B$ is large enough that the arcs in $\delta_T \cup \delta_S$ are all homologically independent in $H_1(B,\partial B)$ and $\delta_T$ and $\delta_S$ give rise to balanced pairings for $T$ and $S$, respectively. Fix a diffeomorphism $h\colon \bar{R}_+ \to \bar{R}_-$ which identifies the boundary components of $T \cup ([-1,1] \x \delta_T)$ on $\bar{R}_+$ with those on $\bar{R}_-$, and similarly for $S \cup ([-1,1] \x \delta_S)$. Let $Y$ be the result of gluing $[1,3] \x \bar{R}_+$ to $X$ using $\Id \colon 1 \x \bar{R}_+ \to \bar{R}_+$ and $h\colon 3\x\bar{R}_+ \to \bar{R}_-$. Let $\bar{T}$ and $\bar{S}$ be the natural extensions of $T \cup ([-1,1] \x \delta_T)$ and $S \cup ([-1,1] \x \delta_S)$, respectively, in $X$.

	Let $Q_X \subset X$ be the union of $F$ with $[-1,1] \x (\delta_T \cup \delta_S)$. The natural closed extension of $F$ in $Y$ is a closed oriented surface $Q$ in the same homology class of $\bar{T} - \bar{S}$. We now isotope $Q_X$ so that it is in the form described in Lemma~\ref{lem:ThurstonNormCalcForQ}. Observe that the intersection of $Q_X$ with $\partial X$ is the union \[
		(F \cap (R_-(\gamma) \cup R_+(\gamma)) \cup (\{-1,1\} \x (\delta_T \cup \delta_S))
	\]where the first term is a $2$-manifold with boundary while the second term is a union of arcs. The intersection of $Q_X$ with the interior of $X$ is a collection of annuli embedded in $(-1,1) \x B$ whose cross sections in $t \x B$ are all the same. Fix one of these annuli, and let $a$ be its intersection with $B$ in any cross section. At least one subarc of $a$ lies on along the boundary of $B$ so there is a unique pushoff of $a$ which still lies in $B$. Let $a^*$ be this pushoff, so that $a^*$ and $a$ cobound an annulus $A$ in $B$. Within $Q_X$, replace the fixed annulus with \[
		(\{-1\} \x A) \cup ([-1,1] \x a^*) \cup (\{1\} \x A).
	\]This can be achieved by an isotopy of $Q_X$, and we apply this to each annulus of $Q_X \cap \Int(X)$. The resulting surface, which we refer to now as $Q_X$, is of the form described in Lemma~\ref{lem:ThurstonNormCalcForQ}. 

	To compute $\chi(Q_X^+) - \chi(Q_X^-)$, we observe that $Q_X^+ \cap B$ and $Q_X^+\cap B$ are identical so \begin{align*}
		\chi(Q_X^+) - \chi(Q_X^-) &= \chi(Q_X^+ \cap R_+(\gamma)) - \chi(Q_X^- \cap R_-(\gamma)) \\
		&= \chi(F \cap R_+(\gamma)) - \chi(F \cap R_-(\gamma)).
	\end{align*}The claim now follows from Lemmas \ref{lem:muMapThurstonNorm} and \ref{lem:ThurstonNormCalcForQ}. 
\end{proof}

\begin{lem}\label{lem:gradingShiftSameSide}
	Let $S$ be an admissible surface in a balanced sutured manifold $(M,\gamma)$. Suppose $T$ is obtained by two stabilizations of $S$ satisfying the following two conditions: \begin{enumerate}[itemsep=-0.5ex]
		\item the Whitney discs introduced in the two stabilizations are disjoint,
		\item any small positive (with respect to the orientation of $S$) parallel pushoff of $\partial S$ within $\partial M$ is disjoint from $\partial T$.
	\end{enumerate}Then $\SHI(M,\gamma,S,i) = \SHI(M,\gamma,T,i+k)$ where \[
		k = \begin{cases}
			1 & \text{the stabilizations are both positive}\\
			0 & \text{one is negative and one is positive}\\
			-1 & \text{both are negative.}
		\end{cases}
	\]
\end{lem}
\begin{proof}
	Apply a small isotopy of $S$, which extends a small positive parallel pushoff of $\partial S$ so that $\partial T$ and $\partial S$ are disjoint. 
	Let $F$ be the oriented subsurface of $\partial M$ with boundary $\partial T - \partial S$. Let $D_1,D_2$ be the Whitney discs of the stabilizations. Note the second condition in the statement of the lemma implies that $D_i$ lies in $R_+(\gamma)$ if and only if the stabilization is negative. Thus \[
		k = \chi(F \cap R_-(\gamma)) - \chi(F \cap R_+(\gamma))
	\]is just the number of positive stabilizations minus the number of negative stabilizations. The claim now follows from Lemma~\ref{lem:gradingShiftF}. 
\end{proof}

\begin{lem}\label{lem:gradingShiftOppSide}
	Let $S$ and $T$ be admissible surfaces, both obtained from a common surface $V$ by a single stabilization of the same sign. Also suppose that any small positive pushoff of $\partial S$ is disjoint from $\partial T$. Then $S$ and $T$ define the same grading. 
\end{lem}
\begin{proof}
	Apply an isotopy of $S$ extending a small positive parallel pushoff of $\partial S$ so that $\partial S$ and $\partial T$ are disjoint, and let $F$ be the oriented subsurface of $\partial M$ with boundary $\partial T - \partial S$. Lemma~\ref{lem:gradingShiftF} implies that $S$ and $T$ define the same grading if $\chi(F \cap R_+(\gamma)) = \chi(F \cap R_-(\gamma))$. Because both $S$ and $T$ are obtained from $V$ by a stabilization of the same sign, one Whitney disc lies in $R_+(\gamma)$ while the other lies in $R_-(\gamma)$ from which it follows that $\chi(F \cap R_+(\gamma)) = \chi(F \cap R_-(\gamma))$. 
\end{proof}

\begin{proof}[Proof of Proposition~\ref{prop:gradingShift}]
	It suffices to prove the result when $T$ is obtained from $S$ by exactly two stabilizations. If \begin{enumerate}[itemsep=-0.5ex]
		\item the two Whitney discs are disjoint,
		\item any small positive parallel pushoff of $\partial S$ is disjoint from $\partial T$
	\end{enumerate}then the result is true by Lemma~\ref{lem:gradingShiftSameSide}. If the first condition holds and any small negative parallel pushoff of $\partial S$ is disjoint from $\partial T$, then $-S$ and $-T$ satisfy the conditions of Lemma~\ref{lem:gradingShiftSameSide}. But since a positive stabilization of $S$ is a negative stabilization of $-S$, we have that \[
		\SHI(M,\gamma,S,i) = \SHI(M,\gamma,-S,-i) = \SHI(M,\gamma,-T,-i-k) = \SHI(M,\gamma,T,i+k).
	\]Suppose the first condition holds but small no parallel pushoff of $\partial S$ is disjoint from $\partial T$. Then we can find a surface $T'$ for which Lemma~\ref{lem:gradingShiftOppSide} applies to $T$ and $T'$ and for which Lemma~\ref{lem:gradingShiftSameSide} applies to $S$ and $T'$. The appropriate grading shift for $S$ and $T$ then follows. 

	Finally, if the first condition fails, then we can find surfaces $U,V$ such that \[
		\begin{tikzcd}
			U \ar[r] & V\\
			S \ar[u] & T \ar[u]
		\end{tikzcd}
	\]where $A \to B$ indicates that $B$ can be obtained from $A$ by two stabilizations of appropriate sign with disjoint Whitney discs, from which it follows that the grading shift property holds for $S$ and $T$. 
\end{proof}

\begin{df*}[Equivalent collections of curves]
	Suppose $\alpha$ and $\beta$ are collections of disjoint oriented simple closed curves on a closed oriented surface $\Sigma$. Then $\alpha$ and $\beta$ are \textit{equivalent} if there is a sequence $\delta_1,\ldots,\delta_n$ of collections of disjoint oriented simple closed curves on $\Sigma$ for which \begin{enumerate}[itemsep=-0.5ex]
		\item $\delta_1 = \alpha$ and $\delta_n = \beta$,
		\item $\delta_i$ and $\delta_{i+1}$ are disjoint, and there is a subsurface $F_i$ of $\Sigma$, oriented compatibly with either $\Sigma$ or $-\Sigma$, for which $\partial F_i = \delta_{i+1} - \delta_{i}$.
	\end{enumerate}This defines an equivalence relation on isotopy classes of disjoint oriented simple closed curves, and if $\alpha$ and $\beta$ are equivalent, then they are homologous. 
\end{df*}

\begin{lem}\label{lem:homologousCurvesEquiv}
	Suppose $\alpha$ and $\beta$ are collections of disjoint oriented simple closed curves on a closed oriented surface $\Sigma$. If $\alpha$ and $\beta$ are homologous, then they are equivalent. 
\end{lem}
\begin{proof}
	First, $\alpha$ is equivalent to some number of parallel copies of a simple closed curve $\mu$ by Lemma 3.10 of \cite{Gab83}. We repeat the argument here for the reader's convenience. Consider the components of $\Sigma\setminus N(\alpha)$ where $N(\alpha)$ is a regular neighborhood of $\alpha$. If any component has just a single boundary component, then we may erase the corresponding component of $\alpha$ from the collection and obtain an equivalent collection with fewer components. If a component $S$ of $\Sigma\setminus N(\alpha)$ contains $3$ or more boundary components, then there is an orientation of $S$ for which the orientation of at least two of its boundary components inherited from $\alpha$ coincides with its boundary orientation. Using a pair of pants within $S$, we may again reduce the number of components in $\alpha$. Hence we may assume every component of $S\setminus N(\alpha)$ has two boundary components. Now fix a component $\mu$ of $\alpha$, and let $S$ be a component of $\Sigma\setminus N(\alpha)$ for which one of its boundary components corresponds to $\mu$. If the other boundary component also corresponds to $\mu$, then $\alpha = \mu$ because $\Sigma$ is connected, so assume that the other boundary component of $S$ corresponds to a different component $\mu'$ of $\alpha$. Then $\alpha$ is equivalent via $S$ to the collection of curves where $\mu'$ is replaced with a parallel copy of $\mu$. Continuing in this way, we see that $\alpha$ is equivalent to some number of parallel copies of $\mu$. 

	Thus we can assume that $\alpha$ consists of $k$ parallel copies of a simple closed curve $\mu$, and similarly $\beta$ consists of $\ell$ parallel copies of a simple closed curve $\nu$. Then $\mu$ and $\nu$ have algebraic intersection number $0$. Think of $\mu$ as a particular component of $\alpha$, and choose an arc $c$ along $\mu$ whose endpoints are two intersection points $\mu \cap \beta$ of opposite sign where the interior of $\mu$ is disjoint from $\beta$. Tubing $\beta$ along $c$ results in an equivalent collection of curves $\beta'$, so by continuing in this way, we may assume that $\alpha$ and $\beta'$ are disjoint, where $\beta'$ no longer must consist only of parallel copies of a simple closed curve. It suffices to show that $\alpha$ and $\beta'$ are equivalent. We can assume that $\Sigma\setminus N(\alpha)$ consists of $k-1$ annuli and a connected component $Z$ containing $\beta'$, because otherwise $\alpha$ is null-homologous. Suppose that no other collection of curves equivalent to $\beta'$ within $Z$ has fewer components than $\beta'$. The same argument as before shows that every component of $Z \setminus N(\beta')$ has at most $2$ boundary components corresponding to curves in $\beta'$, and any component containing a single boundary component corresponding to $\beta'$ also has a boundary component corresponding to $\alpha$. A similar argument as above now shows that $\beta'$ is equivalent to a number of parallel copies of $\mu$. 
\end{proof}

\begin{proof}[Proof of Proposition~\ref{prop:sameRelGrading}]
	By Lemma~\ref{lem:homologousCurvesEquiv}, we know that $\partial S$ and $\partial T$ are equivalent, so let $\delta_1,\ldots,\delta_n$ be a sequence of collections of curves for which $\delta_1 = \partial S$, $\delta_n = \partial T$, and let $F_i$ be a subsurface of $\partial M$ for which $\partial F_i = \delta_{i+1} - \delta_{i}$. Let $S_2$ be the surface obtained by gluing $S$ to $F_1$ along the boundary of $S$ and pushing its interior off of $\partial M$. Let $S_{i+1}$ be obtained in the same way from $S_{i}$ and $F_{i}$ so that $S_n$ is a surface whose boundary is $\delta_n = \partial T$. If $S_2,\ldots,S_{n}$ are all admissible, then $S,S_2,\ldots,S_n$ all determine the same relative $\Z$-grading by Lemma~\ref{lem:gradingShiftF}. Since $S_n$ and $T$ lie in the same relative homology class and have the same boundary, they determine the same absolute $\Z$-grading, so it follows that $S$ and $T$ determine the same relative $\Z$-grading. 

	Suppose $S_i$ is not admissible. Consider the following operation which makes $S_i$ into an admissible surface. Choose a point $p \in s(\gamma)$ which does not lie on $\partial S_i$, and let $D_p$ be a disc neighborhood of $p$ in $\partial M$ whose interior is pushed off of $\partial M$. The union $S_i \cup D_p$ is admissible. By Lemma~\ref{lem:gradingShiftF}, the relative $\Z$-grading determined by $S_i \cup D_p$ does not depend on $p$ while the absolute $\Z$-grading depends on the orientation of $D_p$. We refer to this relative $\Z$-grading as the one that $S_i$ defines when $S_i$ is not admissible. 

	To see that $S$ and $S_2$ define the same relative grading when $S_2$ is not admissible, arbitrarily choose a point $p \in s(\gamma)$ disjoint from both $\partial S = \delta_1$ and $\partial S_2 = \delta_2$, and let $D_p$ be a disc neighborhood of $p$ whose interior is pushed off of $\partial M$, oriented in the same way $F_1$ is oriented. If $p$ lies in the interior of $F_1$, we use the surface $F_1 \setminus D_1$ to apply Lemma~\ref{lem:gradingShiftF} to see that $S$ and $S_2$ determine the same relative $\Z$-grading. If $p$ is not in the interior of $F_1$, we use $F_1 \cup D_p$. The same argument shows that $S_i$ and $S_{i+1}$ determine the same relative $\Z$-gradings so $S$ and $T$ do too. 
\end{proof}

\begin{proof}[Proof of Proposition~\ref{prop:simultaneousClosures}]
	Our proof follows the proof of Theorem 1.12 of \cite{1910.10842}. Arbitrarily choose admissible representatives $T_1,\ldots,T_n$ of the given relative homology classes, and assume that $S,T_1,\ldots,T_n$ all intersect transversely. Let $T_0 = S$. Via isotopies supported near their boundaries, we will achieve the following conditions: \begin{enumerate}[itemsep=-.5ex]
		\item each component of $\partial T_i \cap (R_+(\gamma) \cup R_-(\gamma))$ intersects $\bigcup_{i=0}^n \partial T_i$ in at most one point. 
		\item each positive intersection point $T_i \cap T_j$ with $i < j$ will lie on $R_+(\gamma)$ while each negative intersection point lies on $R_-(\gamma)$. 
	\end{enumerate}Furthermore, no new intersection points between the boundaries of these surfaces will be introduced, and the only isotopies of $S$ that change how it intersects $s(\gamma)$ will be negative stabilizations. 

	Note that further stabilizations preserve the two conditions. Choose balanced pairings for each of $T_0,\ldots,T_n$ in such a way that for each component $\delta$ of $\partial T_i \cap (R_+(\gamma) \cup R_-(\gamma))$ which intersects $\bigcup_i \partial T_i$ in a point, the boundary points of $\delta$ are paired. This is possible after sufficiently many extra stabilizations. Consider the resulting set of curves in a preclosure associated to a auxiliary surface of sufficiently large genus. The first condition implies that each curve on $\bar{R}_+$ intersects at most one other curve, and similarly for $\bar{R}_-$. To see that the curves on $\bar{R}_+$ corresponding to $\partial T_i$ intersect the curves corresponding to $\partial T_j$ in the same number of times that the corresponding curves on $\bar{R}_-$ intersect, note that the intersection of $T_i$ and $T_j$ consists of arcs and circles. The arcs of intersection define a pairing between the positive intersection points and negative. The second condition guarantees that the positive intersection points are on $R_+(\gamma)$ and the negative on $R_-(\gamma)$. Thus there exists a diffeomorphism $h\colon \bar{R}_+\to \bar{R}_-$ taking one set of curves to the other in an orientation-reversing way. 

	We achieve the two conditions stated above in 4 steps. 

	Step 1. Consider an arc $\delta$ of $\partial T_j \cap R(\gamma)$ for some $j \ge 1$. Suppose that either $\delta$ intersects $\partial S$ more than once, or it intersects $\partial S$ exactly once but intersects $\bigcup_{i=1}^n \partial T_i$ at least once. Choose an intersection point of $\delta \cap \partial S$ and apply a finger move to $\delta$ in the direction of $\partial S$ until it reaches $s(\gamma)$ and creates a stabilization. If this arc of $\partial S$ along which the finger move is applied intersects other curves, then do stabilizations to them as well. By repeating this process, we have ensured that every arc of $\partial T_j \cap R(\gamma)$ either intersects $\partial S$ once and nothing else, or it intersects $\bigcup_{i=1}^n \partial T_i$. 

	Step 2. Let $\delta$ be an arc of $\partial S \cap R(\gamma)$, and suppose that it intersects $\bigcup_{i=1}^n \partial T$ more than once. Choose a path in $R(\gamma)$ starting a point on $\delta$ between two intersection points and ending on $s(\gamma)$ whose interior is disjoint from $\partial S$. The interior may intersect $\partial T$. Apply a finger move along this path. We require that the path be chosen so that the resulting stabilization of $S$ is negative. We have now ensured that every arc of $\partial S \cap R(\gamma)$ has at most one intersection point with $\bigcup_{i=1}^n \partial T_i$ and that each arc of $\partial T_j \cap R(\gamma)$ which intersects $\partial S$ does so in a single point and intersects nothing else. 

	Step 3. Now let $\delta$ be an arc of $\partial T_j \cap R(\gamma)$ which intersects $\bigcup_{i=1}^n \partial T_i$ in more than one point. Choose an intersection point with say $\partial T_i$ closest to an endpoint of $\delta$ and apply a finger move to this arc of $\partial T_i$ towards this endpoint of $\delta$ and create a stabilization. The intersection point is now on the other side of $s(\gamma)$ and is with an arc $\delta'$ of $\partial T_j \cap R(\gamma)$. We then do a small stabilization of $\delta'$ as shown in Figure~\ref{fig:twoStabs}. 
	We have now achieved the first condition. 

	\begin{figure}[!ht]
		\centering
		\labellist
		\pinlabel $\delta$ at 125 10
		\pinlabel $\delta'$ at 128 230
		\pinlabel $\partial T_j$ at 85 225
		\pinlabel $s(\gamma)$ at 10 137
		\pinlabel $\partial T_i$ at 10 71
		\pinlabel $\delta$ at 520 10
		\pinlabel $\delta'$ at 523 230
		\pinlabel $\partial T_j$ at 480 225
		\pinlabel $s(\gamma)$ at 405 137
		\pinlabel $\partial T_i$ at 405 71
		\endlabellist
		\vspace{5pt}
		\includegraphics[width=.3\textwidth]{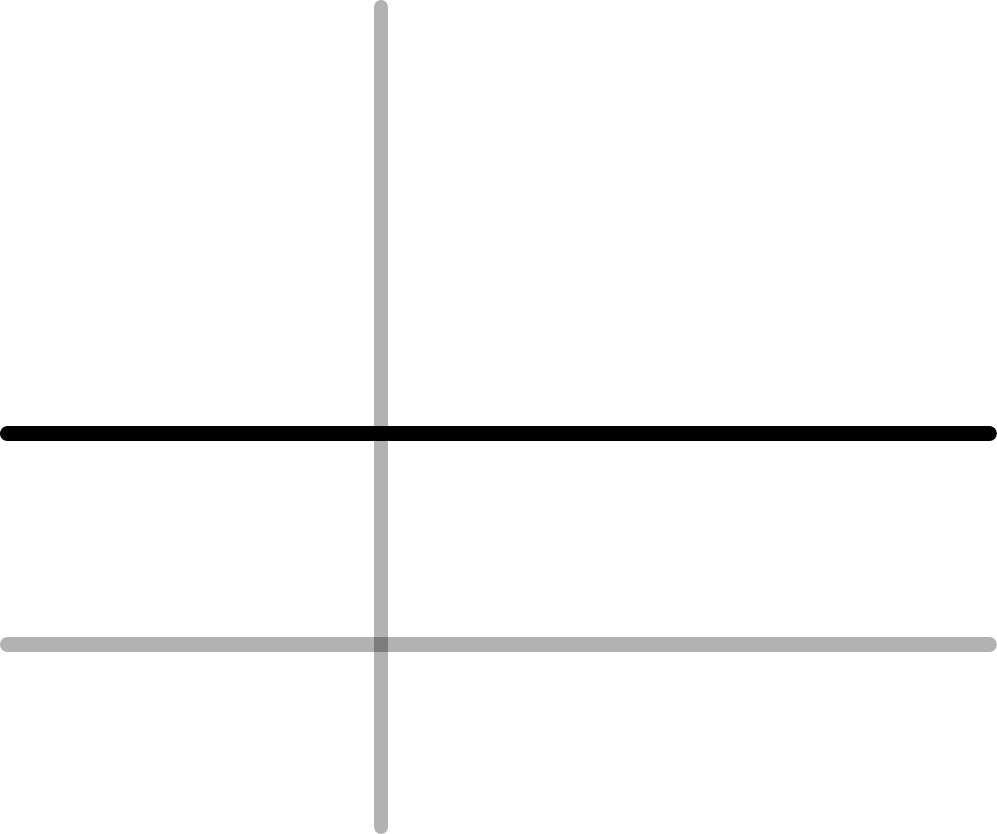}
		\hspace{40pt}
		\includegraphics[width=.3\textwidth]{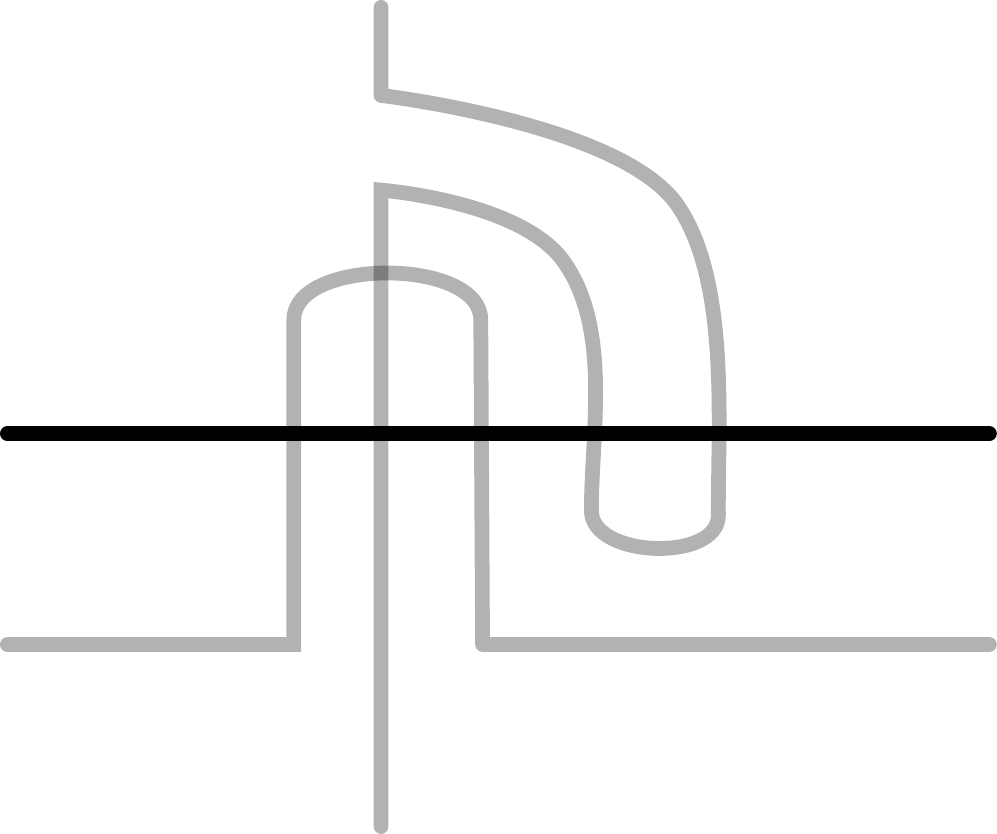}
		\caption{Two stabilizations.}
		\label{fig:twoStabs}
	\end{figure}
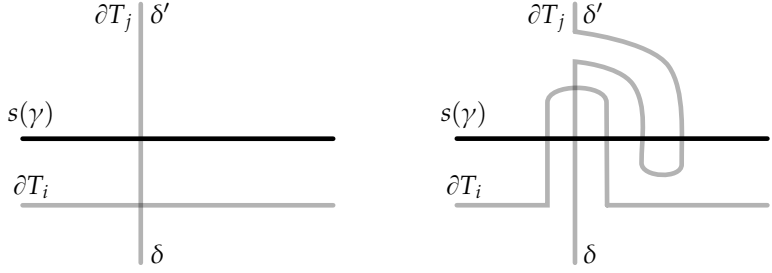

	Step 4. Let $\delta$ be an arc of $\partial T_j$ which intersects $\partial T_i$ with $i < j$. Assume that the sign of the intersection point does not match the sign of $R(\gamma)$ in which it lies. By a finger move of $\delta$ along $\partial T_i$, we can push this intersection point to the other side of $s(\gamma)$. We then do a small stabilization just as in Step 3 to ensure that each arc still intersects $\bigcup_{i=0}^n \partial T_i$ just once. The second condition has now also been satisfied. 
\end{proof}

\subsection{Detecting the trivial band}\label{subsec:DetectTrivBandInstanton}

In this section, we show that $\KHI(K_\#)$ and $\KHI(K_b)$ have the same dimension over $\C$ if and only if the band $b$ is trivial. 

\begin{lem}[Proposition 9.15 of \cite{Juh06}]\label{lem:instantonConnectedSumFormula}
	Let $(M,\gamma)$ and $(M',\gamma')$ be balanced sutured manifolds, and consider their connected sum $(M\# M',\gamma \cup \gamma')$ obtained by identifying $3$-balls in their interiors. Then there is an isomorphism \[
		\SHI(M\# M',\gamma \cup \gamma') \cong \SHI(M,\gamma) \otimes \SHI(M',\gamma') \otimes \C^2.
	\]
\end{lem}
\begin{proof}
	This result is just Proposition 9.15 of \cite{Juh06} in the instanton Floer context. As Juh\'asz observes, there is a product disc decomposition \[
		(M\# M',\gamma \cup\gamma') \rightsquigarrow (M,\gamma) \amalg (M',\gamma')(1)
	\]where $(M',\gamma')(1)$ denotes $(M',\gamma')$ with a sutured puncture. There is another product disc decomposition \[
		(M',\gamma')(1) \rightsquigarrow (M',\gamma') \amalg S^3(2)
	\]where $S^3(2)$ denotes $S^3$ with two sutured punctures. The sutured instanton homology of a disjoint union is the tensor product of the sutured instanton homologies of the individual balanced sutured manifolds. Furthermore, the inverse of a product disc decomposition is a contact $1$-handle attachment, and the contact $1$-handle attachment map is an isomorphism. Thus, it suffices to show that $\SHI(S^3(2)) \cong \C^2$. Since there is a product disc decomposition $S^1 \x S^2(1) \rightsquigarrow S^3(2)$, the result follows from the fact that $\SHI(S^1\x S^2(1)) \cong \C^2$. 
\end{proof}

\begin{prop}\label{prop:instantonRankOfLinkMinusEqualsTwiceDimOfKnotHat}
	Let $K \cup C$ be a two component link in $S^3$. Let $\sigma$ be either of the two sutures on $\partial N(C)$ on the sutured exterior $S^3(K \cup C)$. Then \[
		\rank \SHI^-(S^3(K \cup C),\sigma) = 2\dim \KHI(K).
	\]
\end{prop}
\begin{proof}
	Let $S^3(K)(1)$ denote the sutured exterior of $K$ with a sutured puncture. Observe that $S^3(K)(1)$ is the result of attaching a $2$-handle to the suture $\sigma$. By Corollary~\ref{cor:rankOfMinusInstantonEqualsDimOfHatNBeta} \[
		\rank \SHI^-(S^3(K \cup C),\sigma) = \dim \SHI(S^3(K)(1)).
	\]The result follows from the connected sum formula (Lemma~\ref{lem:instantonConnectedSumFormula}).
\end{proof}

\begin{cor}\label{cor:instantonSufficesToConsiderMinusLink}
	Let $K_b$ be the band sum of a split two-component link, and let $K_\#$ be the connected sum. Then \[
		\dim \KHI(K_b) = \dim \KHI(K_\#)
	\]if and only if \[
		\rank \SHI^-(S^3(K_b \cup C),\sigma) = \rank \SHI^-(S^3(K_\# \cup C),\sigma)
	\]In both cases, $C$ denotes a linking circle, and $\sigma$ is a suture on $\partial N(C)$. 
\end{cor}

\begin{rem}
	By Corollary~\ref{cor:structMinusInstantonSutured}, we know that there is an $\C[U]$-module isomorphism \[
		\SHI^-(S^3(K_b \cup C),\sigma) \cong \C[U]^k \oplus \bigoplus_{i=1}^n \frac{\C[U]}{U^{k_i}}
	\]with $k_i$ positive where the generators of the summands are homogeneous with respect to the splitting along an affine space of Proposition~\ref{prop:splittingofSuturedInstantonMinus}. Although the generators themselves of the free summands are not well-defined, the relative gradings in which they lie are well-defined. 
\end{rem}

\begin{lem}\label{lem:instantonMinusDisjointUnion}
	Suppose $(M,\gamma,\sigma)$ a balanced sutured manifold with a suitable distinguished suture, and $(M',\gamma')$ is another balanced sutured manifold. Then \[
		\SHI^-(M\cup M',\gamma \cup \gamma',\sigma) \cong \SHI(M',\gamma') \otimes_\C \SHI^-(M,\gamma,\sigma)
	\]in a grading-preserving way. 
\end{lem}
\begin{proof}
	There is an isomorphism \[
		\SHI(M \cup M',\gamma \cup \gamma') \cong \SHI(M',\gamma') \otimes \SHI(M,\gamma)
	\]defined by excision by Proposition 6.5 of \cite{MR2652464}. We choose bases for the second relative homologies of $M$ and $M'$, and choose closures in which representatives simultaneously extend to closed surfaces. The excision isomorphism defined on such closures will be preserving gradings. The excision map intertwines \[
		\psi_n^\pm\colon \SHI(M \cup M',\gamma_n \cup \gamma') \to \SHI(M \cup M',\gamma_{n+1}\cup \gamma')
	\]with \[
		\Id \otimes \psi_n^\pm \colon \SHI(M',\gamma') \otimes \SHI(M,\gamma_n) \to \SHI(M',\gamma') \otimes \SHI(M,\gamma_{n+1})
	\]and the result follows. 
\end{proof}

\begin{prop}\label{prop:instantonMinusOfConnectedSum}
	There is an isomorphism of $\C[U]$-modules \[
		\SHI^-(S^3(K_\# \cup C),\sigma) \cong \KHI(K_\#) \otimes_\C \C^2 \otimes_\C \C[U].
	\]With respect to the splitting of $\SHI^-(S^3(K_\# \cup C),\sigma)$ along an affine space of Proposition~\ref{prop:splittingofSuturedInstantonMinus}, no two generators of free summands lie in gradings differing by $[\sigma]$. 
\end{prop}
\begin{proof}
	Since $K_\# \cup C$ is a split link, there is a product disc decomposition \[
		S^3(K_\# \cup C) \rightsquigarrow S^3(K_\#)(1) \amalg S^3(C)
	\]just like in Lemma~\ref{lem:instantonConnectedSumFormula}. By Lemma~\ref{lem:instantonMinusDisjointUnion}, we have an isomorphism of $\C[U]$-modules
	\[
		\SHI^-(S^3(K_\#)(1) \amalg S^3(C),\sigma) \cong \SHI(S^3(K_\#)(1)) \otimes_\C \SHI^-(S^3(C),\sigma).
	\]By Proposition~\ref{prop:instantonMinusUTriangle}, we have the exact triangle \[
		\begin{tikzcd}[column sep=-1em]
			\SHI^-(S^3(C),\sigma) \ar[rr,"U"] & & \SHI^-(S^3(C),\sigma) \ar[dl]\\
			& \KHI(C) \ar[ul]
		\end{tikzcd}
	\]so $\SHI^-(S^3(C),\sigma) \cong \C[U]$ as a module because $\KHI(C) \cong \C$. Since a contact $1$-handle attachment is the reverse of a product disc decomposition, from Lemma~\ref{lem:surfDecomp1handle} we see that $\SHI^-(S^3(K_\# \cup C),\sigma)$ is a free $\C[U]$-module. 

	Again by Proposition~\ref{prop:instantonMinusUTriangle}, we have an exact triangle \[
		\begin{tikzcd}[column sep=-2em]
			\SHI^-(S^3(K_\# \cup C),\sigma) \ar[rr,"U"] & & \SHI^-(S^3(K_\# \cup C),\sigma) \ar[dl]\\
			& \KHI(K_\# \cup C) \ar[ul] &
		\end{tikzcd}
	\]where $\SHI^-(S^3(K_\# \cup C),\sigma) \to \KHI(K_\# \cup C)$ is grading-preserving. Since $\SHI^-(S^3(K_\# \cup C),\sigma)$ is free, this grading-preserving map is surjective. To see that no two generators of free summands of $\SHI^-(S^3(K_\# \cup C),\sigma)$ lie in gradings differing by $[\sigma]$, it suffices to show that no two basis elements of $\KHI(K_\# \cup C)$ lie in gradings differing by $[\sigma]$. 

	Let $D$ be a disc whose boundary is $C$ and is disjoint from $K_\#$. It suffices to show that $\KHI(K_\# \cup C)$ is concentrated in a single grading with respect to the $\Z$-grading that $D$ defines. The disc $D$ extends to a torus $\bar{D}$ in a closure $Y$ of $S^3(K_\# \cup C)$ so the only eigenvalue of $\mu(\bar{D})$ on $\KHI(K_\# \cup C)$ is zero, so the result follows. 
\end{proof}

\begin{cor}\label{cor:instantonDifferentGradingsGeneratorsSuffices}
	If there are generators of free summands of $\SHI^-(S^3(K_b \cup C),\sigma)$ lying in gradings differing by $[\sigma]$, then \[
		\rank \SHI^-(S^3(K_b \cup C),\sigma) > \rank \SHI^-(S^3(K_\# \cup C),\sigma).
	\]
\end{cor}
\begin{proof}
	By Theorem~\ref{thm:instantonMinusRibbon}, there is an injective grading-preserving map \[
		\SHI^-(S^3(K_\# \cup C),\sigma) \hookrightarrow \SHI^-(S^3(K_b \cup C),\sigma)
	\]onto a direct $\C[U]$-module summand. If the two modules have the same rank, then by Proposition~\ref{prop:instantonMinusOfConnectedSum}, no generators of free summands of $\SHI^-(S^3(K_b \cup C),\sigma)$ can lie in gradings differing by $[\sigma]$. 
\end{proof}

We will use the sequence of surface decompositions constructed in Theorem~\ref{thm:SeqOfSurfDecomp} to verify that the condition of Corollary~\ref{cor:instantonDifferentGradingsGeneratorsSuffices} holds when the band $b$ is nontrivial. 

\begin{lem}\label{lem:instantonHopfLink}
	Let $H$ be a two-component Hopf link, and let $\sigma$ be one of the four sutures on $S^3(H)$. Then there are generators of free summands of $\SHI^-(S^3(H),\sigma)$ lying in gradings differing by $[\sigma]$. 
\end{lem}
\begin{proof}
	We first compute $\KHI(H)$. There is a skein exact triangle (Theorem 3.1 of \cite{MR2683750}) relating $H$, the unknot, and the figure eight $4_1$. Since the sutured instanton homology of the sutured exterior of a knot $K$ has Euler characteristic the Alexander polynomial of $K$ (Theorem 1.1 of \cite{MR2683750}), we know that $\dim \KHI(4_1) \ge 5$ so $\dim \KHI(H) \ge 4$. The skein exact triangle involving $H$, the unknot, and the two-component unlink shows that $\dim \KHI(H) \leq 4$ so $\KHI(H)$ is $4$-dimensional. 

	Genus zero Seifert surfaces for the two components each viewed as an annulus in $S^3(H)$ form a basis $\alpha,\beta$ for $H_2(S^3(H),\partial S^3(H))$. Decomposing along either yields the exterior of an unknot with four meridional sutures. The sutured instanton homology of this manifold is $2$-dimensional as proved in Section 3.1 of \cite{MR2683750}. It follows that with respect to the splitting of Proposition~\ref{prop:instantonSplittingAlongAffineSpace}, each graded part of $\KHI(H)$ is either zero or $1$-dimensional, and there is a grading $h$ in which $\KHI(H,h) \cong \C$ for which the other three gradings that support $\KHI(H)$ are $h + \alpha^*,h+\beta^*,h+\alpha^*+\beta^*$ where $\alpha^*,\beta^*$ are the dual basis elements of $\Hom(H_2(S^3(H),\partial S^3(H)),\Z)$ arising from the basis $\alpha,\beta$ of $H_2(S^3(H),\partial S^3(H))$. 

	We now turn to $\SHI^-(S^3(H),\sigma)$. Attaching a $2$-handle to $\sigma$ yields $S^3(U)(1)$, whose sutured instanton homology is $2$-dimensional. By Corollary~\ref{cor:rankOfMinusInstantonEqualsDimOfHatNBeta}, it follows that $\SHI^-(S^3(H),\sigma)$ has rank $2$. The exact triangle \[
		\begin{tikzcd}[column sep=-1em]
			\SHI^-(S^3(H),\sigma) \ar[rr,"U"] & & \SHI^-(S^3(H),\sigma) \ar[dl]\\
			& \KHI(H) \ar[ul] &
		\end{tikzcd}
	\]of Proposition~\ref{prop:instantonMinusUTriangle} implies that \[
		\SHI^-(S^3(H),\sigma) \cong \C[U]^2 \oplus \C.
	\]The torsion $\C$ of $\SHI^-(S^3(H),\sigma)$ contributes two copies of $\C$ to $\KHI(H)$ which lie in gradings differing by $[\sigma]$. The two free summands of $\SHI^-(S^3(H),\sigma)$ must therefore contribute the two other $\C$ summands of $\KHI(H)$ which lie gradings differing by $[\sigma]$ from our explicit description of $\KHI(H)$ so the generators of the free summands must lie in gradings differing by $[\sigma]$. 
\end{proof}

\theoremstyle{plain}
\newtheorem*{thm:suturedDetects}{Theorem~\ref{thm:suturedInstantonHomologyDetectsTrivialBand}}
\begin{thm:suturedDetects}
	Let $K_b$ be the band sum of a split two-component link along a band $b$, and let $K_\#$ be the connected sum. Then \[
		\dim \KHI(K_\#) = \dim \KHI(K_b)
	\]over $\C$ if and only if $b$ is trivial. 
\end{thm:suturedDetects}
\begin{proof}
	Assume that $b$ is nontrivial. By Corollaries \ref{cor:instantonSufficesToConsiderMinusLink} and \ref{cor:instantonDifferentGradingsGeneratorsSuffices}, it suffices to show that there are generators of free summands of $\SHI^-(S^3(K_b \cup C),\sigma)$ lying in gradings differing by $[\sigma]$. By Theorem~\ref{thm:SeqOfSurfDecomp}, there is a sequence of surface decompositions \[
		S^3(K_b \cup C) \overset{S_1}{\rightsquigarrow} (M_1,\gamma_1) \overset{S_2}{\rightsquigarrow} \cdots \overset{S_n}{\rightsquigarrow} (M_n,\gamma_n)
	\]for which each $S_i$ is nice, taut, and disjoint from $\partial N(C)$, and $(M_n,\gamma_n)$ is the union of a product sutured manifold and $S^3(H)$. By Theorem~\ref{thm:instantonMinusSurfDecomp}, it suffices to show that $\SHI^-(M_n,\gamma_n,\sigma)$ has generators of free summands lying in gradings differing by $[\sigma]$. The result now follows from Lemmas \ref{lem:instantonMinusDisjointUnion} and \ref{lem:instantonHopfLink}. 
\end{proof}

\subsection{Invariance under full twists of the band}\label{subsec:invarianceInstanton}

We begin by recalling the oriented skein exact triangle in sutured instanton homology. 

\begin{thm}[Theorem 3.1 and Lemmas 3.3, 3.4 of \cite{MR2683750}]\label{thm:suturedInstantonSkein}
	Let $K_+,K_-$, and $K_0$ be oriented links of an oriented skein relation. The three links agree except in a local region where $K_+,K_-$ have crossings of the specified sign and $K_0$ is the oriented resolution. If $K_0$ has one more component than $K_+$ and $K_-$, then there is an exact triangle \[
		\begin{tikzcd}[column sep=-1em]
			\KHI(K_+) \ar[rr] & & \KHI(K_-) \ar[dl]\\
			& \KHI(K_0) \ar[ul] &
		\end{tikzcd}
	\]Choose a Seifert surface for each of the three links. With respect to the $\Z$-gradings defined by Seifert surfaces, all three maps are grading-preserving. With respect to the canonical mod 2 gradings, the map $\KHI(K_-) \to \KHI(K_0)$ has odd degree while the other two maps have even degree. 
\end{thm}
\begin{rem}
	We may view $K_+$ and $K_-$ as obtained from $K_0$ by band surgery. Let $C$ be a linking circle for the band. The maps in the exact triangle arise as surgery along $C$, viewed as a knot in closures of the sutured exteriors of $K_+,K_-$, and $K_0$. 
\end{rem}

\begin{lem}\label{lem:suturedInstantonTrivialSkeinTri}
	Let $K_\#$ be the connected sum of a split two-component link $L$ in $S^3$. Then the map $\KHI(K_\#) \to \KHI(K_\#)$ in the skein exact triangle \[
		\begin{tikzcd}[column sep=-1em]
			\KHI(K_\#) \ar[rr,"0"] & & \KHI(K_\#) \ar[dl]\\
			& \KHI(L) \ar[ul] & 
		\end{tikzcd}
	\]is zero. 
\end{lem}
\begin{proof}
	By exactness of the triangle, it suffices to prove that $\dim \KHI(L) = 2\cdot \dim \KHI(K_\#)$. If $K_1,K_2$ are the two components of $L$, then there a product annulus decomposition $S^3(K_\#) \rightsquigarrow S^3(K_1) \amalg S^3(K_2)$. By Proposition 6.7 of \cite{MR2652464} and the connected sum formula (Lemma~\ref{lem:instantonConnectedSumFormula}), we have \[
		\dim \KHI(L) = 2 \cdot \dim \KHI(K_1) \cdot \dim \KHI(K_2) = 2 \cdot \dim \KHI(K_\#)
	\]as required. 
\end{proof}

\theoremstyle{plain}
\newtheorem*{thm:suturedInvariance}{Theorem~\ref{thm:SuturedInstantonInvariantUnderFullTwists}}
\begin{thm:suturedInvariance}
	Let $K_b$ be a band sum of a split two-component link, and let $K_{b+n}$ be obtained by adding $n$ full twists to the band. Then \[
		\KHI(K_b) \cong \KHI(K_{b+n})
	\]as vector spaces over $\C$ equipped with Alexander gradings and canonical mod $2$ gradings. 
\end{thm:suturedInvariance}
\begin{proof}
	It suffices to prove the result for $n=1$. Just like in Lemma~\ref{lem:HeegaardInjectBetweenTriangles} and Remark~\ref{rem:HeegaardSurjTri}, we have the following commutative diagram.\[
		\begin{tikzcd}[row sep=small,column sep=small]
			\KHI(K_\#) \ar[rr] & & \KHI(K_\#) \ar[dl]\\
			& \KHI(L) \ar[ul]\\
			\KHI(K_{b+1}) \ar[rr] \ar[uu,two heads] & & \KHI(K_b) \ar[uu,two heads] \ar[dl]\\
			& \KHI(L) \ar[lu] \ar[uu,two heads,crossing over] & \\
			\KHI(K_\#) \ar[uu,hook] \ar[rr] & & \KHI(K_\#) \ar[dl] \ar[uu,hook]\\
			& \KHI(L) \ar[uu,hook,crossing over] \ar[ul] &
		\end{tikzcd}
	\]The injective vertical maps are induced by compatible ribbon concordances while the surjective vertical maps are induced by their reverses. That these maps are injective and surjective as claimed follows from Theorem 6.7 of \cite{1904.09721}. The result now follows in the same way as Theorem~\ref{thm:HeegaardKnotFloerInvariantUnderFullTwists} using Lemma~\ref{lem:suturedInstantonTrivialSkeinTri} and the grading claims in Theorem~\ref{thm:suturedInstantonSkein}.
\end{proof}
\begin{rem}
	Although the Heegaard Floer group playing the role of $\KHI(L)$ above is $\HFKhat(K_\#,S^1 \x S^2)$ in Lemma~\ref{lem:HeegaardInjectBetweenTriangles}, there is a product annulus decomposition relating the two sutured exteriors. From the proof of Lemma~\ref{lem:surfDecompProdAnnulus}, there is a grading-preserving identification of the two sutured instanton Floer homology groups, and the maps in the exact triangle of Theorem~\ref{thm:suturedInstantonSkein} are defined through such an isomorphism. 
\end{rem}

\subsection{Singular instanton homology}\label{subsec:singularInstantonHomology}

Let $J$ be a link in $B^3 \subset S^3$. The (unreduced) \textit{singular instanton homology} $\Isharp(J)$ of $J$ is defined in \cite{MR2805599} in the following way. Associated to $J$ is an auxiliary $3$-dimensional orbifold $Y$ with $\Z/2$ singularities along $J \amalg H$ where $H$ is a fixed auxiliary two-component Hopf link in the complement of our fixed $B^3 \subset S^3$. A chain complex $(C(J),d)$ over $\Z$ is then defined by a version of Morse homology for a Chern-Simons functional defined on a space of orbifold connections. The complex depends on a number of auxiliary choices, but its homology $\Isharp(J)$ is a functorial invariant of $J$. 

There is also a reduced variant $\Inat(J)$ of singular instanton homology when $J$ is equipped with a basepoint $p$. We associate to $J$ an orbifold $Z$ with $\Z/2$ singularities along $J \cup \mu$ where $\mu$ is a meridian of the component of $J$ containing $p$. Again a chain complex over $\Z$ is defined depending on a number of auxiliary choices, and its homology $\Inat(J)$ is a functorial invariant of $J$ equipped with a basepoint. 

\begin{prop}[Proposition 1.4 of \cite{MR2805599}]\label{prop:suturedIsoSingular}
	Let $J$ be a knot in $S^3$. Then $\Inat(J)$ and $\KHI(J)$ are isomorphic over $\Q$. 
\end{prop}

\theoremstyle{plain}
\newtheorem*{cor:singular}{Corollary \ref{cor:singularInstantonHomologyDetectsTrivialBand}}
\begin{cor:singular}
	Let $K_b$ be the band sum of a two-component split link along a band $b$, and let $K_\#$ be the connected sum. Then \[
		\dim \Inat(K_b) = \dim \Inat(K_\#)
	\]over $\Q$ if and only if $b$ is trivial.
\end{cor:singular}
\begin{proof}
	The result follows from Theorem~\ref{thm:suturedInstantonHomologyDetectsTrivialBand} and Proposition~\ref{prop:suturedIsoSingular}. 
\end{proof}

\begin{cor}\label{cor:singularDetectsTrivialBandF2}
	Both $\dim \Inat$ and $\dim\Isharp$ over $\F_2$ detect the trivial band. 
\end{cor}
\begin{proof}
	A ribbon concordance $K_\# \to K_b$ induces a inclusion \[
		\Inat(K_\#;\Z) \hookrightarrow \Inat(K_b;\Z)
	\]onto a direct summand by Zemke's argument \cite{Zem19a} (see Corollary 4.5 of \cite{1904.09721} for example) so we may write \[
		\Inat(K_b;\Z) \cong \Inat(K_\#;\Z) \oplus D_b
	\]where $D_b$ is a finitely-generated abelian group. Since the dimension of $\Inat$ with coefficients in $\C$ detects the trivial band, we know that the rank of $D_b$ is at least one by the universal coefficients theorem. Thus $\Inat$ also detects the trivial band over $\F_2$. 

	By Lemma 7.7 of \cite{MR3880205}, we know that $\dim \Isharp(J) = 2\cdot \dim \Inat(J)$ over $\F_2$ for any knot $J$, so $\Isharp$ over $\F_2$ detects the trivial band. 
\end{proof}

%%%%%%%%%%%%%%%%%%%%%%%%%%%%%%%%%%%%%%%
%%%%%%%%%% Khovanov homology %%%%%%%%%%
%%%%%%%%%%%%%%%%%%%%%%%%%%%%%%%%%%%%%%%

\section{Khovanov homology}\label{sec:Khovanovhomology}

In section~\ref{subsec:KhovanovDetectingTrivBand}, we show that Khovanov homology detects the trivial band using functoriality of Kronheimer-Mrowka's spectral sequence from Khovanov homology to singular instanton homology \cite{MR2805599,MR3903915}. Using this detection result, we show that adding full twists to the band changes Khovanov homology in section~\ref{subsec:KhovanovFullTwists}. 

\subsection{Preliminaries}

We review the basics of Khovanov homology \cite{MR1740682,MR1917056}. After a note on cobordism maps and Levine-Zemke's result on ribbon concordances \cite{MR4041014}, we briefly recall Kronheimer-Mrowka's spectral sequence from Khovanov homology to singular instanton homology \cite{MR2805599}.

\begin{dfs*}[Khovanov homology]
	Let $D$ be a diagram of an oriented link $L$. There are two resolutions of each crossing, which are called the $0$- and $1$-resolutions shown in Figure~\ref{fig:resolutions}.

	\begin{figure}[!ht]
		\centering
		\labellist
		\pinlabel $0$ at 240 120
		\pinlabel $1$ at 510 120
		\endlabellist
		\includegraphics[width=.45\textwidth]{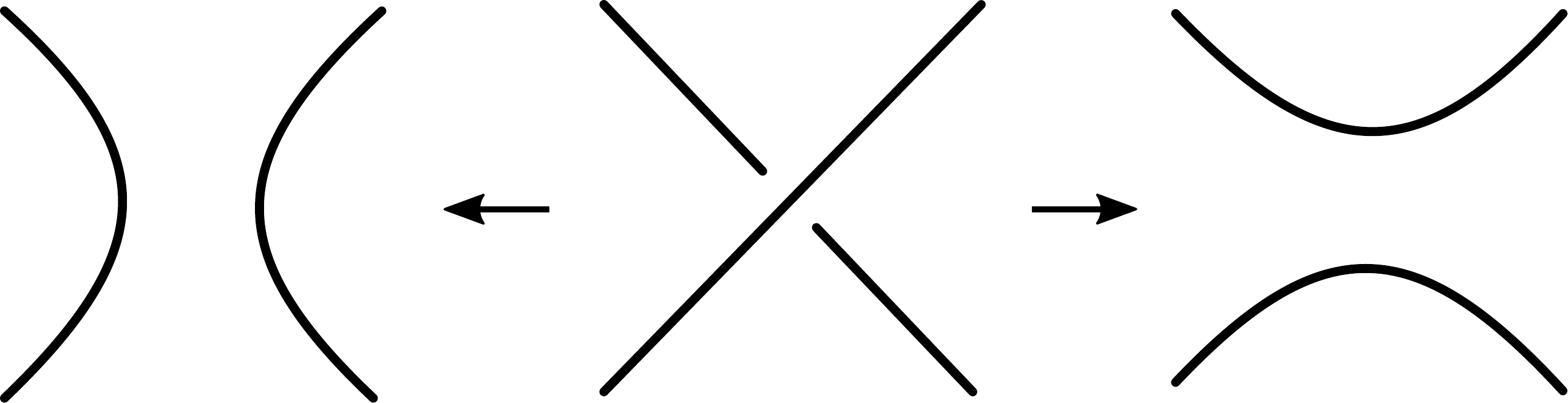}
		\caption{$0$- and $1$-resolutions of a crossing.}
		\label{fig:resolutions}
	\end{figure}

	\noindent A \textit{complete resolution} of $D$ is the result of resolving each crossing of $D$, and is determined by a function $v\colon c(D) \to \{0,1\}$ where $c(D)$ is the set of crossings of $D$. Let $D_v$ denote the diagram obtained by resolving $D$ via $v\colon c(D) \to \{0,1\}$. Associated to each complete resolution $v$ is the free abelian group $\KhC(D_v) = V^{\otimes \pi_0(D_v)}$ where $V = \Z[X]/X^2$. If $w\colon c(D) \to \{0,1\}$ agrees with $v$ except at a single crossing $c$ where $w(c) = 1$ but $v(c) = 0$, then there is a \textit{edge map} $\KhC(D_v) \to \KhC(D_w)$. The resolution $D_w$ is obtained from $D_v$ by either merging two components or by splitting a single component. Associated to a merge and a split are maps $m\colon V\otimes V \to V$ and $\Delta\colon V \to V \otimes V$, respectively, defined by \[
		m = \begin{cases}
			1 \otimes 1 \mapsto 1\\
			1 \otimes X \mapsto X\\
			X \otimes 1 \mapsto X\\
			X \otimes X \mapsto 0
		\end{cases} \qquad \Delta = \begin{cases}
			1 \mapsto 1 \otimes X + X \otimes 1\\
			X \mapsto X \otimes X.
		\end{cases}
	\]The edge map $\KhC(D_v) \to \KhC(D_w)$ is obtained by tensoring the merge or split map with the identity map on the other components. The \textit{Khovanov complex} $\KhC(D)$ is the direct sum of $\KhC(D_v)$ over all complete resolutions $v$, and the differential on $\KhC(D)$ is the sum of the edge maps with signs determined by an ordering of the crossings of $D$ (see \cite{MR1917056} for a precise formula). 

	The complex $\KhC(D)$ comes equipped with two gradings. First, there is a $q$-grading on $V = \Z[X]/X^2$ defined by $q(X) = -1$ and $q(1) = 1$. With respect to this grading, the merge and split maps drop $q$ by $1$. Each $\KhC(D_v)$ is equipped with the induced $q$-grading shifted by $|v| = \sum v(c)$ so that the differential is $q$-grading preserving. The $q$-grading on the total complex $\KhC(D)$ has a further global shift by $n_+(D) - 2n_-(D)$ where $n_\pm(D)$ is the number of crossings of $D$ of given sign. The other grading $h$ is defined by first placing $\KhC(D_v)$ in $h$-grading $|v| = \sum v(c)$ so that the differential increases $h$ by $1$. There is again a global shift on $\KhC(D)$ by $-n_-(D)$. The homology $\Kh(L)$ inherits these bigradings, and up to bigraded isomorphism, $\Kh(L)$ is an oriented link invariant. Khovanov homology is also invariant under simultaneously changing all of the orientations of the components of $L$. The invariant over a commutative ring $R$ is the homology of the bigraded chain complex $\KhC(D) \otimes_\Z R$.

	Given a basepoint $p$ on the diagram $D$ away from the crossings, each $\KhC(D_v)$ has an action of $V = \Z[X]/X^2$ given by the copy of $V$ in $\KhC(D_v)$ corresponding to the component of $D_v$ containing $p$. The differential preserves this action, and we let $X_p$ denote this action on $\KhC(D)$. The subcomplex $X_p\KhC(D)$ is the \textit{reduced Khovanov complex} of $D$, denoted $\KhrC(D)$, and its homology $\Khr(L)$, up to bigraded isomorphism, is an invariant of $L$ with the basepoint $p$. 
\end{dfs*}

There are unoriented skein exact triangles built into the theory. Let $D$ be a diagram for a link $L$ with a distinguished crossing $c$. Let $D_0$ and $D_1$ be the diagrams obtained from $D$ by taking the $0$- and $1$-resolutions, respectively, of $c$. Then after a grading shift, $\KhC(D_1)$ is a subcomplex of $\KhC(D)$, and $\KhC(D_0)$ is the quotient complex. The short exact sequence \[
	\begin{tikzcd}
		0 \ar[r] & \KhC(D_1) \ar[r] & \KhC(D) \ar[r] & \KhC(D_0) \ar[r] & 0
	\end{tikzcd}
\]induces an exact triangle \[
	\begin{tikzcd}[column sep=tiny]
		\Kh(D_1) \ar[rr] & & \Kh(D) \ar[dl]\\
		& \Kh(D_0). \ar[ul] &
	\end{tikzcd}
\]We state the precise grading shifts in the relevant special case. 

\begin{prop}[Proposition 4.2 of \cite{MR2189938}]\label{prop:KhovanovSkeinGradings}
	Let $D$ be a diagram for a knot $K_-$ with a distinguished negative crossing $c$. Let $K_+$ be the knot with a positive crossing at $c$, and let $U$ be knot obtained by the unoriented resolutions at $c$. Suppose that $L$, the two-component link obtained by the oriented resolution at $c$, has linking number $0$. Then the maps in the unoriented skein exact triangles are homogeneous in the following ways: \[
		\begin{tikzcd}[column sep=small]
			\Kh(K_-) \ar[rr,"{(1,2)}"] & & \Kh(U) \ar[rr,"{(1,2)}"] \ar[dl,"{(0,-1)}"] & & \Kh(K_+) \ar[dl,"{(0,-1)}"]\\
			& \Kh(L) \ar[ul,"{(0,-1)}"] & & \Kh(L) \ar[ul,"{(0,-1)}"] &
		\end{tikzcd}
	\]where an arrow labeled $(a,b)$ means that it increases $h$-grading by $a$ and $q$-grading by $b$. If the linking number of $L$ is not zero, then the maps are graded in this way after a global shift of gradings on $\Kh(U)$ based solely on $\ell k(L)$. The same exact triangles with the described grading shifts hold for reduced Khovanov homology.
\end{prop}

\vspace{20pt}

Khovanov homology is functorial with respect to oriented link cobordisms $\Sigma\colon L_0 \to L_1$ in $[0,1] \x S^3$ in the following way. First fix a \textit{planar movie presentation} for $\Sigma$. A planar movie presentation for $\Sigma$ starts with a diagram for $L_0$ and ends in a diagram for $L_1$ and consists of a finite sequence of the following elementary moves: planar isotopy, Reidemeister moves, planar birth, planar saddle, planar death. A planar birth or death is simply the appearance or disappearance of an unknot in the diagram which bounds a disc in the plane disjoint from the diagram. A planar saddle is just a merge or a split as considered before. Associated to each of these elementary moves is a chain map on Khovanov complexes. Up to sign, the map on Khovanov homology induced by the composite of these elementary moves is independent of the choice of planar movie presentation with fixed starting and ending diagrams \cite{MR2113903,MR2174270}. 

There are two features of the cobordism maps which we will need. The first follows from the construction of the maps.

\begin{rem}\label{rem:KhovanovSkeinTriangleCommutesWithCobordismMaps}
	Let $D'$ be a diagram obtained from a diagram $D$ by one of the elementary moves. Suppose that $c$ is a crossing of $D$ for which the elementary move is supported away from a neighborhood of $c$. Then the elementary move can be thought as elementary moves $D_0 \to D_0'$ and $D_1 \to D_1'$ where $D_0,D_1,D_0',D_1'$ are obtained from $D,D'$ by resolutions at $c$. Then the chain maps associated to the elementary move yield a map of short exact sequences \[
		\begin{tikzcd}
			0 \ar[r] & \KhC(D_1') \ar[r] & \KhC(D') \ar[r] & \KhC(D_0') \ar[r] & 0\\
			0 \ar[r] & \KhC(D_1) \ar[r] \ar[u] & \KhC(D) \ar[r] \ar[u] & \KhC(D_0) \ar[r] \ar[u] & 0
		\end{tikzcd}
	\]and therefore a map of exact triangles \[
		\begin{tikzcd}[column sep=small,row sep=small]
			\Kh(D_1') \ar[rr] & & \Kh(D') \ar[dl]\\
			& \Kh(D_0') \ar[ul] &\\
			\Kh(D_1) \ar[rr] \ar[uu] & & \Kh(D) \ar[dl] \ar[uu]\\
			& \Kh(D_0). \ar[ul] \ar[uu,crossing over]
		\end{tikzcd}
	\]The same is true for reduced Khovanov homology. 
\end{rem}

The second feature we need is the behavior under ribbon concordances. 

\begin{thm}[Theorem 1 of \cite{MR4041014}]
	If $C$ is a ribbon concordance from a link $L_0$ to a link $L_1$, the induced grading-preserving map $\Kh(C)\colon \Kh(L_0) \to \Kh(L_1)$ is injective, with left inverse given by $\Kh(\bar{C})$. In particular, $\Kh(L_0)$ is a direct summand of $\Kh(L_1)$ in each bigrading. 
\end{thm}

\noindent The result holds for coefficients in any ring, and the same argument works for reduced Khovanov homology.

\vspace{20pt}

Finally, we recall Kronheimer-Mrowka's spectral sequence from Khovanov homology to singular instanton homology. 

\begin{thm}[Proposition 1.2 and Theorem 8.2 of \cite{MR2805599}, also see Theorem 1.1 of \cite{MR3176310}]\label{thm:dsharp}
	Let $D$ be a diagram of an oriented link $J$ in $S^3$, and let $m(D)$ be the mirror of $D$ representing $m(J)$.
	There is a filtered chain complex $(\bo{C},\bo{F})$, depending on certain auxiliary choices, whose homology is $\Isharp(J)$. The $E_1$-page of the associated spectral sequence is the $h$-graded Khovanov chain complex $\KhC(m(D))$. 

	Similarly, if $D$ is given a basepoint, there is a spectral sequence arising from a filtered complex whose $E_2$-page is the $h$-graded reduced Khovanov homology $\Khr(m(J))$ of the mirror of $J$ and which abuts to $\Inat(J)$. 
\end{thm}

\subsection{Detecting the trivial band}\label{subsec:KhovanovDetectingTrivBand}

\theoremstyle{plain}
\newtheorem*{thm:KhovanovDetects}{Theorem~\ref{thm:KhovanovHomologyDetectsTrivialBand}}
\begin{thm:KhovanovDetects}
	Let $K_b$ be the band sum of a split two-component link along a band $b$, and let $K_\#$ be the connected sum. Then \[
		\dim \Kh(K_\#) = \dim \Kh(K_b)
	\]over $\F_2$ if and only if $b$ is trivial. 
\end{thm:KhovanovDetects}
\begin{proof}
	Let $R$ be a ribbon concordance from $K_\#$ to $K_b$. By Theorem 1.3 of \cite{MR3903915}, there are filtered complexes $(\bo{C}(K_\#),d^\sharp(K_\#)), (\bo{C}(K_b),d^\sharp(K_b))$ as in Theorem~\ref{thm:dsharp} defining spectral sequences from the Khovanov homology of $K_\#,K_b$ to the singular instanton homology of $m(K_\#),m(K_b)$, respectively, along with a filtered chain map \[
		f\colon \bo{C}(K_\#) \to \bo{C}(K_b)
	\]whose induced map $f_2$ on $E_2$-pages of the spectral sequence is the induced map on Khovanov homology $\Kh(R)\colon \Kh(K_\#) \to \Kh(K_b)$. If $\dim \Kh(K_b) = \dim \Kh(K_\#)$, then the map $\Kh(R)$ must be an isomorphism since $\Kh(R)$ is injective by Theorem 1 of \cite{MR4041014}. Since an isomorphism of chain complexes induces an isomorphism on homology, the induced map $f_r$ on $E_r$-pages is an isomorphism for $r \ge 2$ and $r = \infty$. Since $f_\infty$ is an isomorphism between the associated graded vector spaces of $\Isharp(m(K_\#))$ and $\Isharp(m(K_b))$, it follows that $\dim \Isharp(m(K_\#)) = \dim \Isharp(m(K_b))$. Since $\dim \Isharp$ over $\F_2$ detects the trivial band (Corollary~\ref{cor:singularDetectsTrivialBandF2}) we conclude that $b$ is trivial. 
\end{proof}

\begin{rem}\label{rem:reducedKhovanovDetectsToo}
	Since $\dim \Inat$ over $\F_2$ and $\dim \Inat$ over $\Q$ detect the trivial band as well (Corollaries~\ref{cor:singularInstantonHomologyDetectsTrivialBand} and \ref{cor:singularDetectsTrivialBandF2}), the same argument shows that reduced Khovanov homology over $\F_2$ and $\Q$ also detects the trivial band. 
\end{rem}

\subsection{Shifting under full twists of the band}\label{subsec:KhovanovFullTwists}

Let $L$ be a split two-component link, and let $K_b$ the band sum of $L$ along a band $b$. As before, $K_{b+1},K_{b},L$ form an oriented skein triple. Let $K_{b+1/2}$ denote the (arbitrarily oriented) knot obtained by taking the unoriented resolution at this crossing. We have the following two unoriented skein exact triangles over $\F_2$ by Proposition~\ref{prop:KhovanovSkeinGradings}. \[
	\begin{tikzcd}[column sep=tiny]
		\Kh(K_b) \ar[rr] & & \Kh(K_{b + 1/2}) \ar[rr] \ar[dl] & & \Kh(K_{b + 1}) \ar[dl]\\
		& \Kh(L) \ar[ul] & & \Kh(L) \ar[ul] &
	\end{tikzcd}
\]When the band is trivial, $K_b$ is the connected sum $K_\#$ of $L$, and $K_b = K_\# = K_{b+1}$ as oriented knots. The knot $K_{b+1/2}$ is the connected sum of the split link $L$ where one of the components of $L$ is given the opposite orientation. Let $K_\#'$ denote this knot with an arbitrary orientation. There is an isomorphism $\Kh(K_\#') \cong \Kh(K_\#)$ respecting the bigradings (Proposition 23 of \cite{MR1740682}).

\begin{lem}
	Suppose the band $b$ is trivial. Then the maps $\Kh(K_\#) \to \Kh(K_\#')\to \Kh(K_\#)$ in the two unoriented skein triangles \[
		\begin{tikzcd}[column sep=tiny]
			\Kh(K_\#) \ar[rr,"0"] & & \Kh(K_{\#}') \ar[rr,"0"] \ar[dl] & & \Kh(K_\#) \ar[dl]\\
		& \Kh(L) \ar[ul] & & \Kh(L) \ar[ul] &
		\end{tikzcd}
	\]are zero over $\F_2$. 
\end{lem}
\begin{proof}
	It suffices to show that $\dim \Kh(L) = 2\dim \Kh(K_\#)$. Let $K_1,K_2$ be the components of $L$. 
	Since $\dim \Kh(J) = 2\dim \Khr(J)$ over $\F_2$ for any link $J$ \cite{MR3071132}, we have that \begin{align*}
		2\dim \Kh(K_\#) &= 4\dim \Khr(K_\#)\\
		&= 4\dim \Khr(K_1)\cdot\dim \Khr(K_2)\\
		&= \dim \Kh(K_1) \cdot \dim \Kh(K_2) = \dim \Kh(L).\qedhere
	\end{align*}
\end{proof}

\begin{prop}\label{prop:KhovanovSkeinTriangleRibbonMaps}
	There is a commutative diagram of unoriented skein exact triangles \[
		\begin{tikzcd}[column sep=small,row sep=small]
			\Kh(K_\#) \ar[rr] & & \Kh(K_\#') \ar[rr] \ar[dl] & & \Kh(K_\#) \ar[dl]\\
			& \Kh(L) \ar[ul] & & \Kh(L) \ar[ul] &\\
			\Kh(K_b) \ar[rr] \ar[uu,crossing over,two heads] & & \Kh(K_{b+1/2}) \ar[rr] \ar[dl] \ar[uu,crossing over,two heads] & & \Kh(K_{b+1}) \ar[dl] \ar[uu,crossing over,two heads]\\
			& \Kh(L) \ar[ul] \ar[uu,crossing over,two heads] & & \Kh(L) \ar[ul] \ar[uu,crossing over,two heads] &\\
			\Kh(K_\#) \ar[rr] \ar[uu,crossing over,hook] & & \Kh(K_\#') \ar[rr] \ar[dl] \ar[uu,crossing over,hook] & & \Kh(K_\#) \ar[dl] \ar[uu,crossing over,hook]\\
			& \Kh(L) \ar[ul] \ar[uu,crossing over,hook] & & \Kh(L) \ar[ul] \ar[uu,crossing over,hook] &
		\end{tikzcd}
	\]with coefficients in $\F_2$ where the vertical arrows are induced by ribbon concordances and their reverses. 
\end{prop}
\begin{proof}
	Fix a planar movie presentation of a ribbon concordance $K_\# \to K_b$ which is completely supported away from a negative crossing $c$ of a diagram for $K_\#$ which defines the skein exact triangles involving $K_\#,K_\#'$, and $L$. The induced map on Khovanov homology is the map on the left $\Kh(K_\#) \to \Kh(K_b)$ while its reverse is the map above it $\Kh(K_b) \to \Kh(K_\#)$. Changing the crossing $c$ to either the positive crossing, the oriented resolution, or the unoriented resolution for the duration of the entire movie yields the ribbon concordances defining the other four sets of vertical maps. The fact that these cobordism maps commute with the maps in the exact triangles follows from Remark~\ref{rem:KhovanovSkeinTriangleCommutesWithCobordismMaps}.
\end{proof}

\theoremstyle{plain}
\newtheorem*{thm:KhovanovShifts}{Theorem~\ref{thm:KhovanovHomologyChangesUnderFullTwists}}
\begin{thm:KhovanovShifts}
	Let $K_b$ be the band sum of a split two-component link along a nontrivial band $b$, and let $K_{b+n/2}$ be obtained by adding $n$ half twists to the band. Then there is a nonzero finite-dimensional bigraded vector space $H_b$ for which \[
		\Kh(K_{b+n/2}) \cong \Kh(K_\#) \oplus h^{n}q^{2n}\: H_b
	\]as bigraded vector spaces over $\F_2$. In particular, the bigraded vector spaces $\Kh(K_{b+n/2})$ over $\F_2$ for $n \in \Z$ are all distinct. 
\end{thm:KhovanovShifts}
\begin{proof}
	Consider the commutative diagram of Proposition~\ref{prop:KhovanovSkeinTriangleRibbonMaps} induced by a ribbon concordance $K_\# \to K_b$. Since $\Kh(K_\#) \hookrightarrow \Kh(K_b)$ is an inclusion onto a direct summand by Theorem 1 of \cite{MR4041014}, there is an isomorphism of bigraded vector spaces \[
		\Kh(K_b) \cong \Kh(K_\#) \oplus H_b
	\]where $H_b$ is a bigraded vector space of nonzero dimension by Theorem~\ref{thm:KhovanovHomologyDetectsTrivialBand} and the assumption that $b$ is nontrivial. Similarly, we may write \[
		\Kh(K_{b+1/2}) \cong \Kh(K_\#') \oplus H_{b+1/2} \qquad \Kh(K_{b+1}) \cong \Kh(K_\#) \oplus H_{b+1}
	\]where $H_{b+1/2}$ and $H_{b+1}$ are of nonzero dimension. Commutativity of the diagram implies that the maps \[
		\begin{tikzcd}
			\Kh(K_b) \ar[r] & \Kh(K_{b+1/2}) \ar[r] & \Kh(K_{b+1})
		\end{tikzcd}
	\]restrict to isomorphisms $H_b \to H_{b+1/2} \to H_{b+1}$. These isomorphisms are homogeneous of bidegree $(1,2)$ by Proposition~\ref{prop:KhovanovSkeinGradings}. Thus $H_{b+1/2} \cong hq^2\:H_b$ and $H_{b+1} \cong h^2q^4\: H_b$. Since they are finite-dimensional and of nonzero dimension, it follows that they are not isomorphic as bigraded vector spaces. Thus $\Kh(K_b)$, $\Kh(K_{b+1/2})$ and $\Kh(K_{b+1})$ are all distinct. The same argument iterated finishes the result. 
\end{proof}
\begin{rem}
	The argument also works for reduced Khovanov homology over $\F_2$ and $\Q$ by Remark~\ref{rem:reducedKhovanovDetectsToo}.
\end{rem}

Since the bigraded isomorphism type of $\Kh(K)$ over $\F_2$ is an invariant of the unoriented knot type of $K$, we obtain a proof of the main result.

\theoremstyle{plain}
\newtheorem*{thm:nontrivTwists}{Theorem~\ref{thm:nonTrivBandFullTwistsDifferent}}

\begin{thm:nontrivTwists}
	Let $K_b$ be the band sum of a split two-component link along a nontrivial band $b$, and let $K_{b+n/2}$ be obtained from $K_b$ by adding $n$ half twists to the band. Then the unoriented knots $K_{b+n/2}$ for $n \in \Z$ are all distinct. 
\end{thm:nontrivTwists}

\phantomsection
\addcontentsline{toc}{section}{\protect\numberline{}References}
\raggedright
\bibliography{bandsums}
\bibliographystyle{alpha}

\end{document}